\documentclass[a4paper, 10pt]{amsart}
\usepackage[english]{babel}
\usepackage{cite}
\usepackage{amsfonts}
\usepackage{amssymb}
\usepackage{latexsym}
\usepackage{booktabs}
\usepackage{amsthm}
\usepackage{amsopn}
\usepackage{amsmath}
\usepackage[all]{xy}
\usepackage{verbatim}
\usepackage{amssymb}
\usepackage{mathtools}
\usepackage{float}
\usepackage[active]{srcltx}
\usepackage{fullpage}
\usepackage{hyperref}
\usepackage{manyfoot}
\usepackage{amscd}
\usepackage[dvips]{graphics}
\usepackage[dvips]{graphicx}
\usepackage{amscd}
\usepackage{epsfig,color,xcolor}
\usepackage{imakeidx}
\usepackage{pgfplots}
\pgfplotsset{compat=1.15}
\usepackage{mathrsfs}
\usepackage{tikz}
\usepackage{diagbox}
\usepackage{tikz-cd}
\usepackage{geometry}
\usetikzlibrary{arrows}
\usepackage{makecell}
\usetikzlibrary{intersections}

\usepackage{subcaption}

\usetikzlibrary{arrows.meta, positioning, calc}

\usepackage{float}

\usetikzlibrary{arrows.meta}
\usepackage{xcolor}
\usepackage[colorinlistoftodos, textsize=small]{todonotes}

\def\arxiv#1{%
	\href{http://arxiv.org/abs/#1}{arXiv:#1}}%
\def\doi#1{%
	\href{https://doi.org/#1}{doi:#1}}%
    
\theoremstyle{remark}
\newtheorem{remark}{Remark}

\theoremstyle{rmk}

\newtheorem{theorem}{Theorem}[section]
\newtheorem{lemma}[theorem]{Lemma}
\newtheorem{prop}[theorem]{Proposition}
\newtheorem{proposition}[theorem]{Proposition}
\newtheorem{corollary}[theorem]{Corollary}

\newtheorem{definition}[theorem]{Definition}
\newtheorem{example}[theorem]{Example}

\newtheorem*{Question}{Question}

\newcommand{\PP}{\mathbb{P}}

\newcommand{\cO}{\mathcal{O}}

\newcommand{\rk}{\mathrm{rk}}

\newcommand{\NN}{\mathbb{N}}
\newcommand{\ZZ}{\mathbb{Z}}

\newcommand{\QQ}{\mathbb{Q}}
\newcommand{\CC}{\mathbb{C}}

\newcommand{\lin}{\textrm{lin}}
\newcommand{\chig}{\langle \chi, g\rangle} 
\newcommand{\etag}{\langle \eta, g\rangle} 
\newcommand{\ejg}{\langle \epsilon_j, g\rangle} 
\newcommand{\eig}{\langle \epsilon_i, g\rangle} 
	
\def\Spec{\textnormal{Spec}}
\def\NS{\textnormal{NS}}
\def\Mot{\textnormal{Mot}}
\def\Vect{\textnormal{Vect}}
\def\AbMot{\textnormal{AbMot}}
\def\HH{\mathrm{H}_{\textnormal{mot}}}
\def\ChMot{\mathcal{M}_{\textnormal{rat}}}
\def\SmPr{\textnormal{SmPr}}
\def\opp{\textnormal{op}}

\def\AJ{\textnormal{AJ}}
\def\Spec{\textnormal{Spec}}
\def\Mot{\textnormal{Mot}}
\def\AbMot{\textnormal{AbMot}}
\def\HH{\mathrm{H}_{\textnormal{mot}}}

\def\tra{\textnormal{tra}}
\def\alg{\textnormal{alg}}

\renewcommand{\to}{\longrightarrow}

\makeatletter
\def\blfootnote{\xdef\@thefnmark{}\@footnotetext}

\title{The classification of SNC $(p,k)$-planes with Hodge structures of K3 type}

\author{Federico Fallucca, Matteo Penegini and  Anna Ulivi}

\date{}

\begin{document}

\subjclass[2010]{14E20, 14J28, 14J29, 14C30}

\keywords{Plane coverings, K3 surfaces, surfaces of general type, Hodge structures, Mumford--Tate conjecture} 

\maketitle 

\begin{abstract}
A SNC $(p,k)$-plane is a $(\mathbb Z/p)^k$-Galois cover $\pi\colon X \to \PP^2$ with a simple normal crossing branch locus. In this setting, there exist several intermediate quotients $Y_1, Y_2,\ldots, Y_{\phi_{p,k}}$ corresponding to subgroups of $(\mathbb Z/p)^k$ of index $p$. We focus on the case where each $Y_i$ is either a surface with geometric genus $0$ or $1$, with at least one with geometric genus $1$.  We classify all SNC $(p,k)$-planes satisfying these conditions obtaining that an extensive list of cases are minimal regular surfaces of general type with positive geometric genus. Moreover, these assumptions allow us to have strong control over the weight-$2$ Hodge structure of the minimal model of the covering surface $X$. Therefore, we also prove that Tate and Mumford–Tate conjectures hold for such surfaces.
\\
As an independent result of our methods, we obtain a classification of SNC $(p,1)$-planes with geometric genus $1$, which turn out to be precisely all K3 surfaces, and a classification of SNC $(p,1)$-planes with geometric genus $0$, which turn out to be instead rational surfaces.
\end{abstract}


\tableofcontents
\newpage 


\section{Introduction}\label{sec:intro}
The theory of abelian covers was developed over the last decades of the twentieth century by various authors and was ultimately codified in the seminal paper \cite{Par91a}. This theory has proven to be a powerful tool for constructing varieties with rich and intriguing geometric properties. In this work we shall exploit the theory of abelian covers to study the weight--2 Hodge structure of surfaces of general type.  

Our approach is inspired by ideas of Morrison~\cite{Mor87}, who showed that to any surface \( S \) with \( p_g(S) = 1 \) one can associate a K3 surface whose transcendental second cohomology is isomorphic to that of \( S \). In general, however, this isomorphism does not arise from an actual morphism between the surface and its associated K3 surface.

In this paper, we provide explicit examples of surfaces with \( p_g \ge 1 \) whose weight--2 Hodge structures are completely encoded by the Hodge structures of certain K3 surfaces. Moreover, in all our examples these K3 surfaces arise as geometric quotients of the original surfaces.

To be more precise, we consider the following situation, which combines previous work of the first author \cite{FP24} and of the second author \cite{GP22}. A $(p,k)$-plane is a finite morphism $\pi\colon X\to \mathbb P^2$ where $X$ is a normal surface and with the property that there is a subgroup $(\mathbb Z/p)^k$ of $\text{Aut}(X)$ such that $\pi$ factors as the composition of the quotient map $X\to X/(\mathbb Z/p)^k$ and an isomorphism $X/(\mathbb Z/p)^k\cong \mathbb P^2$. Assume that the branch divisor of $\pi$ is simple normal crossing (SNC for short). 
This cover has several intermediate quotients $Y_{H}\coloneqq X/H$ corresponding to the subgroups $H$ of \( (\mathbb Z /p)^k \). In particular, we focus on the \( \phi_{p,k}\coloneqq \frac{p^k - 1}{p - 1}\) subgroups \( H_i \leq (\mathbb Z/p)^k \) of index \(p\), and on the associated intermediate covers \( Y_{H_i} \to \mathbb{P}^2 \) with Galois group $G/H_i\cong \mathbb Z/p$. Hence, from any $(p,k)$-plane we obtain a diagram of the following type.
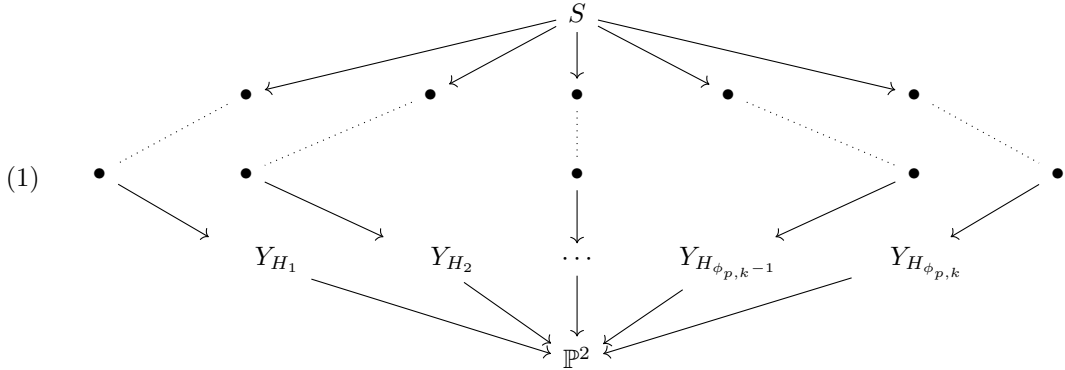
\begin{equation}\label{diag_HesseRev}\begin{tikzcd}
	&&& S \\
	& \bullet & \bullet & \bullet & \bullet & \bullet \\
	\bullet & \bullet && \bullet && \bullet & \bullet \\
	& {\ \ \ \ \ \ \ Y_{H_1}} & {\ \ \ \ \ Y_{H_2}} & \cdots & {Y_{H_{\phi_{p,k}-1}}} & {\ \ \ Y_{H_{\phi_{p,k}}}} \\
	&&& {\mathbb P^2}
	\arrow[from=1-4, to=2-2]
	\arrow[from=1-4, to=2-3]
	\arrow[from=1-4, to=2-4]
	\arrow[from=1-4, to=2-5]
	\arrow[from=1-4, to=2-6]
	\arrow[dotted, no head, from=2-2, to=3-1]
	\arrow[dotted, no head, from=2-3, to=3-2]
	\arrow[dotted, no head, from=2-4, to=3-4]
	\arrow[dotted, no head, from=2-5, to=3-6]
	\arrow[dotted, no head, from=2-6, to=3-7]
	\arrow[from=3-1, to=4-2]
	\arrow[from=3-2, to=4-3]
	\arrow[from=3-4, to=4-4]
	\arrow[from=3-6, to=4-5]
	\arrow[from=3-7, to=4-6]
	\arrow[from=4-2, to=5-4]
	\arrow[from=4-3, to=5-4]
	\arrow[from=4-4, to=5-4]
	\arrow[from=4-5, to=5-4]
	\arrow[from=4-6, to=5-4]
\end{tikzcd}\end{equation}

We focus our attention on SNC $(p,k)$-planes with the following
\begin{description}
	\item[\hypertarget{Property $\bigstar$}{Property $\bigstar$}] Each \( Y_{H_i} \) is a surface with geometric genus either $0$ or $1$, and at least one $Y_{H_i}$ has geometric genus $1$.
\end{description}
We shall prove in Theorems \ref{thm: main_thm_(p,1)_planes_pg=0} and \ref{thm: pg1_Zp_cover_structure} that the intermediate quotients  \( Y_{H_i} \)  with geometric genus zero are rational surfaces, while those with geometric genus one are $K3$ surfaces. In general, the surface \( Y_{H_i} \) need not be smooth; nevertheless, with a slight abuse of terminology, we shall refer to \( Y_{H_i} \) as a \emph{K3 surface} whenever its minimal model is a K3 surface.

This situation has already been considered in the case \((p,k)=(2,2)\) by several authors, e.g., \cite{CatM, Cat99, G, GP22, Lat19, Lat21, Lat22, PZ19}, who investigated various specific properties under the assumption that the cover is smooth. Indeed, in this case, the transcendental cohomology of the cover splits into transcendental parts of K3 type, making the study of some very deep conjectures accessible, such as the Tate and Mumford-Tate conjectures, and the Voisin conjecture together with the more general Bloch-Beilinson conjecture. Furthermore, Laterveer asked in \cite{Lat19} for the existence of other examples with the same property and high geometric genus $\geq 4$:
\begin{Question}\cite[Question 7.7]{Lat19}
		Do there exist surfaces \(S\) of general type with \(p_g(S)=4\) such that the transcendental cohomology of \(S\) splits into 4 pieces of K3 type? And what about \(p_g >4\)? \end{Question} 
In this paper, we provide a positive and exhaustive answer to this question. Moreover, we present a systematic study of the above-mentioned construction for primes \(p \ge 2\) and integers \(k \ge 1\), allowing also for $(p,k)$-planes  with singularities. More precisely, we give a complete description of SNC \((p,k)\)-planes satisfying \hyperlink{Property $\bigstar$}{Property $\bigstar$}: we not only classify all such coverings in families for all possible primes \(p\) and all positive integers \(k\), but also determine their invariants and explicit defining equations for every resulting surface \(X\) in suitable weighted projective spaces. 
Furthermore, we study the weight-2 Hodge structure of the minimal model of $X$. This cohomology group decomposes into Hodge substructures, and we show that the rational transcendental part is completely determined by the cohomology of those intermediate quotients \( Y_{H_i} \) which are K3 surfaces.  Thus, the extensive list of surfaces of general type with \(p_g\geq 4\) (see Tables \ref{tab: inv_(p,k)_planes_1} and \ref{tab: inv_(p,k)_planes_2})  provide new relevant examples for the question posed by Laterveer.  

 This remarkable property allows us also to draw deeper conclusions. In particular, we prove the Mumford--Tate conjecture, and consequently the Tate conjecture, for all the surfaces constructed, following an argument similar to that in \cite{Com19, CP, GP22}.
 
We can summarize our main results in the following theorems. 
\begin{theorem}\label{thm main1}  Let $p,k \in \NN$ with $p$ a prime number and $k\geq 2$. Moreover, let  $\pi\colon X\to \PP^2$ be a SNC $(p,k)$-plane. 
Let $\rho \colon \widetilde{X}\to X$ be the minimal resolution of singularities of $X$, and $\widetilde{X}_{min}$ its minimal model.  

Let us assume that each intermediate $(p,1)$-plane $Y_{H_i}$ of $\pi$ is either a surface with geometric genus $0$ or $1$, and at least one of them has geometric genus $1$.

Then $\pi$ falls in one of the cases listed in tables \ref{tab: (2,2)} from \ref{tab: (7,2)}. More precisely:
\begin{enumerate}
\item 
The only admissible pairs $(p,k)$ are the following:
\begin{table}[h]
    \centering
    \begin{tabular}{|c||c|c|c|c|c|}
    \hline
    \diagbox{p}{k} & 2 &  3 & 4 & 5 & 6\\
    \hline\hline
    2 & \hyperlink{tab: (2,2)}{48}&  \hyperlink{tab: (2,3)}{63}  & \hyperlink{tab: (2,4)}{42} & \hyperlink{tab: (2,5)}{14}  & \hyperlink{tab: (2,6)}{3} \\
    \hline
    3 & \hyperlink{tab: (3,2)}{13}  & \hyperlink{tab: (3,3)}{5}  & \hyperlink{tab: (3,4)}{1} & --& -- \\
    \hline
    5 & \hyperlink{tab: (5,2)}{3} & \hyperlink{tab: (5,3)}{1} & -- & -- & --\\
    \hline
    7 & \hyperlink{tab: (7,2)}{2}  & -- & -- & -- & -- \\
    \hline
    \end{tabular}
    \caption{Number of families of SNC $(p,k)$-planes.}
    \label{tab:all-family}
\end{table}
\item The surface $\widetilde{X}$ is regular. Moreover, $\widetilde{X}$ is a minimal surface except for the cases $C-(3,2)$, $A-(5,2)$, and $B-(7,2)$, whose minimal models are obtained by contracting further a single $(-1)$-curve.
    The surface $\widetilde{X}$ is a surface of general type, except in the following cases:
    \begin{itemize}
        \item $\widetilde{X}_{min}$ is a K3 surface in cases: $A-(2,2)$, $A-(2,3)$, $B-(2,3)$, $A-(2,4)$, $A-(2,5)$, $C-(3,2)$, $A-(5,2)$;
        \item $\widetilde{X}$ is a minimal properly elliptic surface (with Kodaira dimension one) in cases $B-(3,2)$ and $C-(3,3)$.  
    \end{itemize}
    In tables \ref{tab: inv_(p,k)_planes_1} and \ref{tab: inv_(p,k)_planes_2}, we list for each case the self-intersection of a canonical divisor of $\widetilde{X}_{min}$, along with its geometric genus, topological Euler characteristic, and surface type. We also list  in tables \ref{tab: (2,2)} to \ref{tab: (7,2)} the Picard number together with the Hodge numbers of the transcendental part of the Hodge structure of a very general member of the family.
\item $\widetilde{X}_{min}$ has the following decomposition of its  transcendental lattice:
			\[ T_{\widetilde{X}_{min}}\otimes \QQ\simeq (T_{Y_{H_i}}\oplus \ldots \oplus T_{Y_{H_{\phi_{p,k}}}})\otimes \QQ, 
			\]
 where all the transcendental Hodge structures of the surfaces $Y_{H_i}$ are either trivial or the ones of a K3 surface.
 \item Finally, $\widetilde{X}_{min}$ satisfies the Mumford--Tate and Tate conjectures.
\end{enumerate}
\end{theorem}

As a byproduct of our work, we also obtain a classification of all SNC $(p,1)$-planes \( \pi\colon X\to \mathbb{P}^2 \) 
with \( p_g(X)=1 \). Note that these $(p,k)$-planes are exactly those satisfying \hypertarget{Property $\bigstar$}{Property $\bigstar$} for $k=1$; hence it provides together with  Theorem \ref{thm main1} an exhaustive conclusion of the classification.
\\
 We remark once more that \( X \) is not smooth in general; nevertheless, we analyze its minimal resolution of singularities and minimal model, and in doing so we establish the following theorem. This result is of independent interest and may well be known, although we have been unable to locate it in the existing literature.
\begin{theorem}\label{thm main2}  Let \( p \geq 2\) be a prime number and let \( \pi\colon X \to \mathbb{P}^2 \) be a SNC $(p,1)$-plane. 
	Let \( \widetilde{X}_{min} \) denote the minimal model of \( X \), and assume that \( p_g(\widetilde{X}_{min}) = 1 \). Then:
\begin{enumerate}
\item \( p \in \{2,3,5,7,11,13, 17, 19\} \);
\item \( \widetilde{X}_{min}\) is a K3 surface; 
\item for a very general member of the family holds $\rho(X)=1$;
\end{enumerate}
Furthermore, $\pi$ falls in one of the cases listed in tables \ref{tab: (2,1)} to \ref{tab: (19,1)}, where we also provide the Picard number $\rho(\widetilde{X}_{min})$ together with the Hodge numbers of the transcendental part of the Hodge structure of a very general member of the family.
\end{theorem}
A second byproduct of our work consists in the classification of SNC $(p,1)$-planes with geometric genus $0$, which turns out to be rational surfaces. 
\begin{theorem}\label{thm: intro_pg_zero_classification}
	Let $p\geq 2$ be a prime number and let $\pi\colon X\to \mathbb P^2$ be a SNC $(p,1)$-plane with $p_g(X)=0$. Then $X$ is a rational surface and it falls in one of the cases listed in tables \ref{tab: (p,1)_pg=0_p=2} to \ref{tab: (p,1)_pg=0_p>=7_data}. 
\end{theorem}
As the reader will see, the amount of computation required in the classification process is far too large to be carried out by hand. For this reason, we rely on several computer algebra scripts, which can be found at the webpage 
\begin{center}
	\href{https://github.com/Fefe9696/MAGMA_scripts_results_Classification_of_SNC_p_k_planes_with_Hodge_substructures_of_K3_type}{\texttt{MAGMA} scripts and results repository}
\end{center}
We emphasize that the group-theoretic computations were performed using \verb|MAGMA| \cite{BCP97}. 

In the same webpage one can find the outputs of the classification algorithm together with a commented version which studies the intermediate quotients of each SNC $(p,k)$-plane turned out. As a consequence of this study, we also highlight the following remark.
\begin{remark}
	There are two significant intermediate $(p,k)$-planes:
	\begin{itemize} 
		\item a $(2,3)$-plane 
		branched over seven lines in general position, which is a Campdelli surface with maximal Picard rank $8$, see \cite[Sec. 3]{AlPa24}. This appears as one of the intermediate quotients of the surfaces belonging to the families $I-(2,4), O-(2,4), C-(2,5), F-(2,5), I-(2,5), J-(2,5), A-(2,6), B-(2,6)$.  It has been already met in \cite[Subsec. 5.8]{FP24} and was first constructed in this way by Persson in \cite[Ex. 5.8]{Ulf}. We also mention \cite[Thm. 1.1(1)]{DuGao} where the authors give it through equations of a different (not normal) birational model.
		\item 
		a $(5,2)$-plane 
		branched over four lines in general position, which is the original Godeaux surface with maximal Picard rank $9$ constructed by Godeaux. Indeed, it can be also viewed as the \'etale cyclic quotient of $A-(5,3)$, which is the Fermat quintic $x_0^5+x_1^5+x_2^5+x_3^5=0$  in $\mathbb P^3$, where $(\mathbb Z/5)^3$ acts on each coordinate by a different $5$-root of unity. See also \cite[Theorem 65]{Cat17} and \cite{Par91}.
	\end{itemize}
\end{remark}
We conclude the introduction by outlining the structure of the paper.

In Section~\ref{sec: preliminaries} we briefly recall the theory of abelian covers of algebraic varieties, with particular emphasis on the case of Galois covers with group $(\mathbb Z/p)^k$, following \cite{FP24, Par91a}.

In Section~\ref{sec: classification_SNC_p,k_planes_Prop_STAR} we develop the
necessary tools to bound the admissible values of $p$ and $k$, thereby providing
all the ingredients required to construct an explicit algorithm for the
classification of SNC $(p,k)$-planes satisfying Property~$\bigstar$. Consequently, we present the algorithm implementing the classification and we prove Theorem \ref{thm main1} point (1). 

Section \ref{sec: SNC_(p,1)-planes} is devoted to the proof of Theorem~\ref{thm main2}: any $(p,1)$-plane with geometric genus $1$ is  a K3 surface. We provide a detailed analysis of the $(p,1)$-planes 
appearing in the theorem; specifically we study their minimal resolution of singularities to determine the minimal model via the contraction of $(-1)$-curves, see Figures \ref{fig:global_process_p=7}, \ref{fig:global_process_1}, \ref{fig:global_process_p=13}, \ref{fig:global_process_p=17}, \ref{fig:global_process_p=19}. 

In Section \ref{sec:SNC_(p,1)-planes_zero}, we prove Theorem \ref{thm: intro_pg_zero_classification}, see Theorem \ref{thm: main_thm_(p,1)_planes_pg=0}, where we classify all $(p,1)$-planes with geometric genus $0$ in (infinite) families, and then proving that each of those surfaces are rational. 

In Section~\ref{subsec: K3Hodge}, we discuss some general results on the Hodge structures of SNC $(p,k)$-planes. More precisely, we investigate 
the Tate and Mumford-Tate conjectures of these surfaces following mainly the ideas contained in \cite{GP22}. Thus, we prove Theorem \ref{thm main1} points (3) and (4).

In Section \ref{sec: invariants_SNC_p,k_planes_Prop_STAR} we prove Theorem \ref{thm main1} point (2), see Theorem \ref{thm: SNC_(p,k)_planes_invariants}. More precisely, we analyze each SNC $(p,k)$-plane satisfying Property~$\bigstar$ in
the classification list \ref{tab: (2,2)} to \ref{tab: (7,2)}, explicitly computing its numerical invariants and
determining its Kodaira dimension. For the reader’s convenience, the
tables produced by the algorithm are collected in Appendix \ref{app:classification_tables}, in order to
keep the exposition streamlined.

Finally, in Section~\ref{sec: examples_of_SNC_p,k_planes_Prop_STAR}, we present detailed and instructive
examples of SNC $(p,k)$-planes that we used as a guideline of the entire work. In the same examples, we also show how to determine the Picard number and the rank of the transcendental part of the Hodge structure of a very general member of their family. 
 
\begin{center}
	N{\footnotesize  OTATION AND CONVENTIONS}
\end{center}

	
	We work over the field $\mathbb{C}$ of complex numbers. 

    
    With $\lfloor x\rfloor$ we denote the greatest integer less or equal than the real number $x$. 
	



    We denote for short by $(\mathbb Z/p)^k$ the elementary abelian $p$-group $\left(\ZZ/p\ZZ\right)^k$. 
    
	Let $G$ be an abelian group. The set of characters, that is, group homomorphisms
    $\chi\colon G \to \mathbb{C}^\times$, forms an abelian group under pointwise multiplication.
    We denote the character group by $G^*$.
 
      Given a character $\chi\colon G\to \mathbb C^\times$ of $G=(\mathbb Z/p)^k$ and $g\in G$, define $\langle \chi, g\rangle \in \{0,\dots,p-1\}$ by
\[
\chi(g)=\exp\!\left(\frac{2\pi i}{p}\langle \chi,g\rangle
\right).
\]
    
	By \emph{surface} we mean a projective, connected normal two dimensional complex variety $S$, and
	for such a surface $\omega_S=\cO_S(K_S)$ denotes the canonical
	class, $p_g(S)=h^0(S, \, \omega_S)$ is the \emph{geometric genus},
	$q(S)=h^1(S, \, \omega_S)$ is the \emph{irregularity}, 
	$\chi(\mathcal{O}_S)=1-q(S)+p_g(S)$ is the \emph{Euler-Poincar\'e
		characteristic}, $e(S)$ is the \emph{topological Euler characteristic}, and $\rho(S)$ is the \emph{rank of the Picard group} of $S$. If $q(S)=0$, we call $S$ a \emph{regular surface}.
	\\
	$T(S)_{\mathbb Q}$ denotes the \textit{transcendental lattice} $T(S)=\NS(S)^\perp$ with rational coefficients.  
	\\
	$H^{1,1}_{tra}(S)$ denotes the transcendental part of the complex vector space $H^{1,1}(S)$, and by $h^{1,1}(S)$ its dimension.

By an abuse of notation we will call \emph{K3 surface} also singular surfaces, whose minimal model is a K3 surface.

Throughout the paper, we denote Cartier (or Weil) divisors on a variety by capital letters and the corresponding line bundles by italic letters, so we write for instance $\mathcal{L}=\cO_S(L)$. 

We denote by $\equiv$ the numerical equivalence relation among Cartier divisors, while $\equiv_{lin}$ denotes the linear equivalence relation. 

    A cyclic quotient singularity $x$ of a normal surface $X$ is called of type $\frac{1}{n}(1,a)$, with $\gcd(a,n)=1$, if there exists an open neighborhood of $x$ isomorphic to the quotient $\mathbb C^2/(\mathbb Z/n)$, where $\mathbb Z/n$ acts diagonally on $\mathbb C^2$ by sending $(z_1,z_2)\mapsto \left(\exp(\frac{2\pi i}{n})z_1, \exp(\frac{2\pi i}{n}a)z_2\right)$. 
	\medskip
\begin{center}
	A{\footnotesize CKNOWLEDGMENTS}
\end{center}

The authors wish to thank  M. Bianco, F. Catanese, R. Clifford, A. Garbagnati, Y. Hu, R. Pignatelli, and G.P. Pirola for fruitful discussions.  The first author held a research fellowship
from INdAM, Istituto Nazionale di Alta Matematica. The second author was partially supported by the Research Project PRIN 2020 - CuRVI, CUP J37G21000000001. Matteo Penegini and Anna Ulivi wish to thank the MIUR Excellence Department of Mathematics, University of Genoa, CUP D33C23001110001;  The authors are members of the INDAM-GNSAGA.%
\section{General results of abelian covers with Galois group $(\mathbb Z/p)^k$}\label{sec: preliminaries}
We briefly recall the theory of abelian covers following the notation developed in \cite{FP24}. 

    A \emph{Galois cover} with a finite group $G$ is a finite morphism $\pi\colon X\to Y$ onto a smooth complete algebraic variety $Y$, with $X$ normal admitting a subgroup of $\text{Aut(X)}$ isomorphic to $G$ and such that $\pi$ factors through the composition of the quotient map $X\to X/G$ and an isomorphism $X/G\cong Y$. The group $G$ is the Galois group of $\pi$. 

\begin{definition}
An \emph{abelian cover} is a Galois cover with $Y$ smooth, $X$ normal, and $G$ a finite abelian group. A \emph{$(p,k)$-cover} is an abelian cover with Galois group $G= (\mathbb Z/p)^k$. A \emph{$(p,k)$-plane} is a $(p,k)$-cover with $Y\cong \mathbb P^2$. 
\end{definition}
Let $\pi\colon X\to Y$ be an abelian cover with Galois gruop $G$. Following \cite{FP24} and \cite{Par91a}, the direct image of the structure sheaf $\mathcal{O}_X$ on $Y$ splits as a direct sum of line bundles $\mathcal L_\chi^{-1}$ indexed by characters $\chi\in G^*$
\begin{equation}\label{eq: dec_structure_sheaf}
    \pi_*\mathcal{O}_X= \bigoplus_{\chi\in G^*}\mathcal L_\chi^{-1}. 
\end{equation}
By Zariski-Nagata purity theorem, the branch locus of $\pi$ is a divisor. The reduced structure of this divisor will be denoted by $D$. As is customary for Galois covers, we call the divisor $\pi^{-1}(D)$, endowed with the reduced structure, the \emph{ramification divisor} of $\pi$, and we denote it by $R$.
 Let $T$ be an irreducible component of $R$. The elements of $G$ fixing all the points of $T$ form a cyclic subgroup $H$ of $G$, the so called \textit{inertia group} of $T$. Furthermore, there is a unique character $\psi$ of $H$ and a uniformizing parameter $t$ for $\mathcal O_{X,T}$ such that any $h\in H$ acts on $t$ as $t\mapsto \psi(h)t$. In addition, as observed in \cite{FaPa97}, we can choose a \emph{unique} element $g\in G$ such that $H=\langle g \rangle$ with the property that $\chi(g)=\exp\left(\frac{2\pi i}{\vert H\vert}\right)$, i.e., $h$ acts on $t$ by the multiplication of the first $\vert H\vert$-root of the unity. In this case, $g$ is called the \textit{local monodromy} of the component $T$. This discussion leads to the following  natural decomposition of the ramification divisor
\[
  R=\sum_{g\in G} R_g,
\]
where $R_g$ is the (possibly zero) reduced divisor given by the sum of those irreducible components $T \subset R$ whose inertia group is generated by $g$ and whose local monodromy character corresponds to $g$.
In particular, two irreducible components $T_1$ and $T_2$ appear in the same summand $R_g$ if and only if they have the same local monodromy.  Furthermore, since $G$ is a abelian group, 
then $R_g$ is a union of $G$-orbits.
For this reason, there exists a unique reduced divisor $D_g \subset Y$ such that
\[
R_g = \pi^{-1}(D_g).
\] 
Therefore, we have also a decomposition of the reduced branch locus
\[
D=\sum_{g\in G}D_g.
\]
\begin{definition}
   The \emph{building data} of the abelian cover $\pi\colon X\to Y$ consist of the collection of line bundles
$\{\mathcal L_\chi\}_{\chi\in G^*}$ together with the collection of reduced divisors
$\{D_g\}_{g\in G}$ constructed above.
\end{definition}
\begin{remark}\label{rem: X_connected_implies_L_chi-1_no_sections}
 By construction, $D_0$ is the zero divisor and $\mathcal L_1\cong \mathcal O_Y$. Furthermore, $X$ is connected if and only if $h^0(Y, \mathcal L_\chi^{-1})=0$ for all $\chi\neq 1_G$. Indeed, from the decomposition \eqref{eq: dec_structure_sheaf} one has $h^0(X, \mathcal O_X)=\sum_{\chi \in G^*} h^0(Y, \mathcal L_{\chi}^{-1})$. 
\end{remark}
The most powerful result regarding abelian coverings is Pardini's Existence Theorem \cite[Theorem 2.1]{Par91a}, which states that the building data completely determine an abelian cover of a smooth variety $Y$. This allows one to work with abelian coverings simply by managing the building data, in the sense that all the geometric properties of \(X\), such as its smoothness, canonical class, and numerical invariants, can be expressed in terms of the building data, and vice versa. For the sake of simplicity and for our purposes, we only present the theorem in the case where \(G\cong (\mathbb Z/p)^k\) is an elementary abelian \(p\)-group and the base of the cover is smooth, i.e., for $(p,k)$-covers with $p$ prime number and $k\geq 1$.

\begin{theorem}\label{thm: Par91a}\cite[Thm. 2.1]{Par91a}
    Let $\pi\colon X \to Y$ be a $(p,k)$-cover. Then for any pairs of characters $\chi, \eta\in G^*$ it holds
    \begin{equation}\label{eq: Pardini_eq}
       L_\chi+L_\eta\equiv_{\lin} L_{\chi \eta}+\sum_{g\in \mathbb (Z/p)^k}\left \lfloor \frac{\chig+\etag}{p} \right \rfloor D_g.
    \end{equation}
    Conversely, given a smooth variety $Y$ and 
    \begin{itemize}
        \item a collection of line bundles $\{\mathcal L_\chi\}_{\chi\in G^*}$ on $Y$ indexed by the characters of $G$,
        \item a collection of \textup{divisors} $\{D_g\}_{g\in G}$ on $Y$ indexed by the elements of $G$,
    \end{itemize}
    with the property that $D\coloneqq \sum_{g\in G}D_g$ is reduced and the linear equations \eqref{eq: Pardini_eq} hold for any pair $\chi,\eta\in G^*$, then there exists a unique $(p,k)$-cover, up to isomorphisms of $G$-covers, $\pi\colon X\to Y$  whose building data are $\{\mathcal L_\chi\}_{\chi\in G^*}$ and $\{D_g\}_{g\in G}$. 
\end{theorem}
As evident, the number of linear equivalence equations required to guarantee the existence of a $(p,k)$-cover grows quadratically with the cardinality of the group. As a consequence, the explicit determination of a $(p,k)$-cover can be computationally demanding. The following Theorem~\ref{thm: reduction_numb_Pard_eq} reduces the number of equations that need to be checked to $k$, the rank of $(\mathbb{Z}/p)^{k}$.

\begin{theorem}\cite[Prop. 2.1]{Par91a}\label{thm: reduction_numb_Pard_eq}
     Let $\{e_1, \dots, e_k\}$ be a set of generators of $(\ZZ/p)^k=:G$ and let $\{\epsilon_1, \dots, \epsilon_k\}\subset G^*$ be the corresponding dual set of generators of $G^*$.
     
     Let $\pi\colon X\to Y$ be a $(p,k)$-cover. Then 
    \begin{equation}\label{eq: reduction_numb_Pard_eq}
        p\cdot L_{\epsilon_j}\equiv_{\lin} \sum_{g\in G}\ejg D_g, \qquad \makebox{for any} \qquad j=1,\dots, k.
    \end{equation}
    Conversely, given a smooth variety $Y$ and 
    \begin{itemize}
    \item a collection of divisors $\{D_g\}_{g\in G}$ on $Y$ and 
    \item line bundles $\{\mathcal L_{\epsilon_1}, \dots, \mathcal L_{\epsilon_k}\}$ on $Y$
    \end{itemize}
    satisfying equations \eqref{eq: reduction_numb_Pard_eq} for any $j=1, \dots,k$. Then for any $\chi\in G^*$ the divisor $\sum_g\chi(g)D_g$ is $p$-divisible, in the sense that there exists a divisor $L_\chi$ such that $p\cdot L_\chi\equiv \sum_g\chi(g)D_g$. 
    Setting $\mathcal L_\chi\coloneqq \mathcal O_Y(L_\chi)$, then $\{\mathcal L_\chi \}_{\chi\in G^*}$ and $\{D_g\}_{g\in G}$ satisfy equations \eqref{eq: Pardini_eq}. 
    
    Consequently, there exists a unique  $(p,k)$-cover $\pi\colon X\to Y$ (up to isomorphisms of $G$-covers) whose building data are $\{\mathcal L_\chi\}_{\chi\in G^*}$ and $\{D_g\}_{g\in G}$.  
\end{theorem}
It is not so difficult to see that divisors $D_g$ determine $L_\chi$ (namely the line bundle $\mathcal L_\chi$) up to torsion of $\text{Pic}(Y)$. 
\begin{lemma}\label{lem: L_chi_in_function_D_g}\cite[Cor. 1.6]{FP24}
    Any line bundle $\mathcal L_\chi$ in the building data of a $(p,k)$-cover satisfies the linear equivalence equation
    \begin{equation}\label{eq: reduction_Pardini_eq}
        p\cdot L_\chi \equiv_{\lin} \sum_{g\in G}\chig D_g.
    \end{equation}
\end{lemma}
In particular, if $\text{Pic}(Y)$ is torsion-free (as in the case when $Y$ is a rational or K3 surface) then the divisors $D_g$ uniquely determine the line bundles.
\vspace{0.2cm}

\begin{remark}\label{rem: onlyD_g}
By Lemma~\ref{lem: L_chi_in_function_D_g}, if the base space $Y$ has torsion-free Picard group,
the building data of a $(p,k)$-cover are uniquely determined by the collection of divisors
$\{D_g\}_{g\in G}$.
Indeed, in this case there is a unique collection of line bundles
$\{\mathcal L_\chi\}_{\chi\in G^*}$ for which the relations~\eqref{eq: reduction_Pardini_eq} are satisfied,
and these are implicitly fixed throughout.

Moreover, in order to test whether a given collection of divisors $\{D_g\}_{g\in G}$
defines a $(p,k)$-cover, it suffices to verify that the divisors
\[
\sum_{g\in G} \eig\, D_g
\]
are $p$-divisible for a suitable dual set of generators $\{\epsilon_1,\dots,\epsilon_k\}$ of $G^*$,
as follows from Theorem~\ref{thm: reduction_numb_Pard_eq}.
\end{remark}

\begin{remark}\label{rem: D_g_generates_G}
    If $Y$ has a torsion-free Picard group and $X$ is connected, then the set of local monodromies $\{g\in G| D_g\neq 0\}$ generates the entire group $G$. Otherwise, if the  subgroup $H$ generated by local monodromies were proper in $G$, one could define a non trivial character $\chi$ vanishing on $H$. However, it would follow from Equation \eqref{eq: reduction_Pardini_eq} that $\mathcal L_\chi$ has $p$-torsion, and so thus it must be the trivial bundle. This would imply $h^0(Y,\mathcal L_\chi^{-1})=1$, which contradicts the connectedness of $X$, according to Remark \ref{rem: X_connected_implies_L_chi-1_no_sections}. 
\end{remark}

As will become apparent below, most geometric and numerical invariants of a $(p,k)$-cover are naturally expressed
in terms of the line bundles $\mathcal L_\chi$, rather than directly in terms of the divisors $D_g$.
Consequently, although the cover can be defined by specifying the divisors $D_g$, constructing a $(p,k)$-cover with
prescribed numerical invariants requires passing between these two types of building data, a process which can be
computationally demanding.


For this reason, the authors in \cite{FP24} investigated whether the line bundles $\mathcal L_\chi$ determine the divisors $D_g$ for a general abelian cover. As observed in \cite[Example 1.7]{FP24}, it is not so difficult to construct (singular) abelian covers having the same $\mathcal L_\chi$ but different $D_g$, so it is not always possible to determine $D_g$ in terms of $\mathcal L_\chi$. However, a positive answer is given in \cite[Thm. 1.11]{FP24} for groups $G=(\mathbb Z/2)^k$. We generalize this crucial result for elementary abelian $p$-groups. Before doing that let us recall the following definition.
\begin{definition}
The natural isomorphism $G \xrightarrow{\sim} G^{**}$ allows us to regard each element
$g \in G$ as a character of $G^*$, which we denote again by $g$, by setting
\[
g(\chi) \coloneqq  \chi(g).
\]
The kernel of $g$ is then the subgroup of $G^*$ consisting of the characters $\chi$
such that $\chi(g)=1$. In other words,
\[
\chi \in \ker g \;\Longleftrightarrow\; g \in \ker \chi.
\]

Let $\mathcal H \subset G^*$ be a subgroup, possibly of the form $\ker g$.
For any $g \in G$, we denote by $g_{\mid \mathcal H}$ the element of $\mathcal H^*$
obtained by restricting $g$ to $\mathcal H$. Since $G$ is an elementary abelian $p$-group, the order $o(g_{\mid \mathcal H})$ of $g_{\mid \mathcal H}$
is either $1$ or $p$. We define the orthogonal subgroup of $\mathcal H$ as
\[
\mathcal H^{\perp} \coloneqq  \{g \in G \mid g_{\mid \mathcal H} = 1\}.
\]
\end{definition}

\begin{theorem}\label{thm: Dg_elm_group}
	Let $\pi \colon X \rightarrow Y$ be a $(p,k)$-cover with building data $\{\mathcal L_\chi\}_{\chi\in G^*}$, $\{D_g\}_{g\in G}$. Then, given a subgroup $\mathcal H\leq G^*$ of index $p^t$ in $G^*$, 
	we have
	\begin{equation}\label{eq: D_g_function_Lchi_H}
	    \frac{p^{k}(p-1)}{2}\sum_{h\in \mathcal H^\perp} D_h \equiv_{\lin}  p\left( \sum_{\chi \not\in \mathcal H} L_\chi - (p^t-1)\sum_{\chi \in \mathcal H} L_\chi\right).
	\end{equation}
    
	In particular, for $\mathcal H =G^*$, we have 
    \begin{equation}\label{eq: sum_dg_sm_lchi}
        \frac{p^{k}(p-1)}{2}\sum_{g\in G} D_g \equiv_{\lin}\, p\sum_{\chi \in G^*} L_\chi.
    \end{equation}
\end{theorem}
\begin{proof}
    Consider a subgroup $\mathcal H$ of $G^*$. Then applying equation \eqref{eq: reduction_Pardini_eq} to each $\chi \in \mathcal{H}$, one has
	\begin{equation}\label{eq: sum_H Lchi}
      \begin{split}
      		p\sum_{\chi \in \mathcal H}L_\chi & \equiv_{\lin}  \sum_{\chi \in \mathcal H} \left(\sum_{g\in G}\chig D_g\right)\\ 
      		& \equiv_{\lin}  \sum_{g\in G} \left(\sum_{\chi \in \mathcal H}\chig \right)D_g \\ 
      		& \equiv_{\lin}  \sum_{g\in G} \frac{|\mathcal H|}{2}p \left( 1 -\frac{1}{o(g_{|\mathcal H})} \right)D_g,
      \end{split}
	\end{equation}
    where the last equality follows from \cite[Lemma~1.9]{FP24}, which is true for an arbitrary abelian group. Furthermore, when $\mathcal H = G^*$, since $o(g_{\mid \mathcal H}) = p$, the last equation~\eqref{eq: sum_H Lchi}
reduces to~\eqref{eq: sum_dg_sm_lchi}.

 Now, using both \eqref{eq: sum_H Lchi} and  \eqref{eq: sum_dg_sm_lchi} we obtain 
	\[\begin{split}
		\sum_{g\in G }\frac{ p^{k+1}}{2}\left(\frac{1}{o(g_{|\mathcal H})}-\frac{1}{p}\right)D_g \equiv_{\lin}  & \sum_{g\in G}\frac{ p^{k+1}}{2}\left(1-\frac{1}{p}\right)D_g-\sum_{g\in G}\frac{ p^{k+1}}{2}\left(1-\frac{1}{o(g_{|\mathcal H})}\right)D_g\\
		\equiv_{\lin} \, &p\sum_{\chi\in G^*}L_\chi-p[G^*: \mathcal H]\sum_{\chi\in \mathcal H}L_\chi \\
		\equiv_{\lin} \, & p\left( \sum_{\chi\not \in \mathcal H}L_\chi-\left(p^t-1\right)\sum_{\chi\in \mathcal H}L_\chi\right).
	\end{split}
	\]
    However, the coefficient on the left member of the equation is either $\frac{p^k}{2}(p-1)$ or $0$  according if $g\in \mathcal H^\perp$ or not. Thus the left summand is over $\mathcal H^\perp$ and 
    \eqref{eq: D_g_function_Lchi_H} follows. 
\end{proof}
\begin{corollary}\label{cor: sum_<g> D_h equal to}
	Let $g\in G$ with $g\neq 0$. Then 
	\begin{equation}\label{eq: sum_<g> D_h equal to}
		\frac{p^{k}(p-1)}{2}\sum_{h\in \langle g\rangle } D_h \equiv_{\lin} \,p\left(\sum_{\chi \not\in \ker g} L_\chi - (p-1)\sum_{\chi \in \ker g} L_\chi\right).
	\end{equation}
\end{corollary}
\begin{proof}
	It follows directly from Theorem \ref{thm: Dg_elm_group} applied to $\mathcal H:=\ker g$. Indeed, in this case $\mathcal H^\perp=\langle g\rangle$: the inclusion $\langle g\rangle\subseteq \mathcal H^\perp$ is straightforward, whilst the other direction holds directly by the fact that each character defined over a proper subgroup can be extended to a non-trivial character of the entire group. 
\end{proof}
\begin{remark}
	Formula~\eqref{eq: sum_<g> D_h equal to} shows that the numerical class of $\sum_{h\in \langle g\rangle} D_h$ can be expressed in terms of the line bundles $\mathcal L_\chi$.
	Note that when working with an elementary abelian $2$-group, each subgroup $\langle g\rangle$, $g\neq 0$, consists of exactly two elements, namely the identity and $g$ itself, and the above sum simply reduces to $D_g$. Thus, in the case $G\cong (\mathbb Z/2)^k$, the formula in \cite[Thm.~1.11]{FP24} becomes a particular instance of Formula \eqref{eq: sum_<g> D_h equal to}. 
\end{remark}
\subsection{Smooth and singular $(p,k)$-covers and their invariants}
\begin{definition}\label{defn: smooth_p-k_cover}
	A \emph{smooth \((p,k)\)-cover} is a \((p,k)\)-cover \(\pi\colon X\to Y\) such that the reduced branch divisor \(D\) has normal crossings, and for each \(g \in G\), the divisor \(D_{g}\) is either empty or a disjoint union of smooth irreducible components. Furthermore, we require that
	 \begin{equation}\label{eq_condSmooth_0}  \langle g\rangle \cap\langle h\rangle= \{0\}, 
	 \end{equation} 
for any two distinct irreducible components \(\Delta_g \leq D_g\) and \(\Delta_h \leq D_h\) such that $\Delta_g \cap \Delta_h\neq 0$. 
\end{definition}
\begin{proposition}\label{prop: smooth_cover} Let $\pi\colon  X \to Y$ be a smooth $(p,k)$-cover. Then $X$ is smooth. \end{proposition}
\begin{proof} This is a special case of \cite[Prop. ~3.1]{Par91a}.  
First, by \cite[Cor. ~3.1]{Par91a} the variety \(X\) is normal.  
Moreover, the only points of \(X\) that may be singular are those lying over 
intersection points of two irreducible components of the reduced branch divisor $D$.  
However, by \cite[Prop. ~3.1]{Par91a}, condition~\eqref{eq_condSmooth_0}  
ensures that all points lying above such intersections are smooth. Hence the variety \(X\) is smooth.
\end{proof}
%
%
%
To consider more general \((p,k)\)-covers, we must allow certain singularities. Thus, we will introduce a slightly weaker condition on the reduced branch divisor \(D\) than the one required for a smooth \((p,k)\)-cover.
\begin{definition}\label{def: SNC_p,k_cover}
    A \emph{Simple Normal Crossing  $(p,k)$-cover} (SNC for short) is a $(p,k)$-cover $\pi\colon X\to Y$ of a smooth surface $Y$ such that the reduced branch divisor $D$ has normal crossings, i.e. all irreducible components of $D$ are smooth, each two of them either do not intersect or intersect transversally, and no point of \(Y\) belongs to more than two of them. 
\end{definition}
For the rest of this section, we shall consider SNC $(p,k)$-covers of a surface $Y$. By Proposition \ref{prop: smooth_cover} (see also \cite[Prop. 3.1]{Par91a}), the singular points of $X$ arise from the intersection points of pairs of distinct irreducible components $\Delta_g\subseteq D_g$ and $\Delta_h\subseteq D_h$ such that $\langle g\rangle \cap \langle h  \rangle\neq 0$. In other words, the image in $Y$ of the singular locus of $X$ is 
\[
\pi(\text{Sing}(X)) = \bigcup_{\substack{\langle g\rangle \cap \langle h\rangle \neq \{0\} \\ \Delta_g \leq D_g, \, \Delta_h \leq D_h}} (\Delta_g \cap \Delta_h).
\]
Such a surface $X$ has only rational singularities, that are determined precisely as follows.
\begin{proposition}\cite[Prop. 3.3]{Par91a}\label{prop: type_of_sing}
Let $\pi\colon X \to Y$ be an SNC $(p,k)$-cover of a smooth surface $Y$, and let
$\Delta_g \subset D_g$ and $\Delta_h \subset D_h$ be smooth irreducible components
such that $\langle g\rangle \cap \langle h\rangle \neq 0$.
Let $1 \le a \le p-1$ be the integer such that $h = (p-a)g$.
If $\Delta_g$ and $\Delta_h$ intersect transversely at a point $y \in Y$,
then the fiber $\pi^{-1}(y)$ consists of $p^{k-1}$ points, each of which is a
cyclic quotient singularity of $X$ of type $\frac{1}{p}(1,a)$.
\end{proposition}

\begin{proof}
The statement is a special case of \cite[Proposition~3.3]{Par91a}.
\end{proof}
Let $x$ be a singularity of type $\frac{1}{n}(1,a)$ with $\gcd(n,a)=1$. The minimal resolution of singularities of $x$ consists of replacing $x$ by a Hirzebruch-Jung string $E_1\cup \dots \cup E_l$, where the $E_i$ are smooth rational curves with $E_iE_{i+1}=1$, for $i=1,\dots, \ell-1$, and $E_iE_j=0$ otherwise. Their self-intersection is $E_i^2=-b_i$, where $b_i$ are the coefficients of the continued fraction 
\[\frac{n}{a}=b_1-\frac{1}{b_2-\frac{1}{
		b_3-\dots}}=[b_1,\dots, b_\ell].\]
	\begin{definition}\label{def: correct_terms_invariants}
		Let $x$ be a singularity of type $\frac{1}{n}(1,a)$ with $\gcd(n,a)=1$ and let $1\leq a' \leq n-1$, such that $a'=a^{-1}$ in $(\mathbb Z/n)^*$. We define the following 
		\begin{itemize}
			\item $k_x:= k(\frac{1}{n}(1,a)):=-2+\frac{2+a+a'}{n}+\sum_{i=1}^\ell (b_i-2) \geq 0$;
			\item $e_x:=e(\frac{1}{n}(1,a)):=\ell+1-\frac{1}{n}$;
			\item $E_x:=E(\frac{1}{n}(1,a)):=\sum_{i=1}^\ell\left(1-\frac{D_{i,\ell}+D_{i,1}}{n}\right)E_i$, where $D_{i,\ell}$ and $D_{i,1}$ are  the denominators of the truncated continued fractions 
			\[
			  \frac{N_{i,\ell}}{D_{i,\ell}}:=[b_i, b_{i+1},\dots, b_\ell], \qquad \makebox{and} \qquad \frac{N_{i,1}}{D_{i,1}}:=[b_i, b_{i-1},\dots, b_1]. 
			\]
			Clearly, $N_{1,\ell}=n$, $D_{1,\ell}=a$, and $N_{\ell,\ell}=b_\ell$, $D_{\ell,\ell}=1$. It is well known that $N_{\ell, 1}=n$, $D_{\ell, 1}=a'$, and $N_{1,1}=b_1$, $D_{1,1}=1$. 
		\end{itemize}
	\end{definition}
\begin{theorem}\label{thm: invariants_ab_cov_cyclic_sing}
	Let $\pi\colon X\to Y$ be an SNC $(p,k)$-cover of a smooth surface $Y$. Let $\rho\colon \widetilde{X}\to X$ be the minimal resolution of singularities of $X$, and consider the generically finite morphism $\widetilde{\pi}\coloneqq \pi\circ \rho\colon \widetilde{X}\to Y$. Then:
	\begin{equation}\label{eq: pg_ab_cov_quo_sing}
		p_g(\widetilde{X})   =\sum_{\chi\in G^*} h^0(Y, K_Y+\mathcal L_{\chi}),
	\end{equation}
	\begin{equation}\label{eq: q_ab_cov_quo_sing}
		q(\widetilde{X})  = q(Y)+ \sum_{\chi\in G^*\setminus \{1\}}h^1(Y, - \mathcal L_{\chi}),
\end{equation}
\begin{equation}\label{eq: chi_ab_cov_quo_sing}
	\chi(\mathcal{O}_{\widetilde{X}})   =p^{k}\chi(\mathcal{O}_Y)+ \frac{1}{2}\sum_{\chi\in G^*}L_\chi(L_\chi+K_Y),
\end{equation}
\begin{equation}\label{eq: Ksquare_ab_cov_quo_sing}
	K_{\widetilde{X}}^2  =p^{k-2} \left(pK_Y+(p-1)\sum_{g\in G}D_g\right)^2-\sum_{x\in \text{Sing}(X)}k_x,
\end{equation}
\begin{equation}\label{eq: e_ab_cov_quo_sing}
	e(\widetilde{X})  =p^{k}\left(e(Y)-\left(1-\frac{1}{p}\right)\cdot\sum_{\Delta_g}\left(2-2g(\Delta_g)\right)+\frac{1}{2}\left(1-\frac{1}{p}\right)^2\sum_{\Delta_g\neq \Delta_h}\Delta_g\Delta_h\right)+\sum_{x\in \text{Sing}(X)}e_x.
\end{equation}
Furthermore, we have 
\begin{equation}\label{eq: KS_ab_cov_quo_sing}
	pK_{\widetilde{X}} \equiv_{lin}  \widetilde{\pi}^*\left(pK_Y+(p-1)\sum_{g\in G}D_g\right)-\sum_{x\in \text{Sing}(X)}pE_x.
\end{equation} 
\end{theorem}
\begin{remark}\label{rem: easier_invariants_formulae_smooth_cover}
	Theorem \ref{thm: invariants_ab_cov_cyclic_sing} generalizes \cite[Proposition~4.2]{Par91a} to the case where cyclic quotient singularities are allowed. Indeed, if the covering surface is smooth, then $\text{Sing}(X)$ is empty and the formulae reduce to those in \cite[Proposition~4.2]{Par91a}.
\end{remark}
\begin{proof}
    The formulae for the geometric genus \eqref{eq: pg_ab_cov_quo_sing} and the irregularity \eqref{eq: q_ab_cov_quo_sing} follow directly from \eqref{eq: dec_structure_sheaf}, \cite[Satz~1]{Frei71} and \cite[Proposition~4.1]{Par91a}, since the singular locus of $X$ consists only of rational singularities from Proposition \ref{prop: type_of_sing}.  As for \eqref{eq: chi_ab_cov_quo_sing}, the holomorphic Euler characteristic satisfies
\[
\chi(\mathcal O_{\widetilde{X}})
= 1 - q(\widetilde{X}) + p_g(\widetilde{X})
= \sum_{\chi \in G^*} \chi\bigl(K_Y + \mathcal L_\chi\bigr),
\]
and the claim follows from the Riemann--Roch theorem.

Formula \eqref{eq: KS_ab_cov_quo_sing} is obtained by first applying the canonical bundle formula
for the $(\mathbb Z/p)^k$-cover to express $pK_X$ as the pullback of
$pK_Y+(p-1)\sum_{g\in G}D_g$, and then passing to the minimal resolution
$\rho\colon \widetilde X\to X$, where the discrepancy contributions of the cyclic
quotient singularities are computed explicitly using the Hirzebruch-Jung
resolution data (see e.g. \cite{F26} and \cite[page 286]{HTU17}). 

The self-intersection of the canonical class \eqref{eq: Ksquare_ab_cov_quo_sing} is a standard computation obtained from \eqref{eq: KS_ab_cov_quo_sing} (see \cite[Remark~3.3]{GaPe15} and \cite[Thm.~2.4]{fedeLincei}).

Formula \eqref{eq: e_ab_cov_quo_sing} for $e(\widetilde{X})$ is obtained by applying the additivity of the Euler
characteristic with respect to stratifications.
More precisely, one first computes $e(X)$ by decomposing $Y$ into the complement of the
branch locus and its natural strata, using that $\pi$ is a $(\mathbb Z/p)^k$-cover with
simple normal crossings branch divisor, and then corrects for the singular points of $X$
by passing to the minimal resolution of singularities $\rho\colon \widetilde X \to X$.
\end{proof}
\begin{remark}
	Notice that, by Noether’s formula, the Euler characteristic could also be recovered from
	the previously computed invariants.
	Nevertheless, in the sequel we shall compute the Euler characteristic using
	Formula \ref{eq: e_ab_cov_quo_sing}.
\end{remark}
We automatically obtain the following useful result for SNC $(p,k)$-planes. 
\begin{lemma}\label{lem: (p,k)_planes_are_regular}
The minimal resolution of singularities $\widetilde{X}$ of a SNC $(p,k)$-plane $\pi\colon X\to \mathbb P^2$ is a regular surface, namely $q(\widetilde{X})=0$.
\end{lemma}
\begin{proof}
Indeed, every line bundle of $\mathbb P^2$ has trivial first cohomology group; thus, from Theorem \ref{thm: invariants_ab_cov_cyclic_sing}, we have 
    \[
    q(\widetilde{X})=\sum_{\chi\in G^*}h^1(\mathbb P^2, -L_\chi)=0. 
    \]
\end{proof}
\subsection{Decomposition of the branch curves under resolution of singularities}
We conclude this section studying the decomposition of the pullback $\widetilde{\pi}^*\Delta_g$ of a smooth irreducible component
$\Delta_g \subset D_g$, keeping the terminology and notation of Theorem \ref{thm: invariants_ab_cov_cyclic_sing}.
We denote by $\widetilde{\Delta}_g$ the strict transform of $\pi^*\Delta_g$ under the resolution $\rho\colon \widetilde{X} \to X$.
\begin{proposition}\label{prop: dec_of_the_pullback}
Let $\pi \colon X \longrightarrow \PP^2$ be a SNC $(p,k)$-plane, and let $\widetilde{\pi} := \pi \circ \rho \colon \widetilde{X} \to Y$. The reduced pullback $\widetilde{\pi}^*\Delta_g$ decomposes as the sum of a smooth irreducible curve $\widetilde{\Delta}_g$ (the strict transform of $\Delta_g$ under the resolution $\rho$) and the Hirzebruch-Jung strings arising from the singular points of $X$ lying over $\Delta_g$.
	
	More precisely, let $x \in \text{Sing}(X)$ be such that $\pi(x) \in \Delta_g$. In this case, there exists a unique integer $1 \leq a_x \leq p-1$ such that $\pi(x) \in \Delta_g \cap \Delta_h$, where $h = (p-a_x)g$, and $x$ is a cyclic quotient singularity of type $\frac{1}{n}(1,a_x)$ (cf. Proposition~\ref{prop: type_of_sing}).
	
	Then, $\widetilde{\Delta}_g$ meets the exceptional chain $E_{1,x} \cup \cdots \cup E_{\ell_x, x}$ transversally in a single point, and this intersection point lies on the first component $E_{1,x}$. Moreover, we have 
	\begin{equation}\label{eq: self Intersecion strict}
		\widetilde{\Delta}_g^2 = p^{k-2}\Delta_g^2 - \sum_{x \in \text{Sing}(X), \, \pi(x) \in \Delta_g} \frac{a_x}{p}
	\end{equation}
	\begin{equation}\label{eq: genus_cover_branch}
		2g(\widetilde{\Delta}_g)-2 = p^{k-1} \left( 2g(\Delta_g)-2 + \sum_{\Delta_t, \, t \notin \langle g \rangle} \left(1-\frac{1}{p}\right) \Delta_g \cdot \Delta_t \right).
	\end{equation}
	Furthermore,
	\begin{equation}\label{eq: decomp_pullback_branch_comp}
		\widetilde{\pi}^*\Delta_g = p\widetilde{\Delta}_g + \sum_{x \in \text{Sing}(X), \, \pi(x) \in \Delta_g} \left( \sum_{i=1}^{\ell_x} D_{i,\ell_x} E_{i,x} \right),
	\end{equation}
	where $\frac{N_{i,\ell_x}}{D_{i,\ell_x}} = [b_i, b_{i+1}, \dots, b_{\ell_x}]$ is the $i$-th truncated continued fraction of $\frac{p}{a_x} = [b_1, \dots, b_{\ell_x}]$.
\end{proposition}
\begin{proof}
    The description of the intersection between the strict transform $\widetilde{\Delta}_g$
and the Hirzebruch--Jung string $E_{1,x} \cup \cdots \cup E_{\ell_x,x}$ follows from a local
analysis in a neighborhood of the point $x \in X$.
Indeed, since the divisors $\Delta_g$ and $\Delta_h$ intersect transversely at $\pi(x)$,
the local equation of the cover near $x$ shows that $\widetilde{\Delta}_g$ meets the
exceptional chain transversely in a single point on the first component.

We now determine the coefficients $\alpha_{1,x}, \dots, \alpha_{\ell_x,x}$ of the decomposition 
    \[
		\widetilde{\pi}^*\Delta_g = p\widetilde{\Delta}_g + \sum_{x \in \text{Sing}(X), \, \pi(x) \in \Delta_g} \left( \sum_{i=1}^{\ell_x} \alpha_{i,x} E_{i,x} \right),
    \]
    Since $\widetilde{\pi}^*\Delta_g$ is the pullback of a divisor on $Y$, it has zero
intersection with each exceptional curve $E_{j,x}$.
Using the fact that the strict transform $\widetilde{\Delta}_g$ meets the
Hirzebruch--Jung chain only on the first component $E_{1,x}$, this yields a tridiagonal linear system for the coefficients $\alpha_i:=\alpha_{i,x}$, whose unique solution is expressed
explicitly in terms of the continued fraction data of the corresponding cyclic
quotient singularity. More precisely, these coefficients are the solution to the linear system $\widetilde{\pi}^*\Delta_g\cdot E_{i,x}=0$, $i=1,\dots, \ell_x$, which takes the form
    \[
     \begin{cases}
         0 & =p+(-b_1)\alpha_1+\alpha_2 \\
         0 & =0 +\alpha_1-b_2\alpha_2+\alpha_3\\
         \vdots \\
         0 & =0 +\alpha_{\ell_x-1}-b_{\ell_x}\alpha_{\ell_x}\\
     \end{cases}.
    \]
    The coefficient matrix of the linear system is the intersection pair matrix of the exceptional chain, and the constant vector is given by  $(-p, 0, \dots,  0)^t$. Thus, the explicit solution is obtained by multiplying the constant vector by the inverse of the intersection pair matrix, which can be found in \cite{F26}.
Specifically, $\alpha_i$ is the denominator $D_{i,\ell_x}$ of the $i$-th partial continued fraction $\frac{N_{i,\ell_x}}{D_{i,\ell_x}}=[b_i,b_{i+1}, \dots, b_{\ell_x}]$ of $\frac{p}{a_x}=[b_1,\dots, b_{\ell_x}]$. 

    We can now compute the self-intersection of $\widetilde{\Delta}_g$. 
Intersecting the decomposition of $\widetilde{\pi}^*\Delta_g$ with $\widetilde{\Delta}_g$
and using the projection formula, we obtain
\[
(\widetilde{\pi}^*\Delta_g)\cdot \widetilde{\Delta}_g
= \Delta_g\cdot \widetilde{\pi}_*\widetilde{\Delta}_g
= p^{k-1}\Delta_g^2.
\]
On the other hand, since $\widetilde{\Delta}_g$ meets each exceptional chain only on the
first component $E_1^x$, we have
\[
(\widetilde{\pi}^*\Delta_g)\cdot \widetilde{\Delta}_g
= p\,\widetilde{\Delta}_g^2 + \sum_{x \in \text{Sing}(X), \, \pi(x) \in \Delta_g}\ \alpha_{1,x},
\]
where $\alpha_{1,x}=D_{1,x}=a_x$. Therefore,
\[
\widetilde{\Delta}_g^2
= p^{k-2}\Delta_g^2 - \sum_{x \in \text{Sing}(X), \, \pi(x) \in \Delta_g}\frac{a_x}{p}.
\]
 The last formula determines the genus $g(\widetilde{\Delta}_g)$ of the curve $\widetilde{\Delta}_g$.
 We compute it by applying Riemann-Hurwitz to the restriction $\widetilde{\pi}_{|\widetilde{\Delta}_g}\colon \widetilde{\Delta}_g \to \Delta_g$, which is a $G/\langle g\rangle \cong (\mathbb Z/p)^{k-1}$ Galois cover. 
 \\
 Geometrically, restricting the cover to $\Delta_g$ amounts to quotienting the monodromy
 by $\langle g\rangle$, so that $\widetilde{\pi}_{\vert \widetilde{\Delta}_g}$ is branched
 precisely at the intersection points $\Delta_g\cap \Delta_t$ for which the monodromy
 class of $t$ is nontrivial in $G/\langle g\rangle$. This means, more or less, that in a schematic picture of the branch locus, one can recognize the components $\Delta_t$
 contributing to the branching of $\widetilde{\pi}_{\vert \widetilde{\Delta}_g}$ simply as
 those curves intersecting $\Delta_g$ whose associated group element $t$ does not lie in
 $\langle g\rangle$; intersections with components indexed by $t\in\langle g\rangle$
 produce no branching and are therefore invisible in the induced cover. In virtue of this, we claim:

\textbf{Claim:}
The branch locus of the restriction
\[
\widetilde{\pi}_{\vert \widetilde{\Delta}_g}\colon \widetilde{\Delta}_g \to \Delta_g
\]
consists of the points $y\in \Delta_g$ such that $y\in \Delta_t$ for some component
$\Delta_t\neq \Delta_g$ with $t\notin \langle g\rangle$.
Moreover, the local monodromy at such a point $y$ is given by the class
$[t]\in G/\langle g\rangle$.

\medskip

This claim implies that $\widetilde{\Delta}_g$ is a disjoint union of smooth connected
components.
In fact, $\widetilde{\Delta}_g$ is connected.
Indeed, the number of connected components of $\widetilde{\Delta}_g$ equals the index
of the subgroup of $G/\langle g\rangle$ generated by the local monodromies of
$\widetilde{\pi}_{\vert \widetilde{\Delta}_g}$.
We show that this subgroup coincides with the whole group $G/\langle g\rangle$.

Since $X$ is connected and, by Remark~\ref{rem: D_g_generates_G} (and denoting by $d_g \coloneqq  \deg D_g$),
\[
G=\langle\, h\in G \mid d_h\neq 0\,\rangle,
\]
Let $h\in G$ with $d_h\neq 0$, then the curve $\Delta_h$ intersects $\Delta_g$ in at least one point,
since any two curves in $\mathbb P^2$ intersect.
If $h\in \langle g\rangle$, then $[h]=0$ in $G/\langle g\rangle$.
Otherwise, each point of $\Delta_g\cap \Delta_h$ has local monodromy $[h]$ for
the map $\widetilde{\pi}_{\vert \widetilde{\Delta}_g}$.
Hence we obtain
\[
G/\langle g\rangle
= \langle\, [h] \mid d_h\neq 0\,\rangle
\subseteq
\langle\, [t] \mid \Delta_g\cap \Delta_t\neq 0,\ t\notin \langle g\rangle\,\rangle,
\]
which shows that $\widetilde{\Delta}_g$ is connected.

Therefore $\widetilde{\Delta}_g$ is a smooth irreducible curve in $\widetilde{X}$.
Its genus is then computed by applying the Riemann--Hurwitz formula to the morphism
$\widetilde{\pi}_{\vert \widetilde{\Delta}_g}$.

    \textit{Proof of the Claim:} Let us consider a point $y\in \Delta_g$. We distinguish three cases:
    \begin{itemize}
        \item The point $y$ does not belong to any other component of the branch locus of $\pi$. In this case, the fibre $\pi^{-1}(y)$ consists of $p^{k-1}$ points of $X$, which are smooth by \cite[Prop. ~3.1]{Par91a}. Since the resolution is an isomorphism over smooth points, the fibre $\widetilde{\pi}^{-1}(y)$ also consists of $p^{k-1}$ points, hence $y$ is not a branch point;
        \item The point $y$ belongs to another component $\Delta_t\neq \Delta_g$ of the branch locus of $\pi$, and $t\in \langle g\rangle$. Here, the fibre $\pi^{-1}(y)$ still consists of $p^{k-1}$ points on $X$, which are cyclic quotient singularities by Proposition \ref{prop: type_of_sing}. Since $\widetilde{\Delta_g}$ intersects the Hirzebruch-Jung string $E_1^x\cup\dots \cup E_{l_y}^x$ at a single point on $E_1^x$, there is a one-to-one correspondence between the points of $\widetilde{\pi}_{\vert \widetilde{\Delta_g}}^{-1}(y)$ and $\pi^{-1}(y)$. Thus, $y$ is not a branch point.
        \item The point $y\in \Delta_t$, for some $\Delta_t\neq \Delta_g$, with $t\notin \langle g\rangle$. In this case,  the fibre $\pi^{-1}(y)$ consists of $\vert G/\langle g,t\rangle\vert=p^{k-2}$ of $X$, which are smooth for $X$ by \cite[Prop. ~3.1]{Par91a}. Consequently, the fibre $\widetilde{\pi}^{-1}(y)$ consists of $p^{k-2}$ points, making $y$ a branch point for $\widetilde{\pi}_{\vert \widetilde{\Delta_g}}$. 
        
        To determine the local monodromy of $y$, let $x\in \widetilde{X}$ lie over $y$. Since $\rho(x)$ is a smooth point for $X$ and the intersection of $\Delta_g$ and $\Delta_t$ is transversal, there exists an open neighborhood of $x$ with coordinates $(z_0,z_1)$ where $\widetilde{\Delta_g}=(z_0=0)$, $\widetilde{\Delta_t}=(z_1=0)$. The action is $g\colon (z_0,z_1)\mapsto (\exp\left(\frac{2\pi i}{p}\right)z_0,z_1)$ and $t\colon (z_0,z_1)\mapsto (z_0, \exp\left(\frac{2\pi i}{p}\right)z_1)$. The restriction to $\widetilde{\Delta_g}$ shows that $[t]$ acts locally as 
        $(0,z_1)\mapsto (0, \exp\left(\frac{2\pi i}{p}\right)z_1)$, identifying  $[t]$ as the local monodromy.
    \end{itemize} 
\end{proof}
\begin{remark}\label{rem: genus branch case Z/p}
	Notice that in case $k=1$, hence the group $G=\ZZ/p$, then $\Delta_g$ is isomorphic to $\widetilde{\Delta}_g$.
\end{remark}
\medskip
\subsection{Equations of $(p,k)$-planes in weighted projective spaces}\label{subsec: equations_p,k_plane_in_suitable_proj_planes}
Let us consider a collection of building data  $\{\mathcal L_\chi\}_{\chi \in G^*}$ and $\{D_g\}_{g\in G}$ on $\mathbb P^2$, with $G=(\mathbb Z/p)^k$. We conclude this section by constructing the associated $(p,k)$-plane $\pi\colon X\to \mathbb P^2$, realizing $X$ as a subvariety of a suitable weighted projective space, following the approach described in \cite{FP24}.  


Let  $f_g(x_0,x_1,x_2)$ be an homogeneous polynomial of degree $d_g$ in the variables $x_0,x_1,x_2$, whose zero locus is the support of the divisor $D_g$ on $\mathbb P^2$. 

We consider the weighted projective space $\mathbb P\left(1,1,1, (l_{\chi})_{\chi\in G^*\setminus 1}\right)$ whose the first three variables are $x_0,x_1,x_2$, and the remaining ones are $(y_{\chi})_{\chi\in G^*\setminus 1}$  (e.g., for $(p,k)=(2,2)$, we consider the weighted projective space $\mathbb P(1,1,1, l_{\epsilon_1}, l_{\epsilon_2}, l_{\epsilon_1+\epsilon_2})$ with coordinates $(x_0,x_1,x_2,$ $ y_{\epsilon_1}, y_{\epsilon_2}, y_{\epsilon_1+\epsilon_2})$).  

Notice that the projection on the first three variables $\text{pr}\colon \mathbb P\left(1,1,1, (l_{\chi})_{\chi\in G^*\setminus 1}\right)\setminus \{x_0=x_1=x_2=0\}\to \mathbb P^2$ is the vector bundle of $\mathbb P^2$ of rank $p^k-1$ corresponding to $\bigoplus_{\chi\in G^*\setminus 1}\mathcal O_{\mathbb P^2}(l_\chi)$.

We have a natural action of the group $G=(\mathbb Z/p)^k$: 
\[
g\cdot (x_0,x_1,x_2, (y_{\chi})_{\chi\in G^*\setminus 1})=(x_0,x_1,x_2, (\chi(g)\cdot y_{\chi})_{\chi\in G^*\setminus 1}), \qquad g\in G. 
\]
We define $X\subseteq \mathbb P\left(1,1,1, (l_{\chi})_{\chi\in G^*\setminus 1}\right)\setminus \{x_0=x_1=x_2=0\}$ in the following way
\[
X\coloneqq \left\{
y_{\chi}y_{\eta}=y_{\chi\eta}\cdot  \prod_{g\in G} f_g^{\left \lfloor \frac{\chig+\etag}{p} \right \rfloor}, \ \ \text{for all } \chi,\eta\in G^*\setminus 1\right\}.
\]
The action of $G$ naturally extends to $X$ and $\text{pr}_{\vert X}\colon  X\to \mathbb P^2$ is by construction a $(p,k)$-plane with building data $\{\mathcal L_\chi\}_{\chi \in G^*}$ and $\{D_g\}_{g\in G}$.

We refer to Example \ref{ex: h231} for an explicit application, where we find the equations in a suitable weighted projective space for a $(p,k)$-plane $X$ belonging to one of the families H-(2,3) in Table \ref{tab: (2,3)}.

The equations for the other SNC $(p,k)$-planes satisfying \hyperlink{Property $\bigstar$}{Property $\bigstar$} and belonging to one of the families of our classification can be obtained similarly.  
\begin{remark}
	We point out that using these equations we are able to give $X$ defined over $\mathbb Q$. 
\end{remark}
\section{The classification of SNC $(p,k)$-planes with Property~$\bigstar$}
\label{sec: classification_SNC_p,k_planes_Prop_STAR}
We present here the tools and constraints that will be used in the classification of
SNC $(p,k)$-planes 
\[
\pi \colon X \longrightarrow \PP^2
\]
with $p$ prime number and $k\geq 1$ satisfying the extra condition: \hyperlink{Property $\bigstar$}{Property $\bigstar$}.

In particular, we derive conditions on the branch divisors and the associated numerical
invariants that will play a central role in the classification.

The number of subgroups in an elementary abelian \( p \)-group of rank \( k \) is the same as the number of subspaces of a \( k \)-dimensional vector space over the finite field with \( p \) elements. This number can be computed using the Gaussian binomial coefficient
\[
\left[\!\!{k \atop m}\!\!\right]_{p} = \prod_{i=0}^{m-1} \frac{p^k - p^i}{p^m - p^i},
\]
where \( m \) is the dimension of the subspaces to be counted. In particular, subgroups of order \( p \) correspond to subspaces of dimension \( 1 \), and their number is 
\[
\phi_{p,k} \coloneqq  \frac{p^k - 1}{p - 1}.
\] 
Moreover, the number of subgroups of index \( p \) in \( G \cong (\mathbb Z/p)^k\) is equal to the number of subgroups of order \( p \) simply by projective duality, hence also \( \phi_{p,k} \).

For each subgroup \(H \le G\) there is an intermediate quotient \( Y_H\coloneqq  X/H \). We focus on the \( \phi_{p,k} \) maximal subgroups \(H\cong (\mathbb Z/p)^{k-1} \subset G\) and their intermediate  covers $Y_{H} \longrightarrow \mathbb{P}^2$ with Galois group $G/H\cong \mathbb Z/p$. For these intermediate covers 
it holds the following lemma.

\begin{lemma}\label{lem: interm_building_data}
Let $\pi\colon X \to \mathbb P^2$ be a SNC $(p,k)$-plane with building data
$\{\mathcal L_\chi\}_{\chi\in G^*}$ and $\{D_g\}_{g\in G}$, so branched along the divisor $D=\sum_{g\in G} D_g$.
Let $H \le G$ be a subgroup of index $p$, and let $\epsilon \in G^*$ be a nontrivial character such that $\ker\epsilon = H$. Thus, we identify $G/H$ with $\mathbb Z/p$ via the character $\epsilon$.
Then the induced quotient map
\[
\pi_H\colon X/H=Y_H \longrightarrow \mathbb P^2
\]
is a cyclic $G/H\cong (\mathbb Z/p)$-cover whose branch divisor is again SNC and can be expressed as
\[
D^{(H)} \coloneqq  \sum_{\substack{g\in G\\ \langle \epsilon, g\rangle \neq 0}} D_g.
\]
Moreover, the building data $\{\mathcal L^{(H)}_{\overline{j}}\}_{\overline{j}\in \mathbb Z/p}$ and $\{D_{\overline{j}}^{(H)}\}_{{\overline{j}}\in \mathbb Z/p}$ of such intermediate cover are 
\[
  L_{\overline{j}}^{(H)}\equiv_{lin} L_{\epsilon^j}, \qquad \makebox{and} \qquad D_{\overline{j}}^{(H)}=\sum_{\substack{g\in G\\ \langle\epsilon,g\rangle\equiv j\;(\mathrm{mod}\;p)}} D_g, \qquad \overline{j}\in \mathbb Z/p.
\]
\end{lemma}
\begin{proof}
Since $H$ has index $p$ in $G$ and $G$ is elementary abelian, there exists a nontrivial
character $\epsilon\in G^*$ with $\ker\epsilon = H$, and the quotient group $G/H$ is naturally identified with $\mathbb Z/p$ via the homomorphism $\epsilon \colon G\to \mathbb C^\times$.
The quotient map $\pi_H\colon X/H \to \mathbb P^2$ is therefore a cyclic $(\mathbb Z/p)$-cover. Moreover the branch divisor of this cover is given by a subset of the components of the divisor $D$, hence is again SNC. 

A component $D_g$ contributes to the branch locus of $\pi_H$ if and only if the image of
$g$ in $G/H$ is nontrivial, equivalently if $\epsilon(g)\neq 1$, namely  $\langle \epsilon , g\rangle\neq 0$.
Hence the branch divisor of $\pi_H$ is the divisor $D^{(H)}$ defined above. Similarly, a straightforward local computation shows that the divisor $D_g$ has local monodromy $\overline{j}\in (\mathbb Z/p)$ for the intermediate cover $\pi_H$ if and only if $\langle \epsilon, g\rangle\equiv j  \mod p$. Thus, the divisor $D_{\overline{j}}^{(H)}$ with local monodromy  $\overline{j}\in \mathbb Z/p$ is exactly that of the statement. 

Finally, applying Pardini’s equation for the character $\overline{i}\in (\mathbb Z/p)$ yields
\[
\begin{split}
	p L^{(H)}_{\overline{i}}&
	\equiv_{lin} 
	\sum_{\overline{j}\in \mathbb Z/p} [i\cdot j]_pD_j^{(H)}\\ 
	& \equiv_{lin} \sum_{\overline{j}\in \mathbb Z/p} \sum_{\substack{g\in G\\ \langle\epsilon,g\rangle\equiv j\;(\mathrm{mod}\;p)}} [i\cdot j]_pD_g\\ 
	& \equiv_{lin}\sum_{g\in G} \langle \epsilon^i,g\rangle D_g\\ & \equiv_{lin} pL_{\epsilon^i},
\end{split}
\]
where the latter is obtained from Pardini's equation for the character $\epsilon^i\in G^*$. Thus, $p L^{(H)}_{\overline{i}}\equiv_{lin}pL_{\epsilon^i}$ and since $\mathbb P^2$ has a Picard group without torsion, then $L^{(H)}_{\overline{i}}\equiv_{lin} L_{\epsilon^i}$.  
\end{proof}
By considering the Hasse diagram of subgroups of \(G\) and reversing it to reflect the covering structure, we obtain Diagram~\ref{diag_HesseRev}.  
In what follows we analyze the bottom row, i.e.\ the covers lying immediately above \( \mathbb{P}^2 \).

First of all, we need to translate the geometric \hyperlink{Property $\bigstar$}{Property $\bigstar$} in terms of the building data $\{D_g\}_{g\in G}$ and $\{\mathcal L_{\chi}\}_{\chi\in G^*}$ of the cover $\pi$. 

In the projective plane \( \mathbb{P}^2 \), we denote the degrees of the line bundles $\mathcal L_\chi$ and divisors $D_g$ by
\[
l_\chi \coloneqq  \deg \mathcal L_\chi, \qquad \chi \in G^*, 
\qquad\text{and}\qquad 
d_g \coloneqq  \deg D_g, \qquad g \in G.
\]
\begin{proposition}\label{prop: equivalence_star_prop}
    Let $\pi\colon X\to \mathbb P^2$ be a SNC $(p,k)$-plane with $k\geq 1$ and building data $\{\mathcal L_\chi\}_{\chi\in G^*}$ and $\{D_g\}_{g\in G}$. Then \hyperlink{Property $\bigstar$}{Property $\bigstar$} holds if and only if the degrees $\{l_\chi\}_{\chi\in G^*}$ and $\{d_g\}_{g\in G}$ verify the following conditions:
    \begin{enumerate}
        \item $1\leq l_\chi\leq 3$ for each non trivial character $\chi \in G^*$;
        \item there exists $\chi \in G^*$ such that $l_\chi=3$;
        \item \begin{itemize}
            \item if $p=2$, for each non-trivial character $\chi\in G^*$ with $l_\chi=3$, we have $\sum_{g\notin \ker \chi}d_g=6$;
            \item if $p=3$, for each (proper) cyclic subgroup $\langle \chi\rangle\leq G^*$ there exists at most one $\chi$ such that $l_\chi=3$, in which case $l_{\chi^2}=2$, and $\sum_{g\notin \ker \chi}d_g=5$;
            \item if $p\geq 5$, for each (proper) cyclic subgroup $\langle \chi\rangle\leq G^*$ there exists at most one $\chi$ such that $l_\chi=3$, in which case $l_{\chi^{p-1}}=1$, $l_{\chi^i}=2$ for any $i\neq 0,1,p-1$, and $\sum_{g\notin \ker \chi}d_g=4$.
        \end{itemize}
    \end{enumerate}
\end{proposition}
\begin{proof}
    Let $\chi\in G^*$ be a non-trivial character of $G$, and let $H:=\ker \chi $ be its subgroup of index $p$ in $G$.  
By Lemma \ref{lem: interm_building_data} the intermediate quotient $\pi_H\colon Y_{\ker \chi } \longrightarrow \mathbb P^2$ is a SNC $(p,1)$-plane with building data $\{D^{(H)}_{\overline{j}}\}_{\overline{j}\in \mathbb Z/p}$ and $\{\mathcal L_{\chi^i}\}_{i \in \{0, \ldots ,p-1\}}$. Therefore, we can apply Theorem \ref{thm: invariants_ab_cov_cyclic_sing} to $\pi_H$:
   \begin{equation}\label{eq: geom_genus_int_quotient}
       p_g(Y_{\ker \chi })=\sum_{i=1}^{p-1} h^0\left(\mathbb P^2, \mathcal O_{\mathbb P^2}(-3+l_{\chi^i})\right)=\sum_{i=1}^{p-1}\frac{(l_{\chi^i}-1)(l_{\chi^i}-2)}{2}.
   \end{equation} 
   Assuming Property~$\bigstar$, for any nontrivial character
$\chi \in G^*$ the induced $(p,1)$-plane $Y_{\ker \chi}$
necessarily has geometric genus at most one. Equation \eqref{eq: geom_genus_int_quotient} implies $l_{\chi^i}\leq 3$ for all $i$. Specifically, if   $Y_{\ker \chi}$ has geometric genus one, exactly one $l_{\chi^j}$ equals to $3$, while the others satisfy $l_{\chi^i}\leq 2$, for $j \ne i$. Furthermore, $l_\chi\geq 1$ as $X$ is connected, see Remark \ref{rem: X_connected_implies_L_chi-1_no_sections}. This proves $(1)$. Moreover, $(2)$ directly follows from Property~$\bigstar$, which requires that at least one of the intermediate
$(p,1)$-planes $Y_{\ker \chi}$ has geometric genus equal to one. 

Let us prove point $(3)$ with a case-by-case analysis depending on the prime $p$.

   
   First, assume $p=2$, and $l_\chi=3$. Then, by Lemma \ref{lem: L_chi_in_function_D_g} applied for the character $\chi$, we have 
   \[
     \sum_{g\notin \ker \chi}d_g=2l_\chi=6. 
   \]
   Next, assume $p=3$ and $l_\chi=3$, so $l_{\chi^2}\leq 2$. From Pardini's equations \eqref{eq: Pardini_eq}, we have
   \[
     \begin{split}
         l_\chi+l_{\chi^2}&= \sum_{g\notin \ker \chi}d_g, \qquad \makebox{and}\\
         6 =l_\chi+l_\chi & =l_{\chi^2}+\sum_{g\in G}\left\lfloor\frac{2\chig}{3}\right\rfloor d_g\leq l_{\chi^2}+\sum_{g\notin \ker \chi}d_g.
     \end{split}
   \]
   Combining these, we obtain 
   \[
   6\leq l_{\chi^2} +\sum_{g\notin \ker \chi}d_g=3+2l_{\chi^2},
   \] 
   which forces $l_{\chi^2}\geq \frac{3}{2}$, hence $l_{\chi^2}=2$, and consequently  $\sum_{g\notin \ker \chi}d_g=5$. 
   
   Finally, assume $p\geq 5$ and $l_\chi=3$, so $l_{\chi^i}\leq 2$ for each $i\neq 1$. By contradiction, if $l_{\chi^{p-1}}=2$, then Pardini's equations \ref{eq: Pardini_eq} imply
   \[
   l_{\chi^2}+l_{\chi^{p-2}}=\sum_{g\notin \ker \chi} d_g=l_{\chi}+l_{\chi^{p-1}}\geq 5, 
   \]
   which is not possible as both $l_{\chi^2}, l_{\chi^{p-2}}\leq 2$. Thus, $l_{\chi^{p-1}}=1$ and $\sum_{g\notin \ker \chi} d_g=4$. For $2\leq i\leq p-2$, the characters $\chi^i$ and $\chi^{p-i}$ are distinct from $ \chi, \chi^{p-1}$. Using Pardini equations again, we find $l_{\chi^i}+l_{\chi^{p-i}}=\sum_{g\notin \ker \chi} d_g=4$,
   which forces $l_{\chi^i}=l_{\chi^{p-i}}=2$. This proves point $(3)$. 
   
   Conversely, by Equation \ref{eq: geom_genus_int_quotient}, points $(1)-(3)$  imply \hyperlink{Property $\bigstar$}{Property $\bigstar$}.  
\end{proof}
Let us discuss the classification, distinguishing between the cases $k=1$ and $k\geq 2$.
\begin{remark}
For $k=1$, we observe that SNC $(p,1)$-planes $\pi\colon X \to \mathbb P^2$ satisfying
Property~$\bigstar$ have geometric genus $p_g(X)=1$.
In general, the surface $X$ need not be smooth; indeed, as discussed in the
previous section, it is a normal surface with only cyclic quotient singularities.

In this situation, the minimal model of $X$ is a K3 surface, as will be shown in Section~\ref{sec: SNC_(p,1)-planes},  Theorem \ref{thm: pg1_Zp_cover_structure}.
Accordingly, with a slight abuse of notation, we shall refer to $X$ itself as
a \emph{K3 surface} whenever its minimal model is a K3 surface. Such covers will be also classified in Section~\ref{sec: SNC_(p,1)-planes}, Proposition \ref{prop: complete_classification_p,1_planes_pg=1} 
\end{remark}
From now on we assume $k\ge 2$.
Let $\pi\colon X \to \mathbb P^2$ be a SNC $(p,k)$-plane.
By Remark~\ref{rem: D_g_generates_G}, the set of local monodromies
\[
\{\, g \in G \mid d_g \neq 0 \,\}
\]
generates the whole group $G \cong (\mathbb Z/p)^k$. 

We denote from now on by $e_1,\dots,e_k$ a set of generators of $G$ such that $d_{e_i}\geq 1$ for each $i=1,\dots,k$, and by  $\epsilon_1,\dots, \epsilon_k$ the corresponding dual set of generators of $G^*$, i.e. those characters for which $\langle \epsilon_i,e_j\rangle= \delta_{ij}$, for all $i,j=1,\dots, k$, where $\delta_{ij}$ is the Kronecker delta.

Since the rank of $G$ is $k$, this immediately yields the rough inequalities
\begin{equation}\label{eq: rough_ineq_k}
2 \le k \le \bigl|\{\, g \in G \mid d_g \neq 0 \,\}\bigr|
\le \sum_{g\in G} d_g .
\end{equation}
We now show that the numerical constraints \eqref{eq: rough_ineq_k} strongly restrict the possible building data of a SNC $(p,k)$-plane. 
In particular, either the cover belongs to a very special and explicit class, or the total degree of the branch locus must satisfy stronger lower bounds than \eqref{eq: rough_ineq_k}. 
\begin{proposition}\label{prop: hirzebruch_planes}
Let $\pi\colon X \to \mathbb P^2$ be a SNC $(p,k)$-plane with building data $\{\mathcal L_\chi\}_{\chi\in G^*}$ and $\{D_g\}_{g\in G}$. 
Then one of the following holds:
\begin{enumerate}
    \item $X$ is a smooth $(p,k)$-plane whose building data satisfy
    \[
    d_{e_1}=\cdots=d_{e_k}=d_{(p-1)\sum_{i=1}^k e_i}=1,
    \qquad \makebox{and} \qquad 
    d_g=0 \text{ for all other } g\in G.
    \]
    In this case, $X$ satisfies Property~$\bigstar$ if and only if
    \[
    (p,k)\in \{(2,5),(2,6),(3,4),(5,3)\}.
    \]
    \item The branch data satisfy the numerical inequalities
    \[
    k+2 \le \sum_{g\in G} d_g
    \qquad\text{and}\qquad
    k \le \bigl|\{\, g\in G \mid d_g\neq 0 \,\}\bigr|.
    \]
\end{enumerate}
\end{proposition}

\begin{proof}

We first consider the extremal situation where the rough inequality
\eqref{eq: rough_ineq_k} is sharp, that is, when $\sum_{g\in G} d_g = k$.
Then \eqref{eq: rough_ineq_k} forces
\[
k = \bigl|\{\, g\in G \mid d_g\neq 0 \,\}\bigr| = \sum_{g\in G} d_g.
\]
In particular, we must have 
\[
d_{e_1}=\cdots=d_{e_k}=1,
\qquad
d_g=0 \text{ for all } g\in G\setminus\{e_1,\dots,e_k\}.
\]
By Lemma~\ref{lem: L_chi_in_function_D_g} applied to $\epsilon_1$, we obtain
\[
p\cdot L_{\epsilon_1} \equiv_{lin} \sum_{g\in G}\langle \epsilon_1,g\rangle\,D_g
= D_{e_1},
\]
Passing to degrees on $\mathbb P^2$ yields 
\[
pl_{\epsilon_1}= d_{e_1}=1,\]
a contradiction.

Thus, by \eqref{eq: rough_ineq_k}, we must have $k+1 \le \sum_{g\in G} d_g$.
We now consider the case $\sum_{g\in G} d_g = k+1$.
Again from \eqref{eq: rough_ineq_k}, there are two possibilities for the cardinality
$\bigl|\{\, g\in G \mid d_g\neq 0 \,\}\bigr|$, namely either $k$ or $k+1$.
If it was equal to $k$, we would obtain a contradiction by the same argument as in the previous case.

We therefore assume that
\[
\bigl|\{\, g\in G \mid d_g\neq 0 \,\}\bigr| = k+1,
\]
that is, there exists exactly another element $h$ other than $e_1,\dots,e_{k}$  such
that
\[
d_{e_1}=\cdots=d_{e_k}=d_h=1,
\qquad \makebox{and} \qquad 
d_g=0 \text{ for all } g\in G\setminus\{e_1,\dots,e_{k},h\}.
\]
By Lemma~\ref{lem: L_chi_in_function_D_g}, for each $1\le j\le k$ we have
\[
p\,L_{\epsilon_j}
= \sum_{i=1}^{k} \langle \epsilon_j, e_i\rangle  \, D_{e_i}+ \langle \epsilon_j,h\rangle \, D_h,
\]
which implies passing to the degrees
\[
pl_{\epsilon_j}= 1 + \langle \epsilon_j, h \rangle
\le p,
\]
thus
\[
l_{\epsilon_j}=1
\qquad\text{and}\qquad
\langle \epsilon_j, h \rangle=p-1, \quad \makebox{ for each } j=1,\dots, k.  
\]
It follows that there is only one element $h$ for which Pardini's equations \eqref{eq: reduction_Pardini_eq} hold, and it is  equal to 
\[
h
= \sum_{j=1}^k \langle \epsilon_j, h\rangle\, e_j
= (p-1)\sum_{j=1}^k e_j,
\]
as claimed.
This yields case~(1) of Proposition~\ref{prop: hirzebruch_planes}.
If this situation does not occur, then necessarily $\sum_{g\in G} d_g \ge k+2$, which is
case~(2).

Notice that in case~(1), the surface $X$ is smooth, since the branch divisor consists of $k+1$
distinct lines in $\mathbb P^2$ with simple normal crossings satisfying Condition \eqref{eq_condSmooth_0}.

We now turn to a detailed analysis of the exceptional situation described in
case~(1), 
under the additional assumption that $X$ satisfies Property~$\bigstar$.
In order to exploit the numerical constraints, let us consider the following character 
\[
\chi \coloneqq  (p-1)\sum_{i=1}^k \epsilon_i.
\]
A direct computation of Pardini's equation \eqref{eq: reduction_Pardini_eq} for the character $\chi$ yields
\begin{equation}\label{eq: rank_and_prime}
pl_\chi = k(p-1)+[k]_p.
\end{equation}
Moreover, since $X$ satisfies Property~$\bigstar$, we have $\,l_\chi \le 3$, and hence
\begin{equation*}\label{eq: contraints_on_k}
3p \ge p\,l_\chi = k(p-1) + [k]_p \ge k(p-1),
\end{equation*}
which immediately implies
\begin{equation}\label{eq: contraints_on_k}
k \le 3 + \frac{3}{p-1}.
\end{equation}
Let $\eta\in G^*$ be a character such that $l_\eta=3$.
By the bound \eqref{eq: contraints_on_k} on $k$ and
Proposition~\ref{prop: equivalence_star_prop} (3), we obtain
\[
4+\frac{3}{p-1}
\;\ge\;
k+1=\sum_{g\in G} d_g
\;\ge\;
\sum_{g\notin \ker \eta} d_g
=
\begin{cases}
6 & \text{if } p=2,\\
5 & \text{if } p=3,\\
4 & \text{if } p\ge 5.
\end{cases}
\]
It follows that the only possible pairs $(p,k)$ are
\[
(p,k)\in \{(2,5),(2,6),(3,4)\}\,\cup\,\{(p,3)\mid p\ge 5\}.
\]

One can directly verify that $(2,5)$, $(2,6)$, $(3,4)$ satisfy  \hyperlink{Property $\bigstar$}{Property $\bigstar$}. 

Consider now the case $(p,3)$ with $p\ge 5$.
In this situation, the building data are given by
\[
d_{e_1}=d_{e_2}=d_{e_3}
= d_{(p-1)\sum_{i=1}^3 e_i}=1,
\qquad
d_g=0 \text{ for all remaining } g\in G.
\]
Thus, from Equation \eqref{eq: rank_and_prime}, we have for $\chi = (p-1)\sum_{i=1}^3 \epsilon_i$ that
\[
l_\chi = \frac{3(p-1)+[3]_p}{p} = 3,
\]
so that $p_g(Y_{H})=1$, where $H:=\ker \chi$. 

Now, using the notation introduced in Lemma \ref{lem: interm_building_data} we have that the building data of the intermediate $(\mathbb Z/p)$-cover $\pi_H\colon Y_H\to \mathbb P^2$ are
\[
D^{(H)}_{\bar j}
= \sum_{\langle\chi,g\rangle \equiv j \,(\mathrm{mod}\, p)} D_g, \qquad \makebox{for} \quad \overline{j}\in \mathbb Z/p. 
\]
In the present case, $\langle\chi,e_i\rangle=p-1$ for $i=1,2,3$, while
$\langle\chi,(p-1)\sum_{i=1}^3 e_i\rangle=3$.
Therefore,
\[
D^{(H)}_{\overline{p-1}}=D_{e_1}+D_{e_2}+D_{e_3},
\qquad \makebox{and} \qquad 
D^{(H)}_{\overline{3}}=D_{(p-1)\sum_{i=1}^3 e_i},
\]
and $D^{(H)}_{\bar j}=0$ for all other $\bar j\in\mathbb Z/p$.
According to Proposition \ref{prop: complete_classification_p,1_planes_pg=1}, the unique SNC $(p,1)$-plane of geometric genus
one with these building data occurs for $p=5$. More precisely, this is the case $A-(5,1)(3)$ of Table \ref{tab: (5,1)}. 
Consequently, the only pair $(p,3)$ with $p\geq 5$ satisfying Property~$\bigstar$ is $(5,3)$. 
\end{proof}

The $(p,k)$-planes appearing in the previous Proposition~\ref{prop: hirzebruch_planes}(1) form a
distinguished class that was extensively studied by Hirzebruch~\cite{H83}.
Therefore, we introduce the following terminology.

\begin{definition}
A $(p,k)$-plane $\pi\colon X \to \mathbb P^2$ is said to be \emph{of Hirzebruch type}
if its building data satisfy the conditions of Proposition~\ref{prop: hirzebruch_planes}(1).
\end{definition}

The next proposition provides a decisive bound on the parameters $(p,k)$ and the degrees of the components of the branch divisor.
In particular, it implies that only finitely many families of $(p,k)$-planes can satisfy
Property~$\bigstar$, making a complete classification possible.
\begin{proposition}\label{prop: final_bound_on_p,k,sum_dg}
Let $\pi\colon X \to \mathbb P^2$ be an SNC $(p,k)$-plane with $k\geq 2$ and building data
$\{\mathcal L_\chi\}_{\chi\in G^*}$ and $\{D_g\}_{g\in G}$, not of Hirzebruch type.
If $\pi$ satisfies Property~$\bigstar$, then the following conditions hold:
\begin{enumerate}
    \item The parameters $(p,k)$ satisfy
    \[
    2 \le p \le 19
    \qquad\text{and}\qquad
    2 \le k \le 3+\frac{3}{p-1}.
    \]
    \item The number of nonzero branch divisors $D_g$ satisfies
    \[
    \bigl|\{\, g\in G \mid d_g\neq 0 \,\}\bigr| \ge k.
    \]
    \item For every $g\in G$, one has
    \[
    d_g \le 3+\frac{3}{p-1}.
    \]
    Moreover, if $(p,k)=(2,6)$, then $d_g\in\{0,1\}$ for all $g\in G$.
    \item The total degree of the branch divisor satisfies the following bounds,
    depending on $p$:
    \begin{itemize}
        \item[$p=2:$] either $(p,k)=(2,2)$ and
        \[
        d_{e_1}=d_{e_2}=d_{e_1+e_2}=3,
        \]
        in which case $X$ is a SNC degeneration of the Garbagnati surface of type~$G3$, see \cite{Lat22}, 
        or else
        \[
        \max\{6,k+2\} \le \sum_{g\in G} d_g \le 12-2^{\,4-k}.
        \]
        \item[$p=3:$]
        \[
        \max\{5,k+2\} \le \sum_{g\in G} d_g
        \le \frac{5}{2}\left(3-\frac{1}{3^{k-1}}\right).
        \]
        \item[$p\ge 5:$] $\sum_{g\in G} d_g = 4$ and $k=2$.
    \end{itemize}
\end{enumerate}
\end{proposition}
\begin{proof}
    The bound $2\le p\le 19$ follows directly from Property~$\bigstar$ since the $(p,k)$-plane would admit an intermediate SNC $(p,1)$-plane with geometric genus one, and we classified them in Proposition \ref{prop: complete_classification_p,1_planes_pg=1}. 
    
Next, we establish a bound on $k$.
Applying Lemma~\ref{lem: L_chi_in_function_D_g} to the character
$(p-1)\sum_{i=1}^k \epsilon_i$, we obtain
\begin{equation}\label{eq: bound_k}
\begin{aligned}
3p
&\ge p\, l_{(p-1)\sum_{i=1}^k \epsilon_i} \\
&\ge (p-1)\sum_{i=1}^k d_{e_i}
   + \sum_{g\notin \ker(\sum_{i=1}^k \epsilon_i)} d_g \\
&\ge (p-1)k .
\end{aligned}
\end{equation}
It follows that
\[
2 \le k \le 3+\frac{3}{p-1}.
\]
This proves (1). We have already observed (2) in \eqref{eq: rough_ineq_k}. So let us work on (3).

    We first address the exceptional subcase $(p,k)=(2,6)$.
In this case the inequality \eqref{eq: bound_k} reads
\[
6
\ge 2\,l_{\sum_{i=1}^6 \epsilon_i}
=  \sum_{g\notin \ker(\sum_{i=1}^6 \epsilon_i)} d_g \ge \sum_{i=1}^6 d_{e_i} 
\ge 6,
\]
hence all inequalities are equalities and in particular $d_{e_1}=\dots =d_{e_k}=1$, and if $d_h\neq 0$, then either $h$ is one among $e_1,\dots, e_k$, or we have
\[
\Bigl(\sum_{i=1}^6 \epsilon_i\Bigr)(h)=0.
\]
Equivalently, $h$ is the sum of an even number of basis elements:
$h=e_{i_1}+\cdots+e_{i_{2j}}$.
 
Fix such a $h\neq 0$ and choose an index $i_1$ appearing in its expression.
Set
\[
\epsilon \coloneqq  \epsilon_{i_1} + \sum_{j=1}^6 \epsilon_j \in G^*.
\]
Then $\epsilon(e_{i_1})=1$, while $\epsilon(e_i)=-1$ for every $i\neq i_1$.
Moreover, since $h$ is a sum of an even number of basis elements containing $e_{i_1}$,
we have $\epsilon(h)=-1$,  hence $h\notin \ker(\epsilon)$.
Applying Lemma~\ref{lem: L_chi_in_function_D_g} to $\epsilon$ we obtain
\[
2\,l_\epsilon
= \sum_{h\notin \ker\epsilon} d_h
\ge \sum_{i\neq i_1} d_{e_i} + d_h
\ge 5 + d_h.
\]
Since Property~$\bigstar$ gives $l_\epsilon\le 3$, we conclude
\[
6 \ge 2\,l_\epsilon \ge 5 + d_h,
\]
and hence $d_h\le 1$.
In particular, for $(p,k)=(2,6)$ every nonzero branch component has degree $d_g\in\{0,1\}$.

    Let us turn to the proof of $(3)$. We now bound the degrees $d_g$ by a similar argument.
Fix $g\in G$, and choose a character $\chi\in G^*$ such that
$\langle \chi,g\rangle = p-1$. 
By Lemma~\ref{lem: L_chi_in_function_D_g}, we have
\[
3p \ge p\,l_\chi
= \sum_{h\in G} \langle \chi,h\rangle\, d_h
\ge (p-1)d_g,
\]
and therefore
\[
d_g \le 3+\frac{3}{p-1}.
\]

    To prove~(4), let us assume first that $l_\eta=3$ for any non trivial character $\eta\in G^*$. For any non-zero $g\in G$, Corollary \ref{cor: sum_<g> D_h equal to} yields
    \[
      \begin{split}
          \frac{p^{k-1}(p-1)}{2}\sum_{h\in \langle g\rangle} d_h & =3\vert G\setminus \ker g \vert -(p-1)\left(0+3(\vert \ker g\vert-1)\right)\\
          &= 3p^k-3p^{k-1}-3(p-1)p^{k-1}+3(p-1)\\
          &= 3(p-1),
      \end{split}
    \]
    which implies 
    \[
     \sum_{h\in \langle g\rangle} d_h = \frac{6}{p^{k-1}}.
    \]
    Since the left-hand side is an integer and $k\geq 2$ , this forces $(p,k)=(2,2)$ or $(3,2)$. The case $(2,2)$ gives $d_g=3$ for all $g\neq 0$, which are the building data defining either the Garbagnati surface of type G3 or its degenerations.  Conversely, $(p,k)=(3,2)$ is automatically excluded by Proposition \ref{prop: equivalence_star_prop} because for $p=3$ any character $\chi$ with $l_\chi=3$ satisfies $l_{\chi^2}=2$ whilst we were assuming that all non-trivial characters $\eta\in G^*$ have $l_\eta=3$. 
    
 Let us conclude the proof of~(4) by assuming the converse, namely that there exists at least one non-trivial character
$\chi\in G^*$ such that $l_\chi\le 2$.
    
   For $p=2$, we have 
\[
\sum_{\chi\in G^*} l_\chi \le 0+2+3(2^{k}-2)=4\bigl(3\cdot 2^{k-2}-1\bigr).
\]  
For $p=3$,  Proposition \ref{prop: equivalence_star_prop} point (3) provides more information: for each proper cyclic subgroup $\langle \chi\rangle\leq G^*$  with $l_\chi=3$, we have $l_{\chi^2}=2$. Thus, 
    \[
    \sum_{t=0}^2 l_{\chi^t}\leq 0+2+3=5,
    \]
    and we can conclude 
    \[
    \sum_{\chi\in G^*} l_\chi=0+\sum_{\langle \chi\rangle \in \mathbb P((\mathbb Z/3^k)}\left(\sum_{\eta \in \langle \chi\rangle} l_\eta\right)\leq 5\cdot \phi_{3,k}=5\cdot \frac{3^k-1}{2}.
    \]

    Similarly, for $p\geq 5$, Proposition \ref{prop: equivalence_star_prop} point (3) implies that each proper subgroup $\langle \chi\rangle\leq G^*$ with $l_\chi=3$ has $l_{\chi^{p-1}}=1$ and $l_{\chi^i}=2$ otherwise; hence \[
    \sum_{i=0}^{p-1}l_{\chi^i}\leq 0+1+3+2(p-3)=2(p-1).
    \]Consequently, 
    \[
    \sum_{\chi\in G^*} l_\chi=0+\sum_{\langle \chi\rangle \in \mathbb P((\mathbb Z/3)^k)}\left(\sum_{i=0}^{p-1} l_{\chi^i}\right)\leq 2\cdot\phi_{p,k}\cdot (p-1)=2(p^{k}-1).
    \]
    
    Finally, since $X$ is assumed to be not of Hirzebruch type, Proposition \ref{prop: hirzebruch_planes} implies $k+2\leq \sum_{g}d_g$ and $k\leq \vert \{g\in G\vert d_g\neq 0\}\vert $. Applying Theorem \ref{thm: Dg_elm_group}\eqref{eq: sum_dg_sm_lchi}, we obtain 
\begin{equation}\label{eq: bound_sum_dg1}
\sum_{g\in G} d_g
= \frac{2}{p^{k-1}(p-1)} \sum_{\chi\in G^*} l_\chi
\le
\begin{cases}
\displaystyle
\frac{2}{2^{k-1}}\cdot 4\bigl(3\cdot 2^{k-2}-1\bigr) & \text{if } p=2,\\[6pt]
\displaystyle
\frac{2}{3^{k-1}\cdot 2}\cdot 5\cdot \frac{3^k-1}{2} & \text{if } p=3,\\[8pt]
\displaystyle
\frac{2}{p^{k-1}(p-1)}\cdot 2(p^k-1) & \text{if } p\ge 5.
\end{cases}
\end{equation}
In particular,
\begin{equation}\label{eq: bound_sum_dg}
k+2 \le \sum_{g\in G} d_g \le
\begin{cases}
\displaystyle
12-2^{\,4-k} & \text{if } p=2,\\[6pt]
\displaystyle
\frac{5}{2}\left(3-\frac{1}{3^{k-1}}\right) & \text{if } p=3,\\[10pt]
\displaystyle
4+\frac{4}{p-1}-\frac{4}{p^{k-1}(p-1)} & \text{if } p\ge 5.
\end{cases}
\end{equation}

    Note that for $p\geq 5$, since $k\geq 2$, then  $\sum_gd_g$ is forced to be equal to $4$, which implies directly $k=2$. Combining 
    \eqref{eq: bound_sum_dg}  with the constraints from Proposition \ref{prop: equivalence_star_prop} on $\sum_gd_g$, we obtain point (4). 
\end{proof}

This completes the numerical analysis of SNC $(p,k)$-planes satisfying
Property~$\bigstar$. 
We are now in a position to carry out a classification. 

A classification of SNC $(p,1)$-planes satisfying Property $\bigstar$, so those with geometric genus one, will be achieved in Section \ref{sec: SNC_(p,1)-planes}, Proposition \ref{prop: complete_classification_p,1_planes_pg=1}. Thus, we only consider the case $k\geq 2$. In this case, Proposition~\ref{prop: final_bound_on_p,k,sum_dg} provides explicit numerical
constraints on $p$, $k$, and $\sum_{g\in G} d_g$ that reduce the problem to a finite
and computationally manageable set of cases.

We emphasize that Proposition~\ref{prop: final_bound_on_p,k,sum_dg} plays a crucial
role in controlling the computational complexity.
Without these bounds, for large values of $p$, $k$, and $\sum_{g} d_g$, there could
be millions of distinct ways to label the irreducible components of the branch
locus of a cover so that the reduced Pardini equations of
Theorem~\ref{thm: reduction_numb_Pard_eq} are satisfied. Moreover, many of these combinatorial configurations define isomorphic SNC
$(p,k)$-planes.
Indeed, an automorphism of the group $(\mathbb Z/p)^k$ may send the building data of
one cover to that of another, yielding isomorphic $(\mathbb Z/p)^k$-coverings of
$\mathbb P^2$.
\\
The classification is implemented algorithmically through a single fundamental
routine, which generates all admissible building data. We used the computational
algebra system \verb!MAGMA!~\cite{BCP97}. All routines, functions, auxiliary scripts used in the implementation of the
algorithm, and a commented version of the outputs are available online here
\begin{center}
	\href{https://github.com/Fefe9696/MAGMA_scripts_results_Classification_of_SNC_p_k_planes_with_Hodge_substructures_of_K3_type}{\texttt{MAGMA} scripts and results repository}
\end{center}
where the full source code and documentation can be found.

The algorithm proceeds as follows. The function 
\[
\verb!Building_Data(card,d,G)!\] takes as input a elementary $p$-abelian group $G\cong (\mathbb Z/p)^k$ 
and two positive integers \verb!card!, and \verb!d!.
It returns a list of vectors
\[
[(g,d_g)\colon g\in G,\ d_g\neq 0]
\]
satisfying the following conditions:
\begin{enumerate}
	\item for every character $\chi\in G^*$, the numbers
	\[
	l_\chi\coloneqq \frac{1}{p}\sum_{g\in G}\chig d_g
	\]
	are integers;
	\item the collections $\{l_\chi\}_{\chi\in G^*}$ and $\{d_g\}_{g\in G}$ satisfy
	the numerical conditions~(1), (2), and~(3) of
	Proposition~\ref{prop: equivalence_star_prop};
	\item $\bigl|\{\,g\in G \mid d_g\neq 0\,\}\bigr|=\verb!card!$;
	\item $\sum_{g\in G} d_g=\verb!d!$.
\end{enumerate}
As an explicit example, we apply the function
\verb!Building_Data(card,d,G)! to reproduce cases $G-(2,3)1)$ and $H-(2,3)1)$ of
Table~\ref{tab: (2,3)}, corresponding to $(p,k)=(2,3)$.
\\
\verb!G:=AbelianGroup([2,2,2]);!\\
\verb!Building_Data(5,8,G);! \\
\verb!{!\\
	\verb!    [!\\
	\verb!                <G.3, 3>,! \\
	\verb!                <G.1 + G.2, 1>,! \\
	\verb!                <G.2, 2>,! \\
	\verb!                <G.1, 1>,! \\
	\verb!                <G.2 + G.3, 1>!\\
	\verb!    ]! \\
	\verb!    [!\\
	\verb!               <G.3, 1>,! \\
	\verb!               <G.1 + G.2 + G.3, 1>,! \\
	\verb!               <G.2, 1>,! \\
	\verb!               <G.1, 1>,! \\
	\verb!               <G.2 + G.3, 4>! \\
	\verb!    ]! \\
	\verb!}! \\
\\
We refer to Section \ref{sec: examples_of_SNC_p,k_planes_Prop_STAR} where we study the example $H-(2,3)1)$ in full detail.
\vspace{0.2cm}

Given such an output vector, one then chooses a collection of (possibly not prime) divisors
$\{D_g\}_{g\in G}$ on $\mathbb P^2$ such that 
$\deg D_g=d_g$ and the divisor 
\[
D\coloneqq \sum_{g\in G} D_g
\]
is reduced, has smooth irreducible components, and has
simple normal crossings. By the properties (1)--(4) imposed above on the degrees $\{d_g\}_{g\in G}$, the divisors
$\{D_g\}_{g\in G}$ satisfy the reduced Pardini's equations
(cf.~Theorem~\ref{thm: reduction_numb_Pard_eq} and the
Remark \ref{rem: onlyD_g}),
and therefore define a SNC $(p,k)$-plane for which
Property~$\bigstar$ holds, from Proposition \ref{prop: equivalence_star_prop}. 

To avoid redundancy, the function filters out vectors that define isomorphic
coverings.
Consequently, the output consists of vectors defining SNC $(p,k)$-planes that are
pairwise non-isomorphic as $(\ZZ/p)^k$-coverings of $\mathbb P^2$.
\begin{remark}
	The relation of being isomorphic as $(\mathbb Z/p)^k$-coverings of $\mathbb P^2$ is stronger than the standard notion of isomorphism between surfaces. Consequently, it is possible for two SNC $(p,k)$-planes to be isomorphic as surfaces even if they are not isomorphic as $(\mathbb Z/p)^k$-coverings of $\mathbb P^2$. 
\end{remark}
The complete classification of SNC $(p,k)$-planes satisfying Property $\bigstar$ and with $k\geq 2$  is attained using the function \verb!Building_Data(card,d,G)! for each possible \verb!G!, \verb!card!, and \verb!d! according to the finite list of possibilities established in Proposition \ref{prop: final_bound_on_p,k,sum_dg}. 
\\
More precisely, we have that $2\leq p\leq 19$, and for $5\le p\le 19$, then Proposition~\ref{prop: final_bound_on_p,k,sum_dg}(2)
implies that $k=2$ and $\sum_{g\in G} d_g=4$.
Consequently, there are only three possible values for
\[
\verb!card! = \bigl|\{\, g\in G \mid d_g\neq 0 \,\}\bigr|,
\]
namely $2$, $3$, or $4$.
Running the function \verb!Building_Data(card,4,AbelianGroup([p,p]))! for each $p$ and each
$\verb!card!\in\{2,3,4\}$, we find that admissible cases occur only for $p=5$ and
$p=7$.

For the rest of cases, namely $p\in\{2,3\}$, there is greater flexibility in the choice of the parameters
$k$, $\bigl|\{g\in G \mid d_g\neq 0\}\bigr|$, and $\sum_{g} d_g$, although they remain
bounded by Proposition~\ref{prop: final_bound_on_p,k,sum_dg}, point~(2).
We therefore ran the function for all valid combinations of these parameters,
obtaining a more extensive list of cases.

The resulting SNC $(p,k)$-planes, including those of Hirzebruch type
(cf.~Proposition~\ref{prop: hirzebruch_planes}, point~(1)) and the Garbagnati
surfaces of type~$G3$ together with their SNC degenerations
(cf.~Proposition~\ref{prop: final_bound_on_p,k,sum_dg}, point~(4) for $p=2$), are
studied in Section \ref{sec: invariants_SNC_p,k_planes_Prop_STAR}. 

For the sake of readability, all cases produced by the algorithm are collected in
Appendix~\ref{app:classification_tables}, where they are organized into several
tables according to the values of $(p,k)$.
\section{SNC $(p,1)$-planes with geometric genus one}\label{sec: SNC_(p,1)-planes}

We discuss SNC $(p,1)$-planes $\pi\colon Y \to \mathbb{P}^2$ with $p_g(Y)=1$. The reader can find the proof of the main Theorem \ref{thm main2}, which is Theorem \ref{thm: pg1_Zp_cover_structure} in this section. 
 
 First of all, we study automorphisms of prime order on a projective surface $S$. 
In particular, we analyze the induced action on the second cohomology group and, more precisely, on the Néron–Severi group and on the transcendental lattice. 
\begin{proposition}\label{prop: bound_on_the _prime_p}
	Let $S$ be a smooth projective surface. Let $\rho(S)$ be the Picard number of $S$ and let $T(S):=\NS(S)^\perp\subset H^2(S,\mathbb Z)$ be the transcendental lattice of $S$. Then 
	\begin{equation}\label{eq: ineq_rk_T(S)}
		\rho(S)\leq 10\chi(S)+2q(S)-K^2_S \quad \makebox{and} \qquad \rk \ T(S)\leq 12\chi(S)+4q(S)-K_{S_{min}}^2-3, 
	\end{equation}
	where $S_{min}$ is the minimal model of $S$. 
	
	Furthermore, if $g\in\mathrm{Aut}(S)$ has prime order $p$ and acts non-trivially on
	$H^{2,0}(S)$, then
	\begin{equation}\label{eq: ineq_prime_order_automor_of_S}
		p\le 12\chi(S)+4q(S)-K_{S_{min}}^2-2.
	\end{equation}
\end{proposition}
\begin{proof}
	By Lefschetz Theorem on $(1,1)$-classes, then $\rho(S)\leq h^{1,1}(S)$. From Hodge theory, we have
	\[
	b_1(S)=h^{1,0}(S)+h^{0,1}(S)=2q(S), \quad \makebox{and} \quad b_2(S)=h^{0,2}(S)+h^{1,1}(S)+h^{2,0}(S)=2p_g(S)+h^{1,1}(S),
	\]
	so applying Noether's formula we obtain 
	\begin{equation}\label{eq: formula_h^11_general}
		\begin{split}
			10\chi(S)-K^2_S & =e(S)-2\chi \\ & = 2-2b_1(S)+b_2(S)-2+2q(S)-2p_g(S) \\
			& =h^{1,1}(S)-2q(S). 
		\end{split}
	\end{equation}
	This proves $\rho(S)\leq h^{1,1}(S)=10\chi(S)+2q(S)-K^2_S$. Furthermore, $S$ is projective, hence  $\rho(S)\geq 1$ and by the definition of trascendental lattice we obtain 
	\begin{equation}\label{eq: ineq_rkTS_computation}
		\begin{split}
			\rk \ T(S) & =b_2(S)-\rho(S)\\ 
			& \leq b_2(S)-1\\ 
			&  =2p_g(S)+h^{1,1}(S)-1 \\
			& =2p_g(S)+2q(S)+10\chi(S)-K^2_S-1\\
			&=12\chi(S)+4q(S)-K^2_S-3.
		\end{split}
	\end{equation}
	Let $S_{min}$ be the minimal model of $S$, so there exists a blow-up $b \colon S\to S_{\min}$ at finitely many points. Blowing up a smooth projective surface does not change its transcendental lattice, so $T(S_{min})\cong T(S)$ as integral lattices and in particular 
	\[
	\rk \ T(S)=\rk \ T(S_{min})\leq 12\chi(S)+4q(S)-K^2_{S_{min}}-3.    \] 
	Finally, let $g\in \mathrm{Aut}(S)$ has prime order $p$ and acts non-trivially on $H^{2,0}(S)$.
	Then $g^*$ preserves $\mathrm{NS}(S)$ and the cup-product form, hence it preserves $T(S)$.
	Since the irreducible representations of $\mathbb Z/p$ over $\mathbb Q$ are only the trivial one and the cyclotomic representation $\mathbb Q(\zeta_p)$ of dimension $p-1$, then  $T(S)_\mathbb Q$ splits as a direct sum of two isotypic components 
	\begin{equation}\label{eq: dec_isotypic_comp_trascend_lattice}
		T(S)_\mathbb Q \simeq 1^{\oplus n}\oplus \mathbb Q(\zeta_p)^{\oplus m}, \qquad \makebox{for some} \qquad n,m\geq 0. 
	\end{equation}
	However,  the representation of $\mathbb Z/p$ on $T(S)_\mathbb Q$ is not trivial otherwise $g$ would act trivially on $H^{2,0}(S)$, so $m\geq 1$. 
	In particular,
	\[
	p-1 \leq  \mathrm{rk}\,T(S),
	\qquad\text{and hence}\qquad
	p \le \mathrm{rk}\,T(S)+1\le 12\chi(S)+4q(S)-K^2_{S_{min}}-2.
	\]
	This concludes the proof.
\end{proof}
\begin{remark}
	The assumption that the action of $g$ on $H^{2,0}(S)$ is non-trivial is essential. 
	If $g^*|_{H^{2,0}(S)}=\mathrm{id}$, the automorphism is cohomologically symplectic, and the above representation-theoretic argument on the transcendental lattice no longer applies. 
	In this situation, additional phenomena may occur; see for instance the next Section \ref{sec:SNC_(p,1)-planes_zero} for examples of SNC $(p,1)$ planes with $p_g(S)=0$ and unbounded $p$, and Cai \cite{Cai} for examples and a detailed analysis in the elliptic case.
	
	Moreover, by the result of Morrison \cite{Mor87} on the structure of the transcendental lattice of projective surfaces with $p_g=1$, the transcendental lattice is an irreducible Hodge structure of K3 type. 
	Consequently, once the non-trivial action on $H^{2,0}(S)$ is assumed, the above argument specialized to surfaces with $p_g=1$ reduces immediately to the well-known theory of non-symplectic automorphisms of K3 surfaces, for which is known (see for instance \cite{AST} and \cite{GS13}) that $n=0$ in Equation \eqref{eq: dec_isotypic_comp_trascend_lattice} and then we would get the stronger condition $p-1 \mid \rk \ T(S)$ instead of $p-1\leq \rk \ T(S)$.
\end{remark}
\begin{remark}\label{rem: ineq_order_aut_without_Ksquare}
	Assume that $S$ is a smooth projective surface not ruled, and that $g\in \text{Aut}(S)$ has prime order $p$ and acts non-trivially on $H^{2,0}(S)$. From Enriques-Kodaira classification, then $K_{S_{min}}^2\geq 0$, so Equation \ref{eq: ineq_prime_order_automor_of_S} gives
	\[
	p\leq 12\chi(S)+4q(S)-2.
	\]
\end{remark}
We are ready to provide a classification of SNC $(p,1)$-planes with $p_g(Y)=1$. 
\begin{proposition}\label{prop: complete_classification_p,1_planes_pg=1}
	Let $\pi\colon Y\to \mathbb P^2$ be a SNC $(p,1)$-plane with $p_g(Y)=1$. Then
	$p\in \{2,3,5,7,$ $11,13,17,19\}$,
	and the building data of the cover is one among those listed in tables \ref{tab: (2,1)} to \ref{tab: (19,1)} of the Appendix \ref{app:classification_tables}. 
\end{proposition}
\begin{proof}
	From Lemma \ref{lem: (p,k)_planes_are_regular}, then the minimal resolution of singularities $\widetilde{Y}$ of $Y$ has $q(\widetilde{Y})=0$, so $\chi(\widetilde{Y})=2$. 
	
	The action of $\mathbb Z/p$ on $Y$ lifts naturally to an action on $\widetilde{Y}$. The action on $\widetilde{Y}$ induces an action on $H^{2,0}(\widetilde{Y})$, which is not trivial. Otherwise, an invariant global holomorphic 2-form would descend to a global holomorphic $2$-form on $Y/(\mathbb Z/p)\cong \mathbb P^2$, a contradiction. 
	
	Since $\widetilde{Y}$ is a smooth projective surface that is not ruled, then from Proposition \ref{prop: bound_on_the _prime_p} and Remark \ref{rem: ineq_order_aut_without_Ksquare} we have $p\leq 12\chi-2=22$. Finally, from Proposition \ref{prop: equivalence_star_prop}, then $\sum_g \deg(D_g)\leq 6,5,4$ accordingly if $p$ is equal either to $2,3$ or $5\leq p\leq 19$, respectively. Thus, for any prime number $p\leq 19$, one can compute either by hand or using a computer station, via the commented function \verb!Building_Data(card,d,G)! explained in Section \ref{sec: classification_SNC_p,k_planes_Prop_STAR},
	all the few possibilities for the degrees of the divisors $D_g$ for which the reduced Pardini's equations of Theorem \ref{thm: reduction_numb_Pard_eq} are satisfied; the obtained cases are those listed in tables \ref{tab: (2,1)} to \ref{tab: (19,1)}.
\end{proof}
Let us study more in detail the geometry of the obtained SNC $(p,1)$-planes $\pi\colon Y\to \mathbb P^2$ with $p_g(Y)=1$. In particular, we prove that the minimal model of $Y$ is a K3 surface. 

The strategy to prove this is as follows.
\begin{remark}\label{rem: strategy_proof_(p,1)_planes_pg1_are_K3}
	By Proposition \ref{prop: complete_classification_p,1_planes_pg=1}, $Y$ falls in one of the cases listed in tables \ref{tab: (2,1)} to \ref{tab: (19,1)}. We show for each of those cases that the numerical canonical class of the minimal resolution of singularities $\rho\colon \widetilde{Y}\to Y$ is supported on a divisor whose irreducible components are precisely the curves contracted in the passage from $\widetilde{Y}$ to the minimal model $\widetilde{Y}_{min}$. 
	\\
	Thus, we would obtain a minimal surface $\widetilde{Y}_{min}$ with trivial numerical canonical class 
	that has $p_g(\widetilde{Y}_{min})=1$, $q(\widetilde{Y}_{min})=0$. Thus, it is a K3 surface from Enriques-Kodaira classification. 
\end{remark}
The cases from tables \ref{tab: (2,1)} to \ref{tab: (5,1)}, so those with $p\in \{2,3,5\}$, can be treated together and are discussed in the following lemma. 
\begin{lemma}\label{lem: (p,1)_planes_pg1_case_p=2,3,5}
	Let $\pi\colon Y\to \mathbb P^2$ be a SNC $(p,1)$-plane with $p_g(Y)=1$, and $p\in \{2,3,5\}$. Then the minimal model of $Y$ is a K3 surface.
\end{lemma}
\begin{proof}
	From Proposition \ref{prop: complete_classification_p,1_planes_pg=1}, then $Y$ falls in one of the cases in tables \ref{tab: (2,1)} to \ref{tab: (5,1)}. The case $p=2$ is already well-understood and known in the literature: a double cover of $\mathbb P^2$ ramified over a (possibly singular with normal crossings) sextic curve has  a minimal resolution of singularities that is a K3 surface. Indeed, from Formula \eqref{eq: KS_ab_cov_quo_sing}, we have 
	\[
	2K_{\widetilde{Y}}\equiv_{lin} \widetilde{\pi}^*\left(-2K_{\mathbb P^2}+\sum_{g\in G}D_g\right)=\widetilde{\pi}^*(-6H+6H)=0, 
	\]
	so $K_{\widetilde{Y}}$ is numerically trivial, and then $\widetilde{Y}$ is a K3 surface, as explained in Remark \ref{rem: strategy_proof_(p,1)_planes_pg1_are_K3}.  
	
	The remaining cases for \(p=3\) and \(p=5\) are quite similar. We consider the line \(\ell \) of the branch locus of \(\pi \) (cf. Tables \ref{tab: (3,1)} to \ref{tab: (5,1)}) passing through the points of \(\mathbb{P}^{2}\) corresponding to the points of \(Y\) which are non-canonical singularities. It turns out that the strict transform $\widetilde{\ell}$ of $\widetilde{\pi}^*\ell$ in $\widetilde{Y}$ is a $(-1)$-curve and that $K_{\widetilde{Y}}\equiv \widetilde{\ell}$. Thus, after contracting $\widetilde{\ell}$ we obtain a surface with trivial canonical class, hence a K3 surface by the discussion in Remark \ref{rem: strategy_proof_(p,1)_planes_pg1_are_K3}. 
	
	In particular, for $p=3$, looking at the local monodromies of Table \ref{tab: (3,1)}, the line $\ell$ is the branch divisor $\ell=D_{2e_1}$ which intersects transversally in four points the (possibly singular with normal crossings) quartic $D_{e_1}$. The singularities of $Y$ lying over those four points are of type $\frac{1}{3}(1,1)$.  From Proposition \ref{prop: dec_of_the_pullback}, we have
	\[
	\widetilde{\pi}^*\ell=3\widetilde{\ell}+\sum_{i=1}^4 E_{i}^1, 
	\] 
	and from Formulae \eqref{eq: self Intersecion strict} and Remark \ref{rem: genus branch case Z/p}, we have $g(\widetilde{\ell})=0$ and 	
	\[
	\widetilde{\ell}^2=\frac{1}{3}\cdot 1-4\cdot \frac{1}{3}=-1,
	\]
	so $\widetilde{\ell}$ is a $(-1)$-curve of $\widetilde{Y}$. Furthermore, from Formula \ref{eq: KS_ab_cov_quo_sing} we also have 
	\[
	\begin{split}
		K_{\widetilde{Y}} \equiv \frac{1}{3}\left(\widetilde{\pi}^*(3K_{\mathbb P^2}+2\sum_{g\in G}D_g)-\sum_{i=1}^4E_{i}^1\right) =\frac{1}{3}\left(\widetilde{\pi}^*(H)-\sum_{i=1}^4E_{i}^1\right)\equiv \widetilde{\ell},
	\end{split}
	\]
	as we wanted to prove. 
	
	Instead, for $p=5$, looking at Table \ref{tab: (5,1)}, the line $\ell$ has local monodromy $\ell=D_{3e_1}$, and it intersects transversally in three points the (possibly singular with normal crossings) cubic $D_{4e_1}$. The singularities of $\widetilde{Y}$ over those three points are of type $\frac{1}{5}(1,2)$, so that we have 
	\[
	\widetilde{\pi}^*(\ell)=5\widetilde{\ell}+\sum_{i=1}^3 (2E_{1,i}+E_{2,i}),
	\]
	and from Formula \eqref{eq: self Intersecion strict} and Remark \ref{rem: genus branch case Z/p}, we have $g(\widetilde{\ell})=0$ and 
	\[
	\widetilde{\ell}^2=\frac{1}{5}\cdot 1 -3\cdot \frac{2}{5}=-1.
	\]
	In other words, $\widetilde{\ell}$ is a $(-1)$-curve of $\widetilde{Y}$. Furthermore, Formula \eqref{eq: KS_ab_cov_quo_sing} gives
		\[
	\begin{split}
		K_{\widetilde{Y}}  \equiv \frac{1}{5}\left(\widetilde{\pi}^*(5K_{\mathbb P^2}+4\sum_{g\in G}D_g)-\sum_{i=1}^3 (2E_{i}^1+E_{i}^2)\right)   \equiv \frac{1}{5}\left(\widetilde{\pi}^*(H)-\sum_{i=1}^3 (2E_{i}^1+E_{i}^2)\right) \equiv \widetilde{\ell},
	\end{split}
	\] 
	so the thesis follows. 
\end{proof}
Let us discuss the cases for $p=7$. 
\begin{lemma}\label{lem: (p,1)_planes_pg1_case_p=7}
	Let $\pi\colon Y\to \mathbb P^2$ be a SNC $(7,1)$-plane with $p_g(Y)=1$. Then the minimal model of $Y$ is a K3 surface.  
\end{lemma}
\begin{proof}
	By Table \ref{tab: (7,1)},  the branch divisor of $\pi$ consists of two lines $\ell_1$, $\ell_2$ and a conic ${\mathcal C}$ with local monodromies \(4e_1,5e_1\) and \(6e_1\), respectively. We have two possibilities: either $\mathcal C$ is smooth or it decomposes as two distinct lines. We only treat the first case, which corresponds to the case $A-(7,1)(1)$  of Table \ref{tab: (7,1)}. The second case $A-(7,1)(2)$ can be done in a  similar way. The corresponding cyclic quotient singularities of \(Y\) in case $A-(7,1)(1)$ are shown in Figure~\ref{fig: A-(7,1)(1)_config}.
			%
			%
			%
			%
			%
			%
			%
			%
			%
	We now proceed with the calculation of the self-intersection of the strict transform of the lines $\ell_1,\ell_2$ and the conic $\mathcal C$ on $\widetilde{Y}$. Using the Formula \eqref{eq: self Intersecion strict} we see that
	\[
	\widetilde{\ell_1}^2=-1, \quad \widetilde{\ell_2}^2=-1, \quad  \widetilde{C}^2=-2.
	\]
	Clearly the curves $\tilde{\ell_1}$, $\tilde{\ell_2}$ are rational from Remark \ref{rem: genus branch case Z/p}, so they are $(-1)$-curves.  We can therefore determine the dual graph of the exceptional configuration in the minimal resolution $\widetilde{Y}$, shown in Figure~\ref{2fig:p7}. To facilitate the minimalization process, we highlight the $(-1)$-curves in red. Moreover, at each step of the minimalization process, we shall denote the strict transforms of the curves by the same symbols as the original curves in order to avoid unnecessarily heavy notation.
	We proceed by contracting the two $(-1)$-curves $\widetilde{\ell_1}$ and $\widetilde{\ell_2}$. The diagram that we obtain is given in Figure~\ref{3fig:p7}.
	After this contraction, we have only another $(-1)$-curve $E^1_1$ to contract and no others occurs in the exceptional configuration, see Figure \ref{3fig:p7_2}. Notice that the new exceptional $(-2)$-curves $E_2^1$ and $E_4^1$ become tangent. In this case, we use the convention that a double edge indicates that the corresponding curves are tangent. 
	
	\begin{figure}[H]
		\centering
		\begin{subfigure}[b]{0.60\textwidth} 
			\centering
			\begin{tikzpicture}[
				scale=1,
				every node/.style={font=\scriptsize, transform shape}, 
				curve/.style={circle, draw, thick, minimum size=10mm},
				exc/.style={circle, fill=black, inner sep=1.6pt},
				edge/.style={thick}
				]
				\draw[thick] (0,0) ellipse (2.8 and 1.15);
				\node[right] at (2.85,-0.25) {$\mathcal C$};
				
				\draw[thick] (-2.4,-1.5) -- (2.2,1.7) node[above right] {$\ell_1$};
				\draw[thick] (-2.3,1.45) -- (2.4,-1.45) node[below right] {$\ell_2$};
				
				
				\fill[] (1.26,1.02) circle (2.3pt)
				node[above left] {$\frac1{7}(1,2)$};
				
				\fill[] (-1.58,-0.95) circle (2.3pt)
				node[below right] {$\frac1{7}(1,2)$};
				
				\fill[] (-1.52,0.95) circle (2.3pt)
				node[above right] {$\frac1{7}(1,3)$};
				
				\fill[] (1.55,-0.95) circle (2.3pt)
				node[below] {$\frac1{7}(1,3)$};
				
				\fill[] (-0.1,0.1) circle (2.3pt)
				node[left] {$\frac1{7}(1,4)$};
				
			\end{tikzpicture}
			\caption{Branch configuration of $A-(7,1)(1)$.}\label{fig: A-(7,1)(1)_config}
		\end{subfigure}
		
		\par\vspace{0.8cm} 
		
		\begin{subfigure}[b]{0.48\textwidth}
			\centering
			\begin{tikzpicture}[
				scale=1,
				every node/.style={font=\scriptsize, transform shape}, 
				curve/.style={circle, draw, thick, minimum size=10mm},
				exc/.style={circle, fill=black, inner sep=1.6pt},
				edge/.style={thick},
				rededge/.style={violet, thick},
				tang/.style={blue, thick, double, double distance=1.5pt},
				blueedge/.style={blue, thick, double, double distance=1.5pt}
				]
				
				\node[curve] (l1) at (-3,2) {{\color{red}$\widetilde{\ell_1}$}};
				\node[curve] (l2) at (3,2) {{\color{red}$\widetilde{\ell_2}$}};
				\node[curve] (C)  at (0,-2) {$\mathcal C$};
				
				\node[exc,label=above:{$E_1^1$},label=below:{\scriptsize $(-2)$}] (E11) at (-1,2) {};
				\node[exc,label=above:{$E_1^2$},label=below:{\scriptsize $(-4)$}] (E12) at (1,2) {};
				\draw[edge] (l1) -- (E11) -- (E12) -- (l2);
				
				\node[exc,label=left:{$E_2^1$},label=right:{\scriptsize $(-4)$}] (E21) at (-3,0.8) {};
				\node[exc,label=left:{$E_2^2$},label=right:{\scriptsize $(-2)$}] (E22) at (-2.2,-0.7) {};
				\draw[edge] (l1) -- (E21) -- (E22) -- (C);
				
				\node[exc,label=left:{$E_4^1$},label=right :{\scriptsize $(-4)$}] (E41) at (-1.45,0.65) {};
				\node[exc,label=left:{$E_4^2$},label=right :{\scriptsize $(-2)$}] (E42) at (-0.8,-0.75) {};
				\draw[edge] (l1) -- (E41) -- (E42) -- (C);
				
				\node[exc,label=right:{$E_5^1$},label=left:{\scriptsize $(-3)$}] (E51) at (1.35,0.65) {};
				\node[exc,label=right:{$E_5^2$},label=left:{\scriptsize $(-2)$}] (E52) at (0.9,-0.35) {};
				\node[exc,label=right:{$E_5^3$},label=left:{\scriptsize $(-2)$}] (E53) at (0.55,-1.25) {};
				\draw[edge] (l2) -- (E51) -- (E52) -- (E53) -- (C);
				
				\node[exc,label=right:{$E_3^1$},label=below right:{\scriptsize $(-3)$}] (E31) at (3,0.8) {};
				\node[exc,label=right:{$E_3^2$},label=below right:{\scriptsize $(-2)$}] (E32) at (2.2,-0.7) {};
				\node[exc,label=right:{$E_3^3$},label=below right:{\scriptsize $(-2)$}] (E33) at (1.55,-1.55) {};
				\draw[edge] (l2) -- (E31) -- (E32) -- (E33) -- (C);
				
			\end{tikzpicture}
			\caption{Dual graph of the minimal resolution of singularities $\widetilde{Y}$.}
			\label{2fig:p7}
		\end{subfigure}
		\hfill
		\begin{subfigure}[b]{0.48\textwidth}
			\centering
			\begin{tikzpicture}[
				scale=1,
				every node/.style={
					font=\scriptsize, 
					transform shape 
				},
				curve/.style={circle, draw, thick, minimum size=10mm},
				exc/.style={circle, fill=black, inner sep=1.6pt},
				edge/.style={thick},
				rededge/.style={violet, thick},
				tang/.style={blue, thick, double, double=none},
				blueedge/.style={blue, thick, double, double distance=1.5pt},
				tang2/.style={blue, thick, double, double distance=1.5pt}
				]
				\node[curve] (C) at (0,-2.3) {$\mathcal C$};
				
				\node[exc,label=above:{{\color{red}$E_1^1$}},label=left:{{\color{red}\scriptsize $(-1)$}}] (E11) at (-1.4,1.8) {};
				\node[exc,label=above:{$E_1^2$},label=below left:{\scriptsize $(-3)$}] (E12) at (1.4,1.8) {};
				\draw[edge] (E11)--(E12);
				
				\node[exc,label=above left:{$E_2^1$},label=below left:{\scriptsize $(-3)$}] (E21) at (-3.2,-0.1) {};
				\node[exc,label=left:{$E_2^2$},label=below:{\scriptsize $(-2)$}] (E22) at (-2.1,-1.7) {};
				\draw[edge] (E21)--(E22)--(C);
				
				\node[exc,label=above right:{$E_4^1$},label=right:{\scriptsize $(-3)$}] (E41) at (-0.9,-0.5) {};
				\node[exc,label=left:{$E_4^2$},label=below left:{\scriptsize $(-2)$}] (E42) at (-0.4,-1.4) {};
				\draw[edge] (E41)--(E42)--(C);
				
				\node[exc,label=above left:{$E_5^1$},label=left:{\scriptsize $(-2)$}] (E51) at (1.0,0) {};
				\node[exc,label=above right:{$E_5^2$},label=right:{\scriptsize $(-2)$}] (E52) at (0.6,-1) {};
				\node[exc,label=above right:{$E_5^3$},label=right:{\scriptsize $(-2)$}] (E53) at (0.2,-1.7) {};
				\draw[edge] (E51)--(E52)--(E53)--(C);
				
				\node[exc,label=right:{$E_3^1$},label=below right:{\scriptsize $(-2)$}] (E31) at (3.2,-0.1) {};
				\node[exc,label=right:{$E_3^2$},label=below right:{\scriptsize $(-2)$}] (E32) at (2.1,-1.4) {};
				\node[exc,label=right:{$E_3^3$},label=below right:{\scriptsize $(-2)$}] (E33) at (1.3,-2.0) {};
				\draw[edge] (E31)--(E32)--(E33)--(C);
				
				\draw[rededge] (E21)--(E11)--(E41)--(E21);
				
				\draw[rededge] (E12)--(E51)--(E31)--(E12);
				
			\end{tikzpicture}
			\caption{First contraction of the $(-1)$-curves $\widetilde{\ell_1}$ and $\widetilde{\ell_2}$.}
			\label{3fig:p7}
		\end{subfigure}
		\hfill
		
		\vspace{0.6cm}
		
		\hspace*{-1cm}
		\begin{subfigure}[b]{0.48\textwidth}
			\centering
			\begin{tikzpicture}[
				scale=1,
				every node/.style={font=\scriptsize, transform shape}, 
				curve/.style={circle, draw, thick, minimum size=10mm},
				exc/.style={circle, fill=black, inner sep=1.6pt},
				edge/.style={thick},
				rededge/.style={violet, thick},
				tang/.style={blue, thick, double, double=none},
				blueedge/.style={blue, thick, double, double distance=1.5pt},
				tang2/.style={blue, thick, double, double distance=1.5pt}
				]
				
				\node[curve] (C) at (0,-2.3) {$\mathcal C$};
				
				\node[exc,label=above:{$E_1^2$},label=left:{\scriptsize $(-2)$}] (E12) at (1.4,1.8) {};
				
				\node[exc,label=above left:{$E_2^1$},label=below left:{\scriptsize $(-2)$}] (E21) at (-3.2,-0.1) {};
				\node[exc,label=left:{$E_2^2$},label=below:{\scriptsize $(-2)$}] (E22) at (-2.1,-1.7) {};
				\draw[edge] (E21)--(E22)--(C);
				
				\node[exc,label=above left:{$E_4^1$},label=right:{\scriptsize $(-2)$}] (E41) at (-0.9,-0.5) {};
				\node[exc,label=left:{$E_4^2$},label=below left:{\scriptsize $(-2)$}] (E42) at (-0.4,-1.4) {};
				\draw[edge] (E41)--(E42)--(C);
				
				\node[exc,label=above left:{$E_5^1$},label=left:{\scriptsize $(-2)$}] (E51) at (1.0,0) {};
				\node[exc,label=above right:{$E_5^2$},label=right:{\scriptsize $(-2)$}] (E52) at (0.6,-1) {};
				\node[exc,label=above right:{$E_5^3$},label=right:{\scriptsize $(-2)$}] (E53) at (0.2,-1.7) {};
				\draw[edge] (E51)--(E52)--(E53)--(C);
				
				\node[exc,label=right:{$E_3^1$},label=below right:{\scriptsize $(-2)$}] (E31) at (3.2,-0.1) {};
				\node[exc,label=right:{$E_3^2$},label=below right:{\scriptsize $(-2)$}] (E32) at (2.1,-1.4) {};
				\node[exc,label=right:{$E_3^3$},label=below right:{\scriptsize $(-2)$}] (E33) at (1.3,-2.0) {};
				\draw[edge] (E31)--(E32)--(E33)--(C);
				
				\draw[tang] (E21)--(E12);
				\draw[tang] (E41)--(E12);
				\draw[tang2] (E41)--(E21);
				
				\draw[tang] (E12)--(E51);
				\draw[tang] (E12)--(E31);
				\draw[tang] (E51)--(E31);
			\end{tikzpicture}
			\caption{Final contraction of the $(-1)$-curve $E_1^1$.}
			\label{3fig:p7_2}
		\end{subfigure}
		\caption{Complete contraction process for the case $A-(7,1)(1)$.}\label{fig:global_process_p=7}
	\end{figure}
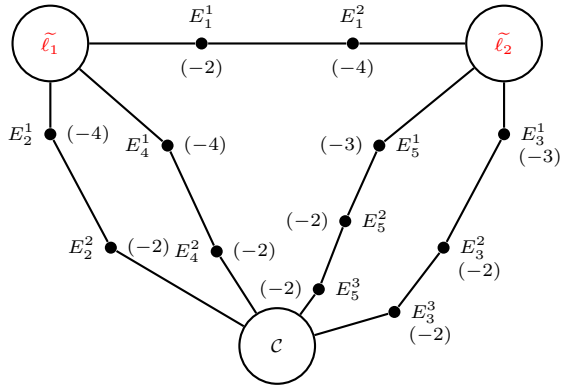
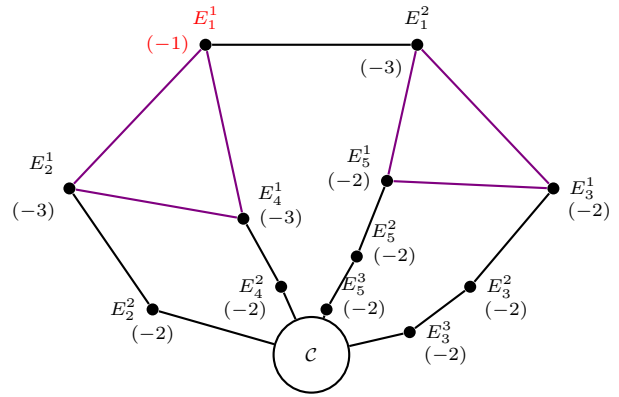
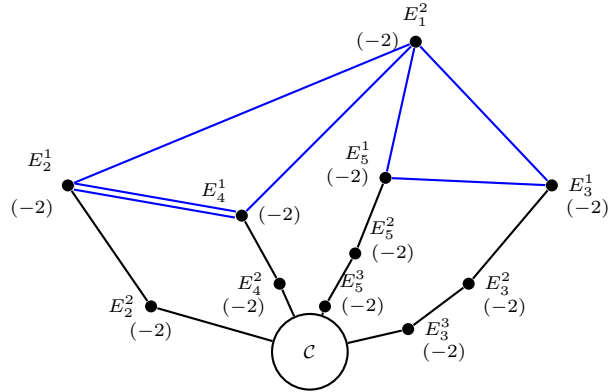
\vspace{-0.4cm}

	As explained in Remark \ref{rem: strategy_proof_(p,1)_planes_pg1_are_K3} and already done in Lemma \ref{lem: (p,1)_planes_pg1_case_p=2,3,5}, we proceed by giving the numerical canonical class of the minimal resolution $\widetilde{Y}$ using Formula \eqref{eq: KS_ab_cov_quo_sing}:
\[
\begin{split}
	7K_{\widetilde{Y}}\equiv_{lin}\widetilde{\pi}^*(3H)-\left(2E_{1}^1+4E_1^2\right)-\left(4E_{2}^1+2E_2^2\right)-\left(4E_{4}^1+2E_4^2\right)-\sum_{j= 3,5}\left(3E_j^1+2E_j^2+E_j^3\right).
\end{split}
\] 
	However, using Formula \ref{eq: decomp_pullback_branch_comp}, then 
\[
\begin{split}
	\widetilde{\pi}^*\ell_1& =7\widetilde{\ell_1}+(4E_1^1+E_1^2)+(2E_{2}^1+E_2^2)+(2E_4^1+E_4^2) \\
	\widetilde{\pi}^*\ell_2& =7\widetilde{\ell_2}+(E_1^1+2E_1^2)+\sum_{j=3,5}(3E_j^1+2E_j^2+E_j^3),
\end{split}
\]
so that the numerical class of $K_{\widetilde{Y}}$ can be written as 
\[
K_{\widetilde{Y}}\equiv 2\widetilde{\ell_1}+\widetilde{\ell_2}+E^1_1. 
\]
Thus, we see that the support of the canonical divisor of $\widetilde{Y}$ is on contracted curves. Hence the minimal model $\widetilde{Y}_{min}$ is a K3 surface, as explained in Remark \ref{rem: strategy_proof_(p,1)_planes_pg1_are_K3}. 
\end{proof}
The remaining cases $p=11,13,17,19$ can be studied by the same computation. 
\begin{lemma}\label{lem: (p,1)_planes_pg1_case_p=11,13,17,19}
	Let $\pi\colon Y\to \mathbb P^2$ be a SNC $(p,1)$-plane with $p_g(Y)=1$, and $p\in \{11,13,17,19\}$. Then the minimal model of $Y$ is a K3 surface. 
\end{lemma}
\begin{proof}
	The argument is essentially identical for \(p=11,13,17\) and \(19\) and it follows the strategy presented in Remark \ref{rem: strategy_proof_(p,1)_planes_pg1_are_K3} already used to prove lemmas \ref{lem: (p,1)_planes_pg1_case_p=2,3,5} and \ref{lem: (p,1)_planes_pg1_case_p=7}. Therefore, we present only the case \(p=11\) in full detail.
	
	By Table~\ref{tab: (11,1)}, for \(p=11\) the branch divisor of $\pi$ consists of four lines with monodromies \(6e_1,8e_1,9e_1\) and \(10e_1\). The corresponding cyclic quotient singularities of \(Y\) are shown in Figure~\ref{1fig:p11}.
	
	We now proceed with the calculation of the self intersection of the strict transform of the lines $\ell_1, \ldots ,\ell_4$. Using Formula \eqref{eq: self Intersecion strict} we see that
	\[
	\tilde{\ell_1}^2=-1, \quad \tilde{\ell_2}^2=-1, \quad  \tilde{\ell_3}^2=-1, \quad  \tilde{\ell_4}^2=-2. 
	\]
	Clearly the curves $\tilde{\ell_1}$, $\tilde{\ell_2}$ and $\tilde{\ell_3}$ are all rational, by Remark \ref{rem: genus branch case Z/p}.  We can therefore determine the dual graph of the exceptional configuration in the minimal resolution of singularities $\widetilde{Y}$, shown in Figure~\ref{2fig:p11}. To facilitate the minimalization process, we highlight the $(-1)$-curves in red. 
	
	Moreover, at each step of the minimalization process, we shall denote the strict transforms of the curves by the same symbols as the original curves in order to avoid unnecessarily heavy notation.
	
	The surface obtained after the first contraction of $\widetilde{\ell}_1,\widetilde{\ell_2}, \widetilde{\ell_3}$ is still not minimal. Indeed, it contains two further $(-1)$-curves shown in red in Figure \ref{3fig:p11}. We therefore proceed by contracting these curves.
	
	In Figure~\ref{fig:p11}, we adopt the convention that a double edge indicates that the corresponding curves are tangent, rather than merely intersecting transversely, as done for the Figure \ref{3fig:p7_2} in Lemma \ref{lem: (p,1)_planes_pg1_case_p=7}. We have one last curve $E^1_2$ to contract,  
	this is done in the last Figure \ref{4fig:p11}. We point out a subtle phenomenon that does not occur in the cases $p\leq 7$, but appears in all the cases with $p\geq 11$. Indeed, during the minimalization process we contract a $(-1)$-curve that is tangent to other curves in the configuration. 

Consequently, the strict transforms of the tangent curves become singular and their self-intersection must be computed with particular care. In the present case, the self-intersection of the strict transform of $E_3^1$ increases by $4$ 
due to the tangential intersection with the contracted curve.

	We finally give the numerical canonical class of the minimal resolution 
	using Theorem \ref{thm: invariants_ab_cov_cyclic_sing}:
	\[
	K_{\widetilde Y}
	\equiv
	4\widetilde{\ell}_1
	+
	2\widetilde{\ell}_2
	+
	\widetilde{\ell}_3
	+
	2\widetilde E_1^{\,1}
	+
	\widetilde E_2^{\,1}
	+
	\widetilde E_5^{\,1}.
	\]
	Proceeding with the contraction as explaind above we see that the support of the canonical divisor $K_{\widetilde{Y}}$ was on contracted curves. Hence the minimal model $\widetilde{Y}_{min}$ is a K3 surface, as explained in Remark \ref{rem: strategy_proof_(p,1)_planes_pg1_are_K3}. 
			%
			%
			%
			%
			%
			%
			%
			%
			%
	\begin{figure}[H]
		\centering
		\begin{subfigure}[b]{0.60\textwidth} 
			\centering
			\begin{tikzpicture}[
				scale=0.95,
				every node/.style={font=\scriptsize, transform shape}, 
				curve/.style={circle, draw, thick, minimum size=10mm},
				exc/.style={circle, fill=black, inner sep=1.6pt},
				edge/.style={thick}
				]
				\draw[thick] (-3,0) -- (3,0) node[right] {$\ell_1$};
				\draw[thick] (-2.4,-1.2) -- (2.1,3.0) node[right] {$\ell_2$};
				\draw[thick] (-1.8,2.7) -- (2.3,-1.0) node[right] {$\ell_3$};
				\draw[thick] (-0.2,-1.3) -- (1.35,2.9) node[right] {$\ell_4$};
				
				\fill[] (-1.10,0) circle (2.3pt)
				node[above left] {$\frac1{11}(1,6)$};
				
				\fill[] (0.28,0) circle (2.3pt)
				node[below left] {$\frac1{11}(1,2)$};
				
				\fill[] (1.2,0) circle (2.3pt)
				node[below right] {$\frac1{11}(1,4)$};
				
				\fill[] (0,1.05) circle (2.3pt)
				node[left] {$\frac1{11}(1,3)$};
				
				\fill[] (0.50,0.6) circle (2.3pt)
				node[above right] {$\frac1{11}(1,5)$};
				
				\fill[] (1,2.0) circle (2.3pt)
				node[above left] {$\frac1{11}(1,7)$};
				
				
			\end{tikzpicture}
			\caption{Branch configuration for the case $A-(11,1)$.}\label{1fig:p11}
		\end{subfigure}
		
		\par\vspace{0.6cm} 
		
		\hspace*{-1cm}
		\begin{subfigure}[b]{0.48\textwidth}
			\centering
			\begin{tikzpicture}[
				scale=0.9,
				every node/.style={font=\scriptsize, transform shape}, 
				curve/.style={circle, draw, thick, minimum size=10mm},
				exc/.style={circle, fill=black, inner sep=1.6pt},
				edge/.style={thick},
				rededge/.style={violet, thick},
				tang/.style={blue, thick, double, double distance=1.5pt},
				blueedge/.style={blue, thick, double, double distance=1.5pt}
				]
				
				\node[curve] (l1) at (-3,2) {{\color{red}$\tilde{\ell_1}$}};
				\node[curve] (l2) at (3,2) {{\color{red}$\tilde{\ell_2}$}};
				\node[curve] (l3) at (-3,-2) {{\color{red}$\tilde{\ell_3}$}};
				\node[curve] (l4) at (3,-2) {$\tilde{\ell_4}$};
				
				\node[exc,label=above:{$E_1^1$},label=below:{\scriptsize $(-2)$}] (E11) at (-1.2,2) {};
				\node[exc,label=above:{$E_1^2$},label=below:{\scriptsize $(-6)$}] (E12) at (0.9,2) {};
				
				\draw[edge] (l1) -- (E11);
				\draw[edge] (E11) -- (E12);
				\draw[edge] (E12) -- (l2);
				
				\node[exc,label=right:{$E_2^1$},label=left:{\scriptsize $(-3)$}] (E21) at (-3,0.7) {};
				\node[exc,label=right:{$E_2^2$},label=left:{\scriptsize $(-4)$}] (E22) at (-3,-0.7) {};
				
				\draw[edge] (l1) -- (E21);
				\draw[edge] (E21) -- (E22);
				\draw[edge] (E22) -- (l3);
				
				\node[exc,label=above:{$E_3^1$},label=below:{\scriptsize $(-6)$}] (E31) at (-1.15,0.55) {};
				\node[exc,label=right:{$E_3^2$},label=below:{\scriptsize $(-2)$}] (E32) at (0.8,-0.55) {};
				
				\draw[edge] (l1) -- (E31);
				\draw[edge] (E31) -- (E32);
				\draw[edge] (E32) -- (l4);
				
				\node[exc,label=above:{$E_4^1$},label=below:{\scriptsize $(-4)$}] (E41) at (1.15,0.55) {};
				\node[exc,label=left:{$E_4^2$},label=below:{\scriptsize $(-3)$}] (E42) at (-0.8,-0.55) {};
				
				\draw[edge] (l2) -- (E41);
				\draw[edge] (E41) -- (E42);
				\draw[edge] (E42) -- (l3);
				
				\node[exc,label=right:{$E_5^1$},label=left:{\scriptsize $(-2)$}] (E51) at (3,1.1) {};
				\node[exc,label=right:{$E_5^2$},label=left:{\scriptsize $(-3)$}] (E52) at (3,0.25) {};
				\node[exc,label=right:{$E_5^3$},label=left:{\scriptsize $(-2)$}] (E53) at (3,-0.55) {};
				\node[exc,label=right:{$E_5^4$},label=left:{\scriptsize $(-2)$}] (E54) at (3,-1.25) {};
				
				\draw[edge] (l2) -- (E51);
				\draw[edge] (E51) -- (E52);
				\draw[edge] (E52) -- (E53);
				\draw[edge] (E53) -- (E54);
				\draw[edge] (E54) -- (l4);
				
				\node[exc,label=above:{$E_6^1$},label=below:{\scriptsize $(-3)$}] (E61) at (-1.65,-2) {};
				\node[exc,label=above:{$E_6^2$},label=below:{\scriptsize $(-2)$}] (E62) at (-0.7,-2) {};
				\node[exc,label=above:{$E_6^3$},label=below:{\scriptsize $(-2)$}] (E63) at (0.25,-2) {};
				\node[exc,label=above:{$E_6^4$},label=below:{\scriptsize $(-2)$}] (E64) at (1.15,-2) {};
				\node[exc,label=above:{$E_6^5$},label=below:{\scriptsize $(-2)$}] (E65) at (2.0,-2) {};
				
				\draw[edge] (l3) -- (E61);
				\draw[edge] (E61) -- (E62);
				\draw[edge] (E62) -- (E63);
				\draw[edge] (E63) -- (E64);
				\draw[edge] (E64) -- (E65);
				\draw[edge] (E65) -- (l4);
				
			\end{tikzpicture}
			\caption{Dual graph of the minimal resolution of singularities $\widetilde{Y}$.}
			\label{2fig:p11}
		\end{subfigure}
		\hfill
		\begin{subfigure}[b]{0.48\textwidth}
			\centering
			\begin{tikzpicture}[
				scale=0.9,
				every node/.style={
					font=\scriptsize, 
					transform shape 
				},
				curve/.style={circle, draw, thick, minimum size=10mm},
				exc/.style={circle, fill=black, inner sep=1.6pt},
				edge/.style={thick},
				rededge/.style={violet, thick}
				]
				
				\node[curve] (l4) at (3,-2.2) {$\ell_4$};
				
				\node[exc,label=above:{{\color{red}$E_1^1$}},label=above right :{{\color{red}\scriptsize $(-1)$}}] (E11) at (-1.1,2.2) {};
				\node[exc,label=above:{$E_1^2$},label=above right :{\scriptsize $(-5)$}] (E12) at (1.2,2.2) {};
				\draw[edge] (E11) -- (E12);
				
				\node[exc,label=left:{$E_2^1$},label=below left:{\scriptsize $(-2)$}] (E21) at (-3,1.2) {};
				\node[exc,label=left:{$E_2^2$},label=below left:{\scriptsize $(-3)$}] (E22) at (-3,-0.8) {};
				\draw[edge] (E21) -- (E22);
				
				\node[exc,label=right:{$E_3^1$},label=below:{\scriptsize $(-5)$}] (E31) at (-1.1,0.9) {};
				\node[exc,label=right:{$E_3^2$},label=below:{\scriptsize $(-2)$}] (E32) at (1.0,-0.75) {};
				\draw[edge] (E31) -- (E32);
				\draw[edge] (E32) -- (l4);
				
				\node[exc,label=left:{$E_4^1$},label=below:{\scriptsize $(-3)$}] (E41) at (1.0,0.95) {};
				\node[exc,label=right:{$E_4^2$},label=below:{\scriptsize $(-2)$}] (E42) at (-1.2,-0.7) {};
				\draw[edge] (E42) -- (E41);
				
				\node[exc,label=right:{{\color{red} $E_5^1$}},label=below right :{{\color{red}\scriptsize $(-1)$}}] (E51) at (3,1.05) {};
				\node[exc,label=right:{$E_5^2$},label=below right:{\scriptsize $(-3)$}] (E52) at (3,0.25) {};
				\node[exc,label=right:{$E_5^3$},label=below right:{\scriptsize $(-2)$}] (E53) at (3,-0.55) {};
				\node[exc,label=right:{$E_5^4$},label=below right:{\scriptsize $(-2)$}] (E54) at (3,-1.35) {};
				\draw[edge] (E51) -- (E52) -- (E53) -- (E54) -- (l4);
				
				\node[exc,label=above:{$E_6^1$},label=below:{\scriptsize $(-2)$}] (E61) at (-1.7,-2.2) {};
				\node[exc,label=above:{$E_6^2$},label=below:{\scriptsize $(-2)$}] (E62) at (-0.75,-2.2) {};
				\node[exc,label=above:{$E_6^3$},label=below:{\scriptsize $(-2)$}] (E63) at (0.15,-2.2) {};
				\node[exc,label=above:{$E_6^4$},label=below:{\scriptsize $(-2)$}] (E64) at (1.05,-2.2) {};
				\node[exc,label=above:{$E_6^5$},label=below:{\scriptsize $(-2)$}] (E65) at (1.9,-2.2) {};
				\draw[edge] (E61) -- (E62) -- (E63) -- (E64) -- (E65) -- (l4);
				
				\draw[rededge] (E21) -- (E11) -- (E31) -- (E21);
				\draw[rededge] (E22) -- (E61) -- (E42) -- (E22);
				\draw[rededge] (E12) -- (E41) -- (E51) -- (E12);
				
			\end{tikzpicture}
			\caption{First contraction of the $(-1)$-curves $\widetilde{\ell}_1$, $\widetilde{\ell}_2$, $\widetilde{\ell}_3$.}
			\label{3fig:p11}
		\end{subfigure}
		\hfill
		
		\vspace{0.6cm}
		
		\hspace*{-1cm}
		\begin{subfigure}[b]{0.48\textwidth}
			\centering
			\begin{tikzpicture}[
				scale=0.9,
				every node/.style={font=\scriptsize, transform shape}, 
				curve/.style={circle, draw, thick, minimum size=10mm},
				exc/.style={circle, fill=black, inner sep=1.6pt},
				edge/.style={thick},
				rededge/.style={violet, thick},
				tang/.style={blue, thick, double, double distance=1.5pt},
				blueedge/.style={blue, thick, double, double distance=1.5pt}
				]
				
				\node[curve] (l4) at (3,-2.2) {$\ell_4$};
				
				\node[exc,label=above:{$E_1^2$},label=above right:{\scriptsize $(-3)$}] (E12) at (0.25,2.0) {};
				
				\node[exc,label=left:{{\color{red}$E_2^1$}},label=below left:{{\color{red}\scriptsize $(-1)$}}] (E21) at (-3,0.8) {};
				\node[exc,label=left:{$E_2^2$},label=below left:{\scriptsize $(-3)$}] (E22) at (-3,-0.75) {};
				
				\node[exc,label=above:{$E_3^1$},label=below:{\scriptsize $(-4)$}] (E31) at (-1.1,0.55) {};
				\node[exc,label=right:{$E_3^2$},label=below:{\scriptsize $(-2)$}] (E32) at (0.75,-0.55) {};
				
				\node[exc,label=above right:{$E_4^1$},label=below:{\scriptsize $(-2)$}] (E41) at (0.85,0.55) {};
				\node[exc,label=below:{$E_4^2$},label=below right :{\scriptsize $(-2)$}] (E42) at (-1.0,-0.65) {};
				
				\node[exc,label=right:{$E_5^2$},label=below right:{\scriptsize $(-2)$}] (E52) at (2.45,0.45) {};
				\node[exc,label=right:{$E_5^3$},label=below right:{\scriptsize $(-2)$}] (E53) at (2.45,-0.45) {};
				\node[exc,label=right:{$E_5^4$},label=below right:{\scriptsize $(-2)$}] (E54) at (2.45,-1.25) {};
				
				\node[exc,label=above:{$E_6^1$},label=below:{\scriptsize $(-2)$}] (E61) at (-1.85,-2.2) {};
				\node[exc,label=above:{$E_6^2$},label=below:{\scriptsize $(-2)$}] (E62) at (-0.95,-2.2) {};
				\node[exc,label=above:{$E_6^3$},label=below:{\scriptsize $(-2)$}] (E63) at (-0.05,-2.2) {};
				\node[exc,label=above:{$E_6^4$},label=below:{\scriptsize $(-2)$}] (E64) at (0.85,-2.2) {};
				\node[exc,label=above:{$E_6^5$},label=below:{\scriptsize $(-2)$}] (E65) at (1.75,-2.2) {};
				
				\draw[edge] (E21) -- (E22);
				
				\draw[edge] (E31) -- (E32);
				\draw[edge] (E32) -- (l4);
				
				\draw[edge] (E42) -- (E41);
				
				\draw[edge] (E52) -- (E53) -- (E54) -- (l4);
				
				\draw[edge] (E61) -- (E62) -- (E63) -- (E64) -- (E65) -- (l4);
				
				\draw[edge] (E31) -- (E32);
				\draw[edge] (E32) -- (l4);
				
				\draw[tang] (E21) -- (E31);
				\draw[blue, thick] (E21) -- (E12);
				\draw[blue, thick] (E31) -- (E12);
				
				\draw[tang] (E12) -- (E41);
				\draw[blue, thick] (E41) -- (E52);
				\draw[blue, thick] (E12) -- (E52);
				
				\draw[violet, thick] (E22) -- (E42);
				\draw[violet, thick] (E42) -- (E61);
				\draw[violet, thick] (E61) -- (E22);
				
			\end{tikzpicture}
			\caption{Second contraction of the $(-1)$-curves $E_1^1$ and $E_5^1$.}
			\label{fig:p11}
		\end{subfigure}
		\hfill
		\begin{subfigure}[b]{0.48\textwidth}
			\centering
			\begin{tikzpicture}[
				scale=0.9,
				every node/.style={font=\scriptsize, transform shape}, 
				curve/.style={circle, draw, thick, minimum size=10mm},
				exc/.style={circle, fill=black, inner sep=1.6pt},
				edge/.style={thick},
				rededge/.style={violet, thick},
				tang/.style={blue, thick, double, double distance=1.5pt},
				blueedge/.style={blue, thick, double, double=none},
				tang2/.style={blue, thick, double, double distance=1.5pt},
				tanggreen/.style={green!45!black, thick, double, double distance=1.5pt}
				]
				
				\node[curve] (l4) at (3,-2.2) {$\ell_4$};
				
				\node[exc,label=above:{$E_1^2$},label=above right:{\scriptsize $(-2)$}] (E12) at (0.15,2.0) {};
				
				\node[exc,label=left:{$E_2^2$},label=below left:{\scriptsize $(-2)$}] (E22) at (-3,-0.55) {};
				
				\node[exc,label=right:{$E_3^1$},label=below:{\scriptsize $(0)$}] (E31) at (-0.85,0.35) {};
				\node[exc,label=right:{$E_3^2$},label=below:{\scriptsize $(-2)$}] (E32) at (0.65,-0.35) {};
				
				\node[exc,label=above right:{$E_4^1$},label=below:{\scriptsize $(-2)$}] (E41) at (0.55,0.65) {};
				\node[exc,label=below:{$E_4^2$},label=below right:{\scriptsize $(-2)$}] (E42) at (-1.1,-0.55) {};
				
				\node[exc,label=right:{$E_5^2$},label=below right:{\scriptsize $(-2)$}] (E52) at (2.45,0.45) {};
				\node[exc,label=right:{$E_5^3$},label=below right:{\scriptsize $(-2)$}] (E53) at (2.45,-0.45) {};
				\node[exc,label=right:{$E_5^4$},label=below right:{\scriptsize $(-2)$}] (E54) at (2.45,-1.25) {};
				
				\node[exc,label=above:{$E_6^1$},label=below:{\scriptsize $(-2)$}] (E61) at (-1.85,-2.2) {};
				\node[exc,label=above:{$E_6^2$},label=below:{\scriptsize $(-2)$}] (E62) at (-0.95,-2.2) {};
				\node[exc,label=above:{$E_6^3$},label=below:{\scriptsize $(-2)$}] (E63) at (-0.05,-2.2) {};
				\node[exc,label=above:{$E_6^4$},label=below:{\scriptsize $(-2)$}] (E64) at (0.85,-2.2) {};
				\node[exc,label=above:{$E_6^5$},label=below:{\scriptsize $(-2)$}] (E65) at (1.75,-2.2) {};
				
				\draw[edge] (E31) -- (E32);
				\draw[edge] (E32) -- (l4);
				
				\draw[edge] (E42) -- (E41);
				
				\draw[edge] (E52) -- (E53) -- (E54) -- (l4);
				
				\draw[edge] (E61) -- (E62) -- (E63) -- (E64) -- (E65) -- (l4);
				
				\draw[edge] (E31) -- (E32);
				
				\draw[rededge] (E22) -- (E42);
				\draw[rededge] (E42) -- (E61);
				\draw[rededge] (E61) -- (E22);
				
				\draw[tang2] (E12) -- (E41);
				\draw[blueedge] (E41) -- (E52);
				\draw[blueedge] (E52) -- (E12);
				
				\draw[green!45!black, thick] (E31) -- (E22);
				\draw[green!45!black, thick] (E22) -- (E12);
				
				\draw[tanggreen] (E31) -- (E12);
				
			\end{tikzpicture}
			\caption{Final contraction of the $(-1)$-curve $E_2^1$.}
			\label{4fig:p11}
		\end{subfigure}
		\caption{Complete contraction process for the case $A-(11,1)$.}\label{fig:global_process_1}
	\end{figure}
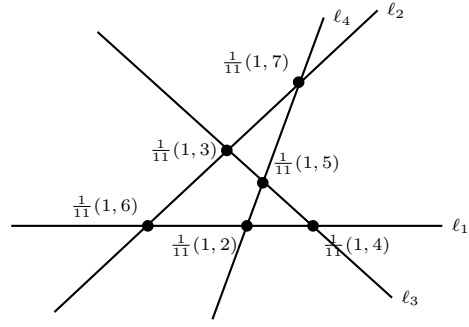
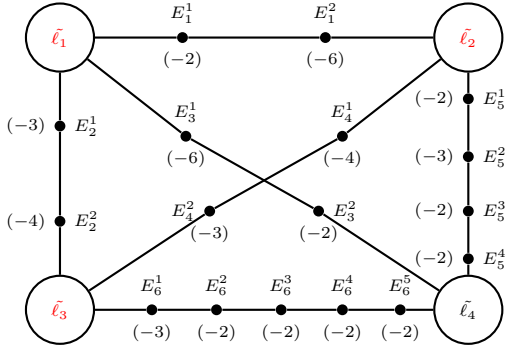
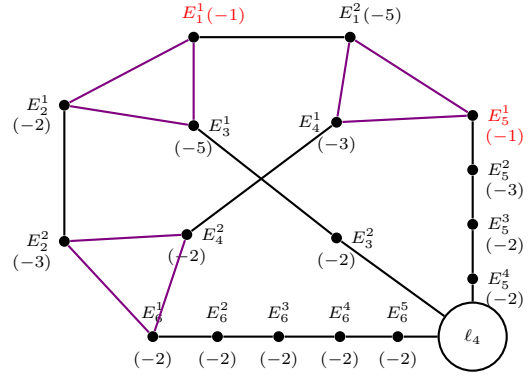
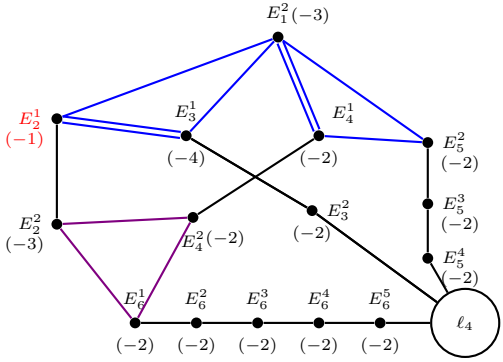
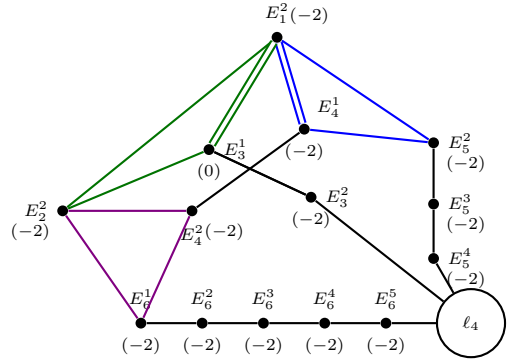 
	We conclude the proof showing the numerical canonical class of each remaining case $p=13,17,19$ together with its minimalization process. In this way, we exhibit the numerical canonical class of $\widetilde{Y}$ as supported on the contracted curves of the minimalization, so that the minimal model of $Y$ turns out to be a K3 surface, still by Remark \ref{rem: strategy_proof_(p,1)_planes_pg1_are_K3}.
\begin{figure}[H]
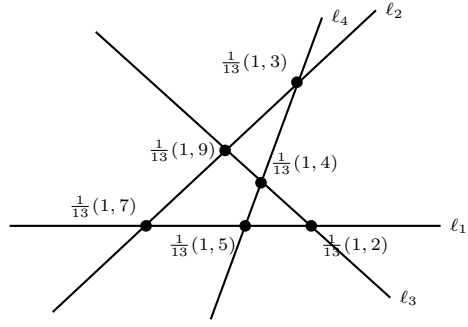
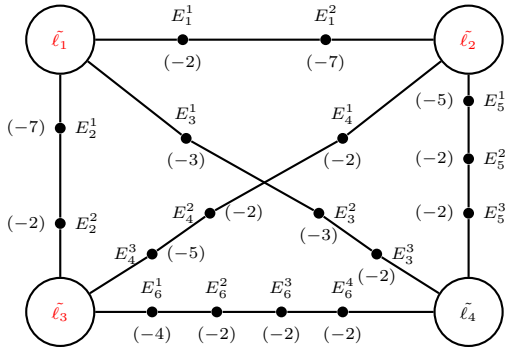
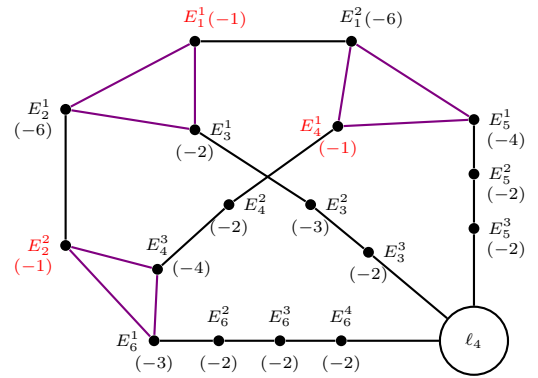
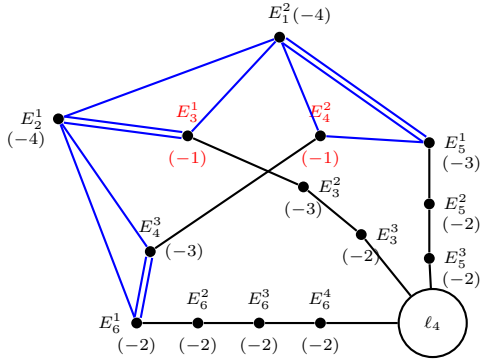
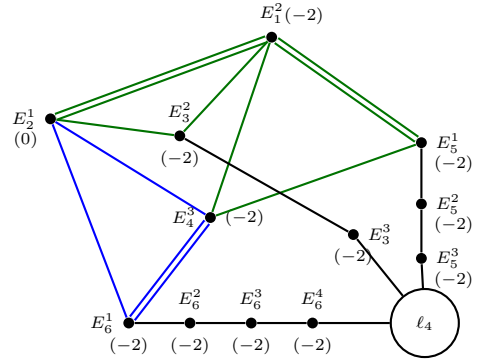

	\centering
	\begin{subfigure}[b]{0.60\textwidth} 
		\centering

		\caption{Branch configuration for the case $A-(13,1)$.}\label{1fig:p13}
	\end{subfigure}
	
	\par\vspace{0.8cm} 
	
	\hspace*{-1cm}
	\begin{subfigure}[b]{0.48\textwidth}
		\centering
		%
		\caption{Dual graph of the minimal resolution of singularities $\widetilde{Y}$.}
		\label{2fig:p13}
	\end{subfigure}
	\hfill
	\begin{subfigure}[b]{0.48\textwidth}
		\centering
		%
		\caption{First contraction of the $(-1)$-curves $\widetilde{\ell}_1$, $\widetilde{\ell}_2$, $\widetilde{\ell}_3$.}
		\label{3fig:p13}
	\end{subfigure}
	\hfill
	
	\vspace{0.6cm}
	
	\hspace*{-1cm}
	\begin{subfigure}[b]{0.48\textwidth}
		\centering
		%
		\caption{Second contraction of the $(-1)$-curves $E_1^1$, $E_2^2$, and $E_4^1$.}
		\label{fig:p13}
	\end{subfigure}
	\hfill
	\begin{subfigure}[b]{0.48\textwidth}
		\centering
		%
		\caption{Final contraction of the $(-1)$-curves $E_3^1$ and $E_4^2$.}
		\label{4fig:p13}
	\end{subfigure}
	\caption{Complete contraction process for the case $A-(13,1)$.}\label{fig:global_process_p=13}
\end{figure}
\begin{figure}[H]
	\centering
	\begin{subfigure}[b]{0.60\textwidth} 
		\centering
		%
		\caption{Branch configuration for the case $A-(17,1)$.}\label{1fig:p17}
	\end{subfigure}
	
	\par\vspace{0.8cm} 
	
	\hspace*{-1cm}
	\begin{subfigure}[b]{0.48\textwidth}
		\centering
		%
		\caption{Dual graph of the minimal resolution of singularities $\widetilde{Y}$.}
		\label{2fig:p17}
	\end{subfigure}
	\hfill
	\begin{subfigure}[b]{0.48\textwidth}
		\centering
		%
		\caption{First contraction of the $(-1)$-curves $\widetilde{\ell}_1$, $\widetilde{\ell}_2$, $\widetilde{\ell}_3$ and $\widetilde{\ell}_4$.}
		\label{3fig:p17}
	\end{subfigure}
	\hfill
	
	\vspace{0.6cm}
	
	\hspace*{-1cm}
	\begin{subfigure}[b]{0.48\textwidth}
		\centering
		%
		\caption{Second contraction of the $(-1)$-curves $E_1^1$, $E_2^2$, and $E_4^1$.}
		\label{fig:p17}
	\end{subfigure}
	\hfill
	\begin{subfigure}[b]{0.48\textwidth}
		\centering
		%
		\caption{Final contraction of the $(-1)$-curves $E_3^1$, $E_4^2$, and then $E_3^2$,$E_4^3$.}
		\label{4fig:p17}
	\end{subfigure}
	\caption{Complete contraction process for the case $A-(17,1)$.}\label{fig:global_process_p=17}
\end{figure}
\begin{figure}[H]
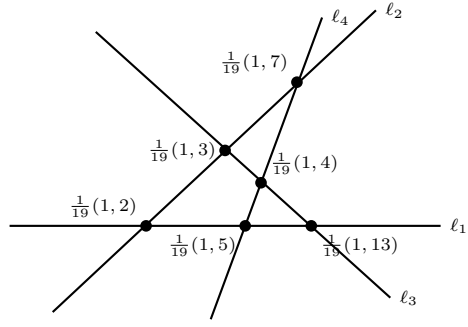
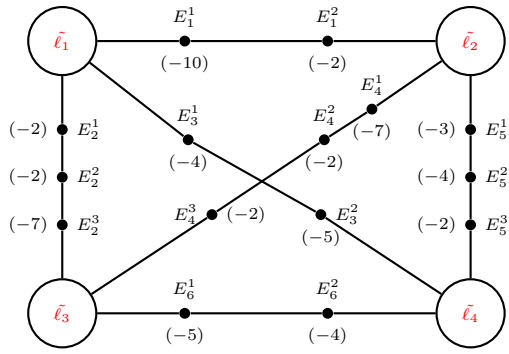
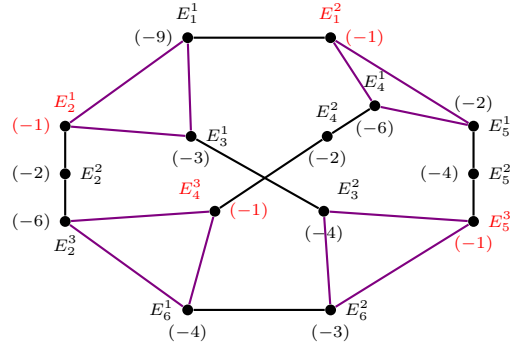
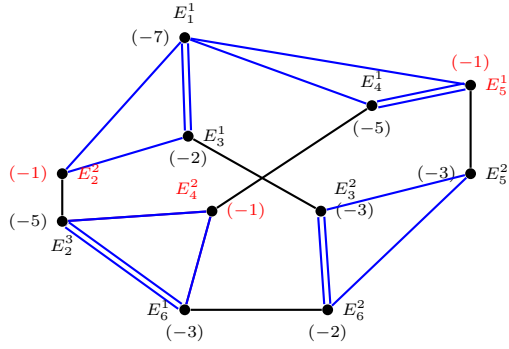
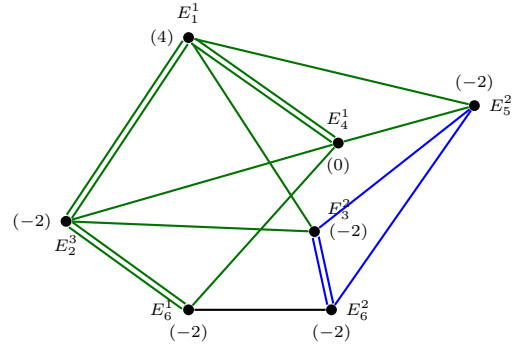

	\centering
	\begin{subfigure}[b]{0.60\textwidth} 
		\centering
		%
		\caption{Branch configuration for the case $A-(19,1)$.}\label{1fig:p19}
	\end{subfigure}
	
	\par\vspace{0.8cm} 
	
	\hspace*{-1cm}
	\begin{subfigure}[b]{0.48\textwidth}
		\centering
		%
		\caption{Dual graph of the minimal resolution of singularities $\widetilde{Y}$.}
		\label{2fig:p19}
	\end{subfigure}
	\hfill
	\begin{subfigure}[b]{0.48\textwidth}
		\centering
		%
		\caption{First contraction of the $(-1)$-curves $\widetilde{\ell}_1$, $\widetilde{\ell}_2$, $\widetilde{\ell}_3$ and $\widetilde{\ell}_4$.}
		\label{3fig:p19}
	\end{subfigure}
	\hfill
	
	\vspace{0.6cm}
	
	\hspace*{-1cm}
	\begin{subfigure}[b]{0.48\textwidth}
		\centering
		%
		\caption{Second contraction of the $(-1)$-curves $E_1^2$, $E_2^1$, $E_4^3$, and $E_5^3$.}
		\label{fig:p19}
	\end{subfigure}
	\hfill
	\begin{subfigure}[b]{0.48\textwidth}
		\centering
		%
		\caption{Final contraction of the $(-1)$-curves $E_2^2$, $E_4^2$, $E_5^1$, and then $E_3^1$.}
		\label{4fig:p19}
	\end{subfigure}
	\caption{Complete contraction process for the case $A-(19,1)$.}\label{fig:global_process_p=19}
\end{figure}
	For the case $A-(13,1)$, we have that the canonical numerical class is 
\[
K_{\widetilde{Y}_{min}}\equiv 4\widetilde{\ell}_1 +3\widetilde{\ell}_2+2\widetilde{\ell}_3+2E_{1}^1+E_2^2+E_3^1+2E_4^1+E_4^2,
\]
and the minimalization process is Figure \ref{fig:global_process_p=13}. 

For the case $A-(17,1)$, we have that the canonical numerical class is 
\[
K_{\widetilde{Y}_{min}}\equiv 6\widetilde{\ell}_1 +4\widetilde{\ell}_2+2\widetilde{\ell}_3+\widetilde{\ell}_4+3E_{1}^1+E_2^2+2E_3^1+E_3^2+3E_4^1+2E_4^2+E_4^3,
\]
and the minimalization process is Figure \ref{fig:global_process_p=17}. 

Finally, for the case $A-(19,1)$, we have that the canonical numerical class is 
\[
K_{\widetilde{Y}_{min}}\equiv 6\widetilde{\ell}_1 +4\widetilde{\ell}_2+3\widetilde{\ell}_3+2\widetilde{\ell}_4+2E_{1}^2+4E_2^1+2E_2^2+E_3^1+E_4^2+2E_4^3+E_5^1+E_5^3,
\]
and the minimalization process is Figure \ref{fig:global_process_p=19}. 
\end{proof}
To conclude the section, we determine the Picard number and the transcendental part of the Hodge structure of a very general SNC $(p,1)$-plane with geometric genus one. 

 We denote by \(\mathcal M^p_{K3}\) the moduli space of K3 surfaces admitting a non-symplectic automorphism of order \(p\). For each family of \((p,1)\)-planes, let \(\eta\) denote the action on the K3 lattice induced by the covering automorphism on the minimal K3 model, up to conjugacy. We shall study the corresponding moduli locus \(\mathcal M_\eta\subseteq\mathcal M^p_{K3}\), parametrizing K3 surfaces equipped with a non-symplectic automorphism whose cohomological action is conjugate to \(\eta\).
\begin{proposition}\label{prop: formulae_rho_Y_min_pg=1}
Let $\pi: Y \to \mathbb{P}^2$ be a very general member of one of the families of SNC $(p,1)$-planes with $p_g(Y) = 1$ with building data  $\{D_g\}_{g\in G}$. 
Let $ \widetilde{Y}_{\min}$ be the minimal K3 model of the minimal resolution of singularities $\rho\colon \widetilde{Y}\to Y$, and let $\sigma$ be the non-symplectic automorphism of order $p$ on $\widetilde{Y}_{min}$ given by a generator of the Galois group of $\pi$. Let $S(\sigma):=\{x\in H^2(\widetilde{Y}_{min}, \mathbb Z)\colon \sigma^*(x)=x\}$. Then
\begin{itemize}
	\item $\operatorname{NS}(\widetilde{Y}_{min})_{\mathbb{Q}}=S(\sigma)_{\mathbb{Q}}$; 
	\item the Picard number of the minimal model $\widetilde{Y}_{min}$ of $Y$ is 
	\begin{equation}\label{eq: rho_min_model_K3_surf}
		\rho(\widetilde{Y}_{min})=22+(6+[p]_2)(p-1)-(p-1)\sum_{\underset{\Delta_g\leq D_g}{g\in G}}\left(\frac{(\deg\Delta_g+2)(\deg \Delta_g+1)}{2}-1\right),
	\end{equation}
	where $[p]_2\in \{0,1\}$ is the residue of $p$ modulo $2$;
	\item the Picard number of $Y$ is $\rho(Y)=1$;
	\item the transcendental part of the Hodge structure of $Y$ is of K3 type, and precisely is equal to 
	\[
	(h^{2,0}(Y), h^{1,1}(Y), h^{0,2}(Y))_{tra}=(1,20-\rho(\widetilde{Y}_{min}),1).
	\]
\end{itemize}
Moreover, the corresponding family of pairs
$(\widetilde Y_{\min},\sigma)$ is dense in the irreducible
moduli locus $\mathcal{M}_\eta$ associated with the cohomological action
$\eta=\sigma^*$. 
\end{proposition}
\begin{proof}
	Since $\sigma$ is non-symplectic, we have
\[
S(\sigma) \subseteq \operatorname{NS}(\widetilde{Y}_{min}),
\]
see  e.g., \cite[Thm. 2.1(a)]{AST}. We prove the equality of their ranks for the very general member of each family of $(p,1)$-planes in tables \ref{tab: (2,1)} to \ref{tab: (19,1)}.

The pushforward of the strict transforms of the ramification components of $\pi$ together with the exceptional curves of the resolution, which are invariant divisor classes on the minimal K3 model $\widetilde{Y}_{min}$, generate a sublattice $V$ of $S(\sigma)$. For $p=2$, the pullback of a line and the $N$ exceptional curves give $1+N$ independent invariant classes; for $p=3$, the preceding computations give $4+2N$ independent invariant classes, where $N$ denotes the number of nodes of the branch quartic. For $5\leq p\leq 19$, using \verb*|SAGE| \cite{Sag95} (see \
\href{https://github.com/Fefe9696/MAGMA_scripts_results_Classification_of_SNC_p_k_planes_with_Hodge_substructures_of_K3_type/tree/main/cyclic_cover_lattices}{\texttt{SAGE} scripts and results repository}), we computed the intersection pair matrix among those invariant divisor classes, which has been shown case-by-case that the rank of such a sublattice is $\operatorname{rk}(V)=1+\vert {\rm Exc}(\rho)\vert+K_{\widetilde{Y}}^2$, where ${\rm Exc(\rho)}$ is the set of irreducible curves contracted by $\rho$. Note that the same formula agrees with $1+N$ for $p=2$ and $4+2N$ for $p=3$. Thus, for any $p$, it gives then the inequality 
\begin{equation}\label{eq: rk_Sigma_is_at_least_TOT}
	1+\vert {\rm Exc}(\rho)\vert+K_{\widetilde{Y}}^2=\operatorname{rk}(V)\leq \operatorname{rk}(S(\sigma)).
\end{equation}
For the case when the branch locus of $\pi$ consists of four lines, namely one of the cases among $A-(5,1)3), A-(7,1)2), A-(11,1)1), A-(13,1)1), A-(17,1)1), A-(19,1)1)$, we have $\rk(V)=1+\vert {\rm Exc}(\rho)\vert+K_{\widetilde{Y}}^2=23-p$. 
Furthermore, as shown in the proof of Proposition \ref{prop: bound_on_the _prime_p}, we have $p-1\leq \operatorname{rk}T(\widetilde{Y}_{min})$, which implies $\rho(\widetilde{Y}_{min})\leq 23-p$. Hence
\[
23-p\leq \operatorname{rk}S(\sigma)\leq \rho(\widetilde{Y}_{min})\leq 23-p,
\]
and the equality of the ranks follows. Note that we did not use arguments involving a very general member of the family. Indeed, in such cases the corresponding family of the $(p,1)$-cover $\pi$ consists only of one point, which is the cover itself.  The period-domain dimension formula gives \(\dim\mathcal M_\eta=0\); hence the nonempty family coincides with \(\mathcal M_\eta\).

For the remaining families of $(p,1)$-planes to be studied, let $\mathcal B$ denote the parameter space of the corresponding branch configurations modulo projectivity, and consider the period map
\[
\Psi:\mathcal B\longrightarrow\mathcal M_\eta,
\qquad [\pi:Y\to\mathbb P^2]\longmapsto[\widetilde Y_{min},\sigma],
\]
where $\widetilde Y_{min}$ is the minimal K3 model of $Y$ and $\sigma$ is the induced automorphism.

{\bf Claim.} $\Psi$ has zero-dimensional fibres. 

Once the claim is proved, we can compare the dimension of $\mathcal B$ with the dimension of $\mathcal M_\eta$, computed using the period-domain dimension formula \cite[Sec. 9]{AST}. What we obtain is the inequality 
\[
 \dim \mathcal B= \sum_{\underset{\Delta_g\leq D_g}{g\in G}}\dim \vert \Delta_g\vert-8\leq \dim M_{\eta}= \frac{22-\rk (S(\sigma))}{p-1}-2+[p]_2,
\]
which gives the upper-bound 
\begin{equation}\label{eq: upper_bound_rk_Sigma}
	\operatorname{rk}(S(\sigma))\leq 22-\left(\sum_{\underset{\Delta_g\leq D_g}{g\in G}}\dim \vert \Delta_g\vert-6-[p]_2\right)(p-1).
\end{equation}
Surprisingly, one can check that the left-hand side of  \eqref{eq: rk_Sigma_is_at_least_TOT} and the right hand side of \eqref{eq: upper_bound_rk_Sigma} are equal case-by-case. This directly implies  \eqref{eq: upper_bound_rk_Sigma} is an equality. 
Since $\Psi$ has zero-dimensional fibres and $\dim\mathcal B=\dim\mathcal M_\eta$, it is dominant.
A very general member of our family therefore avoids the proper Noether–Lefschetz loci where additional classes become algebraic, yielding the desired equality of the ranks of $S(\sigma)$ and $\NS(\widetilde{Y}_{min})$. Thus, $\rho(\widetilde{Y}_{min})=\operatorname{rk} S(\sigma)$ and the equality \eqref{eq: upper_bound_rk_Sigma} gives Formula \eqref{eq: rho_min_model_K3_surf}.  
\\
 Moreover, being also \eqref{eq: rk_Sigma_is_at_least_TOT} an equality, we also have $\rho(\widetilde{Y}_{min})=1+\vert{\rm Exc}(\rho)\vert +K^2_{\widetilde{Y}}$. We observe that since $\widetilde{Y}_{min}$ is a K3 surface, then the number of $(-1)$-curves to contract to pass from $\widetilde{Y}$ to the minimal model equals exactly $-K^2_{\widetilde{Y}}$. Thus, the Picard number of $\widetilde{Y}$ is equal to 
$\rho(\widetilde{Y})=1+\vert{\rm Exc}(\rho)\vert$.
Consequently, the equality $\rho(Y)=1$ follows by subtracting from $\rho(\widetilde{Y})$ the number of curves contracted by the  resolution $\rho$.

Finally, resolutions and blow-ups only add algebraic classes, so the transcendental Hodge structure of \(Y\) is isomorphic to that of \(\widetilde Y_{\min}\).

${}$

{\it proof of the Claim.} Let us consider the case $p = 2$.
Note that in this case, the preceding computations show that $\widetilde Y$ is already minimal, so $\widetilde Y=\widetilde Y_{\min}$. On a fixed K3 surface $\widetilde{Y}_{min}$, the pullback $\mathcal L:=\widetilde\pi^*\mathcal O_{\mathbb P^2}(1)$ satisfies  $h^0(\widetilde{Y}_{min},\mathcal L)=3$ and $L^2=2$, so the complete linear system \(|L|\) determines the degree $2$ morphism $\widetilde{\pi}$ up to projective transformations of the target. Since \(\operatorname{Pic}^0(\widetilde Y_{\min})=0\), then ${\rm Pic}(\widetilde{Y}_{min})$ is a countable set, so there are only countably many line bundles $\mathcal L$ on the fixed K3 surface that can arise as \(\widetilde\pi^*\mathcal O_{\mathbb P^2}(1)\) from such $(2,1)$-plane presentations.  Thus the fibres of $\Psi$ are countable and hence zero-dimensional.

For $p \in \{ 3, 5\}$ we proceed in a similar way.  Note that the resolution $\widetilde{Y}$ is recovered by blowing up an isolated fixed point of $\widetilde{Y}_{min}$ (see Lemma \ref{lem: (p,1)_planes_pg1_case_p=2,3,5}). For the fixed pair \((\widetilde{Y}_{min},\sigma)\), there are only finitely many such points. On each resulting surface $\widetilde{Y}$, the possible pullback line bundles form a countable set, as in the case \(p=2\), and the three-dimensional space of invariant sections for the natural linearization determines a degree $3$ (resp. 5) morphism $\widetilde{\pi}$ up to projective transformations. Thus the fibres of the map $\Psi$ to moduli are zero-dimensional.

It only remains to discuss the case $A-(7,1)1)$. Here we reverse the contractions from $\widetilde{Y}$ to $\widetilde{Y}_{min}$ described in Lemma \ref{lem: (p,1)_planes_pg1_case_p=7}, see Figure \ref{fig: A-(7,1)(1)_config}. Thus, first, we blow up the point \(q\) on the K3 surface obtained by contracting \(E_1^1\), and then a point \(a\) on the resulting exceptional curve, recovering the ramification curve \(\widetilde\ell_1\). Since \(\widetilde\ell_1\) is fixed pointwise, the tangent action at \(a\) is scalar and nontrivial. The local blow-up calculation therefore shows that the tangent eigenvalues at \(q\) are \(\lambda,\lambda^2\), with \(\lambda\) a nontrivial seventh root of unity. Thus \(q\) is an isolated fixed point, and the action on its exceptional curve is nontrivial, with exactly two fixed points from which \(a\) can be chosen. Finally, \(\widetilde\ell_2\) is recovered by blowing up another isolated fixed point of the K3 surface, away from \(q\). All blow-up centres therefore have only finitely many possibilities. The argument is then the same as in the previous cases, which implies that the fibres of the map $\Psi$ to moduli are zero-dimensional. 
\end{proof}
\begin{remark}\label{rem:diPirola}
The rank of the invariant lattice determines the dimension of the moduli space with the prescribed non-symplectic action. The Picard number may nevertheless increase on Noether–Lefschetz loci, where additional non-invariant classes become algebraic. A family contained in such a locus may therefore have \(\operatorname{NS}(S)_{\mathbb Q}\supsetneq S(\sigma)_{\mathbb Q}\) even for its very general member. The density established in Proposition \ref{prop: formulae_rho_Y_min_pg=1} 
\end{remark}
	
\begin{remark}
For a very general member of each family in Tables \ref{tab: (5,1)}--\ref{tab: (19,1)}, the
invariant lattice, and hence the Néron--Severi lattice of
$\widetilde Y_{\min}$ by Proposition \ref{prop: formulae_rho_Y_min_pg=1} 
is as follows:
\[
\begin{array}{c|c|c}
\text{case} & \operatorname{rk}S(\sigma) & S(\sigma) \\ \hline
A-(5,1)1)  &  6 & H_5\oplus A_4\\
A-(5,1)2)  & 14 & H_5\oplus A_4\oplus E_8\\
A-(5,1)3)  & 18 & H_5\oplus E_8^{\oplus 2}\\
A-(7,1)1)  & 10 & U(7)\oplus E_8\\
A-(7,1)2)  & 16 & U\oplus E_8\oplus A_6\\
A-(11,1)1) & 12 & U\oplus A_{10}\\
A-(13,1)1) & 10 & H_{13}\oplus E_8\\
A-(17,1)1) &  6 & U\oplus L_{17}\\
A-(19,1)1) &  4 & U\oplus K_{19}
\end{array}
\]
The intersection matrices and their lattice identifications are
recorded in the \verb|SAGE| script and its output log
\verb|cyclic-cover-lattices.log|. The resulting invariant lattices agree
with the classification of Artebani--Sarti--Taki~\cite{AST}.

These families describe the moduli loci associated with their
respective cohomological actions; our $\mathcal M_{\eta}$ need not be maximal
irreducible components of the moduli space $\mathcal M_{K3}^p$.

For $p=3$, write $Q=D_{e_1}$ and let $N$ and $s$ denote
the number of nodes and irreducible components of $Q$,
respectively. The fixed locus on the minimal K3 model
consists of the $s$ ramification curves lying over the
components of $Q$, one isolated point over each node of $Q$,
and the isolated point obtained by contracting the ramification
curve over $D_{2e_1}$.
Thus the pairs $(n,k)$ (in the notation of \cite[Table~2]{AS}) are, respectively,
$(1,1)$, $(4,2)$, $(5,2)$, $(6,3)$ and $(7,4)$;
the following lattice identifications follow from
\cite[Table~2]{AS}.
\[
\begin{array}{c|c}
\text{case} & S(\sigma) \\ \hline
A-(3,1)1) & U(3)\oplus A_2\\
A-(3,1)2) & U\oplus A_2^{\oplus4}\\
A-(3,1)3) & U\oplus A_2^{\oplus5}\\
A-(3,1)4) & U\oplus E_6\oplus A_2^{\oplus3}\\
A-(3,1)5) & U\oplus E_6^{\oplus2}\oplus A_2
\end{array}
\]
For $p=2$, the invariant lattices can be identified using
the classification of $2$-elementary lattices; see for example \cite{vGS} and ~\cite{GP22} for computations of this kind and the works of Nikulin \cite{NKummer, N, NikSym}.
We do not pursue these identifications here, since the ranks
established in Proposition \ref{prop: formulae_rho_Y_min_pg=1} 
suffice for our purposes.
\end{remark}	
	
We conclude the section with the following theorem, which summarizes the results obtained above. 
\begin{theorem}\label{thm: pg1_Zp_cover_structure}
	Let $\pi\colon Y\to \mathbb P^2$ be a SNC $(p,1)$-plane with $p_g(Y)=1$. Then the minimal model of $Y$ is a K3 surface and the building data 
	of the cover falls in one of the cases listed in tables \ref{tab: (2,1)} to \ref{tab: (19,1)} in Appendix \ref{app:classification_tables}. 
	Furthermore, a very general member of the family has Picard number one, i.e. $\rho(Y)=1$, and we also provide for such member the Picard number $\rho(\widetilde{Y}_{min})$  of the minimal model of $Y$ together with the transcendental part of its Hodge structure. 
\end{theorem}
\begin{proof}
	The cases listed in tables \ref{tab: (2,1)} to \ref{tab: (19,1)} are attained in Proposition \ref{prop: complete_classification_p,1_planes_pg=1}. The proof that each obtained case has a minimal model that is a K3 surface follows by lemmas \ref{lem: (p,1)_planes_pg1_case_p=2,3,5}, \ref{lem: (p,1)_planes_pg1_case_p=7}, and \ref{lem: (p,1)_planes_pg1_case_p=11,13,17,19}. 
	
	Finally, the proof that a very general member of the family has Picard number one and the computation of the Picard number of its minimal model together with its transcendental Hodge structure can be found in tables \ref{tab: (2,1)} to \ref{tab: (19,1)} and are computed case-by-case using the formulae in Proposition \ref{prop: formulae_rho_Y_min_pg=1}. 
\end{proof}
\section{SNC $(p,1)$-planes with geometric genus zero}\label{sec:SNC_(p,1)-planes_zero}
In this section,  we classify SNC $(p,1)$-planes $\pi\colon Y \to \mathbb{P}^2$ with $p_g(Y)=0$. The main theorem of the section is the following
\begin{theorem}\label{thm: main_thm_(p,1)_planes_pg=0}
	 Let $\pi\colon Y\to \mathbb P^2$ be a SNC $(p,1)$-plane. Let $\rho\colon \widetilde{Y}\to Y$ be the minimal resolution of singularities of $Y$. Then $\widetilde{Y}$ is a rational surface if and only if $p_g(Y)=0$. In this case, the building data of the cover falls in one of the cases listed in tables \ref{tab: (p,1)_pg=0_p=2} to \ref{tab: (p,1)_pg=0_p>=7_data}, where we also provide the Picard number $\rho(\widetilde{Y})=h^{1,1}(\widetilde{Y})$ of $\widetilde{Y}$ together with the Picard number $\rho(Y)$ of $Y$. 
\end{theorem}
\begin{proof}
	The cases listed from tables \ref{tab: (p,1)_pg=0_p=2} to \ref{tab: (p,1)_pg=0_p>=7_data} are attained in Proposition \ref{prop: p,1-cover_with_pg=0}. The rationality of the obtained cases is proved in Proposition \ref{prop: rationality_p,1_planes_pg=0} below. Furthermore, rational surfaces have $\chi(\widetilde{Y})=1$, $q(\widetilde{Y})=0$  and maximal Picard number $\rho(\widetilde{Y})=h^{1,1}(\widetilde{Y})$, so from Equation \eqref{eq: formula_h^11_general} we have
	\[
	\rho(\widetilde{Y})=h^{1,1}(\widetilde{Y})=10-K^2_{\widetilde{Y}}.
	\]
	The self-intersection of the canonical class of $\widetilde{Y}$ can be easily computed with Formula \eqref{eq: KS_ab_cov_quo_sing} once we determine the type of cyclic quotient singularities of $Y$, using Proposition \ref{prop: type_of_sing}. 
	
	Finally, the Picard group of $Y$ is obtained from $\widetilde{Y}$ by contracting all the exceptional curves of the resolution of $Y$, so $\rho(Y)=\rho(\widetilde{Y})-\sum_{x\in \text{Sing}(Y)}l_x$, where $l_x$ is the length of the Hirzebruch-Jung string of the resolution of the singular point $x$. 
\end{proof}
 \begin{table}[H]
	\footnotesize
	\setlength{\tabcolsep}{5pt}
	\renewcommand{\arraystretch}{1.2}
	\begin{tabular}{|c||c|c|c|c|c||c|}\hline
		case& $D_{1}$&$\textrm{Sing}(Y)$ & $K_{\widetilde{Y}}^2$ & $\rho(\widetilde{Y})$ & $\rho(Y)$ & \textrm{Surface} \\
		\hline
		 $1$&conic&{\rm $\emptyset$} & 8 &  2&  2&$\mathbb P^1\times \mathbb P^1$ \\
		\hline
		$2$ &2 lines & {\rm $1/2$}&  8 & 2& 1& $\mathbb F_2$\\
		\hline
		$3$ &quartic& {\rm $\emptyset$}& 2&8& 8&{\rm del Pezzo  of degree 2}\\
		\hline
		$4$ & line + cubic& {\rm $1/2^3$}&  2&  8& 5& {\rm Weak del Pezzo  of degree 2}\\
		\hline
		$5$ &2 conics& {\rm $1/2^4$}&  2&  8& 4&{\rm Weak  del Pezzo  of degree 2} \\
		\hline
		$6$ &  2 lines + conic& {\rm $1/2^5$}&  2&  8& 3& {\rm Weak del Pezzo  of degree 2}\\
		\hline
		$7$ & 4 lines& {\rm $1/2^6$}&  2&  8& 2& {\rm Weak del Pezzo  of degree 2}\\
		\hline
	\end{tabular}
	\caption{SNC $(2,1)$-planes $\pi\colon Y\to \mathbb P^2$ with $p_g(Y)=0$.} \label{tab: (p,1)_pg=0_p=2}
\end{table}
 \begin{table}[H]
	\footnotesize
	\setlength{\tabcolsep}{5pt}
	\renewcommand{\arraystretch}{1.2}
	\begin{tabular}{|c||c|c|c|c|c|c||c|}\hline
		case& $D_{1}$& $D_2$&$\textrm{Sing}(Y)$ & $K_{\widetilde{Y}}^2$ & $\rho(\widetilde{Y})$ & $\rho(Y)$ & \textrm{Surface} \\
		\hline
		$1$&line&line&{\rm $1/3$} & 8 &  2&  1&$\mathbb F_3$ \\
		\hline
		$2$ &cubic &0& {\rm $\emptyset$}&  3 & 7& 7& {\rm del Pezzo  of degree 3}\\
		\hline
		$3$ &line + conic&0& {\rm $2/3^2$}& 3&7& 3&{\rm Weak del Pezzo  of degree 3}\\
		\hline
		$4$ & 3 lines&0& {\rm $2/3^3$}&  3&  7& 1& {\rm Weak del Pezzo  of degree 3}\\
		\hline
		$5$ & conic&conic& {\rm $1/3^4$}&  -1&  11& 7& {\rm Rational}\\
		\hline
		$6$ & 2 lines&conic& {\rm $2/3, 1/3^4$}&  -1&  11& 5& {\rm Rational}\\
		\hline
		$7$ & 2 lines&2 lines& {\rm $2/3^2, 1/3^4$}&  -1&  11& 3& {\rm Rational}\\
		\hline
	\end{tabular}
	\caption{SNC $(3,1)$-planes $\pi\colon Y\to \mathbb P^2$ with $p_g(Y)=0$.} \label{tab: (p,1)_pg=0_p=3}
\end{table}
 \begin{table}[H]
	\footnotesize
	\setlength{\tabcolsep}{5pt}
	\renewcommand{\arraystretch}{1.2}
	\begin{tabular}{|c||c|c|c|c|c|c|c|c||c|}\hline
		case& $D_{1}$& $D_2$&$D_3$&$D_4$&$\textrm{Sing}(Y)$ & $K_{\widetilde{Y}}^2$ & $\rho(\widetilde{Y})$ & $\rho(Y)$ & \textrm{Surface} \\
		\hline
		$1$&line&0&0&line&{\rm $1/5$} & 8 &  2&  1&$\mathbb F_5$ \\
		\hline
		$2$ &line&conic&0&0& {\rm $3/5^2$}&  1 & 9& 5& {\rm Rational}\\
		\hline
		$3$ &line&2 lines& 0&0&{\rm $3/5^2,4/5$}& 1&9& 1&{\rm Rational}\\
		\hline
		$4$ & conic&0& 0&conic&{\rm $1/5^4$}&  -7&  17& 13& {\rm Rational}\\
		\hline
		$5$ & 2 lines&0&0 &conic&{\rm $1/5^4, 4/5$}&  -7&  17& 9& {\rm Rational}\\
		\hline
		$6$ & 2 lines&0& 0&2 lines&{\rm $1/5^4, 4/5^2$}&  -7&  17& 5& {\rm Rational}\\
		\hline
		$7$ & line&line&line &line&{\rm $1/5^2, 2/5^2, 3/5^2$}&  -5&  15& 5& {\rm Rational}\\
		\hline
	\end{tabular}
	\caption{SNC $(5,1)$-planes $\pi\colon Y\to \mathbb P^2$ with $p_g(Y)=0$.} \label{tab: (p,1)_pg=0_p=5}
\end{table}
\vspace{-0.3cm}

\begin{table}[H]
	\small
	\begin{center}
		\begin{tabular}{|c||c|c|c|c|c|c|c|}\hline
				case& $D_{1}$& $D_g$&$D_{\frac{p-1}{2}}$&$D_{p-1-g}$ & $D_{p-g}$&$D_{p-1}$&$\textrm{Sing}(Y)$ \\
				\hline
				$1$&line&0&0&0&0&line&{\rm $1/p$} \\
				\hline
				$2$ &conic&0&0&0& 0&conic&{\rm $(1/p)^4$}\\
				\hline
				$3$ &2 lines&0& 0&0&0&conic&{\rm $(1/p)^4,(p-1)/p$}\\
				\hline
				$4$ & 2 lines&0& 0&0&0&2 lines &{\rm $(1/p)^4,((p-1)/p)^2$}\\
				\hline
				$5$ & line&0&conic &0&0&0&{\rm $(((p+1)/2)/p)^2$}\\
				\hline
				$6$ & line&0& 2 lines&0&0&0&\makecell{$(((p+1)/2)/p)^2,$ \\$(p-1)/p$}\\
				\hline
				\makecell{$7$ \\ $2\leq g<\frac{p-1}{2}$} & line&line&0 &line&0&0& \makecell{$(p-g)/p, (g+1)/p,$\\ $ (g^{-1}+1)/p$}\\
				\hline
				\makecell{$8$ \\ $2\leq g\leq \frac{p-1}{2}$} & line&line&0 &0&line&line&\makecell{$(p-g)/p, g/p, (1/p)^2,$ \\ $g^{-1}/p,(p-g^{-1})/p$}\\
				\hline
		\end{tabular}
		\caption{SNC $(p,1)$-planes $\pi\colon Y\to \mathbb P^2$ with $p_g(Y)=0$, and $p\geq 7$.} \label{tab: (p,1)_pg=0_p>=7_data}
	\end{center}
\end{table}
\begin{table}[H]
	\small
	\begin{center}
		\begin{tabular}{|c||c|c|c||c|}\hline
				case&  $K_{\widetilde{Y}}^2$ & $\rho(\widetilde{Y})$ & $\rho(Y)$ & \textrm{Surface} \\
				\hline
				$1$& 8 &  2&  1&$\mathbb F_p$ \\
				\hline
				$2$ &  $8-3p$ & $2+3p$& $3p-2$& {\rm Rational}\\
				\hline
				$3$ & $8-3p$&$2+3p$& $2p-1$&{\rm Rational}\\
				\hline
				$4$ &  $8-3p$&  $2+3p$& $p$& {\rm Rational}\\
				\hline
				$5$ & $6-p$&  $4+p$& $p$& {\rm Rational}\\
				\hline
				$6$ &   $6-p$&  $4+p$& $1$& {\rm Rational}\\
				\hline
				\makecell{$7$ \\ $2\leq g<\frac{p-1}{2}$} &  $9-\left(l_{\frac{p-g}{p}}+l_{\frac{g+1}{p}}+l_{\frac{g^{-1}+1}{p}}\right)
				$&  $1+\left(l_{\frac{p-g}{p}}+l_{\frac{g+1}{p}}+l_{\frac{g^{-1}+1}{p}}\right)
				$& $1$& {\rm Rational}\\
				\hline
				\makecell{$8$ \\ $2\leq g\leq \frac{p-1}{2}$} &  $8-p-2\left(l_{\frac{g}{p}}+l_{\frac{p-g}{p}}\right)$&  $2+p+2\left(l_{\frac{g}{p}}+l_{\frac{p-g}{p}}\right)$& $p$& {\rm Rational}\\
				\hline
		\end{tabular}
		\caption{Invariants of SNC $(p,1)$-planes with $p_g(Y)=0$, and $p\geq 7$ of Table \ref{tab: (p,1)_pg=0_p>=7_data}.} \label{tab: (p,1)_pg=0_p>=7_invariants}
	\end{center}
\end{table}
The rest of this section is devoted to give a proof of the two propositions \ref{prop: p,1-cover_with_pg=0} and \ref{prop: rationality_p,1_planes_pg=0} used to prove Theorem \ref{thm: main_thm_(p,1)_planes_pg=0}. 
\begin{proposition}\label{prop: p,1-cover_with_pg=0}
	Let $\pi\colon Y \to \mathbb{P}^2$ be a  SNC $(p,1)$-plane. Then  $p_g(Y)=0$ if and only if we have one of the following possibilities for the degrees of the building data $\{\mathcal L_j\}_{j\in \mathbb Z/p}$ and $\{D_j\}_{j\in \mathbb Z/p}$ of the cover, depending on $p$:
	\begin{itemize}
		\item[$p=2$:] \begin{enumerate}
			\item[(1)] $l_1=1$, and $d_1=2$; or 
			 \item[(2)] $l_1=2$, and $d_1=4$.
		\end{enumerate}
	\item[$p=3$:] \begin{enumerate}
		\item[(1)] $l_1=l_2=1$, and $d_1=d_2=1$; 
		\item[(2)] $l_1=1, l_2=2$, and $d_1=3$;
		\item[(3)] $l_1=l_2=2$, and $d_1=d_2=2$.
	\end{enumerate} 
   \item[$p=5$:] \begin{enumerate}
   	  \item[(1)] $l_j=1$, for $j\neq 0$, and $d_1=d_4=1$; 
   	  \item[(2)] $l_1=l_3=1$, $l_2=l_4=2$, and $d_1=1$, $d_2=2$; 
   	  \item[(3)] $l_j=2$, for $j\neq 0$, and $d_1=d_4=2$; 
   	  \item[(4)] $l_j=2$, for $j\neq 0$,and $d_1=d_2=d_3=d_4=1$.
   \end{enumerate}
  \item[$p\geq 7$] \begin{enumerate}
  	\item[(1)] $l_j=1$, for $j\neq 0$, and $d_1=d_{p-1}=1$; 
  	\item[(2)] $l_j=2$, for $j\neq 0$, and $d_1=d_{p-1}=2$; 
  	\item[(3)] $l_j=2-[j]_2$, for $j\neq 0$, and $d_1=1$, $d_{\frac{p-1}{2}}=2$; 
  	\item[(4)] for each chosen element $g\in \mathbb Z/p$ with $2\leq g<\frac{p-1}{2}$, we have
  	\\
  	$l_{j}=1+\left\lfloor \frac{[jg]_p+j-1}{p}\right\rfloor\in\{1,2\}$, for $j\neq 0$, and $d_1=d_g=d_{p-1-g}=1$; 
  	\item[(5)] for each chosen element $g\in \mathbb Z/p$ with $2\leq g\leq \frac{p-1}{2}$, we have 
  	\\
  	$l_j=2$, for $j\neq 0$, and $d_1=d_g=d_{p-g}=d_{p-1}=1$. 
  \end{enumerate}
	\end{itemize}
\end{proposition}
\begin{proof}
	From Equation \eqref{eq: pg_ab_cov_quo_sing} in Theorem \ref{thm: invariants_ab_cov_cyclic_sing}, we have 
	\[
	0=p_g(Y)=\sum_{j=0}^{p-1}h^0(\mathbb P^2, \mathcal{O}_{\mathbb P^2}(-3+l_j))=\sum_{j=1}^{p-1}\frac{(l_j-1)(l_j-2)}{2},
	\]
	so that $l_j\in \{1,2\}$ for all $j\neq 0$. However, from Theorem \ref{thm: Dg_elm_group}(7), we also have 
	\begin{equation}\label{eq: sum_dg_less_4}
			\frac{p-1}{2}\sum_{j=1}^{p-1} d_g=\sum_{j=1}^{p-1}{l_j}\leq 2(p-1), \qquad \makebox{so} \qquad \sum_{j=1}^{p-1}d_j\leq 4. 
	\end{equation}
	We analyse case-by-case according to the possible values of $ \sum_{j=1}^{p-1}d_j\in \{1,2,3,4\}$. 
	We automatically discard the case $\sum_{j=1}^{p-1}d_j=1$ because otherwise from Lemma \ref{lem: L_chi_in_function_D_g} we would have $pl_1=\sum_{j=1}^{p-1}jd_j<p$, so $l_1=0$, a contradiction. 
	
	Up to automorphisms of $\mathbb Z/p$, we can always assume from the rest of the proof that $d_1\neq 0$.
	Let us consider the case where $\sum_{j=1}^{p-1}d_j=2$.  If $d_1=2$, then from Lemma \ref{lem: L_chi_in_function_D_g} we would have $pl_1=2$, which forces $p=2$, $l_1=1$ and $d_1=2$. This is the case $p=2(1)$. 
	
	Instead, if $d_1=1$, then there is only another  $g\neq 0, 1$ such that $d_g=1$. In this case, still by Lemma \ref{lem: L_chi_in_function_D_g}, we would have $pl_1=1+g\cdot 1\leq 1+(p-1)=p$, so $l_1=1$ and $g=p-1$. Furthermore, from Lemma \ref{lem: L_chi_in_function_D_g}, we have $pl_j=j+(p-j)\cdot 1=p$, so $l_j=1$ for each $j\neq0$. This is $p=3(1)$, $p=5(1)$, and $p\geq 7(1)$. 
	
	We completely covered all the possibilities for $\sum_{j=1}^{p-1}d_j=2$, so let us consider the next case $\sum_{j=1}^{p-1}d_j=3$.  
	
	If $d_1=3$, then we would simply have $pl_1=3$, which forces $p=3$, $l_1=1$, and $d_1=3$. This is the case $p=3(2)$. 
	
	If $d_1=1$ and $d_g=2$ for some other $g\neq 0,1$, then $pl_1=1+g\cdot 2\leq 1+2(p-1)=2p-1$. This implies $l_1=1$ and $g=\frac{p-1}{2}$; moreover 
	\[
	pl_j=j+\left[\frac{p-1}{2}j\right]_p\cdot 2=2p-p[j]_2, \qquad \makebox{so} \qquad l_j=2-[j]_2. 
	\]
	This is the case $p=5(2)$ and $p\geq 7(3)$. Note that the case $d_1=2$ and $d_g=1$ can be skipped because we would get a corresponding covering isomorphic to the case just studied, via the automorphism $g^{-1}$.
   \\
   Thus, let us discuss the case $d_1=d_{g_1}=d_{g_2}=1$ for some $1<g_1<g_2\leq p-1$. From Lemma \ref{lem: L_chi_in_function_D_g}, we have $pl_1=1+g_1\cdot 1+g_2\cdot 1\leq 1+(p-1)+(p-1)=2p-1$, so $l_1=1$ and then $g_2=p-1-g_1$. For simplicity, we denote $g:=g_1$, so now the building data are $d_1=d_g=d_{p-1-g}=1$. In this case, we have 
   \[
     pl_j=j+[jg]_p+[(p-1-g)j]_p=j+[jg]_p+\left(p-j-[jg]_p+\left\lfloor \frac{[jg]_p+j-1}{p}\right\rfloor\cdot p\right),
   \]
   so $l_j=1+\left\lfloor \frac{[jg]_p+j-1}{p}\right\rfloor$; we obtained the case $p\geq 7(4)$. Note that from the symmetry of the local monodromies $g$ and $p-1-g$, all the possible cases are obtained choosing $2	\leq g\leq \frac{p-1}{2}$. Furthermore, the choice $g=\frac{p-1}{2}$ must be discarded otherwise we would have $g_1=g_2=\frac{p-1}{2}$. 
   
   We continue with the last case $\sum_{j=1}^{p-1}d_j=4$. In this case, the left and right hand side of Equation \eqref{eq: sum_dg_less_4} are the same, that is equal to $2(p-1)$, and this can happen only if
   \begin{equation}\label{eq: sumd_eq_4_imply_lj_eq_2}
   	l_j=2 \qquad \text{for \ each}\qquad j\neq 0.    
   \end{equation}
   Assume that $d_1=4$. Then from Lemma \ref{lem: L_chi_in_function_D_g} we have $pl_1=4$, which forces $p=2$, $l_1=2$ and $d_1=4$. This is $p=2(2)$ case. 
   \\
   Let us consider now the case $d_1=1$ and $d_g=3$ for some $g\neq 0,1$. Then $2p=pl_1=1+g\cdot 3$ so we must have $p\equiv 2 \mod 3$ and $g=\frac{2p-1}{3}$. However, 
   \[
     pl_{3}=3+\left[3g\right]_p\cdot 3=3+(p-1)3=3p, \qquad \makebox{so} \qquad l_3=3,
   \]
   a contradiction. We can exclude also the symmetric case  $d_1=3$ and $d_g=1$, $g\neq 0,1$, since it is isomorphic to the previous case by applying the automorphism  $g^{-1}$.
   \\
   Let us study what happens for $d_1=d_g=2$, for some $g\neq 0,1$. In this case, still from Lemma \ref{lem: L_chi_in_function_D_g}, we have that $2p=pl_1=2+g\cdot 2$, so $g=p-1$. This is the case $p=3(3)$, $p=5(3)$, and $p\geq 7(2)$. 
   \\
   Let us consider the case $d_1=d_{g_1}=1$ and $d_{g_2}=2$, where $g_1,g_2\neq 0,1$ and $g_1\neq g_2$. 
   Pardini equations \eqref{eq: reduction_numb_Pard_eq} are satisfied when
  \begin{equation}\label{eq: Pardini_eq_cyclic_prime_p}
  	2p=pl_j=j+[jg_1]_p+[jg_2]_p\cdot 2 \qquad \makebox{for all} \qquad j=1,\dots, p-1.
  \end{equation}
   This would imply that the zero-sum sequence (cf. \cite[Def. 1.1]{IndexZeroSum}) $s=(1,g_1,g_2,g_2)$ over the cyclic group $\mathbb Z/p$ of prime order $p$ has index $ind(s)=2$. However, by the Index Theorem for length $4$ zero-sum sequences \cite[Thm. 1.2]{IndexZeroSum}, any minimal zero-sum sequence of length $4$ have to be of index $ind(s)=1$. This forces our $s$ to not be a minimal zero-sum sequence, so there exists a proper subsequence of $s$ that is a zero-sum sequence. This proper subsequence could not have length $1$ or $3$, otherwise one of the elements of $s$ would be a multiple of $p$, because from \eqref{eq: Pardini_eq_cyclic_prime_p} we have $2p=1+g_1+2g_2$. Thus, it there exists a zero-sum length $2$ subsequence of $s$. If this subsequence is not containing  $1$, then the condition $2p=1+g_1+2g_2$ would  implies that the complementary length $2$ subsequence of $s$, which is containing $1$, has zero-sum. Thus, one between $g_1$ and $g_2$ has to be equal to $p-1$, forcing either $2g_2=p$, if $g_1=p-1$, or $g_1=1$ if $g_2=p-1$. In both cases, we have a contradiction, so we can exclude the case $d_1=d_{g_1}=1$ and $d_{g_2}=2$. 
   
   It only remains to study the case $d_1=d_{g_1}=d_{g_2}=d_{g_3}=1$. Here we adopt a similar strategy as in the previous case. Indeed, remembering  \eqref{eq: sumd_eq_4_imply_lj_eq_2} and  Pardini equations \eqref{eq: reduction_numb_Pard_eq} we have 
    \[
   2p=pl_j=j+[jg_1]_p+[jg_2]_p+[jg_3]_p\qquad \makebox{for all} \qquad j=1,\dots, p-1,
   \] 
   so  $s=(1,g_1,g_2,g_3)$ is a zero-sum length $4$ sequence of index $ind(s)=2$, which then must be not minimal, by \cite[Thm. 1.2]{IndexZeroSum}. Hence, such as in the previous case, it there exists a zero-sum length $2$ subsequence of $s$, which is containing $1$ (use the condition $2p=1+g_1+g_2+g_3$). Thus,  one among $g_1,g_2,g_3$ has to be equal to $p-1$, let us say $g_3$, and so this also forces $g_2=p-g_1$. We have obtained the building data $d_1=d_g=d_{p-g}=d_{p-1}=1$, which corresponds to the cases $p=5(4)$ and $p\geq 7(5)$.   
\end{proof}

We are ready to prove that any surface belonging to one of the  families obtained in Proposition \ref{prop: p,1-cover_with_pg=0} is rational. 
\begin{proposition}\label{prop: rationality_p,1_planes_pg=0}
	Let $\pi\colon Y\to \mathbb P^2$ be a SNC $(p,1)$-plane with $p_g(Y)=0$. Then the minimal resolution of singularities $\widetilde{Y}$ of $Y$ is a rational surface. 
\end{proposition} 
\begin{proof}
	From Lemma \ref{lem: (p,k)_planes_are_regular}, then $\widetilde{Y}$ is a regular surface, so $q(\widetilde{Y})=0$. Thus, if we prove that $h^0(\widetilde{Y}, 2K_{\widetilde{Y}})=0$, then $\widetilde{Y}$ is a rational surface from Castelnuovo Rationality criterion. 
	
	The strategy is to find a nef divisor $D$ on $\widetilde{Y}$ such that $K_{\widetilde{Y}}D<0$. Indeed, in this case no multiples of $K_{\widetilde{Y}}$ would be effective divisors and so $h^0(\widetilde{Y}, 2K_{\widetilde{Y}})=0$. 
	
	From Equation \eqref{eq: KS_ab_cov_quo_sing}, then we have 
	\begin{equation}
		pK_{\widetilde{Y}}\equiv \widetilde{\pi}^*((-3p+(p-1)\sum_gd_g)H)- E,
	\end{equation}
where $E$ is a divisor supported on the Hirzebruch-Jung strings of the resolution of $Y$. In particular, if  $-3p+(p-1)\sum_gd_g<0$, which means $\sum_gd_g<3+\frac{3}{p-1}$, then $D:=\widetilde{\pi}^*(H)$ is a nef divisor with 
\[
K_{\widetilde{Y}}D=-3p+(p-1)\sum_gd_g<0,
\]
so $\widetilde{Y}$ is a rational surface. This automatically proves the rationality of the cases $p=2(1)$ and $(2)$, $p=3(1)$,$(2)$, and $(3)$, $p=5(1)$ and $(2)$, and finally $p\geq7(1)$, $(3)$, and $(4)$. 

Let us discuss the remain cases $p=5(3)$ and $(4)$, and $p\geq 7(2)$ and $(5)$. For these cases, the divisor $D$ is constructed as follows. We prove that there exist two $(-1)$-curves $E_1$ and $E_2$ of a birational model $S$ of $\widetilde{Y}$ intersecting each other. Indeed, in this case $D:=E_1+E_2$ is a nef divisor with $K_{S}D=-2<0$. Thus, $S$ is rational and then $\widetilde{Y}$ is rational too. 

All of these remain cases can be treated in an analogue way by observing that one can select a line $\ell\subseteq \mathbb P^2$ not belonging to the branch locus of $\pi$ but passing through exactly two intersection points of the branch locus, the points whose inverse images $p_1$ and $p_2$ on $Y$ are two singularities of type $\frac{1}{p}(1,1)$; see Figure \ref{fig: p>=7_2_is_rational} where we consider the case $p\geq7(2)$. 

We claim that $\widetilde{\pi}^*(\ell)$ consists of $p$ disjoint $(-1)$-curves on $\widetilde{Y}$, each of them intersecting transversally in one point two exceptional curves $E_1$ and $E_2$ of self-intersection $-p$, those obtained by the resolution of the two singularities $p_1$ and $p_2$. If this would be true, then after contracting $p-1$ of these $(-1)$-curves, we would obtain two $(-1)$-curves, the images of $E_1$ and $E_2$, intersecting each other. This would prove the rationality of $\widetilde{Y}$. 

To prove the claim, let us blowing-up ${\rm Bl}_2(\mathbb P^2)\to \mathbb P^2$ the two intersection points of the branch locus of $\pi$ where $\ell$ is passing from. One can verify via Pardini equations \eqref{eq: Pardini_eq} that we can construct another $(\mathbb Z/p)$-cover $\pi'\colon Y'\to \text{Bl}_2(\mathbb P^2)$ having building data that are the strict transforms on $\text{Bl}_2(\mathbb P^2)$ of the components of the branch locus of $\pi$. We point out that $Y'$ could still be singular but it is resolving the singularities of the two singular points $p_1$ and $p_2$ of $Y$. Indeed,  denoted by $T_1$ and $T_2$ the two $(-1)$-curves of the blow-up $\text{Bl}_2(\mathbb P^2)$, then $E_i:=(\pi')^*(T_i)$ is a smooth rational curve of self-intersection $-p$ by projection formula, for   $i=1,2$. The contraction of these exceptional curves $E_1$ and $E_2$ is the resolution $Y'\to Y$ of the singular points $p_1$ and $p_2$. 
We have the following commutative diagram
\[
\begin{tikzcd}[row sep=small, column sep=small]
	\widetilde{Y} 
	\arrow[r] 
	\arrow[rr, "\rho", bend left=25] 
	\arrow[to=P2, "{\widetilde{\pi}}"', bend right=45] 
	& Y' 
	\arrow[r] 
	\arrow[d, "{\pi'}"'] 
	& Y 
	\arrow[d, "\pi"] \\
	& {\text{Bl}_2(\mathbb{P}^2)} 
	\arrow[r] 
	& |[alias=P2]| {\mathbb{P}^2}
\end{tikzcd}
\]
Let us denote by $\widetilde{\ell}$ the strict transform of $\ell$ on $\text{Bl}_2(\mathbb P^2)$, which is a $(-1)$-curve. Then  $\widetilde{\ell}$ is intersecting transversally in one point $T_1$ and $T_2$ and it is disjoint from the branch locus of $\pi'$.  
, \begin{figure}[H]
	\centering
	\begin{tikzpicture}[
		>=Stealth, 
		dot/.style={circle, fill=black, inner sep=1.5pt},
		label font/.style={font=\small, color=black}, scale=0.6, transform shape
		]
		
		\coordinate (O) at (-2.5, 0);
		
		\draw[thick, black, name path=ellipseD1] (O) ellipse (2cm and 0.8cm);
		\node[right, label font] at ($(O)+(1.2, 0.8)$) {$D_1$};
		
		\draw[thick, black, name path=ellipseDp] (O) ellipse (0.8cm and 1.6cm);
		\node[below=3.1cm, label font] at ($(O)+(0.6, 1.6)$) {$D_{p-1}$};
		
		\path[name intersections={of=ellipseD1 and ellipseDp}];
		
		\coordinate (Q1) at (intersection-1); 
		\coordinate (Q2) at (intersection-2); 
		\coordinate (Q3) at (intersection-3); 
		\coordinate (Q4) at (intersection-4); 
		
		\draw[thick, black, shorten <= -1cm, shorten >= -1cm] (Q3) -- (Q1) 
		node[pos=1, anchor=south west, label font, xshift=0.7cm, yshift=0.7cm] {$\ell$};
		
		\node[dot] at (Q1) {};
		\node[dot] at (Q2) {};
		\node[dot] at (Q3) {};
		\node[dot] at (Q4) {};
		
		\node[above left, xshift=-0.1cm, yshift=0cm, font=\tiny, text=black] at (Q1) {$\frac{1}{p}(1,1)$};
		\node[above left, font=\tiny, text=black] at (Q2) {$\frac{1}{p}(1,1)$};
		\node[below left, xshift=-0.2cm, yshift=0.1cm, font=\tiny, text=black] at (Q3) {$\frac{1}{p}(1,1)$};
		\node[below right, font=\tiny, text=black] at (Q4) {$\frac{1}{p}(1,1)$};
		
		\node[below right=2cm of O, font=\large, text=black] {$\mathbb{P}^2$};
		
		\draw[<-] ($(O)+(0, 2.6)$) -- ($(O)+(0, 4.1)$) 
		node[midway, right, font=\large] {$\pi$};
		
		\coordinate (P1) at (-2.5, 6.8);
		\coordinate (P2) at (-2.5, 5.2);
		
		\draw[thick, black] ($(P1)+(-1.2, 0)$) ellipse (0.25cm and 0.5cm);
		\draw[thick, black] ($(P1)+(-1.2, 0.5)$) -- (P1);
		\draw[thick, black] ($(P1)+(-1.2, -0.5)$) -- (P1);
		
		\draw[thick, black] ($(P1)+(1.2, 0)$) ellipse (0.25cm and 0.5cm);
		\draw[thick, black] ($(P1)+(1.2, 0.5)$) -- (P1);
		\draw[thick, black] ($(P1)+(1.2, -0.5)$) -- (P1);
		
		\node[dot] at (P1) {};
		\node[above=0.15cm, label font, font=\large] at (P1) {$p_1$};
		
		\draw[thick, black] ($(P2)+(-1.2, 0)$) ellipse (0.25cm and 0.5cm);
		\draw[thick, black] ($(P2)+(-1.2, 0.5)$) -- (P2);
		\draw[thick, black] ($(P2)+(-1.2, -0.5)$) -- (P2);
		
		\draw[thick, black] ($(P2)+(1.2, 0)$) ellipse (0.25cm and 0.5cm);
		\draw[thick, black] ($(P2)+(1.2, 0.5)$) -- (P2);
		\draw[thick, black] ($(P2)+(1.2, -0.5)$) -- (P2);
		
		\node[dot] at (P2) {};
		\node[below=0.15cm, label font, font=\large] at (P2) {$p_2$};
		
		\draw[dashed, black, thin] ($(P1)+(-1.45, 0)$) to[out=-110, in=110] ($(P2)+(-1.45, 0)$);
		\draw[dashed, black, thin] ($(P1)+(-0.95, 0)$) to[out=-70, in=70] ($(P2)+(-0.95, 0)$);
		
		\draw[dashed, black, thin] ($(P1)+(0.95, 0)$) to[out=-110, in=110] ($(P2)+(0.95, 0)$);
		\draw[dashed, black, thin] ($(P1)+(1.45, 0)$) to[out=-70, in=70] ($(P2)+(1.45, 0)$);
		
		\node[right=1.4cm of P1, font=\large, yshift=-2.3cm] {$Y$};
		
		\coordinate (BL) at (-9, 0);
		
		\draw[thick, black, name path=curveMain] (BL) ellipse (2.2cm and 1.5cm);
		\node[above left,  xshift=-1.1cm, yshift=-3.9cm, label font] at ($(BL)+(0, 2.25)$) {$D_{1}$};
		
		\draw[above=1cm, thick, black, name path=curveVertical] ($(BL)+(-0.8, -0.3)$) to[out=105, in=180] ($(BL)+(0, 2.25)$) to[out=0, in=95] ($(BL)+(1, -0.3)$);
		\node[above left,  xshift=-0.1cm, yshift=-0.2cm, label font] at ($(BL)+(0, 2.25)$) {$D_{p-1}$};
		
		\path[name intersections={of=curveMain and curveVertical}];
		\coordinate (B1) at (intersection-1); 
		\coordinate (B2) at (intersection-2); 
		
		\node[dot] at (B1) {};
		\node[dot] at (B2) {};
		
		\node[above right, font=\tiny, text=black] at (B1) {$\frac{1}{p}(1,1)$};
		\node[above left, font=\tiny, text=black] at (B2) {$\frac{1}{p}(1,1)$};
		
		\draw[thick, black] ($(BL)+(-1.5, 0.3)$) parabola bend ($(BL)+(0, -1.1)$) ($(BL)+(1.7, 0.3)$);
		\node[below left=-0.5, label font] at ($(BL)+(-1.5, -0.5)$) {$\tilde{\ell}$};
		
		\draw[thick, black, name path=lineT1] ($(BL)+(-2.5, -0.1)$) -- ($(BL)+(-0.5, -0.1)$) 
		node[left=2cm, label font] {$T_1$} node[above right, xshift=-0.4cm, yshift=-0.1cm, font=\tiny] {$(-1)$};
		
		\draw[thick, black, name path=lineT2] ($(BL)+(0.5, -0.1)$) -- ($(BL)+(2.5, -0.1)$) 
		node[below right, xshift=-0.4cm, yshift=-0.01cm, label font] {$T_2$} node[above left, xshift=-1.6cm, yshift=-0.1cm, font=\tiny] {$(-1)$};

		\node[below, xshift=-6.8cm, yshift=-1.3cm, font=\large, text=black] {${\rm Bl}_2(\mathbb{P}^2)$};

		\coordinate (YPRIME) at (-9, 6);
		
		\draw[thick, black] ($(YPRIME)+(-2, 0.8)$) -- ($(YPRIME)+(2, 0.8)$) node[right, label font] {$E_1$};
		\draw[thick, black] ($(YPRIME)+(-2, -0.8)$) -- ($(YPRIME)+(2, -0.8)$) node[right, label font] {$E_2$};
		
		\draw[thick, black] ($(YPRIME)+(-1.2, 1.4)$) -- ($(YPRIME)+(-1.2, -1.4)$) node[below, font=\tiny] {$(-1)$};
		\node[dot] at ($(YPRIME)+(-1.2, 0.8)$) {};
		\node[dot] at ($(YPRIME)+(-1.2, -0.8)$) {};
		
		\draw[thick, black] ($(YPRIME)+(-0.4, 1.4)$) -- ($(YPRIME)+(-0.4, -1.4)$) node[below, font=\tiny] {$(-1)$};
		\node[dot] at ($(YPRIME)+(-0.4, 0.8)$) {};
		\node[dot] at ($(YPRIME)+(-0.4, -0.8)$) {};
		
		\node[font=\large] at ($(YPRIME)+(0.2, 0)$) {$\dots$};
		
		\draw[thick, black] ($(YPRIME)+(0.9, 1.4)$) -- ($(YPRIME)+(0.9, -1.4)$) node[below, font=\tiny] {$(-1)$};
		\node[dot] at ($(YPRIME)+(0.9, 0.8)$) {};
		\node[dot] at ($(YPRIME)+(0.9, -0.8)$) {};
		
		\node[above, font=\scriptsize] at ($(YPRIME)+(-0.15, 1.4)$) {$p\text{-copies of }\tilde{\ell}$};
		\node[left, font=\tiny] at ($(YPRIME)+(-2, 0.8)$) {$(-p)$};
		\node[left, font=\tiny] at ($(YPRIME)+(-2, -0.8)$) {$(-p)$};
		
		\node[below right,  xshift=-7.2cm, yshift=4.8cm, font=\large, text=black] {$Y'$};

		\draw[->] ($(BL)+(2.6, 0.3)$) -- ($(O)+(-2.4, 0.3)$);
		
		\draw[->] ($(YPRIME)+(0, -1.9)$) -- ($(BL)+(0, 2.6)$) node[midway, right, font=\small] {$\pi'$};
		
		\draw[->] ($(YPRIME)+(2.6, 0.3)$) -- ($(O)+(-2.4, 6.3)$);
		
		\coordinate (TOPLEFT) at (-14, 3);
		
		\node[dot] (sp1) at ($(TOPLEFT)+(-0.8, 0)$) {};
		\node[dot] (sp2) at ($(TOPLEFT)+(0, 0)$) {};
		\node[dot] (sp3) at ($(TOPLEFT)+(0.8, 0)$) {};
		
		\node[font=\small] at ($(TOPLEFT)+(-0.4, 0)$) {$\dots$};
		\node[font=\small] at ($(TOPLEFT)+(0.4, 0)$) {$\dots$};
		
		\node[below=0.15cm, font=\small] at (sp1) {};
		\node[below=0.15cm, font=\small] at (sp3) {};
		
		\draw[thick, black] ($(TOPLEFT)+(-1.4, 0.8)$) to[out=-45, in=135] (sp1) 
		to[out=-45, in=225] (sp2) 
		to[out=45, in=135] (sp3) 
		to[out=-45, in=180] ($(TOPLEFT)+(1.4, 0.8)$)
		node[right, label font] {$E_1$} node[above left, xshift=-1.9cm, yshift=-2cm, font=\tiny] {$(-1)$};
		
		\draw[thick, black] ($(TOPLEFT)+(-1.4, -0.8)$) to[out=45, in=225] (sp1) 
		to[out=45, in=135] (sp2) 
		to[out=-45, in=225] (sp3) 
		to[out=45, in=180] ($(TOPLEFT)+(1.4, -0.8)$)
		node[right, label font] {$E_2$} node[above left, xshift=-1.9cm, yshift=1.4cm, font=\tiny] {$(-1)$};
		
		
		\node[below right=1.3cm of TOPLEFT, font=\large, text=black] {$S$};
		
		\coordinate (YTILDE) at (-14, 6);
		
		\node[font=\large, text=black] at (YTILDE) {$\widetilde{Y}$};

		\draw[<-] ($(YPRIME)+(-1.7, 0)$) -- ($(YTILDE)+(0.3, 0)$);
		
		\draw[->] ($(YTILDE)+(0, -0.5)$) -- ($(TOPLEFT)+(0, 1.1)$);
	\end{tikzpicture}
	\caption{The figure shows the case $p\geq 7(2)$. As one can see, $(\pi')^*(\widetilde{\ell})$ is the disjoint union of $p$ copies of $\widetilde{\ell}$. There is the same configuration of curves on $\widetilde{Y}$. After contracting $p-1$ of those $(-1)$-curves on $\widetilde{Y}$, we obtain a surface $S$ admitting two $(-1)$ curves $E_1$ and $E_2$ intersecting to each other.}
	\label{fig: p>=7_2_is_rational}
\end{figure}
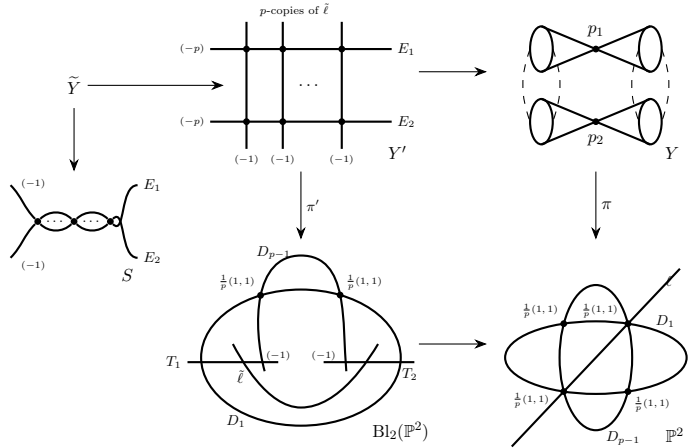
Thus, $(\pi')^*(\widetilde{\ell})$ consists of $p$ copies of $\widetilde{\ell}$ on $Y'$, that are hence $(-1)$-curves  intersecting transversally each $E_i$ in one point, see Figure \ref{fig: p>=7_2_is_rational}. However, these $(-1)$-curves together with $E_1$ and $E_2$ are not passing through other singular points of $Y'$, so we would obtain the same configuration of exceptional curves on the minimal resolution $\widetilde{Y}$. This proves the claim and hence the thesis. 
\end{proof}

\section{The Hodge structures of SNC \texorpdfstring{$(p,k)$}{(p,k)}-planes satisfying Property \texorpdfstring{$\bigstar$}{*}}\label{subsec: K3Hodge}

In this section, we shall analyze the Hodge structure of smooth projective surfaces $S$ endowed with
an action of an elementary abelian $p$-group $G \cong (\mathbb{Z}/p)^k$, where $p$ is a prime number. Since $G$ is finite and abelian, the induced complex
representation on $H^2(S,\mathbb{C})$ decomposes into character eigenspaces:
\begin{equation}\label{eq: dec_into_isoty_comppn_over_C}
	H^2(S,\mathbb{C})
	=
	\bigoplus_{\chi\in G^*} H^2(S,\mathbb{C})_\chi,
\end{equation}
where
\[
H^2(S,\mathbb{C})_\chi
:=
\left\{
\alpha\in H^2(S,\mathbb{C})
\ \middle|\
g^*\alpha=\chi(g)\alpha
\text{ for every } g\in G
\right\}
\]
is the isotypic component of the character $\chi\in G^*$. 

Since the action of $G$ on $S$ is holomorphic, every $g\in G$ preserves the Hodge
decomposition. Hence each character eigenspace decomposes according to Hodge type:
\[
	H^2(S,\mathbb{C})_\chi
	=
	H^{2,0}(S)_\chi
	\oplus
	H^{1,1}(S)_\chi
	\oplus
	H^{0,2}(S)_\chi.
\]
Moreover, complex conjugation exchanges the eigenspaces corresponding to conjugate
characters:
\[
\overline{H^{r,s}(S)_\chi}
=
H^{s,r}(S)_{\overline{\chi}}.
\]
For every nontrivial character $\chi\in G^*$, the subgroup $H:=\ker(\chi)$ has index $p$ in $G$. Conversely, every subgroup $H\leq G$ of index $p$ arises as the
kernel of a nontrivial character. Moreover, two nontrivial characters have the same kernel
if and only if they generate the same cyclic subgroup of $G^*$. Equivalently, if
$\ker(\chi)=H$, then the nontrivial characters with kernel $H$ are precisely
$\chi,\chi^2,\ldots,\chi^{p-1}$.

For every subgroup $H\leq G$ of index $p$, choose a nontrivial character
$\chi\in G^*$ such that $\ker(\chi)=H$, and define
\begin{equation}\label{eq: V_H}
	V_H
	:=
	\bigoplus_{j=1}^{p-1} H^2(S,\mathbb{C})_{\chi^j}.
\end{equation}
Clearly, this definition is independent of the choice of $\chi$. 
We also set
\[
V_0
:=
H^2(S,\mathbb{C})^G
=
H^2(S,\mathbb{C})_1.
\]
\begin{prop} \label{prop:complex_Hodge_decomposition}
	Let $S$ be a smooth projective surface endowed with a $G$-action, $G\cong (\mathbb Z/p)^k$.
	With the notation introduced above, one has a decomposition
	\begin{equation}\label{eq: dec_H2_S_C_into_Comp_index_p}
		H^2(S,\mathbb{C})
		=
		V_0
		\oplus
		\bigoplus_{\substack{H\leq G\\ [G:H]=p}} V_H.
	\end{equation}
	Each summand is compatible with the Hodge decomposition of $H^2(S,\mathbb{C})$. Moreover, for every subgroup $H\leq G$ of index $p$, one has
	\[
	H^2(S,\mathbb{C})^H
	=
	V_0\oplus V_H.
	\]
\end{prop}
\begin{proof}
The decomposition \eqref{eq: dec_H2_S_C_into_Comp_index_p} follows directly replacing  \eqref{eq: V_H} into \eqref{eq: dec_into_isoty_comppn_over_C} for each subgroup $H\leq G$ of index $p$. 
Finally, we observe that 
\[
H^2(S,\mathbb{C})^H
=
H^2(S,\mathbb{C})_1
\oplus
\bigoplus_{j=1}^{p-1}H^2(S,\mathbb{C})_{\chi^j}
=
V_0\oplus V_H,
\]
and the thesis follows. 
\end{proof}

We now turn to the rational structure underlying the decomposition \ref{eq: dec_into_isoty_comppn_over_C}.
Although the individual isotypic components 
$H^2(S,\mathbb{C})_\chi$ need not be defined over $\mathbb{Q}$,
the blocks $V_H$ introduced above admit a natural rational structure.

For every subgroup $H\leq G$, consider the operator
\begin{equation}\label{eq:averaging_op}
P_H
:=
\frac{1}{|H|}
\sum_{h\in H} h^*
\end{equation}
defined on $H^2(S,\mathbb{Q})$. 
We also define $Q_H:=P_H-P_G$.

Since every $g^*$, $g\in G$, is a morphism of rational Hodge structures, the operators
$P_H$, $P_G$, and $Q_H$ are morphisms of rational Hodge structures. In particular, we have
\[ 
\textrm{Im}(P_H)=H^2(S,\QQ)^H.
\]
For this equality, notice that $P_H$ is indeed the projector onto $H^2(S,\mathbb{Q})^H$. We set $W_0:=\operatorname{Im}(P_G)=H^2(S,\mathbb{Q})^G$, and  
%
for every subgroup $H\leq G$ of index $p$, 
\[
W_H
:=
\operatorname{Im}(Q_H).
\]

The rational analogue of Proposition~\ref{prop:complex_Hodge_decomposition}
is the following.

\begin{proposition}
\label{prop:rational_Hodge_decomposition}

Let $S$ be a smooth projective surface endowed with a  $G$-action, $G\cong (\mathbb Z/p)^k$. Then one has a decomposition of rational Hodge structures
\[
H^2(S,\mathbb{Q})
=
W_0
\oplus
\bigoplus_{\substack{H\leq G\\ [G:H]=p}} W_H.
\]
Moreover, $W_0\otimes_{\mathbb{Q}}\mathbb{C}=V_0$
and, for every subgroup $H\leq G$ of index $p$, $W_H\otimes_{\mathbb{Q}}\mathbb{C}=V_H$.

\end{proposition}

\begin{proof}
Let $H\leq G$ be a subgroup of index $p$. 
On a isotypic component $H^2(S,\mathbb{C})_\eta$ of character $\eta\in G^*$, the operator
$P_H$ (seen as an operator of $H^2(S, \CC)$) acts as the multiplication by the complex number 
\[
\frac{1}{|H|}
\sum_{h\in H}\eta(h).
\]
Therefore,
\[
P_H\big|_{H^2(S,\mathbb{C})_\eta}
=
\begin{cases}
\operatorname{id} & \text{if } \eta|_H=1,\\
0 & \text{otherwise}.
\end{cases}
\]
It follows that $\operatorname{Im}(P_H)\otimes_{\mathbb{Q}}\mathbb{C}
=
V_0\oplus V_H, \textrm{ and }
\operatorname{Im}(P_G)\otimes_{\mathbb{Q}}\mathbb{C}
=
V_0$. Hence
$
W_H\otimes_{\mathbb{Q}}\mathbb{C}
=
\operatorname{Im}(Q_H)\otimes_{\mathbb{Q}}\mathbb{C}
=
V_H.
$

Let now $H$ and $K$ be two distinct subgroups of $G$ of index $p$.
Since they are distinct maximal proper subgroups, one has $HK=G$. As $G$ is abelian, the corresponding operators satisfy
$P_H\circ P_K=P_G$.
Moreover,  $P_H\circ P_G=P_G$ and $P_G^2=P_G$.

Consequently, $Q_H\circ Q_K
=
(P_H-P_G)\circ (P_K-P_G)
=
0$
for $H\neq K$, whereas
$
Q_H^2
=
(P_H-P_G)^2
=
Q_H.
$
Thus the $Q_H$ are pairwise orthogonal idempotents.
 Furthermore, one has 
 \[P_G+\sum_{\underset{[G:H]=p}{H\leq G}}Q_H=id_{H^{2}(S,\mathbb Q)}.\]
Therefore,
\[
H^2(S,\mathbb{Q})
=
W_0
\oplus
\bigoplus_{\substack{H\leq G\\ [G:H]=p}}W_H.
\]
After tensoring with $\mathbb{C}$, their images, together with the image of
$P_G$, give precisely the decomposition \eqref{eq: dec_H2_S_C_into_Comp_index_p}.
 
Finally, since all the projectors involved are morphisms of rational Hodge
structures, each $W_H$ and $W_0$ is a rational Hodge substructure of
$H^2(S,\mathbb{Q})$.
\end{proof}
\

\begin{remark}
\label{rem:p2_rational_decomposition}

When $p=2$, the above discussion simplifies considerably. Indeed, every character of
$G$ takes values in $\{\pm 1\}\subset\mathbb{Q}$, and every subgroup $H\leq G$ of index $2$ is the kernel of a unique nontrivial
character $\chi$. Hence the individual isotypic component $H^2(S,\mathbb{Q})_\chi$
is already defined over $\mathbb{Q}$, and no grouping of characters is necessary.
In this case, $W_H\otimes_{\mathbb{Q}}\mathbb{C}
=
H^2(S,\mathbb{C})_\chi$.
\end{remark}

We now give a geometric interpretation of the rational Hodge substructures introduced above. Let $H\leq G$ be any subgroup and consider the quotient map
\[
\pi_H:S\longrightarrow Y_H:=S/H.
\]
The surface $Y_H$ may be singular. Nevertheless, since $\pi_H$ is a finite morphism
by a finite group, its rational cohomology is naturally identified with the $H$-invariant
part of the cohomology of $S$.

\begin{prop} \label{prop:quotient_cohomology_invariants}

Let $S$ be a smooth projective surface endowed with a $G$-action, $G\cong (\mathbb Z/p)^k$, 
and let $H\leq G$ be a subgroup. Then the pullback map
\[
\pi_H^*:H^2(Y_H,\mathbb{Q})\longrightarrow H^2(S,\mathbb{Q})
\]
induces an isomorphism
\[
H^2(Y_H,\mathbb{Q})
\overset{\sim}{\longrightarrow}
H^2(S,\mathbb{Q})^H.
\]
This is an isomorphism of rational Hodge structures.

In particular, we have
\begin{equation}\label{eq : decomposition_intermidiat_quot}
	H^2(Y_H,\mathbb{Q})
	\cong
	W_0\oplus W_H.
\end{equation}

\end{prop}
\begin{proof}
Since the quotient map
$
\pi_H:S\longrightarrow Y_H
$
is finite, its fibers are finite and hence
$
R^i\pi_{H*}\mathbb{Q}_S=0 $ for all  $i>0
$.
Therefore,
$
H^2(S,\mathbb{Q})
\cong
H^2(Y_H,\pi_{H*}\mathbb{Q}_S).
$

The group $H$ acts naturally on the sheaf $\pi_{H*}\mathbb{Q}_S$. Since every fiber of
$\pi_H$ is an $H$-orbit, its invariant subsheaf is the constant sheaf:
$
(\pi_{H*}\mathbb{Q}_S)^H
=
\mathbb{Q}_{Y_H}.
$
Moreover, since we work with rational coefficients, taking $H$-invariants is an exact
functor. It follows that
\[
H^2(Y_H,\mathbb{Q})
\cong
H^2(S,\mathbb{Q})^H.
\]

The isomorphism is induced by the algebraic morphism $\pi_H$, hence it is compatible with
the natural mixed Hodge structures. Since $H^2(S,\mathbb{Q})^H$ is a pure Hodge structure
of weight $2$, the same is true for $H^2(Y_H,\mathbb{Q})$, and the above isomorphism is
therefore an isomorphism of rational Hodge structures.

Finally, $P_G+Q_H=P_H$ and $P_G\circ Q_H=0$ gives
\[
H^2(S,\mathbb{Q})^H
=
W_0\oplus W_H,
\]
and the last assertion follows.
\end{proof}
Applying Proposition~\ref{prop:quotient_cohomology_invariants} to the action of $G$ itself,
we obtain
\[
H^2(Y_G,\mathbb{Q})
\cong
H^2(S,\mathbb{Q})^G
=
W_0.
\]
Hence, under the identification
\[
H^2(Y_H,\mathbb{Q})
\cong
W_0\oplus W_H,
\]
the summand $W_0$ is precisely the part coming from the total quotient $Y_G$, whereas
$W_H$ is the complementary rational Hodge substructure.

Thus $W_H$ may be regarded as the new cohomological contribution associated with the
$G/H\cong \mathbb Z/p$ -intermediate quotient $Y_H\longrightarrow Y_G$ of the $G$-quotient $\pi_G\colon S\to Y_G$. 
\vspace{0.3cm}

We now study how the second cohomology changes when passing from a quotient surface
to a resolution of its singularities. In our setting, all quotient surfaces have only
quotient singularities, and therefore their exceptional divisors consist of configurations
of smooth rational curves.

\begin{lemma}
\label{lem:cohomology_resolution_quotient}

Let $Y$ be a normal projective surface with quotient singularities, and let
$\rho:\widetilde{Y}\longrightarrow Y$
be a resolution of singularities. Denote by
\[
E:=\bigcup_{x\in\operatorname{Sing}(Y)}E_x
\]
the exceptional divisor, and let
\[
\mathcal{E}
:=
\left\langle [E_{x,j}]
\ \middle|\
x\in\operatorname{Sing}(Y),\ 
E_{x,j}\text{ an irreducible component of }E_x
\right\rangle_{\mathbb{Q}}
\subset H^2(\widetilde{Y},\mathbb{Q}).
\]
Then
\[
H^2(\widetilde{Y},\mathbb{Q})
=
\rho^*H^2(Y,\mathbb{Q})
\oplus
\mathcal{E}.
\]
Moreover,
\[
\mathcal{E}\subseteq \operatorname{NS}(\widetilde{Y})\otimes\mathbb{Q}.
\]
In particular, the resolution introduces only algebraic classes of type $(1,1)$.

\end{lemma}

\begin{proof}
This is standard for resolutions of quotient surface singularities; see, for
instance, \cite{mum}. The exceptional fiber over each
singular point is a tree of smooth rational curves, and its intersection matrix
is negative definite.

The long exact sequence of the pair $(\widetilde Y,E)$, together with excision,
gives
\[
0
\longrightarrow
H^2(Y,\QQ)
\xrightarrow{\rho^*}
H^2(\widetilde Y,\QQ)
\longrightarrow
H^2(E,\QQ)
\longrightarrow
0.
\]
Since the intersection matrix of the exceptional curves is nondegenerate, their
classes map isomorphically onto $H^2(E,\QQ)$. Hence
$H^2(\widetilde Y,\QQ)
=
\rho^*H^2(Y,\QQ)
\oplus
\mathcal E$.

Finally, the classes generating $\mathcal E$ are classes of algebraic curves, so $\mathcal E
\subset
\operatorname{NS}(\widetilde Y)\otimes\QQ$.
\end{proof}

We can now apply the previous lemma to the intermediate quotient
$Y_H=S/H$ and compare the transcendental Hodge structures of its resolution
with the $H$-invariant part of the transcendental cohomology of $S$. 

\begin{lemma}
\label{lem:transcendental_quotient}

Let $S$ be a smooth projective surface endowed with a $G$-action, $G\cong( \mathbb Z/p)^k$.
Let $H\leq G$, set $Y_H:=S/H$,
and let $\rho_H:\widetilde{Y}_H\longrightarrow Y_H$
be a resolution of singularities. Then there is a natural isomorphism of the transcendental part of the rational
Hodge structures
\[
T(\widetilde{Y}_H)_{\mathbb{Q}}
\cong
T(S)_{\mathbb{Q}}^H.
\]

\end{lemma}
\begin{proof}
By Proposition~\ref{prop:quotient_cohomology_invariants} and
Lemma~\ref{lem:cohomology_resolution_quotient}, we have a decomposition
of rational Hodge structures
\begin{equation}\label{eq: H2_min_res_int_quo}
	H^2(\widetilde{Y}_H,\mathbb{Q})
	\cong
	H^2(S,\mathbb{Q})^H
	\oplus
	\mathcal{E}_H,
\end{equation}
where $\mathcal{E}_H$ is the subspace generated by the classes of the irreducible
components of the exceptional divisor of $\rho_H$. In particular, $\mathcal{E}_H
\subseteq
\operatorname{NS}(\widetilde{Y}_H)_\mathbb{Q}$.

On the other hand, the action of $H$ preserves both
$\operatorname{NS}(S)_{\mathbb{Q}}$ and the intersection pairing, hence it also
preserves the transcendental Hodge structure $T(S)_{\mathbb{Q}}
=
\operatorname{NS}(S)_{\mathbb{Q}}^\perp$.

Thus the orthogonal decomposition $
H^2(S,\mathbb{Q})
=
\operatorname{NS}(S)_{\mathbb{Q}}
\oplus
T(S)_{\mathbb{Q}}$
is $H$-equivariant, so 
\[
H^2(S,\mathbb{Q})^H
=
\operatorname{NS}(S)_{\mathbb{Q}}^H
\oplus
T(S)_{\mathbb{Q}}^H.
\]

Furthermore, note that  
\[
\operatorname{NS}(\widetilde{Y}_H)_{\mathbb{Q}}
\cong
\operatorname{NS}(S)_{\mathbb{Q}}^H
\oplus
\mathcal{E}_H.
\]
Indeed, the rational $(1,1)$-classes in the summand
$H^2(S,\mathbb{Q})^H$ are precisely $\textrm{NS}(S)_{\mathbb{Q}}^H$
by the Lefschetz theorem on $(1,1)$-classes. Moreover, all classes belonging to
$\mathcal{E}_H$ are algebraic. Conversely, if a rational algebraic class on
$\widetilde{Y}_H$ is decomposed according to \eqref{eq: H2_min_res_int_quo}, 
then its component in $H^2(S,\mathbb{Q})^H$ is again of Hodge type $(1,1)$,
and therefore belongs to $\operatorname{NS}(S)_{\mathbb{Q}}^H$. Consequently,
\[
\frac{H^2(\widetilde{Y}_H,\mathbb{Q})}
     {\operatorname{NS}(\widetilde{Y}_H)_{\mathbb{Q}}}
\cong
\frac{H^2(S,\mathbb{Q})^H}
     {\operatorname{NS}(S)_{\mathbb{Q}}^H}.
\]
Using the above orthogonal decompositions, the left-hand side is isomorphic to $T(\widetilde{Y}_H)_{\mathbb{Q}}$, while the right-hand side is isomorphic to $T(S)_{\mathbb{Q}}^H$.
\end{proof}

\begin{remark}
\label{rem:polarization_quotient}

The isomorphism of Lemma~\ref{lem:transcendental_quotient} is compatible with the natural intersection pairings up to a scalar factor. More precisely, let
\[
\Phi_H:
T(\widetilde{Y}_H)_{\mathbb{Q}}
\longrightarrow
T(S)_{\mathbb{Q}}^H
\]
denote the isomorphism induced by the quotient map $\pi_H$. Then, for every
$\alpha,\beta\in T(\widetilde{Y}_H)_{\mathbb{Q}}$, one has
\[
\bigl\langle
\Phi_H(\alpha),\Phi_H(\beta)
\bigr\rangle_S
=
|H|\,
\langle\alpha,\beta\rangle_{\widetilde{Y}_H}.
\]
Thus $\Phi_H$ is a similitude of polarized rational Hodge structures with
multiplier $|H|$. Equivalently,
\[
\left(
T(\widetilde{Y}_H)_{\mathbb{Q}},
|H|\langle\ ,\ \rangle_{\widetilde{Y}_H}
\right)
\cong
\left(
T(S)_{\mathbb{Q}}^H,
\langle\ ,\ \rangle_S
\right)
\]
as polarized rational Hodge structures.
\end{remark}

\begin{corollary} \label{cor:transcendental_decomposition_smooth}

Let $S$ be a smooth projective surface endowed with a $G$-action, $G\cong (\mathbb Z/p)^k$ and assume that the quotient $S/G$ is a rational surface.

For every subgroup $H\leq G$ of index $p$, let $Y_H:=S/H$
and let $\widetilde{Y}_H$ be a resolution of singularities of $Y_H$. Then there is
an isomorphism of rational Hodge structures
\begin{equation}\label{eq: dec_trasc_part_smooth_case}
	T(S)_{\mathbb{Q}}
	\cong
	\bigoplus_{\substack{H\leq G\\ [G:H]=p}}
	T(\widetilde{Y}_H)_{\mathbb{Q}},
\end{equation}
where the decomposition is orthogonal with respect to the intersection
pairing, and induces a similitude
as polarized rational Hodge structures.

\end{corollary}

\begin{proof}
Applying Lemma~\ref{lem:transcendental_quotient} to the subgroup $G$ itself give $T(S)_{\mathbb{Q}}^G
\cong
T(\widetilde{Y_G})_{\mathbb{Q}}$.
Since $Y_G$ is rational, its resolution is rational as well, and hence
$
T(S)_{\mathbb{Q}}^G=0.
$

The rational decomposition of
Proposition~\ref{prop:rational_Hodge_decomposition} is preserved by the
transcendental part, and therefore
\[
T(S)_{\mathbb{Q}}
=
T(S)_{\mathbb{Q}}^G
\oplus
\bigoplus_{\substack{H\leq G\\ [G:H]=p}}
\left(T(S)_{\mathbb{Q}}\cap W_H\right)=\bigoplus_{\substack{H\leq G\\ [G:H]=p}}
\left(T(S)_{\mathbb{Q}}\cap W_H\right).
\]
For a subgroup $H$ of index $p$, we have
\[
T(S)_{\mathbb{Q}}^H
=
T(S)_{\mathbb{Q}}^G
\oplus
\left(T(S)_{\mathbb{Q}}\cap W_H\right)
=
T(S)_{\mathbb{Q}}\cap W_H.
\]
Combining this together with Lemma~\ref{lem:transcendental_quotient}, we obtain 
\[
T(\widetilde{Y}_H)_{\mathbb{Q}}
\cong
T(S)_{\mathbb{Q}}\cap W_H.
\]

Finally, the projectors $P_H$, $P_G$, and $Q_H$ are self-adjoint with respect to
the intersection pairing. Hence the different summands
$T(S)_{\mathbb{Q}}\cap W_H$ are mutually orthogonal. Since $|H|=p^{k-1}$
for every subgroup $H\leq G$ of index $p$, the last assertion follows from
Remark~\ref{rem:polarization_quotient}.
\end{proof}

We can now extend the decomposition \eqref{eq: dec_trasc_part_smooth_case} to SNC $(p,k)$-planes.

\begin{theorem}
\label{prop:transcendental_decomposition_SNC}

Let $\pi\colon X\to Y$ be a SNC $(p,k)$-cover of a rational surface $Y$. Let $\rho:\widetilde{X}\longrightarrow X$
be the minimal resolution of singularities, and let $\widetilde{X}_{min}$ be its minimal model.

For every subgroup $H\leq G$ of index $p$, set $Y_H:=X/H$,
and let $\widetilde{Y}_H$ be the minimal resolution of singularities of $Y_H$.
Then there is an isomorphism of rational Hodge structures
\[
T(\widetilde{X}_{min})_{\mathbb{Q}}
\cong
\bigoplus_{\substack{H\leq G\\ [G:H]=p}}
T(\widetilde{Y}_H)_{\mathbb{Q}}.
\]

Moreover, this decomposition is orthogonal with respect to the intersection pairing,
and induces a similitude
as polarized rational Hodge structures.
\end{theorem}
\begin{proof}
The action of $G$ on $X$ lifts to its minimal resolution
$\widetilde{X}$. Moreover, the $G$-equivariant morphism $\rho$
induces a birational morphism $\widetilde{X}/G\to X/G\cong Y$. Hence $\widetilde{X}/G$ is rational.

We can therefore apply
Corollary~\ref{cor:transcendental_decomposition_smooth} to $\widetilde{X}$ and obtain
\[
T(\widetilde{X})_{\mathbb{Q}}
\cong
\bigoplus_{\substack{H\leq G\\ [G:H]=p}}
T(\widetilde{\widetilde{X}/H})_{\mathbb{Q}},
\]
where $\widetilde{\widetilde{X}/H}$ denotes any smooth resolution of
$\widetilde{X}/H$. For every subgroup $H\leq G$, the morphism
$\widetilde{X}\to X$ induces a birational morphism $\widetilde{X}/H\to Y_H$.
Thus $\widetilde{\widetilde{X}/H}$ and $\widetilde{Y}_H$ are birational smooth
projective surfaces. Since the transcendental Hodge structure is a birational
invariant of smooth projective surfaces, $T(\widetilde{\widetilde{X}/H})_{\mathbb{Q}}
\cong
T(\widetilde{Y}_H)_{\mathbb{Q}}$. Similarly, $\widetilde{X}$ and $\widetilde{X}_{min}$ are birational smooth projective
surfaces, hence $T(\widetilde{X})_{\mathbb{Q}}
\cong
T(\widetilde{X}_{min})_{\mathbb{Q}}$.

Finally, birational transformations between smooth projective surfaces add or
remove only algebraic classes generated by exceptional curves, and therefore
preserve the intersection form on the transcendental part. The statement
concerning the polarization follows from
Corollary~\ref{cor:transcendental_decomposition_smooth}.
\end{proof}

We now specialize Theorem~\ref{prop:transcendental_decomposition_SNC}
to SNC $(p,k)$-planes satisfying Property $\bigstar$.

\begin{corollary}
\label{cor:K3_Hodge_decomposition}

Let $X$ be a SNC $(p,k)$-plane satisfying Property $\bigstar$.
For every subgroup $H\leq G$ of index $p$, let $Y_H:=X/H$,
and denote by $\widetilde{Y_H}$ the minimal resolution of singularities of $Y_H$ and with $Y_H^{\circ}$ its minimal model. 
Then
\[
T(X_{\min})_{\mathbb{Q}}
\cong
\bigoplus_{\substack{H\leq G,\ [G:H]=p\\ p_g(Y_H)=1}}
T(Y_H^\circ)_{\mathbb{Q}}.
\]
In particular, every nonzero summand on the right-hand side is the
transcendental Hodge structure of a K3 surface. Moreover, it induces a similitude 
as polarized rational Hodge structures.
\end{corollary}

\begin{proof}
By Property $\bigstar$, every intermediate $(p,1)$-plane $Y_H$ has geometric
genus either $0$ or $1$.

By Proposition~\ref{prop: rationality_p,1_planes_pg=0}, if $p_g(Y_H)=0$, then the
minimal resolution of $Y_H$ is rational. Hence $T(\widetilde{Y}_H)_{\mathbb{Q}}=0$. On the other hand, by Theorem~\ref{thm: pg1_Zp_cover_structure}, if $p_g(Y_H)=1$, then its
minimal model $Y_H^\circ$ is a K3 surface. Since $\widetilde{Y}_H$ and
$Y_H^\circ$ are birational smooth projective surfaces, $T(\widetilde{Y}_H)_{\mathbb{Q}}
\cong
T(Y_H^\circ)_{\mathbb{Q}}$. The result now follows immediately from
Theorem~\ref{prop:transcendental_decomposition_SNC}.
\end{proof}
%
%
%
\begin{corollary}
\label{cor:formule_rk_T_X_and_Pic_numb_p_k_plane}

Let $X$ be a SNC $(p,k)$-plane, and let $\rho:\widetilde{X}\longrightarrow X$
be its minimal resolution of singularities. For every subgroup $H\leq G$ of
index $p$, let $\widetilde{Y}_H$ be the minimal resolution of $Y_H:=X/H$.
Then
\[
\operatorname{rk}T(\widetilde{X})
=
\sum_{\substack{H\leq G\\ [G:H]=p}}
\operatorname{rk}T(\widetilde{Y}_H),
\]
and
\begin{equation}\label{eq: rk_(p,k)_plane_and_Pic_numb}
\rho(\widetilde{X})
=
12\chi(\mathcal{O}_{\widetilde{X}})
-
K_{\widetilde{X}}^2
-
2
-
\sum_{\substack{H\leq G\\ [G:H]=p}}
\operatorname{rk}T(\widetilde{Y}_H).
\end{equation}

Furthermore,
\[
\rho(X)
=
12\chi(\mathcal{O}_{\widetilde{X}})
-
K_{\widetilde{X}}^2
-
2
-
\sum_{\substack{H\leq G\\ [G:H]=p}}
\operatorname{rk}T(\widetilde{Y}_H)
-
\sum_{x\in\operatorname{Sing}(X)}\ell_x,
\]
where $\ell_x$ is the number of irreducible components of the exceptional
divisor over the singular point $x$.
\end{corollary}

\begin{proof}
The first equality follows directly from Theorem~\ref{prop:transcendental_decomposition_SNC}. While Equation \ref{eq: rk_(p,k)_plane_and_Pic_numb} is standard.
Finally, the exceptional curves of the resolution of the quotient singularities of
$X$ are linearly independent in
$\operatorname{NS}(\widetilde{X})_{\mathbb{Q}}$. Hence resolving a singular point
$x$ increases the Picard number by precisely $\ell_x$. Thus
\[
\rho(X)
=
\rho(\widetilde{X})
-
\sum_{x\in\operatorname{Sing}(X)}\ell_x,
\]
which gives the last formula.
\end{proof}
\subsection{The Mumford-Tate conjecture for SNC \texorpdfstring{$(p,k)$}{(p,k)}-planes satisfying Property \texorpdfstring{$\bigstar$}{*}}

We conclude this section explaining briefly how to address the Mumford-Tate conjecture for SNC
$(p,k)$-planes constructed in this paper.

A detailed and comprehensive treatment of the Mumford-Tate groups arising from
Hodge structures of K3 type associated with abelian covers of rational surfaces
can be found in \cite{GP22}.
We start with some facts on categories of motives and fix the notation that will be used here. For a more detailed introduction to the subject, we refer to \cite{Moonen,Moonenh1}.

\medskip 

Let $K \subseteq \CC$ be a number field. We consider the category of André motives over~$K$, that we denote with ~$\Mot_K$, introduced in \cite{A96} to which we refer for all the definition and fact about this category. 
Therefore a \emph{motive} is a triple $(X,t,m)$, where $X$ is a smooth projective variety over~$K$, $t$ is an idempotent motivated cycle on $X \times X$,
and $m\in\ZZ$. Given a smooth projective variety $X$ over~$K$, we denote by $\HH(X)$ the object $(X,\Delta_X,0)$, where $\Delta_X$ is the class of the diagonal in $X \times X$.

This category of motives has several good properties, for example is Tannakian, semisimple, graded and polarized, see \cite[Thm 0.2]{A96}. Moreover the  K\"unneth decomposition always exists 
\[
\HH(X)=\bigoplus_i \HH^i(X),
\]
where $\HH^i(\_)$ is the \emph{$i$-th K\"unneth component} and is a contravariant functor from the category of smooth projective varieties over~$K$ to~$\Mot_K$, see \cite[Proposition~2.2]{A96}.

We denote with $\AbMot_K$ the Tannakian subcategory of \emph{abelian motives} that are generated by the motives of the form $\HH(X)$, where $X$ is either an abelian variety or $X=\Spec(L)$ for some finite extension $L/K$ with $L\subset\CC$.

To conclude with the notation, for a given Weil cohomology $\mathcal{H}$ (here we restrict only to  the Betti and the \'etale cohomology),  we introduce the \emph{realization functor} $H_{\mathcal{H}}(\_)$ from $\Mot_K$ to $\Vect_{\QQ}$ for the Betti cohomology and to $\Vect_{\QQ_{\ell}}$ for the \'etale cohomology, such that $H_{\mathcal{H}} \circ \HH = \mathcal{H}$. The Betti realization functor on \(\mathrm{Mot}_K\) is faithful, see 
 \cite[~\S6]{A96}.

\medskip 

We now want to consider the case of SNC $(p,k)$-planes. 


Throughout the remainder of this subsection, we work over a number
field $K\subset\mathbb C$ over which the covering, the action of $G\cong (\mathbb Z/p)^k$,
and all the intermediate models and maps used below are defined.
We also assume that $\operatorname{NS}(S_{\mathbb C})_{\mathbb Q}$
is generated by classes of divisors defined over $K$.
These conditions can be achieved after a single finite extension
of the original field of definition, see Subsection \ref{subsec: equations_p,k_plane_in_suitable_proj_planes}.

Let $\pi: X\longrightarrow \mathbb{P}^2$
be a SNC $(p,k)$-plane, and let $\rho: \widetilde{X}\longrightarrow X$
be its minimal resolution. We denote $S:=\widetilde{X}$ for short. We shall work motivically with $S$, rather than with the singular model $X$. The action of $G$ on $X$ lifts naturally to $S$.
So it makes sense to ask if the decomposition in \eqref{eq: dec_trasc_part_smooth_case} is induced by motivated
correspondences.

	Following the notation introduced in Proposition \ref{prop:rational_Hodge_decomposition}, for every subgroup $H\leq G$ of index $p$, we define
	\[
	P_H:=\frac{1}{|H|}\sum_{h\in H}\Gamma_h,
	\qquad
	P_G:=\frac{1}{|G|}\sum_{g\in G}\Gamma_g,
	\qquad
	Q_H:=P_H-P_G,
	\]
	where $\Gamma_g\subset S\times S$ denotes the graph of the automorphism induced by $g$.
	
\begin{lemma}
\label{lem:motivated_projectors}
Let $S=\widetilde{X}$ be the minimal resolution of singularities of a SNC $(p,k)$-plane. Then $P_G$ and the $Q_H$'s define mutually orthogonal idempotent
endomorphisms of $H^2_{\mathrm{mot}}(S)$, and
\[
P_G+
\sum_{\substack{H\leq G\\ [G:H]=p}}
Q_H
=
\operatorname{id}_{H^2_{\mathrm{mot}}(S)}.
\]
Consequently, defined  $M_0:=\operatorname{Im}(P_G)$ and $
M_H:=\operatorname{Im}(Q_H)$, we have 
\[
H^2_{\mathrm{mot}}(S)
\simeq
M_0
\oplus
\bigoplus_{\substack{H\leq G\\ [G:H]=p}}
M_H.
\]
Moreover, under Betti realization, we have $H_B(M_0)=W_0$ and $H_B(M_H)=W_H$, where $W_0$ and $W_H$ are the $\mathbb Q$-vector spaces that appear in Proposition \ref{prop:rational_Hodge_decomposition}.
\end{lemma}
\begin{proof}
	The proof goes as the proof of Proposition \ref{prop:rational_Hodge_decomposition} once one observes that $P_H$, $P_G$ and  $Q_H$ are motived. 
\end{proof}
The summands \(M_H\) may contain algebraic classes, including those arising from the resolution of singularities. We therefore separate their algebraic and transcendental parts and identify the latter with the transcendental motives of the intermediate K3 surfaces, whenever the SNC $(p,k)$-plane satisfies the further Property $\bigstar$. 

\begin{proposition}
\label{prop:motivated_K3_summand}

Let $\pi:X\to\mathbb P^2$ be a SNC $(p,k)$-plane satisfying
Property $\bigstar$, and let $S=\widetilde X$ be its minimal
resolution of singularities. For every subgroup $H\leq G$ of index $p$, let $M_H$ be the
motivated summand of $H^2_{\mathrm{mot}}(S)$ defined by $Q_H$.

There exist motivated projectors $\pi_{\mathrm{alg}}$ and
$\pi_{\mathrm{tr}}$ on $H^2_{\mathrm{mot}}(S)$ whose Betti
realizations project onto $\operatorname{NS}(S)_{\mathbb Q}$
and $T(S)_{\mathbb Q}$, respectively.

Setting
\[
M_H^{\mathrm{alg}}
:=
\operatorname{Im}(Q_H\circ \pi_{\mathrm{alg}}),
\qquad
M_H^{\mathrm{tr}}
:=
\operatorname{Im}(Q_H\circ \pi_{\mathrm{tr}}),
\]
one obtains
\[
M_H
\simeq
M_H^{\mathrm{alg}}\oplus M_H^{\mathrm{tr}}.
\]

If $p_g(Y_H)=1$, then there exists a motivated direct
summand
\[
T_{\mathrm{mot}}(Y_H^\circ)
\subset H^2_{\mathrm{mot}}(Y_H^\circ)
\]
whose Betti realization is $T(Y_H^\circ)_{\mathbb Q}$,
and an isomorphism
\[
M_H^{\mathrm{tr}}
\simeq
T_{\mathrm{mot}}(Y_H^\circ)
\]
in $\operatorname{Mot}_K$. Moreover, $M_H^{\mathrm{alg}}$ is a direct sum of Lefschetz motives.
\end{proposition}

\begin{proof}
Choose divisor classes $D_1,\ldots,D_{\rho}$ defined over $K$ and forming a basis of
$\operatorname{NS}(S)_{\mathbb Q}$.

Let
\[
A=(D_i\cdot D_j)_{i,j},
\qquad
A^{-1}=(a^{ij})_{i,j}.
\]
Since the intersection form on
$\operatorname{NS}(S)_{\mathbb Q}$ is nondegenerate,
the algebraic correspondence
\[
\pi_{\mathrm{alg}}
:=
\sum_{i,j=1}^{\rho}
a^{ij}D_i\times D_j
\]
induces the orthogonal projector $H^2(S,\mathbb Q)
\longrightarrow
\operatorname{NS}(S)_{\mathbb Q}.
$

Thus, $\pi_{\mathrm{alg}}$ is a motivated projector
on $H^2_{\mathrm{mot}}(S)$. We set
$
\pi_{\mathrm{tr}}
:=
\operatorname{id}_{H^2_{\mathrm{mot}}(S)}
-\pi_{\mathrm{alg}}.
$
Its Betti realization is the orthogonal projector onto
$T(S)_{\mathbb Q}$.

Since the action of $G$ preserves both
$\operatorname{NS}(S)_{\mathbb Q}$ and the intersection pairing, these projectors commute with every $Q_H$, $H\leq G$.
Consequently,
\[
M_H
\simeq
M_H^{\mathrm{alg}}\oplus M_H^{\mathrm{tr}},
\]
with
\[
H_B(M_H^{\mathrm{tr}})
=
W_H\cap T(S)_{\mathbb Q}.
\] Suppose now that $p_g(Y_H)=1$, so from Theorem \ref{thm: pg1_Zp_cover_structure} the minimal model $Y_H^\circ$ of $Y_H$ is a K3 surface.
Looking at the proof of Corollary \ref{cor:transcendental_decomposition_smooth}, we have 
\[
W_H\cap T(S)_{\mathbb Q}
\cong
T(Y_H^\circ)_{\mathbb Q}.
\]
Consider the rational map
$
S\dashrightarrow Y_H^\circ
$
induced by the quotient map $\pi_H\colon S\to Y_H$ and the corresponding
birational transformations. We may resolve its indeterminacies over $K$:
\[
\begin{tikzcd}
	Z_H \arrow[r, "f_H"] \arrow[d, "\mu_H"'] & Y_H^\circ \\
	S \arrow[ur, dashed]            &
\end{tikzcd}
\]
where $\mu_H$ is birational and $f_H$ is generically
finite of degree $d_H=|H|$.

Let $C_H$ be the algebraic correspondence induced by
this diagram, acting on Betti cohomology as
\[
(C_H)_B=\mu_{H*}\circ f_H^*.
\]
Since $C_H$ is algebraic, it is motivated. Define also 
\[
\alpha_H
:=
(\pi_{\mathrm{tr}}\circ Q_H)\circ C_H:
H^2_{\mathrm{mot}}(Y_H^\circ)
\longrightarrow M_H^{\mathrm{tr}}
\]
and
\[
\beta_H := \frac1{d_H}\,
{}^tC_H\circ (Q_H\circ \pi_{\mathrm{tr}}):
M_H^{\mathrm{tr}}
\longrightarrow
H^2_{\mathrm{mot}}(Y_H^\circ).
\]
By the transcendental decomposition established above,
the Betti realization of $\alpha_H$ restricts to an
isomorphism
\[
(\alpha_H)_B:
T(Y_H^\circ)_{\mathbb Q}
\xrightarrow{\sim}
W_H\cap T(S)_{\mathbb Q}.
\]

We claim that $\beta_H$ induces its inverse. Indeed, for $u\in T(Y_H^\circ)_{\mathbb Q}$,
the class $f_H^*u$ belongs to $T(Z_H)_{\mathbb Q}$.
Since $\mu_H$ is birational, one has
\[
\mu_H^*\circ \mu_{H*}\circ f_H^*u=f_H^*u.
\]
Therefore, using the projection formula,
\[
\begin{aligned}
({}^tC_H\circ C_H)u
&=
(f_{H*}\circ \mu_H^*\circ \mu_{H*}\circ f_H^*)u\\
&=
(f_{H*}\circ f_H^*)u\\
&=
d_Hu.
\end{aligned}
\]
Since $Q_H\pi_{\mathrm{tr}}$ acts as the identity
on the transcendental image of $C_H$, it follows that
$
(\beta_H\circ\alpha_H)_B
=
\operatorname{id}_{T(Y_H^\circ)_{\mathbb Q}}
$. As $(\alpha_H)_B$ is an isomorphism onto
$H_B(M_H^{\mathrm{tr}})$, we also obtain
$
(\alpha_H\circ\beta_H)_B
=
\operatorname{id}_{H_B(M_H^{\mathrm{tr}})}
$.

By the faithfulness of the Betti realization functor
on motivated motives,
\[
\alpha_H\circ\beta_H
=
\operatorname{id}_{M_H^{\mathrm{tr}}}.
\]
Setting
\[
e_H:=\beta_H\circ\alpha_H
\in
\operatorname{End}_{\operatorname{Mot}_K}
\bigl(H^2_{\mathrm{mot}}(Y_H^\circ)\bigr), 
\]
note that $e_H^2=e_H$. Since $C_H$ maps algebraic classes to algebraic
classes, the Betti realization of $\alpha_H$ vanishes on
$\operatorname{NS}(Y_H^\circ)_{\mathbb Q}$. Together with
$(\beta_H\circ \alpha_H)_B|_{T(Y_H^\circ)_{\mathbb Q}}
=\operatorname{id}_{T(Y_H^\circ)_{\mathbb Q}}$, this shows that
$(e_H)_B$ is the orthogonal projector onto
$T(Y_H^\circ)_{\mathbb Q}$.

Defining
\[
T_{\mathrm{mot}}(Y_H^\circ)
:=
\operatorname{Im}(e_H),
\]
we obtain mutually inverse motivated isomorphisms
\[
M_H^{\mathrm{tr}}
\simeq
T_{\mathrm{mot}}(Y_H^\circ).
\]

Finally, $M_H^{\mathrm{alg}}$ is a direct summand of
the motivated algebraic part of $H^2(S)$, which is a direct sum
of Lefschetz motives.
\end{proof}

\begin{theorem}
\label{thm:Mumford_Tate_SNC}

Let $\pi:X\longrightarrow\mathbb P^2$ be a SNC $(p,k)$-plane satisfying Property $\bigstar$,
defined over a number field $K\subset\mathbb C$, and let $S:=\widetilde X$ be its minimal resolution of singularities. Then both the Mumford--Tate conjecture and the Tate
conjecture hold for $S$. Consequently, they also hold
for its minimal model $\widetilde{X}_{\min}$.
\end{theorem}

\begin{proof}
By Lemma~\ref{lem:motivated_projectors}, there is a
decomposition in the category of motivated motives
\[
H^2_{\mathrm{mot}}(S)
\simeq
M_0\oplus
 \bigoplus_{\substack{
 H\leq G \\
 [G:H]=p}}
 M_H.
\]

By Proposition~\ref{prop:motivated_K3_summand}, each
summand admits a decomposition
\[
M_H
\simeq
M_H^{\mathrm{alg}}\oplus M_H^{\mathrm{tr}}.
\]

Since the total quotient $X/G\cong \mathbb P^2$ is rational, the
$G$-invariant transcendental part vanishes.
Hence $M_0$ is entirely algebraic.

If $p_g(Y_H)=0$, then $Y_H$ is rational by
Proposition~\ref{prop: rationality_p,1_planes_pg=0},
Consequently, $M_H^{\mathrm{tr}}=0$,
and the corresponding summand is algebraic.

If $p_g(Y_H)=1$, Proposition~
\ref{prop:motivated_K3_summand} gives
\[
M_H^{\mathrm{tr}}
\simeq
T_{\mathrm{mot}}(Y_H^\circ),
\]
where $Y_H^\circ$ is a K3 surface, by Theorem \ref{thm: pg1_Zp_cover_structure}.

Combining these observations, we obtain
\[
H^2_{\mathrm{mot}}(S)
\simeq
M_{\mathrm{alg}}
\oplus
\bigoplus_{\substack{H\leq G,\ [G:H]=p\\
                     p_g(Y_H)=1}}
T_{\mathrm{mot}}(Y_H^\circ),
\]
where $M_{\mathrm{alg}}$ is an algebraic motive
which is a
direct sum of Lefschetz motives.

By Andr\'e's theorem, see \cite{A96}, the motivated cohomology
of every K3 surface is an abelian motive.
The Mumford–Tate conjecture holds for the transcendental motive of each K3 surface. Since these are abelian motives, Commelin's result \cite[Theorem 10.3]{Com19} (see also \cite{Com16}),  implies that it also holds for their finite direct sum, together with the Lefschetz summands. It follows that the Mumford--Tate conjecture
holds for $H^2(S)$.

In degree two, the Tate conjecture follows
from the Mumford--Tate conjecture and the
Lefschetz theorem on $(1,1)$-classes. The remaining degrees are immediate for a
smooth projective regular surface.
Therefore the Tate conjecture holds for $S$. Finally, passing from $S$ to its minimal
model $\widetilde{X}_{\min}$ only removes algebraic
Lefschetz summands from the degree-two motive.
Thus both conjectures hold for $\widetilde{X}_{\min}$
as well.
\end{proof}

\section{Invariants of SNC $(p,k)$-planes satisfying Property $\bigstar$}\label{sec: invariants_SNC_p,k_planes_Prop_STAR}
We now proceed to study further properties of the SNC $(p,k)$-planes
$\pi\colon X\to \mathbb P^2$ satisfying Property $\bigstar$ achieved in Section \ref{sec: classification_SNC_p,k_planes_Prop_STAR}.
In particular, we determine their Kodaira dimension and discuss the minimality of
the minimal resolution of singularities
$\rho\colon \widetilde{X}\to X$. Moreover, we also list their numerical invariants.

We remind that SNC $(p,1)$-planes satisfying Property $\bigstar$, so those having geometric genus one, do not need to be studied: from Theorem \ref{thm: pg1_Zp_cover_structure} their minimal model is always a K3 surface. 
Furthermore, the minimal resolution of singularities $\widetilde{X}$ is a minimal surface only for $p=2$, while it is not minimal for the remaining cases $3\leq p\leq 19$, and we need to contract $(-K^2_{\widetilde{X}})$-curves to obtain the minimal model. 

Thus, in this section we only consider the case $k\geq 2$. The main theorem is as follows
\begin{theorem}\label{thm: SNC_(p,k)_planes_invariants}
    Let us consider a SNC $(p,k)$-plane $\pi\colon X\to \mathbb P^2$ satisfying \hyperlink{Property $\bigstar$}{Property $\bigstar$} and with $k\geq 2$. The minimal resolution of singularities $\widetilde{X}$ of $X$ is a regular surface. It is a minimal surface with the exception of cases $C-(3,2)$, $A-(5,2)$, and $B-(7,2)$, whose minimal models are obtained by contracting a single $(-1)$-curve on $\widetilde{X}$.
    
    The surface $\widetilde{X}$ is always of general type, except in the following cases:
    \begin{itemize}
        \item $\widetilde{X}_{min}$ is a K3 surface in cases $A-(2,2)$, $A-(2,3)$, $B-(2,3)$, $A-(2,4)$, $A-(2,5)$, $C-(3,2)$, $A-(5,2)$; 
        \item $\widetilde{X}$ is a minimal properly elliptic surface (with Kodaira dimension one) in cases $B-(3,2)$ and $C-(3,3)$.  
    \end{itemize}
    In the following tables, we list for each case the self-intersection of a canonical divisor of the minimal model $\widetilde{X}_{min}$, along with its geometric genus, 
    topological Euler characteristic, Kodaira dimension, and surface type. 
    Finally, we also list in tables \ref{tab: (2,2)} to \ref{tab: (7,2)} the Picard number of $\widetilde{X}_{min}$ together with the transcendental part of the Hodge structure of a very general member of the family. 
    \begin{table}[H]
\begin{tabular}{|c||c|c|c|c|c||c|c|}\hline
 case&$p_g$
 &$q$&$K^2_{\widetilde{X}_{min}}$& $e(\widetilde{X}_{min})$&$\kappa$&$\widetilde{X}$ minimal& Surface Name\\
  \hline
  A-(2,2)&$1$&$0$&$0$&$24$&$0$ &$\checkmark$ &K3  \\
  
 B-(2,2)&$1$&$0$&$1$&$23$&$2$ &$\checkmark$ & Special Kunev  \\

 C-(2,2)&$2$&$0$&$1$&$35$&$2$  &$\checkmark$ &(1,2)-surface, i.e. I-surface \\
 D-(2,2)&$2$&$0$&$4$&$32$&$2$  &$\checkmark$ & Garbagnati Type $2$  \\
E-(2,2)&$3$&$0$&$9$&$39$&$2$ &$\checkmark$ &Garbagnati Type $3$\\

A-(2,3)& $1$& $0$& $0$& $24$&  $0$&$\checkmark$ & K3   \\
B-(2,3)& $1$& $0$& $0$& $24$& $0$ & $\checkmark$& K3  \\
C-(2,3)& $3$& $0$& $2$& $46$&  $2$&$\checkmark$ & Horikawa  \\
D-(2,3)& $2$&$0$ & $2$& $34$& $2$ &$\checkmark$ & Special Horikawa  \\
E-(2,3)& $3$& $0$& $8$& $40$& $2$ &$\checkmark$ & C3 surface \cite{FP24}  \\
F-(2,3)& $1$& $0$& $2$& $22$& $2$ &$\checkmark$ & Todorov surface  \\
G-(2,3)& $3$& $0$& $8$& $40$& $2$ &$\checkmark$ & E3 surface \cite{FP24}  \\
H-(2,3)& $4$&$0$ & $8$& $52$& $2$ &$\checkmark$ &  \\
I-(2,3)& $5$&$0$ & $18$& $54$& $2$ &$\checkmark$ &   \\
J-(2,3)& $2$& $0$&  $8$& $28$& $2$ &$\checkmark$ &   \\
K-(2,3)& $4$&$0$ &  $18$& $42$& $2$ &$\checkmark$ &   \\
L-(2,3)& $4$&$0$ &  $18$& $42$&$2$ &$\checkmark$ &   \\
M-(2,3)& $6$& $0$& $32$& $52$& $2$ &$\checkmark$ &   \\
A-(2,4)& $1$& $0$& $0$& $24$& $0$ &$\checkmark$ & K3  \\
B-(2,4)& $4$& $0$& $4$ & $56$& $2$&$\checkmark$ & Horikawa   \\
C-(2,4)& $3$& $0$&  $4$ & $44$& $2$ &$\checkmark$ & D4 surface \cite{FP24}  \\
D-(2,4)& $2$& $0$&  $4$ & $32$&  $2$&$\checkmark$ & \\
E-(2,4)& $5$& $0$& $16$& $56$&  $2$&$\checkmark$ &   \\
F-(2,4)& $5$& $0$&  $16$& $56$& $2$&$\checkmark$ &   \\
G-(2,4)& $6$& $0$& $16$& $68$& $2$&$\checkmark$ &   \\
H-(2,4)& $9$& $0$& $36$& $84$& $2$&$\checkmark$ &   \\
I-(2,4)& $1$& $0$&  $4$& $20$& $2$&$\checkmark$ &   \\
J-(2,4)& $4$& $0$&  $16$& $44$& $2$ &$\checkmark$ &   \\
K-(2,4)& $8$& $0$&  $36$& $72$& $2$ &$\checkmark$ &   \\
L-(2,4)& $9$& $0$&  $36$& $84$& $2$&$\checkmark$ &   \\
M-(2,4)& $3$& $0$&  $16$& $32$& $2$ &$\checkmark$ & C4 surface \cite{FP24} \\
N-(2,4)& $7$& $0$&  $36$& $60$& $2$&$\checkmark$ &   \\
O-(2,4)& $7$& $0$&  $36$&  $60$& $2$&$\checkmark$ &   \\
P-(2,4)& $6$& $0$&  $36$& $48$& $2$ &$\checkmark$ &  \\
Q-(2,4)& $11$& $0$&  $64$& $80$&$2$ &$\checkmark$ &  \\
R-(2,4)& $10$& $0$& $64$& $68$& $2$&$\checkmark$ &   \\
\hline
\end{tabular}
\caption{Invariants of the minimal model $\widetilde{X}_{min}$ for the resolution of singularities  $\widetilde{X}$ of a SNC $(p,k)$-plane $X$.} \label{tab: inv_(p,k)_planes_1}
\end{table}
\begin{table}[H]
\begin{tabular}{|c||c|c|c|c|c||c|c|}\hline
 case&$p_g$
 &$q$&$K^2_{\widetilde{X}_{min}}$& $e(\widetilde{X}_{min})$&$\kappa$&$\widetilde{X}$ minimal& Surface Name\\
  \hline
  A-(2,5)& $1$& $0$& $0$& $24$& $0$&$\checkmark$ & K3, Hirzebruch type  \\
  B-(2,5)& $5$& $0$& $8$& $64$& $2$&$\checkmark$ &   \\
  C-(2,5)& $3$& $0$& $8$& $40$& $2$&$\checkmark$ & D5 surface \cite{FP24}  \\
  D-(2,5)& $9$& $0$& $32$& $88$& $2$&$\checkmark$ &  \\
  E-(2,5)& $9$& $0$& $32$& $88$& $2$&$\checkmark$ &  \\
  F-(2,5)& $7$& $0$& $32$& $64$& $2$&$\checkmark$ &   \\
  G-(2,5)& $15$& $0$& $72$& $120$& $2$&$\checkmark$ &   \\
  H-(2,5)& $13$& $0$& $72$& $96$& $2$&$\checkmark$ &   \\
  I-(2,5)& $13$& $0$& $72$& $96$& $2$&$\checkmark$ &   \\
  J-(2,5)& $21$& $0$& $128$& $136$& $2$&$\checkmark$ & \\ 
  A-(2,6)& $7$& $0$& $16$& $80$& $2$&$\checkmark$ & Hirzebruch type \\
B-(2,6)& $15$& $0$& $64$& $128$& $2$&$\checkmark$ &   \\
C-(2,6)& $27$& $0$& $144$& $192$& $2$&$\checkmark$ &  \\

A-(3,2)& $1$& $0$& $1$& $23$& $2$&$\checkmark$ & Special Kunev  \\
B-(3,2)& $2$& $0$& $0$& $36$& $1$&$\checkmark$ & Properly elliptic \\
C-(3,2)& $1$& $0$& $0$& $24$& $0$&$\times$ & K3  \\
D-(3,2)& $3$& $0$& $9$& $27$& $2$&$\checkmark$ &  \\
E-(3,2)& $2$& $0$& $8$& $28$& $2$&$\checkmark$ &  \\

A-(3,3)& $3$& $0$& $3$& $45$& $2$&$\checkmark$ &   \\
B-(3,3)& $1$& $0$& $3$& $21$& $2$&$\checkmark$ &   \\
C-(3,3)& $2$& $0$& $0$& $36$& $1$&$\checkmark$ & Properly elliptic  \\
D-(3,3)& $5$& $0$& $27$& $45$& $2$&$\checkmark$ &   \\

A-(3,4)& $5$& $0$& $9$& $63$& $2$&$\checkmark$ & Hirzebruch type  \\
A-(5,2)& $1$& $0$& $0$& $24$& $0$&$\times$ & K3   \\
B-(5,2)& $2$& $0$& $1$& $35$& $2$&$\checkmark$ & (1,2)-surface, i.e. I-surface  \\
A-(5,3)& $4$& $0$& $5$& $55$& $2$&$\checkmark$ & Fermat quintic in $\mathbb P^3$  \\
A-(7,2)& $2$& $0$& $9$& $27$& $2$&$\checkmark$ &  \\
B-(7,2)& $3$& $0$& $2$& $46$& $2$&$\times$ & Horikawa  \\
\hline
\end{tabular}
\caption{Invariants of the minimal model $\widetilde{X}_{min}$ for the resolution of singularities  $\widetilde{X}$ of a SNC $(p,k)$-plane $X$.} \label{tab: inv_(p,k)_planes_2}
\end{table}
\end{theorem}
\begin{proof}
     Any SNC $(p,k)$-plane is a regular surface by Lemma \ref{lem: (p,k)_planes_are_regular}. Using Theorem \ref{thm: invariants_ab_cov_cyclic_sing}  one can compute the invariants of $\widetilde{X}$ from the building data of the covering. Thus, using Corollary \ref{cor:formule_rk_T_X_and_Pic_numb_p_k_plane}, one can compute case-by-case the Picard number of $\widetilde{X}$ together with the transcendental part of the Hodge structure for very general member of the family, as explained in the later Example \ref{ex: h231}. 
      
    Finally, the remaining information such as the minimality of $\widetilde{X}$, the Kodaira dimension, and the surface type are  obtained applying lemmas \ref{lem: they_are_reg_min_surf_K3_or_gen_type_with_at_most_cyclic_quo}, \ref{lem: B-(3,2)_prop_elliptic}, \ref{lem: C-(3,2)_is_K3}, \ref{lem: E-(3,2)_is_min_of_gen_type}, \ref{lem: C-(3,3)_prop_elliptic}, \ref{lem: A-(5,2)_is_K3}, \ref{lem: B-(7,2)_is_not_min_Horikawa}, which are proved below. 
\end{proof}
\begin{lemma}\label{lem: they_are_reg_min_surf_K3_or_gen_type_with_at_most_cyclic_quo}
    Any $(p,k)$-plane, with $k\geq 2$, satisfying \hyperlink{Property $\bigstar$}{Property $\bigstar$} and with at most canonical singularities is a minimal surface. Furthermore, they are of general type, with the exception of the cases $A-(2,2)$, $A-(2,3)$, $B-(2,3)$, $A-(2,4)$, $A-(2,5)$, which are $K3$ surfaces. 
\end{lemma}
\begin{proof}
    According to Enriques-Kodaira classification, any surface $S$ with $p_g(S)\geq 1 $ and $K_S^2\geq 1$ is a surface of general type. 

    Let us consider a SNC $(p,k)$-plane $\pi\colon X\to \mathbb P^2$ satisfying \hyperlink{Property $\bigstar$}{Property $\bigstar$}, with $k\geq 2$. Then from 
    Theorem \ref{thm: invariants_ab_cov_cyclic_sing}, we have 
    \[
      \begin{split}
            p_g(\widetilde{X}) 
           = \vert\{\chi \colon l_\chi=3\}\vert,
      \end{split}
    \]
    which is greater than or equal to one from \hyperlink{Property $\bigstar$}{Property $\bigstar$} and Proposition \ref{prop: equivalence_star_prop}.

    Assume that $X$ has at most canonical singularities. Then from Theorem \ref{thm: invariants_ab_cov_cyclic_sing} one has
    \[
      pK_{\widetilde{X}}\equiv \widetilde{\pi}^*\left(\left(-3p+(p-1)\sum_g d_g\right)H\right). 
    \]
    Tables from \ref{tab: (2,2)} to \ref{tab: (7,2)} show that $-3p+(p-1)\sum_g d_g\geq 0$. Since the pullback of a nef divisor by a generically finite map is itself nef, $pK_{\widetilde{X}}$ (and thus $K_{\widetilde{X}}$) is nef.

    Therefore, if $K_{\widetilde{X}}^2\geq 1$, then $\widetilde{X}$ is a minimal surface of general type.
    
    On the other hand, if $K_{\widetilde{X}}^2=0$, then $(p-1)\sum_gd_g=3p$, which forces $p=2$ and $\sum_gd_g=6$. In this case, we would have $2K_{\widetilde{X}}\equiv 0$. Furthermore, 
    by inspecting tables \ref{tab: (2,2)} 
    through \ref{tab: (2,6)}, the only admissible cases are $A-(2,2)$, $A-(2,3)$, $B-(2,3)$, $A-(2,4)$, $A-(2,5)$. These cases also satisfy  $p_g(\widetilde{X})=1$, $q(\widetilde{X})=0$; thus, by Enriques-Kodaira classification, they are $K3$ surfaces. 
\end{proof}
\begin{remark}
	Lemma \ref{lem: they_are_reg_min_surf_K3_or_gen_type_with_at_most_cyclic_quo} already covers most of the cases of tables \ref{tab: (2,2)} to \ref{tab: (7,2)}. Indeed, it remains to discuss only cases $B-(3,2)$, $C-(3,2)$, $E-(3,2)$, $C-(3,3)$, $A-(5,2)$, $B-(7,2)$. 
\end{remark}
For the subsequent lemmas, we fix the following notation. The map \(\rho \colon \widetilde{X}\rightarrow X\) denotes the minimal resolution of singularities of \(X\), and \(\widetilde{\pi }\) denotes the composition \(\widetilde{\pi }\coloneqq \pi \circ \rho \colon \widetilde{X}\rightarrow \mathbb{P}^{2}\). Let \(\Delta _{g}\) be a smooth irreducible curve in the branch locus of \(\pi \), with local monodromy$g\in G$. According to Proposition \ref{prop: dec_of_the_pullback}, the reduced divisor of \(\widetilde{\pi }^{*}\Delta _{g}\) decomposes into the sum of an irreducible smooth curve \(\widetilde{\Delta _{g}}\) (the strict transform of \(\pi ^{*}\Delta _{g}\) under \(\rho \)) and some Hirzebruch-Jung strings of the resolution. We shall repeatedly invoke Proposition \ref{prop: dec_of_the_pullback} to determine the configuration of \(\widetilde{\Delta _{g}}\) and these Hirzebruch-Jung strings, and to compute the self-intersection and genus of \(\widetilde{\Delta _{g}}\).
\vspace{0.2cm}

The case $B-(3,2)$ in Lemma \ref{lem: B-(3,2)_prop_elliptic} is discussed in more details in the later Example \ref{ex: B-(3,2)}. One can use it as a guide for all the other remaining lemmas. 
\begin{lemma}\label{lem: B-(3,2)_prop_elliptic}
    Any $(3,2)$-plane $\pi\colon X\to \mathbb P^2$ of the case $B-(3,2)$ is a minimal properly elliptic surface of Kodaira dimension one. The canonical map $\varphi_{\vert K_{\widetilde{X}}\vert}\colon \widetilde{X}\to \mathbb P^1$ is a morphism with a generic fibre that is a smooth elliptic curve.  
\end{lemma}
\begin{proof}
    Let us consider the line $l\coloneqq D_{e_2}$. 
    From Proposition \ref{prop: dec_of_the_pullback}, the reduced pullback of $l$ under $\widetilde{\pi}\colon \widetilde{X}\to \mathbb P^2$ consists of the three exceptional $(-3)$-curves $E_1$, $E_2$, and $E_3$ of the resolution of $X$, together with the strict transform $\widetilde{l}$. More precisely, $\widetilde{\pi}^*l=3\widetilde{l}+E_1+E_2+E_3$.  
    Furthermore, $\widetilde{l}$ intersects each exceptional curve  $E_1, E_2$ and $E_3$  transversally at a single point. Proposition \ref{prop: dec_of_the_pullback} gives the self-intersection and the genus of $\widetilde{l}$: 
    \[
      \widetilde{l}^2=1-3\cdot \frac{1}{3}=0, \qquad \makebox{and} \qquad 2g(\widetilde{l})-2=3\left(-2+3\left(1-\frac{1}{3}\right)\right)\implies g(\widetilde{l})=1. 
    \]
    Thus, $\widetilde{l}$ is a smooth elliptic curve of self-intersection zero.
    
    From Theorem \ref{thm: invariants_ab_cov_cyclic_sing}, we have 
    \[
    3K_{\widetilde{X}}\equiv \widetilde{\pi}^*H-(E_1+E_2+E_3)\equiv 3\widetilde{l}.
    \]
     Consequently, $\widetilde{X}$ is a minimal surface. Indeed, a hypothetical $(-1)$-curve $E$   would have $3K_{\widetilde{X}}E=-3<0$, which would force $E$ to be a fixed component of $\vert 3K_{\widetilde{X}}\vert $. However, this is impossible as $E$ would have to be $\widetilde{l}$, which is instead an elliptic curve. 

     Therefore, $\widetilde{X}$ is a minimal surface with $K^2_{\widetilde{X}}=0$, and, by Theorem \ref{thm: invariants_ab_cov_cyclic_sing}, we have $p_g(\widetilde{X})=2, q(\widetilde{X})=0$. By the  Enriques-Kodaira classification, $\widetilde{X}$ is a properly elliptic surface (Kodaira dimensione one). More precisely, we claim that the canonical map of $\widetilde{X}$ is an elliptic fibration. Indeed, $\varphi_{\vert K_{\widetilde{X}}\vert}$ is a morphism since $K^2_{\widetilde{X}}=0$ and it has not fixed part since any divisor of the canonical linear system $\vert K_{\widetilde{X}}\vert$ is a component of the fibration $\varphi_{\vert K_{\widetilde{X}}\vert}$. Finally, by Bertini the general curve $F$ of the linear system is smooth and irreducible and from the genus formula it has $g(F)=1+K^2_{\widetilde{X}}=1$, so it is an elliptic curve. 
\end{proof}
\begin{lemma}\label{lem: C-(3,2)_is_K3}
    Any $(3,2)$-plane $\pi\colon X\to \mathbb P^2$ of the case $C-(3,2)$ has a minimal resolution of singularities $\widetilde{X}$ that is not a minimal surface. The minimal model of $\widetilde{X}$ is a K3 surface and it is obtained after contracting a single $(-1)$-curve of $\widetilde{X}$.
\end{lemma}
\begin{proof}
    Consider the line $l\coloneqq D_{e_1+e_2}$. 
    According to Proposition \ref{prop: dec_of_the_pullback}, we have
    \[
    \widetilde{\pi}^*l\coloneqq 3\widetilde{l}+E_1+\cdots + E_6 \qquad 
    \]
    where $E_1,\dots, E_6$ are the $(-3)$-curves arising from the resolution of $X$ and $\widetilde{l}$ 
    is the strict transform of $l$. The curve $\widetilde{l}$ intersects transversally each $(-3)$-curve $E_i$ at a single point. Applying Proposition \ref{prop: dec_of_the_pullback}, we compute its self-intersection and genus:
    \[
    \widetilde{l}^2=1-3\cdot 2\cdot \frac{1}{3}=-1, \qquad \makebox{and} \qquad 2g(\widetilde{l})-2=3\left(-2+2\cdot \left(1-\frac{1}{3}\right)\right) \implies g(\widetilde{l})=0.
    \]
    Thus, $\widetilde{l}$ is a $(-1)$-curve of $\widetilde{X}$. 
    
    From Theorem \ref{thm: invariants_ab_cov_cyclic_sing}, we have
    \[
    3K_{\widetilde{X}}= \widetilde{\pi}^*H-(E_1+\dots+E_6)\equiv_{lin} 3\widetilde{l}. 
    \]
    After contracting the $(-1)$-curve $\widetilde{l}$, the divisor $3K_{\widetilde{X}'}$ of the resulting surface $\widetilde{X}'$ becomes trivial, implying $\widetilde{X}'$ is a minimal surface. Furthermore, since $K_{X'}^2=0$, $p_g(\widetilde{X}')=1$ and $q(\widetilde{X}')=0$ by Theroem \ref{thm: invariants_ab_cov_cyclic_sing}, it follows that  $\widetilde{X}'$ is a K3 surface.
\end{proof}
\begin{lemma}\label{lem: E-(3,2)_is_min_of_gen_type}
    Any $(3,2)-$plane $\pi\colon X\to \mathbb P^2$ of the case $E-(3,2)$ is a minimal surface of general type with $p_g(\widetilde{X})=2$, $q(\widetilde{X})=0$, and $K^2_{\widetilde{X}}=8$. 
\end{lemma}
\begin{proof}
    Consider the line $l\coloneqq D_{e_2}$. According to Proposition \ref{prop: dec_of_the_pullback}, the pullback is 
    \[
    \widetilde{\pi}^*l\coloneqq 3\widetilde{l}+E_1+E_2+ E_3 
    \]
    where $E_1,E_2, E_3$ are the $(-3)$-curves of the resolution of $X$ and $\widetilde{l}$ 
    is the strict transform of $l$. 
    The curve $\widetilde{l}$ 
    intersects  each $(-3)$-curve $E_i$ transversally at a single point. Applying Proposition \ref{prop: dec_of_the_pullback}, we obtain 
    \[
    \widetilde{l}^2=1-3\cdot 1\cdot \frac{1}{3}=0, \qquad \makebox{and} \qquad 2g(\widetilde{l})-2=3\left(-2+4\left(1-\frac{1}{3}\right)\right) \implies g(\widetilde{l})=2. 
    \]
    Now, let us consider the remain lines $ t_1\coloneqq D_{2e_1+e_2}$ and $ t_2\coloneqq D_{2e_1+2e_2}$ together with their strict transforms $\widetilde{t_1}$ and $\widetilde{t_2}$. We have $\widetilde{\pi}^* t_i=3\widetilde{ t_i}$, with $\widetilde{t_i}^2=\widetilde{t_i}\cdot \widetilde{l}=\widetilde{t_1}\cdot \widetilde{t_2}=1$. From Theorem \ref{thm: invariants_ab_cov_cyclic_sing}, we have
    \[
       3K_{\widetilde{X}}\equiv_{lin} \widetilde{\pi}^*3H-(E_1+E_2+E_3)=3(\widetilde{t_1}+\widetilde{t_2}+\widetilde{l}) \implies K_{\widetilde{X}}^2=8. 
    \]
    Furthermore, $\widetilde{X}$ is a minimal surface, as any hypothetical $(-1)-$curve would have to be a fixed component of $\vert 3K_{\widetilde{X}}\vert $, namely one among $\widetilde{t_1}$, $\widetilde{t_2}$ or $\widetilde{l}$, neither of which is a  $(-1)$-curve. Finally,  using Theoroem \ref{thm: invariants_ab_cov_cyclic_sing}, we compute $q(\widetilde{X})=0$ and $p_g(\widetilde{X})=2$. 
\end{proof}
\begin{lemma}\label{lem: C-(3,3)_prop_elliptic}
    Any $(3,3)-$plane $\pi\colon X\to \mathbb P^2$ of the case $C-(3,3)$ is a minimal properly elliptic surface of Kodaira dimension one. The canonical map $\varphi_{\vert K_{\widetilde{X}}\vert}\colon \widetilde{X}\to \mathbb P^1$ is a morphism with a generic fibre that is a smooth elliptic curve.  
\end{lemma}
\begin{proof}
    Consider the line $l\coloneqq D_{e_2}$. Applying Proposition \ref{prop: dec_of_the_pullback}, we have  
    \[
    \widetilde{\pi}^*l=3\widetilde{l}+E_1+\dots+E_9,
    \]
    where the $E_i$ are the $(-3)$-curves of the resolution of $X$. This gives 
    \[
    \widetilde{l}^2=3-9\cdot \frac{1}{3}=0 \qquad \makebox{and} \qquad 2g(\widetilde{l})-2=9\left(-2+3\cdot \left(1-\frac{1}{3}\right)\right) \implies g(\widetilde{l})=1.
    \]
    Furthermore, from Theorem \ref{thm: invariants_ab_cov_cyclic_sing}, we have
    \[
    3K_{\widetilde{X}}=\widetilde{\pi}^*H-(E_1+\dots+E_9)\equiv 3\widetilde{l}
    \]
    Thus, $\widetilde{X}$ is a minimal surface, since any hypothetical $(-1)$-curve $E$ would be a fixed component of the linear system $\vert 3K_{\widetilde{X}}\vert $, namely $E$ must be  $\widetilde{l}$, which is instead an elliptic curve. Finally, we have a minimal surface with $K_{\widetilde{X}}^2=0$. By Theorem \ref{thm: invariants_ab_cov_cyclic_sing}, then $p_g(\widetilde{X})=2$ and $q(\widetilde{X})=0$, so it follows from Enriques-Kodaira classification that $\widetilde{X}$ is a properly elliptic surface. Finally, the canonical map of $\widetilde{X}$ is an elliptic fibration following the same argument in the proof of Lemma \ref{lem: B-(3,2)_prop_elliptic}. 
\end{proof}
\begin{lemma}\label{lem: A-(5,2)_is_K3}
    The $(5,2)$-plane $\pi\colon X\to \mathbb P^2$ of the case $A-(5,2)$ has a minimal resolution of singularities $\widetilde{X}$ that is not a minimal surface. The minimal model of $\widetilde{X}$  is a K3 surface and it is obtained after contracting a single $(-1)$-curve of $\widetilde{X}$. 
\end{lemma}
\begin{proof}
    Consider the line $l\coloneqq D_{e_1+e_2}$. Applying Proposition \ref{prop: dec_of_the_pullback}, we have 
    \[
    \widetilde{\pi}^*l\coloneqq 5\widetilde{l}+\sum_{i=1}^5\left(2E_1^i+E_2^i\right),
    \]
    where $E_1^i$ and $E_2^1$ are respectively the $(-3)$ and $(-2)$-curves of the resolution of the $i$-th cyclic quotient singularity $\frac{1}{5}(1,2)$ on $X$, for $i=1,\dots, 5$. The strict transform $\widetilde{l}$ of $l$ 
    intersects each  $(-3)$-curve $E_1^i$  transversally at a single point and it is disjoint from the $(-2)$-curves $E_2^i$. Thus, we have from Proposition \ref{prop: dec_of_the_pullback}
    \[
    \widetilde{l}^2=1-5\cdot \frac{2}{5}=-1, \qquad \makebox{and} \qquad 2g(\widetilde{l})-2=5\left(-2+2\cdot \left(1-\frac{1}{5}\right)\right) \implies g(\widetilde{l})=0.
    \]
    Thus, $\widetilde{l}$ is a $(-1)$-curve of $\widetilde{X}$. 
    From Theorem \ref{thm: invariants_ab_cov_cyclic_sing}, we have
    \[
    5K_{\widetilde{X}}= \widetilde{\pi}^*H-\sum_{i=1}^5\left(2E_1^i+E_2^i\right)\equiv_{lin} 5\widetilde{l}. 
    \]
    After contracting the $(-1)$-curve $\widetilde{l}$, the divisor $5K_{\widetilde{X}'}$ of the resulting surface $\widetilde{X}'$ becomes trivial, implying that  $\widetilde{X}'$ is a minimal surface. Furthermore, since $p_g(\widetilde{X}')=1$ and $q(\widetilde{X}')=0$ by Theorem \ref{thm: invariants_ab_cov_cyclic_sing}, it follows that $\widetilde{X}'$ is a K3 surface.
\end{proof}
\begin{lemma}\label{lem: B-(7,2)_is_not_min_Horikawa}
    The $(7,2)$-plane $\pi\colon X\to \mathbb P^2$ of the case $B-(7,2)$ has a minimal resolution of singularities $\widetilde{X}$ that is not a minimal surface.  The minimal model $\widetilde{X}_{min}$ is obtained after contracting a single $(-1)$-curve of $\widetilde{X}$. Moreover, $\widetilde{X}_{min}$ is a Horikawa surface with $K_{\widetilde{X}_{min}}^2=2$, $q(\widetilde{X}_{min})=0$ and $p_g(\widetilde{X}_{min})=3$.
\end{lemma}
\begin{proof}
    Consider the line $l\coloneqq D_{e_1+e_2}$. 
    From Proposition \ref{prop: dec_of_the_pullback}, we have
    \[\widetilde{\pi}^*l\coloneqq 7\widetilde{l}+\sum_{i=1}^7\left(2E_1^i+E_2^i\right),
    \]
    where $E_1^i$ and $E_2^i$ are, respectively, the $(-4)$ and $(-2)$-curves of the resolution of the $i$-th cyclic quotient singularity $\frac{1}{7}(1,2)$ on $X$, for $i=1,\dots, 7$. 
    The strict transform $\widetilde{l}$ intersects each $(-4)$-curve $E_1^i$ transversally at a single point, and is disjoint from the $(-2)$-curves $E_2^i$. Thus, we find from Proposition \ref{prop: dec_of_the_pullback}
    \[
    \widetilde{l}^2=1-7\cdot \frac{2}{7}=-1, \qquad \makebox{and} \qquad 2g(\widetilde{l})-2=7\left(-2+2\cdot \left(1-\frac{1}{7}\right)\right) \implies g(\widetilde{l})=0.
    \]
    Accordingly, $\widetilde{l}$ is a $(-1)$-curve on $\widetilde{X}$. Now, consider the line $t\coloneqq D_{e_1}$ and its strict transform $\widetilde{t}$. We have $\widetilde{\pi}^*t=7\widetilde{t}$ and $\widetilde{t}\cdot \widetilde{l}=1$. 
    From Theorem \ref{thm: invariants_ab_cov_cyclic_sing}, then 
    \begin{equation}\label{eq: case_B-7,2,7K_X}
        7K_{\widetilde{X}}\equiv \widetilde{\pi}^*3H-\sum_{i=1}^7\left(4E_1^i+2E_2^i\right)\equiv 7(\widetilde{t}+2\widetilde{l}). 
    \end{equation}
    This implies $K_{\widetilde{X}}^2=1+4(-1)+2\cdot 2=1$.
    
    Let $b\colon \widetilde{X}\to \widetilde{X}'$ be the contraction of the $(-1)$-curve $\widetilde{l}$. Then $\eqref{eq: case_B-7,2,7K_X}$ yields $7b^*K_{\widetilde{X}'}\equiv 7\widetilde{t}+7\widetilde{l}$, so $7K_{\widetilde{X}'}\equiv_{lin} 7b_*\widetilde{t}$. Since $\widetilde{t}$ intersects $\widetilde{l}$ transversally at one point, we have $b^*b_*\widetilde{t}=\widetilde{t}+\widetilde{l}$, which implies $(b_*\widetilde{t})^2=1+(-1)+2=2$. Therefore, $b_*\widetilde{t}$ is not a $(-1)$-curve, and $\widetilde{X}'$ is a minimal surface 
     with $K_{\widetilde{X}'}^2=K_{\widetilde{X}}^2+1=2$. Finally, from Theorem \ref{thm: invariants_ab_cov_cyclic_sing}, we have  $q(\widetilde{X}')=0$ and $p_g(\widetilde{X}')=3$,
     confirming that $\widetilde{X}'$ is a Horikawa surface.
\end{proof}
\section{Explicit examples of SNC  \texorpdfstring{$(p,k)$}{(p,k)}-planes satisfying Property \texorpdfstring{$\bigstar$}{*}}\label{sec: examples_of_SNC_p,k_planes_Prop_STAR}
We present two concrete examples of $(p,k)$-planes  that illustrate how the output of the algorithm explained in Section \ref{sec: classification_SNC_p,k_planes_Prop_STAR} is translated into geometry. As a byproduct, we will show some interesting phenomenons can occur.

\begin{example}\label{ex: h231}
	{\rm We consider the case $H-(2,3)1)$ of Table~\ref{tab: (2,3)}. So we are dealing with a smooth $(2,3)$-plane $\pi \colon S \longrightarrow \PP^2$ . This corresponds to the second output produced by the algorithm explained at the end of Section \ref{sec: classification_SNC_p,k_planes_Prop_STAR}.
	
	In this case, the building data $\{D_g\}_{g\in G}$ and $\{\mathcal L_{\chi}\}_{\chi \in G^*}$  of the $(2,3)$-plane, where $G\cong (\mathbb Z/2)^3=\langle e_1,e_2,e_3\rangle$, have the following degrees: 
	\[
	\begin{split}
		d_{e_2+e_3}=4, 
		\qquad
		d_{e_1}=d_{e_2}=d_{e_3}=d_{e_1+e_2+e_3}=1, \quad  &\makebox{and} \quad d_g=0 \quad \makebox{otherwise}. \\ 
			l_{\epsilon_1}= l_{\epsilon_2\epsilon_3} =1, 
		\qquad
		l_{\epsilon_1\epsilon_2\epsilon_3}=2,
		\qquad
		l_{\chi}=3 \quad &\text{for all other } \chi\in G^*\setminus 1.
	\end{split}
	\]
	Thus, the branch divisor consists of four lines and a smooth quartic curve, that we colour as follows respectively for the convenience:
	\begin{itemize}
		\item
		$D_{e_1}$ (black),
		{\color{orange}$D_{e_2}$} (orange),
		{\color{green}$D_{e_3}$} (green),
		{\color{red}$D_{e_1+e_2+e_3}$} (red), and
		\item 
		{\color{blue}$D_{e_2+e_3}$} (blue),
	\end{itemize}
	as shown in Figure \ref{fig: branch_locus_H-(2,3)}.

	In the picture, we also indicate the following intermediate $(2,1)$-planes and their branch locus:
	\[
	\begin{aligned}
		Y_1 &\coloneqq  S/\langle e_2, e_3 \rangle, &
		Y_2 &\coloneqq  S/\langle e_1, e_3 \rangle, &
		Y_3 &\coloneqq  S/\langle e_1, e_2 \rangle,\\
		Y_4 &\coloneqq  S/\langle e_1+e_2, e_3 \rangle, &
		Y_5 &\coloneqq  S/\langle e_1+e_3, e_2 \rangle, &
		Y_6 &\coloneqq  S/\langle e_2+e_3, e_1 \rangle,\\
		Y_7 &\coloneqq  S/\langle e_2+e_3, e_1+e_2 \rangle. &
	\end{aligned}
	\]
	We now describe one of these intermediate quotients, namely $Y_1$, that corresponds to the subgroup $\ker(\epsilon_1)=\langle e_2,e_3\rangle\le G$. The others are 
	treated analogously. 
	
	According to Lemma~\ref{lem: interm_building_data}, since $\epsilon_1(e_j)=(-1)^{\delta_{1j}}$, then $\langle \epsilon_1,e_j\rangle=\delta_{1j}$ and so the branch locus of the intermediate $(2,1)$-plane $Y_1\to \mathbb P^2$ consists of the two lines $D_{e_1}+D_{e_1+e_2+e_3}$. 
	Accordingly, in Figure \ref{fig: branch_locus_H-(2,3)} the branch locus of $Y_1$ is given by the black and the red lines.
	
	Since the double cover is branched along exactly two lines, it follows from
	Theorem \ref{thm: main_thm_(p,1)_planes_pg=0} (see Table \ref{tab: (p,1)_pg=0_p=2})) that $Y_1$ is birational to $\mathbb{F}_2$.
	More precisely, the surface $Y_1$ has a single node, and blowing up this point yields $\mathbb{F}_2$. 
\begin{figure}
	\centering
	\begin{tikzpicture}
		
		\begin{scope}[shift={(-0,-1.2)}, scale=0.7, transform shape]
			\draw[thick, black] (-3.5, -2.2) rectangle (3.5, 3.8);
			\draw[thick, black] (-3.3, -0.65) -- (3.3, -0.65);
			\draw[thick, red] (-2.64, -1.98) -- (1.52, 3.65);
			\draw[thick, orange] (2.64, -1.98) -- (-1.52, 3.65);
			\draw[thick, green!60!black] (-3.24, 1.1) -- (2.92, -1.3);
			\draw[blue, thick] (0, 2.7) .. controls (0.6, 2.7) and (0.5, 1.3) .. (0.4, 0.8)
			.. controls (0.3, 0.4) and (0.7, 0.3) .. (1.0, 0.1)
			.. controls (1.5, -0.2) and (2.5, -0.9) .. (2.1, -1.6)
			.. controls (1.8, -2.1) and (1.0, -1.6) .. (0.6, -0.8)
			.. controls (0.2, -0.2) and (0.3, 0.4) .. (0, 0.4)
			.. controls (-0.3, 0.4) and (-0.2, -0.4) .. (-0.8, -1.1)
			.. controls (-1.2, -1.6) and (-1.8, -2.1) .. (-2.1, -1.6)
			.. controls (-2.5, -0.9) and (-1.5, -0.2) .. (-1.0, 0.1)
			.. controls (-0.7, 0.3) and (-0.3, 0.4) .. (-0.4, 0.8)
			.. controls (-0.5, 1.3) and (-0.6, 2.7) .. (0, 2.7) -- cycle;
		\end{scope}
		
		\begin{scope}[shift={(-6.3, 4.5)}, scale=0.3, transform shape]
			\draw[thick, black] (-3.3, -2.2) rectangle (3.1, 3.8);
			\draw[thick, black] (-3.1, -0.65) -- (2.7, -0.65);
			\draw[thick, red] (-2.64, -1.98) -- (1.52, 3.65);
		\end{scope}
		
		\begin{scope}[shift={(-4.3, 4.5)}, scale=0.3, transform shape]
			\draw[thick, black] (-3.3, -2.2) rectangle (3.1, 3.8);
			\draw[thick, red] (-2.64, -1.98) -- (1.52, 3.65);
			\draw[thick, orange] (2.64, -1.98) -- (-1.52, 3.65);
			\draw[blue, thick] (0, 2.7) .. controls (0.6, 2.7) and (0.5, 1.3) .. (0.4, 0.8) .. controls (0.3, 0.4) and (0.7, 0.3) .. (1.0, 0.1) .. controls (1.5, -0.2) and (2.5, -0.9) .. (2.1, -1.6) .. controls (1.8, -2.1) and (1.0, -1.6) .. (0.6, -0.8) .. controls (0.2, -0.2) and (0.3, 0.4) .. (0, 0.4) .. controls (-0.3, 0.4) and (-0.2, -0.4) .. (-0.8, -1.1) .. controls (-1.2, -1.6) and (-1.8, -2.1) .. (-2.1, -1.6) .. controls (-2.5, -0.9) and (-1.5, -0.2) .. (-1.0, 0.1) .. controls (-0.7, 0.3) and (-0.3, 0.4) .. (-0.4, 0.8) .. controls (-0.5, 1.3) and (-0.6, 2.7) .. (0, 2.7) -- cycle;
		\end{scope}
		
		\begin{scope}[shift={(-2.3, 4.5)}, scale=0.3, transform shape]
			\draw[thick, black] (-3.3, -2.2) rectangle (3.1, 3.8);
			\draw[thick, red] (-2.64, -1.98) -- (1.52, 3.65);
			\draw[thick, green!60!black] (-3, 1.4) -- (2.92, -1.3);
			\draw[blue, thick] (0, 2.7) .. controls (0.6, 2.7) and (0.5, 1.3) .. (0.4, 0.8) .. controls (0.3, 0.4) and (0.7, 0.3) .. (1.0, 0.1) .. controls (1.5, -0.2) and (2.5, -0.9) .. (2.1, -1.6) .. controls (1.8, -2.1) and (1.0, -1.6) .. (0.6, -0.8) .. controls (0.2, -0.2) and (0.3, 0.4) .. (0, 0.4) .. controls (-0.3, 0.4) and (-0.2, -0.4) .. (-0.8, -1.1) .. controls (-1.2, -1.6) and (-1.8, -2.1) .. (-2.1, -1.6) .. controls (-2.5, -0.9) and (-1.5, -0.2) .. (-1.0, 0.1) .. controls (-0.7, 0.3) and (-0.3, 0.4) .. (-0.4, 0.8) .. controls (-0.5, 1.3) and (-0.6, 2.7) .. (0, 2.7) -- cycle;
		\end{scope}
		
		\begin{scope}[shift={(-0.3, 4.5)}, scale=0.3, transform shape]
			\draw[thick, black] (-3.3, -2.2) rectangle (3.1, 3.8);
			\draw[thick, black] (-2.8, -0.65) -- (2.7, -0.65);
			\draw[thick, orange] (2.64, -1.98) -- (-1.52, 3.65);
			\draw[blue, thick] (0, 2.7) .. controls (0.6, 2.7) and (0.5, 1.3) .. (0.4, 0.8) .. controls (0.3, 0.4) and (0.7, 0.3) .. (1.0, 0.1) .. controls (1.5, -0.2) and (2.5, -0.9) .. (2.1, -1.6) .. controls (1.8, -2.1) and (1.0, -1.6) .. (0.6, -0.8) .. controls (0.2, -0.2) and (0.3, 0.4) .. (0, 0.4) .. controls (-0.3, 0.4) and (-0.2, -0.4) .. (-0.8, -1.1) .. controls (-1.2, -1.6) and (-1.8, -2.1) .. (-2.1, -1.6) .. controls (-2.5, -0.9) and (-1.5, -0.2) .. (-1.0, 0.1) .. controls (-0.7, 0.3) and (-0.3, 0.4) .. (-0.4, 0.8) .. controls (-0.5, 1.3) and (-0.6, 2.7) .. (0, 2.7) -- cycle;
		\end{scope}
		
		\begin{scope}[shift={(1.7, 4.5)}, scale=0.3, transform shape]
			\draw[thick, black] (-3.3, -2.2) rectangle (3.1, 3.8);
			\draw[thick, black] (-2.8, -0.65) -- (2.7, -0.65);
			\draw[thick, green!60!black] (-3, 1.4) -- (2.92, -1.3);
			\draw[blue, thick] (0, 2.7) .. controls (0.6, 2.7) and (0.5, 1.3) .. (0.4, 0.8) .. controls (0.3, 0.4) and (0.7, 0.3) .. (1.0, 0.1) .. controls (1.5, -0.2) and (2.5, -0.9) .. (2.1, -1.6) .. controls (1.8, -2.1) and (1.0, -1.6) .. (0.6, -0.8) .. controls (0.2, -0.2) and (0.3, 0.4) .. (0, 0.4) .. controls (-0.3, 0.4) and (-0.2, -0.4) .. (-0.8, -1.1) .. controls (-1.2, -1.6) and (-1.8, -2.1) .. (-2.1, -1.6) .. controls (-2.5, -0.9) and (-1.5, -0.2) .. (-1.0, 0.1) .. controls (-0.7, 0.3) and (-0.3, 0.4) .. (-0.4, 0.8) .. controls (-0.5, 1.3) and (-0.6, 2.7) .. (0, 2.7) -- cycle;
		\end{scope}
		
		\begin{scope}[shift={(3.7, 4.5)}, scale=0.3, transform shape]
			\draw[thick, black] (-3.3, -2.2) rectangle (3.1, 3.8);
			\draw[thick, orange] (2.64, -1.98) -- (-1.52, 3.65);
			\draw[thick, green!60!black] (-3, 1.4) -- (2.92, -1.3);
		\end{scope}
		
		\begin{scope}[shift={(5.7, 4.5)}, scale=0.3, transform shape]
			\draw[thick, black] (-3.3, -2.2) rectangle (3.1, 3.8);
			\draw[thick, black] (-3.1, -0.65) -- (2.9, -0.65);
			\draw[thick, red] (-2.64, -1.98) -- (1.52, 3.65);
			\draw[thick, orange] (2.64, -1.98) -- (-1.52, 3.65);
			\draw[thick, green!60!black] (-3, 1.4) -- (2.92, -1.5);
		\end{scope}
			
\draw[-stealth, thick, gray!80] (-6.0, 3.70) -- (-0.8, 1.65); 
\draw[-stealth, thick, gray!80] (-4.0, 3.70) -- (-0.5, 1.65); 
\draw[-stealth, thick, gray!80] (-2.0, 3.70) -- (-0.2, 1.65); 
\draw[-stealth, thick, gray!80] (0.0, 3.70)  -- (0.0, 1.65);  
\draw[-stealth, thick, gray!80] (2.0, 3.70)  -- (0.2, 1.65);  
\draw[-stealth, thick, gray!80] (4.0, 3.70)  -- (0.5, 1.65);  
\draw[-stealth, thick, gray!80] (6.0, 3.70)  -- (0.8, 1.65);  

\node[below right, font=\small] at (-7.35, 3.84) {$Y_1$};
\node[below right, font=\small] at (-5.35, 3.84) {$Y_2$};
\node[below right, font=\small] at (-3.35, 3.84) {$Y_3$};
\node[below right, font=\small] at (-1.35, 3.84) {$Y_4$};
\node[below right, font=\small] at (0.65, 3.84)  {$Y_5$};
\node[below right, font=\small] at (2.65, 3.84)  {$Y_6$};
\node[below right, font=\small] at (4.65, 3.84)  {$Y_7$};
\node[right] at (2.65, 0)  {$\mathbb P^2$};

\node[font=\large, above] at (-6.0, 8.5) {$S/\langle e_1\rangle$};
\node[font=\large, above] at (-4.0, 8.5) {$S/\langle e_2\rangle$};
\node[font=\large, above] at (-2.0, 8.5) {$S/\langle e_3\rangle$};
\node[font=\large, above] at (0.0, 8.5)  {$\frac{S}{\langle e_1+e_2\rangle}$};
\node[font=\large, above] at (2.0, 8.5)  {$\frac{S}{\langle e_1+e_3\rangle}$};
\node[font=\large, above] at (4.0, 8.5)  {$\frac{S}{\langle e_2+e_3\rangle}$};
\node[font=\large, above] at (6.0, 8.5)  {$\frac{S}{\langle e_1+e_2+e_3\rangle}$};

\node[font=\large, above] at (0, 11)  {$S$};

\draw[-stealth, thick, gray!80] (0.0, 11.0) -- (-5.7, 9.1); 
\draw[-stealth, thick, gray!80] (0.0, 11.0) -- (-3.7, 9.1); 
\draw[-stealth, thick, gray!80] (0.0, 11.0) -- (-1.7, 9.1); 
\draw[-stealth, thick, gray!80] (0.0, 11.0) -- (0.0, 9.3);  
\draw[-stealth, thick, gray!80] (0.0, 11.0) -- (1.7, 9.2);  
\draw[-stealth, thick, gray!80] (0.0, 11.0) -- (3.7, 9.2);  
\draw[-stealth, thick, gray!80] (0.0, 11.0) -- (5.7, 9.2);  

\draw[-stealth, thick, gray!80] (-4.0, 8.5) -- (-6.3, 5.8); 
\draw[-stealth, thick, gray!80] (-2.0, 8.5) -- (-6.3, 5.7); 
\draw[-stealth, thick, gray!80] (4.0, 8.5)  -- (-6.1, 5.7); 

\draw[-stealth, thick, gray!80] (-6.0, 8.5) -- (-4.2, 5.7); 
\draw[-stealth, thick, gray!80] (-2.0, 8.5) -- (-4.0, 5.7); 
\draw[-stealth, thick, gray!80] (2.0, 8.5)  -- (-3.8, 5.7); 

\draw[-stealth, thick, gray!80] (-6.0, 8.5) -- (-2.2, 5.7); 
\draw[-stealth, thick, gray!80] (-4.0, 8.5) -- (-2.0, 5.7); 
\draw[-stealth, thick, gray!80] (-0, 8.5)  -- (-2.0, 5.7); 

\draw[-stealth, thick, gray!80] (0, 8.5) -- (0, 5.7); 
\draw[-stealth, thick, gray!80] (-2.0, 8.5) -- (0, 5.7); 
\draw[-stealth, thick, gray!80] (6, 8.5)  -- (0, 5.7); 

\draw[-stealth, thick, gray!80] (2, 8.5) -- (2, 5.7); 
\draw[-stealth, thick, gray!80] (-4.0, 8.5) -- (2, 5.7); 
\draw[-stealth, thick, gray!80] (6, 8.5)  -- (2, 5.7); 

\draw[-stealth, thick, gray!80] (4, 8.5) -- (4, 5.7); 
\draw[-stealth, thick, gray!80] (-6.0, 8.5) -- (4, 5.7); 
\draw[-stealth, thick, gray!80] (6, 8.5)  -- (4, 5.7); 

\draw[-stealth, thick, gray!80] (4, 8.5) -- (6, 5.7); 
\draw[-stealth, thick, gray!80] (.0, 8.5) -- (6, 5.7); 
\draw[-stealth, thick, gray!80] (2, 8.5)  -- (6, 5.7); 
	\end{tikzpicture}
	\caption{The branch locus of the case $H-(2,3)1)$ together with the branch loci of its intermediate $(2,1)$-planes $Y_1,\dots, Y_7$.}
\label{fig: branch_locus_H-(2,3)}
\end{figure}
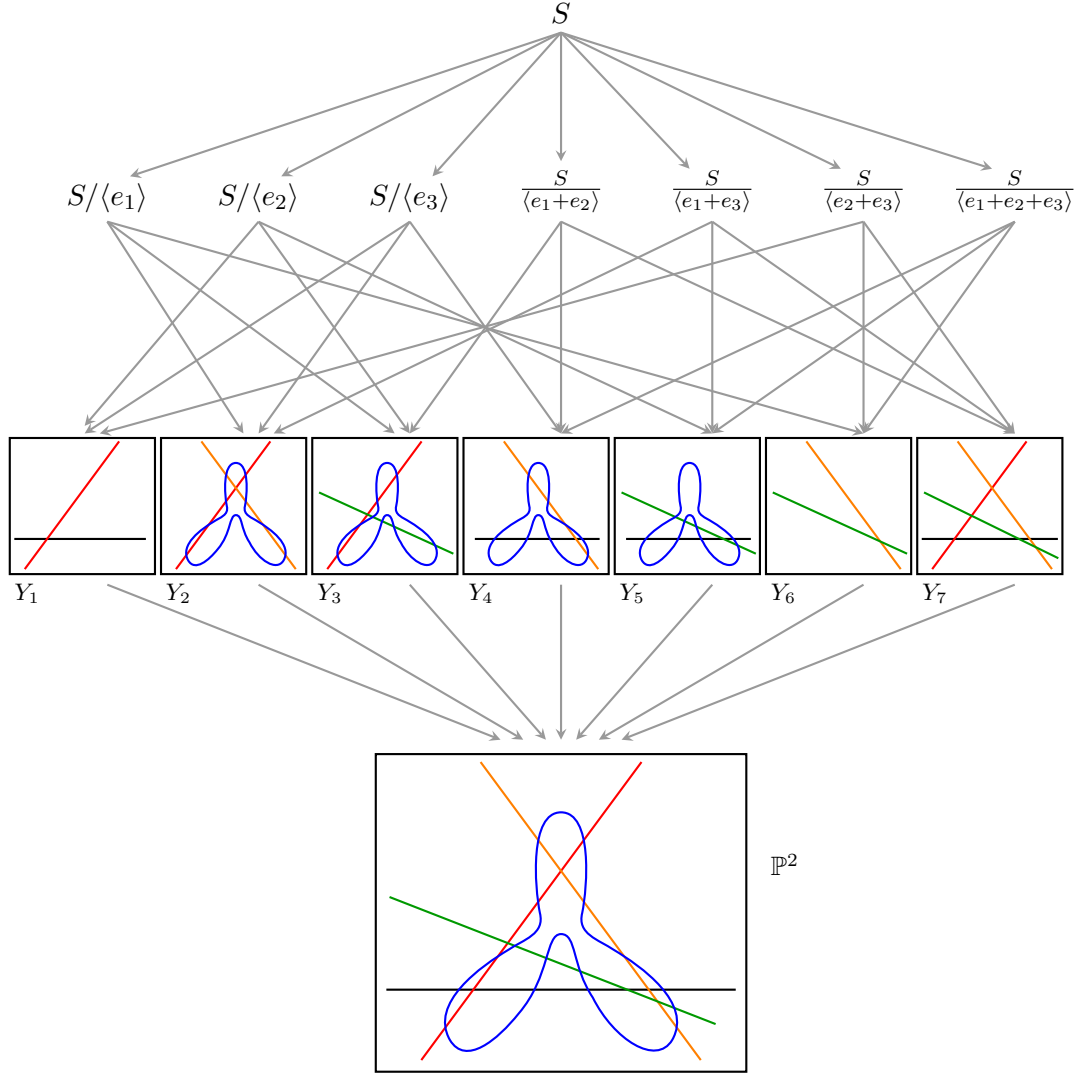
	In this example, the intermediate quotients $Y_2$, $Y_3$, $Y_4$, and $Y_5$ are all $(2,1)$-planes branched along a reduced singular sextic curve.
	Consequently, each of them has geometric genus one, and their minimal models are
	K3 surfaces, see Theorem \ref{thm: pg1_Zp_cover_structure} and Table \ref{tab: (2,1)}. 
	The remaining intermediate quotients $Y_6$ and $Y_7$ can be described similarly.
	As for $Y_6$, the double cover is again branched along two lines and is
	therefore birational to $\mathbb{F}_2$.
	The remaining quotient is branched along four lines and is birational to a Weak del
	Pezzo surface of degree~$2$, see again Theorem \ref{thm: main_thm_(p,1)_planes_pg=0} and Table \ref{tab: (p,1)_pg=0_p=2}.
	
	We obtained that the intermediate $(2,1)$-planes are either rational or K3 surfaces, so $S$ satisfies Property $\bigstar$.
	 
	Hence, counting the number of the intermediate $(2,1)$-planes that are K3 surfaces, the geometric genus of $S$ is $p_g(S)=4$.
	Since in this case the branch divisor of $\pi$ is simple normal crossing and $p=2$, the surface $S$ is smooth, see Proposition \ref{prop: smooth_cover}. Furthermore, $S$ is regular from Lemma \ref{lem: (p,k)_planes_are_regular}.  
	Denoting with $H$ the hyperplane section, the remaining numerical invariants of $S$ can be computed directly using the
	formulae of Proposition \ref{thm: invariants_ab_cov_cyclic_sing} (see also Remark \ref{rem: easier_invariants_formulae_smooth_cover})
	\[
	\chi(\mathcal O_S)= 8 + \frac{1}{2}\big( H(H - 3H)+H(H - 3H)+2H(2H - 3H)\big)=8+\frac{1}{2}(-2-2-2)=5,\]
	\[
	K_S\equiv \frac{1}{2}\pi^*(-6 H+(1+1+1+1+4)H)=\pi^*(H), \qquad \makebox{and} \qquad K_S^2
	= 2^{3-2}\left(-6H + 8H\right)^2=8.
	\]
	Consequently, by Noether formula we have $e(S)=52$. We conclude that $S$ is a minimal surface of general type with ample canonical class. 
	
	We now examine more closely the weight--$2$ Hodge structure of $S$.
	By 
	Corollary~\ref{cor:transcendental_decomposition_smooth},
	 the transcendental lattice $T(S)_{\mathbb Q}$ 
	 decompose as direct sum of the corresponding contributions coming from the intermediate $(2,1)$-planes 
	$Y_1, \dots, Y_7$.
	
	The Hodge numbers of the intermediate double covers can be found in tables \ref{tab: (p,1)_pg=0_p=2} and \ref{tab: (2,1)} depending if $Y_i$ is a rational or a K3 surface respectively.
	More precisely, the rational surfaces $Y_i$ with $i=1,6$ have
	\[
	\bigl(h^{2,0}(Y_i),\,h^{1,1}(Y_i),\,h^{0,2}(Y_i)\bigr)=(0,1,0),
	\]
	while for the rational surface $Y_7$ we obtain
	\[
	\bigl(h^{2,0}(Y_7),\,h^{1,1}(Y_7),\,h^{0,2}(Y_7)\bigr)=(0,2,0).
	\]
	
	Since these surfaces are rational, their Hodge structures are entirely algebraic
	and have no transcendental part.
	Moreover, in each case one generator of $H^{1,1}(Y_i)$ is given by the pullback of
	the hyperplane class on $\mathbb P^2$.
	
	We now turn to the intermediate quotients $Y_2,Y_3,Y_4,Y_5$, whose minimal models are K3
	surfaces. From Theorem \ref{thm: pg1_Zp_cover_structure} we have that the very general member of their families have $\rho(Y_i)=1$, and the generator of the Picard group is exactly the pull back of the hyperplane section of $\PP^2$. Moreover, since all $(2,1)$-planes $Y_i \longrightarrow \PP^2$ are ramified over a quartic and two lines, looking at Table \ref{tab: (2,1)} we have that the Hodge components $H^2_{tra}(Y_i)$ are of dimension $(1,10,1)$, respectively, 
	for $i=2,3,4,5$. We have that the Hodge diamond of $S$ is then
	\[
	\begin{array}{ccccc}
		&& 1 && \\
		& 0 && 0 & \\
		4 && 42 && 4 \\
		& 0 && 0 & \\
		&& 1 &&
	\end{array}
	\]
	and $H_{tra}^2(S)=(4,40,4)$ while the rational Neron-Severi group is generated by the pullback of the hyperplane section and an algebraic class coming from $Y_7$.  
	
	To conclude the example we present $S$ in a suitable weighted projective space according to the method explained in Subsection \ref{subsec: equations_p,k_plane_in_suitable_proj_planes}. 
	\medskip
	Let $(x_0:x_1:x_2)$ be homogeneous coordinates on $\PP^2$ and let
	\[
	f_{e_1},\, f_{e_2},\, f_{e_3},\, f_{e}\in H^0(\PP^2,\mathcal O_{\PP^2}(1)),
	\qquad
	f_{e_2+e_3}\in H^0(\PP^2,\mathcal O_{\PP^2}(4))
	\]
	be defining equations of the branch components 
	$D_{e_1}, D_{e_2}, D_{e_3},D_{e_1+e_2+e_3}$ (lines) and $D_{e_2+e_3}$ (smooth quartic),
	respectively. Setting $\varepsilon\coloneqq \epsilon_1\epsilon_2\epsilon_3$ for short, we recall that in this case
	\[
	\mathcal L_{\epsilon_1}\cong \mathcal O_{\PP^2}(1),\qquad
	\mathcal L_{\epsilon_2+\epsilon_3}\cong \mathcal O_{\PP^2}(1),\qquad
	\mathcal L_{\varepsilon}\cong \mathcal O_{\PP^2}(2),
	\]
	and $\mathcal L_{\chi}\cong \mathcal O_{\PP^2}(3)$ for all other nontrivial $\chi\in G^*$.
	
	We embed $S$ into the weighted projective space $\PP(1,1,1,1,1,2,3,3,3,3)$ with weighted homogeneous coordinates
	\[
	(x_0:x_1:x_2:\ y_{\epsilon_1}:y_{\epsilon_2+\epsilon_3}:\ y_{\varepsilon}:\ 
	y_{\epsilon_2}:y_{\epsilon_3}:y_{\epsilon_1\epsilon_2}:y_{\epsilon_1\epsilon_3}),
	\]
	with $\deg x_i=1$, $\deg y_{\epsilon_1}=\deg y_{\epsilon_2\epsilon_3}=1$,
	$\deg y_{\varepsilon}=2$, and the remaining $y_\chi$ have degree $3$.
	
	Note from Theorem \ref{thm: Par91a} (cf. Subsection \ref{subsec: equations_p,k_plane_in_suitable_proj_planes}) that one of the equations defining $S$ is:
	\[
	y_{\epsilon_1}\,y_{\epsilon_2\epsilon_3} = y_{\varepsilon}.
	\]
	This tells us that we can eliminate the variable  $y_{\varepsilon}$ and 
	consider the weighted projective space with one variable less $
	\PP(1,1,1,1,1,3,3,3,3)$, by projection. Thus, $S$ is defined by the following relations, see Subsection \ref{subsec: equations_p,k_plane_in_suitable_proj_planes}:
		\begin{equation}\label{eq: equations_H-(2,3)_cases}
			\rk
			\begin{pmatrix}
				f_{e_1} \,f_{e_2+e_3} & y_{\epsilon_1\epsilon_2} & y_{\epsilon_1\epsilon_3} &  y_{\epsilon_1}\, f_{e_2+e_3}\\
				y_{\epsilon_1\epsilon_2} & f_{e_2} &  y_{\epsilon_2\epsilon_3} &y_{\epsilon_2} \\
				y_{\epsilon_1\epsilon_3} & y_{\epsilon_2\epsilon_3} & f_{e_3} & y_{\epsilon_3}\\
				y_{\epsilon_1}\, f_{e_2+e_3} & y_{\epsilon_2} &y_{\epsilon_3} & f_\varepsilon\, f_{e_2+e_3}
			\end{pmatrix}
			\leq 1, \qquad y_{\epsilon_1}^2=f_{e_1}f_{\epsilon}.
		\end{equation}
	Note that the projection onto the first three coordinates gives the $(2,3)$-plane
	$\pi\colon S\to\PP^2$.
	
	Finally, from Theorem \ref{thm:Mumford_Tate_SNC} the Mumford--Tate and Tate conjectures hold or $S$. 
}
\end{example}
\begin{remark}
	We point out that the equations \eqref{eq: equations_H-(2,3)_cases} defining the $(2,3)$-plane $S$ of the case $H-(2,3)1)$ directly give a presentation for all the other cases $H-(2,3)2),\dots, H-(2,3)5)$. Indeed, it sufficies to replace the irreducible polynomial $f_{e_2+e_3}$ with a suitable factorization of it to provide equations of the (singular) corresponding SNC $(2,3)$-plane $\pi\colon X\to \mathbb P^2$. 
\end{remark}
\begin{example}\label{ex: B-(3,2)}
	{\rm 
	We consider case \(B-(3,2)1)\) in Table \ref{tab: (3,2)}. Thus, we are dealing with a SNC \((3,2)\)-plane $\pi \colon X \longrightarrow \PP^2$, where $X$ has three cyclic quotient singularities of type \(\frac{1}{3}(1,1)\). This case was roughly studied in Lemma \ref{lem: B-(3,2)_prop_elliptic}, where it was proved that $\widetilde{X}$ is a properly elliptic surface (of Kodaira dimension one). 
	\\
	Hence, this case exhibits an interesting phenomenon: the singularities of a \((p,k)\)-plane may cause the Kodaira dimension of its minimal model to jump from the expected value. 
	
	For this reason, we will study this example in greater detail and show an explicit elliptic fibration of the minimal model, which will coincide with its canonical map. 
	
	The building data $\{D_g\}_{g\in G}$ and $\{\mathcal L_{\chi}\}_{\chi \in G^*}$  of the $(3,2)$-plane have the following degrees, where $G\cong (\mathbb Z/3)^2=\langle e_1,e_2\rangle$:
	\begin{equation}\label{eq: building_data_ex_B-(3,2)1}
			\begin{split}
			d_{e_1}=3, 
			\qquad
			d_{e_2}=d_{2e_2}=1, \quad  &\makebox{and} \quad d_g=0 \quad \makebox{otherwise}. \\ 
			l_{\epsilon_1}= l_{\epsilon_2} =l_{\epsilon_2^2}=1, 
			\qquad
			l_{\epsilon_1^2\epsilon_2^2}=l_{\epsilon_1^2\epsilon_2}=3,
			\qquad
			l_{\chi}=2 \quad &\text{for all other } \chi\in G^*\setminus 1.
		\end{split}
	\end{equation}
	Thus, the branch divisor consists of two lines $\ell_1, \ell_2$ and a smooth cubic curve $\mathcal C$ in general position, that we colour as follows:
	\begin{itemize}
		\item
		{\color{green!60!black} $\mathcal C:=D_{e_1}$} (green),
		\item 
		{\color{blue}$\ell_1:=D_{e_2}$} (blue),
		{\color{red}$\ell_2:=D_{2e_2}$} (red),
	\end{itemize}
	as shown in Figure \ref{fig: branch_locus_B-(3,2)}.
	
		In Figure \ref{fig: branch_locus_B-(3,2)}, we also indicate the three intermediate $(3,1)$-planes and their branch locus:
	\[
	\begin{aligned}
		Y_1 &\coloneqq  X/\langle e_1 \rangle, &
		Y_2 &\coloneqq  X/\langle e_2 \rangle, &
		Y_3 &\coloneqq  X/\langle e_1+e_2 \rangle, \quad \makebox{and} \quad &
		Y_4 &\coloneqq  X/\langle e_1+2e_2 \rangle.
	\end{aligned}
	\]
	According to Lemma~\ref{lem: interm_building_data}, the branch locus of the intermediate $(3,1)$-plane $Y_1\to \mathbb P^2$ consists of the two lines $\ell_1+\ell_2$. 	Since the triple cover is branched along exactly two lines, it follows from
	Theorem \ref{thm: main_thm_(p,1)_planes_pg=0} (see Table \ref{tab: (p,1)_pg=0_p=3})) that $Y_1$ is birational to $\mathbb{F}_3$.
	More precisely, the surface $Y_1$ has a single cyclic quotient singularity of type $\frac{1}{3}(1,1)$, and blowing up this point yields $\mathbb{F}_3$. 
	
	Similarly, one can say that $Y_2\to \mathbb P^2$ is branched over only a smooth cubic and still from Theorem \ref{thm: main_thm_(p,1)_planes_pg=0} (see Table \ref{tab: (p,1)_pg=0_p=3}) we have $Y_2$ is a smooth del Pezzo of degree $3$. 
	
	In addition, the remaining two cases $Y_3$ and $Y_4$ are both branched over the total branch locus of $X\to \mathbb P^2$ (but with different local monodromies). Hence, from Theorem \ref{thm: pg1_Zp_cover_structure} these surfaces are K3 and correspond to the case $A-(3,1)2)$ in  Table \ref{tab: (3,1)}. 
	
	Thus, we have obtained that the intermediate quotients are either rational or K3 surfaces, so $\pi\colon X\to \mathbb P^2$ satisfies Property $\bigstar$. 
	
	Hence, counting the number of the intermediate $(3,1)$-planes that are K3 surfaces, the geometric genus of $X$ is $p_g(X)=2$.
	
	As mentioned above, the $(3,2)-$plane is singular only over the intersection point of the branch components $l_1$ and $l_2$, and above such a point there are $3$ singularities of type $\frac{1}{3}(1,1)$, according to Proposition \ref{prop: type_of_sing}. The same proposition tells us that over the other intersection points of the branch locus no other singularities occur. 
	
	We denote by $\rho\colon \widetilde{X}\to X$ the minimal resolution of singularities of $X$, by $E_1,E_2$ and $E_3$ the $(-3)$-exceptional curves, and by $\widetilde{\pi}\colon \widetilde{X}\to \mathbb P^2$ the composition $\widetilde{\pi}:=\pi\circ \rho$. 
	
	The surface $\widetilde{X}$ is regular from Lemma \ref{lem: (p,k)_planes_are_regular}.  
	The remaining numerical invariants of $\widetilde{X}$ can be computed directly using the
	formulae of Proposition \ref{thm: invariants_ab_cov_cyclic_sing}
	\[
	\begin{split}
		\chi(\mathcal O_{\widetilde{X}})& = 9 + \frac{1}{2}\big( 3H(H - 3H)+6H(2H - 3H)+2H(3H - 3H)\big) =9+\frac{1}{2}(-6-6)=3, \\
		K_{\widetilde{X}} & \equiv \frac{1}{3}\widetilde{\pi}^*(-9H+(3-1)(1+1+3)H)-(E_1+E_2+E_3)=\frac{1}{3}\pi^*(H)-(E_1+E_2+E_3),\\
		K_{\widetilde{X}}^2 & = 3^{2-2}\left(-9H + 10H\right)^2-3\frac{1}{3}=0.
	\end{split}
	\]
	Consequently, by Noether formula we have $e(\widetilde{X})=36$. Furthermore, from Proposition \ref{prop: dec_of_the_pullback}, the strict transforms $\widetilde{\ell}_1$ and $\widetilde{\ell}_2$ of the lines $\ell_1$ and $\ell_2$ on $\widetilde{X}$ are 
	\[
	  \widetilde{\ell_1}\equiv \widetilde{\ell_2}\equiv \frac{1}{3}\widetilde{\pi}^*(H)-(E_1+E_2+E_3).
	\]
	Thus, $K_{\widetilde{X}}\equiv \widetilde{\ell_1}\equiv \widetilde{\ell_2}$ and as already explained in Lemma \ref{lem: B-(3,2)_prop_elliptic}, then $\widetilde{X}$ is minimal and properly elliptic.

We now exhibit an elliptic fibration of $\widetilde{X}$. We consider first the resolution $\widetilde Y_1\to Y_1$ of $Y_1$, which is the Hirzebruch surface $\mathbb F_3$ as explained above. We note that $\widetilde{Y_1}$ is also a triple cover of the blow-up ${\rm Bl}_{pt}(\mathbb P^2)$ branched over the strict transforms $\ell_1$ and $\ell_2$, where $pt$ is the intersection point of $\ell_1$ and $\ell_2$. From here one can easily see the natural ruling $q\colon \widetilde{Y_1}\to \mathbb P^1$ of  $\widetilde{Y_1}=\mathbb F_3$: for any line $\ell$ in $\mathbb P^2$ passing through $pt$ there are three fibres of the ruling, that becomes a single fibre counted with multiplicity $3$ when $\ell$ coincides with one between $\ell_1$ and $\ell_2$. 

We now consider the strict transform $\widetilde{\mathcal C}$ of the cubic curve $\mathcal {C}$ on $\widetilde{Y_1}$. We observe that $\widetilde{\mathcal C}$ is $3$-divisible in ${\rm Pic}(\widetilde{Y_1})$, so there is a triple cover $\pi'\colon X'\to \widetilde{Y_1}$ branched along $\widetilde{C}$. Since $\mathcal C$ intersects any line passing through $pt$ in three points, then $\widetilde{\mathcal C}$ intersects each fibre of $q$ in three points. Thus, $X'$ inherits a natural fibration $\pi'\circ q \colon X'\to \mathbb P^1$ whose general fibre is a triple cover of $\mathbb P^1$ branched over three points, hence is an elliptic curve. 

We also observe that $X'$ has three disjoint $(-3)$-curves, that are the pull-back of the unique $(-3)$-curve of $\widetilde{Y_1}$, which is not intersecting $\widetilde{\mathcal C}$ and is the resolution of the singular point of type $\frac{1}{3}(1,1)$ lying over the point $pt$. Furthermore, $X'$ is a smooth minimal surface and if we contract those three $(-3)$-curves on $X'$, then we obtain the surface $X$, by construction. In other words, $X'\to X$ is the minimal resolution of the singularities of $X$, $\widetilde{X}=X'$. Thus, we have exhibited an elliptic fibration $\pi'\circ q\colon \widetilde{X}\to \mathbb P^1$, which according to Lemma \ref{lem: B-(3,2)_prop_elliptic} is the canonical map of $\widetilde{X}$, since they have the same general fibres. 
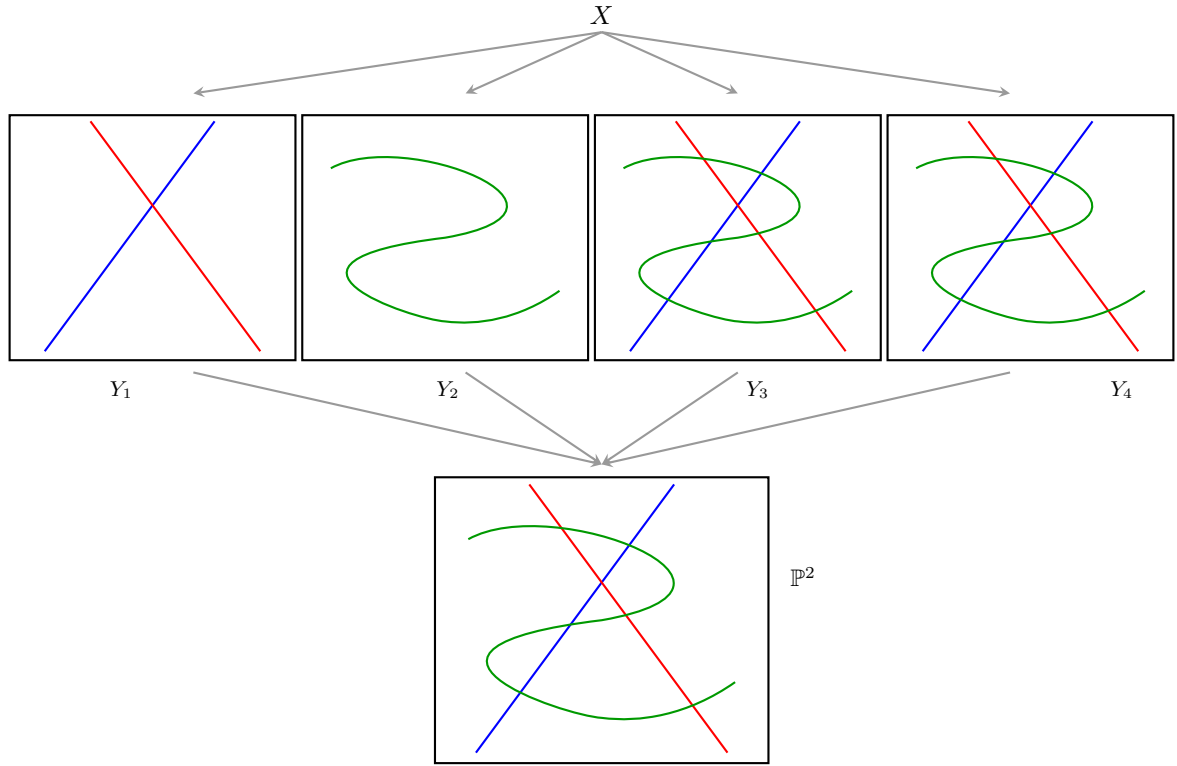
\begin{figure}[H]
	\begin{center}
		\begin{tikzpicture}[scale=0.9, transform shape]
			
			\begin{scope}[shift={(0,-1.2)}, scale=0.7, transform shape]
				\draw[thick, black] (-3.5, -2.2) rectangle (3.5, 3.8);
				
				\draw[thick, blue] (-2.64, -1.98) -- (1.52, 3.65);
				\draw[thick, red] (2.64, -1.98) -- (-1.52, 3.65);
				
				\draw[green!60!black, thick] (-2.8, 2.5) 
				
				.. controls (-1.0, 3.5) and (4.0, 1.5) .. (0, 0.8)
				.. controls (-5, 0.2) and (-0.8, -1.2) .. (0, -1.25)
				.. controls (0.8, -1.35) and (1.8, -1.2) .. (2.8, -0.5);
			\end{scope}
			
			\begin{scope}[shift={(-6.6, 4.5)}, scale=0.6, transform shape]
				\draw[thick, black] (-3.5, -2.2) rectangle (3.5, 3.8);
				\draw[thick, blue] (-2.64, -1.98) -- (1.52, 3.65);
				\draw[thick, red] (2.64, -1.98) -- (-1.52, 3.65);
				
				
				.. controls (-1.0, 3.5) and (4.0, 1.5) .. (0, 0.8)
				.. controls (-5, 0.2) and (-0.8, -1.2) .. (0, -1.25)
				.. controls (0.8, -1.35) and (1.8, -1.2) .. (2.8, -0.5);
			\end{scope}
			
			\begin{scope}[shift={(-2.3, 4.5)}, scale=0.6, transform shape]
				\draw[thick, black] (-3.5, -2.2) rectangle (3.5, 3.8);
				
				\draw[green!60!black, thick] (-2.8, 2.5) 
				
				.. controls (-1.0, 3.5) and (4.0, 1.5) .. (0, 0.8)
				.. controls (-5, 0.2) and (-0.8, -1.2) .. (0, -1.25)
				.. controls (0.8, -1.35) and (1.8, -1.2) .. (2.8, -0.5);
			\end{scope}
			
			\begin{scope}[shift={(2, 4.5)}, scale=0.6, transform shape]
				\draw[thick, black] (-3.5, -2.2) rectangle (3.5, 3.8);
				\draw[thick, blue] (-2.64, -1.98) -- (1.52, 3.65);
				\draw[thick, red] (2.64, -1.98) -- (-1.52, 3.65);
				
				\draw[green!60!black, thick] (-2.8, 2.5) 
				
				.. controls (-1.0, 3.5) and (4.0, 1.5) .. (0, 0.8)
				.. controls (-5, 0.2) and (-0.8, -1.2) .. (0, -1.25)
				.. controls (0.8, -1.35) and (1.8, -1.2) .. (2.8, -0.5);
			\end{scope}
			
			\begin{scope}[shift={(6.3, 4.5)}, scale=0.6, transform shape]
				\draw[thick, black] (-3.5, -2.2) rectangle (3.5, 3.8);
				\draw[thick, blue] (-2.64, -1.98) -- (1.52, 3.65);
				\draw[thick, red] (2.64, -1.98) -- (-1.52, 3.65);
				
				\draw[green!60!black, thick] (-2.8, 2.5) 
				
				.. controls (-1.0, 3.5) and (4.0, 1.5) .. (0, 0.8)
				.. controls (-5, 0.2) and (-0.8, -1.2) .. (0, -1.25)
				.. controls (0.8, -1.35) and (1.8, -1.2) .. (2.8, -0.5);
			\end{scope}

			\draw[-stealth, thick, gray!80] (-6.0, 3) -- (-0, 1.65); 
			\draw[-stealth, thick, gray!80] (-2.0, 3) -- (-0, 1.65); 
			\draw[-stealth, thick, gray!80] (2.0, 3) -- (-0, 1.65); 
			\draw[-stealth, thick, gray!80] (6.0, 3)  -- (0, 1.65);  
			
			\node[below right, font=\small] at (-7.35, 3) {$Y_1$};
			\node[below right, font=\small] at (-2.55, 3) {$Y_2$};
			\node[below right, font=\small] at (2, 3) {$Y_3$};
			\node[below right, font=\small] at (7.35, 3) {$Y_4$};
			
			\node[right] at (2.65, 0)  {$\mathbb P^2$};

			\node[font=\large, above] at (0, 8)  {$X$};
			
			\draw[-stealth, thick, gray!80] (0.0, 8.0) -- (-6, 7.1); 
			\draw[-stealth, thick, gray!80] (0.0, 8.0) -- (-2, 7.1); 
			\draw[-stealth, thick, gray!80] (0.0, 8.0) -- (2, 7.1); 
			\draw[-stealth, thick, gray!80] (0.0, 8.0) -- (6, 7.1);  
		\end{tikzpicture}
		\caption{The branch locus of the case $B-(3,2)1)$ together with the branch loci of its intermediate $(3,1)$-planes $Y_1,\dots, Y_4$.}
		\label{fig: branch_locus_B-(3,2)}
	\end{center}
\end{figure}
	
	We now examine more closely the weight--$2$ Hodge structure of $\widetilde{X}$.
	By  Corollary~\ref{cor:K3_Hodge_decomposition},
 the transcendental lattice $T(\widetilde{X})_{\mathbb Q}$ 
 decomposes
	as direct sum of the corresponding contributions coming from the intermediate $(3,1)$-planes that are K3 surface, hence 
	$Y_3$ and  $Y_4$. Looking at the case $A-(3,1)2)$ in Table \ref{tab: (3,1)}, then we have $\bigl(h^{2,0}(Y_i),\,h^{1,1}(Y_i),\,h^{0,2}(Y_i)\bigr)_{tra}=(1,10,1)$ for $i=3,4$. To complete the study of $H^2(\widetilde{X},\mathbb Q)$ we need to look at $Y_1$ and $Y_2$. Indeed, for these two rational surfaces we have $h^{1,1}(Y_1)=1$ and $h^{1,1}(Y_2)=7$, see Table \ref{tab: (p,1)_pg=0_p=3}.  Thus, $\bigl(h^{2,0}(\widetilde{X}),\,h^{1,1}(\widetilde{X}),\,h^{0,2}(\widetilde{X})\bigr)_{tra}=(2,20,2)$ and $\rho(\widetilde{X})=7+3=10$, where $7$ is the contribution given by $Y_2$ and $3$ are the three $(-3)$-curves of the resolution of $X$. We point out that the computation holds for a very general member $\widetilde{X}$ of the family. 
	The Hodge diamond of $\widetilde{X}$ is
	\[
	\begin{array}{ccccc}
		&& 1 && \\
		& 0 && 0 & \\
		2 && 30 && 2 \\
		& 0 && 0 & \\
		&& 1 &&
	\end{array}.
	\]
	The same kind of result can be obtained with different methods using one among \cite{GPTriple}, \cite{PP13}, and \cite{M85} to the (singular) triple cover $X\to Y_3$. 

	To conclude this example we present $X$ in a suitable weighted projective space according to the method explained in Subsection \ref{subsec: equations_p,k_plane_in_suitable_proj_planes}. 
	\medskip
	Let $(x_0:x_1:x_2)$ be homogeneous coordinates on $\PP^2$ and let
	\[
	f_{e_1}\in H^0(\PP^2,\mathcal O_{\PP^2}(3)),
	\qquad
	\, f_{e_2}\, f_{2e_2},\in H^0(\PP^2,\mathcal O_{\PP^2}(1))
	\]
	be defining equations of the branch components 
	$D_{e_1}$ (smooth cubic) and $D_{e_2}, D_{2e_2}$ (lines),
	respectively. We recall the degrees of the $\mathcal L_\chi$, $\chi\in G^*$, listed in Equation \ref{eq: building_data_ex_B-(3,2)1}. 
	
	Thus, we embed $X$ into the weighted projective space $\PP(1,1,1,1,1,1,2,2,2,3,3)$ with weighted homogeneous coordinates
	\[
	(x_0:x_1:x_2:\ y_{\epsilon_1}:y_{\epsilon_2}:\ y_{\epsilon_2^2}:\ y_{\epsilon_1^2}:\ 
	y_{\epsilon_1\epsilon_2^2}:y_{\epsilon_1\epsilon_2}:y_{\epsilon_1^2\epsilon_2^2}:y_{\epsilon_1^2\epsilon_2}),
	\]
	with $\deg x_i=1$, $\deg y_{\epsilon_1}=\deg y_{\epsilon_2}=\deg y_{\epsilon_2^2}=1$, $\deg y_{\epsilon_1^2\epsilon_2^2}=\deg y_{\epsilon_1^2\epsilon_2}=3$, and the remaining $y_\chi$ have degree $2$.
	
	Note that from Theorem \ref{thm: Par91a} (cf. Subsection \ref{subsec: equations_p,k_plane_in_suitable_proj_planes}) five of the equations defining $X$ are:
	\[
	y_{\epsilon_1}^2=y_{\epsilon_1^2}, \qquad y_{\epsilon_1}y_{\epsilon_2^2}=y_{\epsilon_1\epsilon_2^2}, \qquad y_{\epsilon_1}y_{\epsilon_2}=y_{\epsilon_1\epsilon_2}, \qquad y_{\epsilon_1^2}y_{\epsilon_2^2}=y_{\epsilon_1^2\epsilon_2^2}, \qquad y_{\epsilon_1^2}y_{\epsilon_2}=y_{\epsilon_1^2\epsilon_2}.
	\]
	This tells us that we can eliminate the last five variables and 
	consider $X$ defined in $\PP^5$, by projection. Thus, $X$ is defined by the following equations, see Subsection \ref{subsec: equations_p,k_plane_in_suitable_proj_planes}:
		\begin{equation}\label{eq: equations_B-(3,2)}
			\rk
			\begin{pmatrix}
				f_{2e_2}  & y_{\epsilon_2} & y_{\epsilon_2^2} \\
				y_{\epsilon_2} & y_{\epsilon_2^2} & f_{e_2},\\
			\end{pmatrix}
			\leq 1, \qquad y_{\epsilon_1}^3=f_{e_1}.
		\end{equation}
	Finally, the projection onto the first three coordinates gives the $(3,2)$-plane
	$\pi\colon X\to\PP^2$.
	
	Furthermore, one can easily see that the elliptic fibration studied above is the (unique) fibration $\widetilde{X}\to \mathbb P^1$ obtained by resolving the indeterminacy locus of the rational map
	\[
	X\dashrightarrow \mathbb P^1, \quad 	(x_0:x_1:x_2:\ y_{\epsilon_1}:y_{\epsilon_2}:\ y_{\epsilon_2^2}:\ y_{\epsilon_1^2})\mapsto (f_{2e_2}(x_0,x_1,x_2): -f_{e_2}(x_0,x_1,x_2)). 
	\]
}
\end{example}
\newpage
\section{Appendix: Tables}\label{app:classification_tables}

In this appendix, we collect the complete list of SNC $(p,k)$-planes $\pi\colon X\to \mathbb P^2$ satisfying
Property~$\bigstar$ produced by the algorithm described in
Section~\ref{sec: classification_SNC_p,k_planes_Prop_STAR}, together with the Picard number of the minimal model and the transcendental part of the Hodge structure of a very general member  of the family of $\widetilde{X}_{min}$. The results are organized into several tables according to the corresponding
values of $(p,k)$.

For each admissible pair $(p,k)$, we list all possible configurations of divisors
$\{D_g\}_{g\in G}$.
Cases arising from the same collection of degrees $\{d_g\}_{g\in G}$ are grouped
together.
Indeed, for fixed degrees $d_g$, the divisor $D_g$ may decompose into several smooth
irreducible components, provided that their total degree is $d_g$ and that the
overall branch divisor remains in simple normal crossing position.

For each case, we also record the basket of singularities of the corresponding
$(p,k)$-plane, computed using Proposition~\ref{prop: type_of_sing}.
Any divisor $D_g$ not explicitly appearing in the tables is understood to be zero.

Each case is identified by a letter, the pair $(p,k)$, and a numerical index.
The letter distinguishes cases corresponding to different collections of degrees
$\{d_g\}_{g\in G}$, while the numerical index separates cases with identical degree
data but distinct choices for the decomposition of the divisors $\{D_g\}_{g\in G}$.
\begin{table}[H]
    \footnotesize
    \setlength{\tabcolsep}{5pt}
    \renewcommand{\arraystretch}{1.2}
\begin{tabular}{|c||c|c||c|c|c|}\hline
			case& No.& $D_{e_1}$&$\textrm{Sing}(X)$ &$\rho(\widetilde{X}_{min})$& $h^2(X)_{tra}$\\
			\hline
            \hline
        $A-(2,1)$ &	1)&sextic&{\rm $\emptyset$}&1&(1,19,1)\\
       &	2)&quintic + line&{\rm $1/2^5$}&6&(1,14,1)\\
        &	3)&quartic + conic&{\rm $1/2^8$}&9&(1,11,1)\\
        &	4)&quartic + 2 lines&{\rm $1/2^9$}&10&(1,10,1)\\
        &	5)&cubic + cubic&{\rm $1/2^9$}&10&(1,10,1)\\
         &	6)&cubic + conic +line&{\rm $1/2^{11}$}&12&(1,8,1)\\
        &	7)&cubic + 3 lines&{\rm $1/2^{12}$}&13&(1,7,1)\\
    &	8)&3 conics &{\rm $1/2^{12}$}&13&(1,7,1)\\
         &	9)&2 conics + 2 lines&{\rm $1/2^{13}$}&14&(1,6,1)\\
         &	10)&conic + 4 lines&{\rm $1/2^{14}$}&15&(1,5,1)\\
        &	11)&6 lines&{\rm $1/2^{15}$}&16&(1,4,1)\\
            \hline
         \end{tabular}
         \caption{Building data 
         	of a SNC $(2,1)$-plane $\pi\colon X\to \mathbb P^2$ with $p_g(X)=1$.} \label{tab: (2,1)}
         \end{table}
\begin{table}[H]
    \footnotesize
    \setlength{\tabcolsep}{5pt}
    \renewcommand{\arraystretch}{1.2}
		\begin{tabular}{|c||c|c|c||c|c|c|}\hline
			case& No.& $D_{e_1}$&$D_{2e_1}$&$\textrm{Sing}(X)$&$\rho(\widetilde{X}_{min})$&$h^2(X)_{tra}$\\
			\hline
            \hline
        $A-(3,1)$ &	1)&quartic&line&{\rm $1/3^4$}&4&(1,16,1)\\
        &	2)&cubic + line &line&{\rm $1/3^4, 2/3^3$}&10&(1,10,1)\\
        &	3)&2 conics &line&{\rm $1/3^4, 2/3^4$}&12&(1,8,1)\\
         &	4)&conic + 2 lines &line&{\rm $1/3^4, 2/3^5$}&14&(1,6,1)\\
         &	5)&4 lines &line&{\rm $1/3^4, 2/3^6$}&16&(1,4,1)\\
            \hline
         \end{tabular}
         \caption{Building data 
         	of a SNC $(3,1)$-plane $\pi\colon X\to \mathbb P^2$ with $p_g(X)=1$.} \label{tab: (3,1)}
         \end{table}
         \vspace{-0.5cm}
         
         \begin{table}[H]
    \footnotesize
    \setlength{\tabcolsep}{5pt}
    \renewcommand{\arraystretch}{1.2}
		\begin{tabular}{|c||c|c|c||c|c|c|}\hline
			case& No.& $D_{3e_1}$&$D_{4e_1}$&$\textrm{Sing}(X)$&$\rho(\widetilde{X}_{min})$&$h^2(X)_{tra}$\\
			\hline
            \hline
        $A-(5,1)$ &	1)&line&cubic&{\rm $3/5^3$}&6&(1,14,1)\\
           &	2)&line&conic + line &{\rm $3/5^3, 4/5^2$}&14&(1,6,1)\\
           &	3)&line&3 lines&{\rm $3/5^3, 4/5^3$}&18&(1,2,1)\\
            \hline
         \end{tabular}
         \caption{Building data 
         	of a SNC $(5,1)$-plane $\pi\colon X\to \mathbb P^2$ with $p_g(X)=1$.} \label{tab: (5,1)}
         \end{table}
          \vspace{-0.4cm}
          
          \begin{table}[H]
    \footnotesize
    \setlength{\tabcolsep}{5pt}
    \renewcommand{\arraystretch}{1.2}
		\begin{tabular}{|c||c|c|c|c||c|c|c|}\hline
			case& No.& $D_{4e_1}$&$D_{5e_1}$&$D_{6e_1}$&$\textrm{Sing}(X)$&$\rho(\widetilde{X}_{min})$&$h^2(X)_{tra}$\\
			\hline
            \hline
        $A-(7,1)$ &	1)&line&line&conic & {\rm $2/7^2,3/7^2,4/7$}&10&(1,10,1)\\
           &	2)&line& line&2 lines  &{\rm $2/7^2,3/7^2, 4/7,6/7$}&16&(1,4,1)\\
            \hline
         \end{tabular}
         \caption{Building data 
         	of a SNC $(7,1)$-plane $\pi\colon X\to \mathbb P^2$ with $p_g(X)=1$.} \label{tab: (7,1)}
         \end{table}
          \vspace{-0.4cm}
          
          \begin{table}[H]
    \footnotesize
    \setlength{\tabcolsep}{5pt}
    \renewcommand{\arraystretch}{1.2}
		\begin{tabular}{|c||c|c|c|c|c||c|c|c|}\hline
			case& No.& $D_{6e_1}$&$D_{8e_1}$&$D_{9e_1}$&$D_{10e_1}$&$\textrm{Sing}(X)$&$\rho(\widetilde{X}_{min})$&$h^2(X)_{tra}$\\
			\hline
            \hline
        $A-(11,1)$ &	1)&line&line&line&line& {\rm \makecell{$2/11,3/11,$\\ $4/11,5/11,$ \\ $6/11,7/11$}}&12&(1,8,1)\\
            \hline
         \end{tabular}
         \caption{Building data 
         	of a SNC $(11,1)$-plane $\pi\colon X\to \mathbb P^2$ with $p_g(X)=1$.} \label{tab: (11,1)}
         \end{table}
              \begin{table}[H]
    \footnotesize
    \setlength{\tabcolsep}{5pt}
    \renewcommand{\arraystretch}{1.2}
		\begin{tabular}{|c||c|c|c|c|c||c|c|c|}\hline
			case& No.& $D_{8e_1}$&$D_{9e_1}$&$D_{10e_1}$&$D_{12e_1}$&$\textrm{Sing}(X)$&$\rho(\widetilde{X}_{min})$&$h^2(X)_{tra}$\\
			\hline
            \hline
        $A-(13,1)$ &	1)&line&line&line&line& {\rm \makecell{$2/13,3/13,$\\$4/13,5/13,$\\$7/13,9/13$}}&10&(1,10,1)\\
            \hline
         \end{tabular}
         \caption{Building data 
         	of a SNC $(13,1)$-plane $\pi\colon X\to \mathbb P^2$ with $p_g(X)=1$.} \label{tab: (13,1)}
         \end{table}
         \begin{table}[H]
    \footnotesize
    \setlength{\tabcolsep}{5pt}
    \renewcommand{\arraystretch}{1.2}
		\begin{tabular}{|c||c|c|c|c|c||c|c|c|}\hline
			case& No.& $D_{10e_1}$&$D_{12e_1}$&$D_{14e_1}$&$D_{15e_1}$&$\textrm{Sing}(X)$&$\rho(\widetilde{X}_{min})$&$h^2(X)_{tra}$\\
			\hline
            \hline
        $A-(17,1)$ &	1)&line&line&line&line& {\rm \makecell{$2/17,3/17,$\\$5/17,7/17,$\\$9/17,13/17$}}&6&(1,14,1)\\
            \hline
         \end{tabular}
         \caption{Building data 
         	of a SNC $(17,1)$-plane $\pi\colon X\to \mathbb P^2$ with $p_g(X)=1$.} \label{tab: (17,1)}
         \end{table}
         \begin{table}[H]
    \footnotesize
    \setlength{\tabcolsep}{5pt}
    \renewcommand{\arraystretch}{1.2}
		\begin{tabular}{|c||c|c|c|c|c||c|c|c|}\hline
			case& No.& $D_{12e_1}$&$D_{14e_1}$&$D_{15e_1}$&$D_{16e_1}$&$\textrm{Sing}(X)$&$\rho(\widetilde{X}_{min})$&$h^2(X)_{tra}$\\
			\hline
            \hline
        $A-(19,1)$ &	1)&line&line&line&line& {\rm \makecell{$2/19,3/19,$\\$4/19,5/19,$\\$7/19,13/19$}}&4&(1,16,1)\\
            \hline
         \end{tabular}
         \caption{Building data 
         	of a SNC $(19,1)$-plane $\pi\colon X\to \mathbb P^2$ with $p_g(X)=1$.} \label{tab: (19,1)}
         \end{table}
    \begin{table}[H]
    \footnotesize
		\begin{tabular}{|c||c|c|c|c||c|c||c|c|c|c|c|}\hline
			case& No.& $D_{e_1}$&$D_{e_2}$&$D_{e_{12}}$&$\textrm{Sing}(X)$&$\rho(\widetilde{X}_{min})$&$h^2(X)_{tra}$\\
			\hline
		$A-(2,2)$ &	1)&quartic&conic&$0$ &{\rm $\emptyset$}&9&(1,11,1)\\
			
		&	2)&cubic + line &conic&$0$&{\rm $1/2^6$}&12&(1,8,1)\\
			
		&	3)&2 conics&conic&$0$&{\rm $1/2^8$}&13&(1,7,1)\\

         &   4)&conic + 2 lines &conic&$0$&{\rm $1/2^{10}$}&14&(1,6,1)\\
			
		&	5)&4 lines&conic&$0$&{\rm $1/2^{12}$}&15&(1,5,1)\\
			
		&	6)&quartic&2 lines&$0$ &{\rm $1/2^2$}&10&(1,10,1)\\
			
		&	7)&cubic + line &2 lines&$0$&{\rm $1/2^8$}&13&(1,7,1)\\
\hline
		\end{tabular}
        \hypertarget{tab: (2,2)}{\caption{Building data 
        		of a SNC $(2,2)$-plane $\pi\colon X\to \mathbb P^2$ satisfying Property~$\bigstar$.}} \label{tab: (2,2)}
	\end{table}

     \begin{table}[H]
    \footnotesize
    \hspace*{-3.3cm}
		\begin{tabular}{|c||c|c|c|c|c|c|c|c||c|c|c|}\hline
			case& No.& $D_{e_1}$&$D_{e_2}$&$D_{e_{3}}$&$D_{e_{1}+e_2}$&$D_{e_{1}+e_3}$&$D_{e_{2}+e_3}$&$D_{e_{1}+e_2+e_3}$&$\textrm{Sing}(X)$&$\rho(\widetilde{X}_{min})$&$h^2(X)_{tra}$\\
			\hline
		$A-(2,3)$ &	1)&conic&conic&conic &0&0&0&0&{\rm $\emptyset$}&13&(1,7,1)\\
			
		&	2)&2 lines &conic&conic&0&0&0&0&{\rm $1/2^4$}&14&(1,6,1)\\
			
		&	3)&2 lines&2 lines &conic&0&0&0&0&{\rm $1/2^8$}&15&(1,5,1)\\

         &   4)&2 lines &2 lines &2 lines &0&0&0&0&{\rm $1/2^{12}$}&16&(1,4,1)\\
        \hline
		$B-(2,3)$ &	1)&line&cubic&line &0&0&0&line&{\rm $\emptyset$}&13&(1,7,1)\\
			
		&	2)&line &conic + line&line&0&0&0&line&{\rm $1/2^8$}&15&(1,5,1)\\
			
		&	3)&line&3 lines &line&0&0&0&line&{\rm $1/2^{12}$}&16&(1,4,1)\\
\hline
		\end{tabular}
        \hypertarget{tab: (2,3)}{\caption{Building data 
        		of a SNC $(2,3)$-plane $\pi\colon X\to \mathbb P^2$ satisfying Property~$\bigstar$.}} \label{tab: (2,3)}
	\end{table}
    \begin{table}[H]
    \footnotesize
		\begin{tabular}{|c||c|c|c|c|c|c|c|c||c|c|c|}\hline
			case& No.& $D_{e_1}$&$D_{e_2}$&$D_{e_{3}}$&$D_{e_{1}+e_2}$&$D_{e_{1}+e_3}$&$D_{e_{2}+e_3}$&$D_{e_{1}+e_2+e_3}$&$\textrm{Sing}(X)$&$\rho(\widetilde{X}_{min})$&$h^2(X)_{tra}$\\
			\hline
            \hline
        $K-(2,3)$ &	1)&line&line&conic &line&conic&conic&0&{\rm $\emptyset$}&7&(4,25,4)\\
        &	2)&line&line&conic &line&conic&2 lines&0&{\rm $1/2^4$}&10&(4,22,4)\\
        &	3)&line&line&conic &line&2 lines&2 lines&0&{\rm $1/2^8$}&13&(4,19,4)\\
        &	4)&line&line&2 lines &line&2 lines&2 lines&0&{\rm $1/2^{12}$}&16&(4,16,4)\\
        \hline
        $L-(2,3)$ &	1)&line&line&line &line&line&cubic&line&{\rm $\emptyset$}&4&(4,28,4)\\
        &	2)&line&line&line &line&line&conic+line&line&{\rm $1/2^8$}&12&(4,20,4)\\
        &	3)&line&line&line &line&line&3 lines&line&{\rm $1/2^{12}$}&16&(4,16,4)\\
        \hline
        $M-(2,3)$ &	1)&line&line&conic &conic&line&line&conic&{\rm $\emptyset$}&2&(6,36,6)\\
        &	2)&line&line&conic &conic&line&line&2 lines&{\rm $1/2^4$}&6&(6,32,6)\\
        &	3)&line&line&conic &2 lines&line&line&2 lines&{\rm $1/2^8$}&10&(6,28,6)\\
        &	4)&line&line&2 lines &2 lines&line&line&2 lines&{\rm $1/2^{12}$}&14&(6,24,6)\\
        \hline
		\end{tabular}
        \caption{Building data 
        	of a SNC $(2,3)$-plane $\pi\colon X\to \mathbb P^2$ satisfying Property~$\bigstar$.} \label{tab: (2,3)-part2}
	\end{table}
    \vspace{-0.9cm}
    
    \begin{table}[H]
    \begin{flushleft} 
    \footnotesize
    \setlength{\tabcolsep}{5pt}
    \renewcommand{\arraystretch}{1.2}
		\begin{tabular}{|c||c|c|c|c|c|c||c|c|c|}\hline
			case& No.& $D_{e_1}$&$D_{e_2}$&$D_{e_{3}}$&$D_{e_4}$&$D_{e_1+e_2+e_4}$&$\textrm{Sing}(X)$&$\rho(\widetilde{X}_{min})$&$h^2(X)_{tra}$\\
			\hline
            \hline
        $A-(2,4)$ &	1)&line&line&conic &line&line&{\rm $\emptyset$}&15&(1,5,1)\\
         &	2)&line&line&2 lines &line&line&{\rm $1/2^8$}&16&(1,4,1)\\
            \hline
         \end{tabular}

    \vspace{0.5em} 
    \begin{tabular}{|c||c|c|c|c|c|c||c|c|c|}
    \hline
        case& No.& $D_{e_1}$&$D_{e_2}$&$D_{e_{3}}$&$D_{e_4}$&$D_{e_1+e_2+e_3+e_4}$&$\textrm{Sing}(X)$&$\rho(\widetilde{X}_{min})$&$h^2(X)_{tra}$\\
			\hline
            \hline
        $B-(2,4)$ &	1)&line&line&cubic &line&line&{\rm $\emptyset$}&18&(4,28,4)\\
        &	2)&line&line&conic+line&line&line&{\rm $1/2^{16}$}&26&(4,20,4)\\
        &	3)&line&line&3 lines&line&line&{\rm $1/2^{24}$}&30&(4,16,4)\\
        \hline
         \end{tabular}

    \vspace{0.5em} 
    \begin{tabular}{|c||c|c|c|c|c|c||c|c|c|}
    \hline
        case& No.& $D_{e_1}$&$D_{e_2}$&$D_{e_{3}}$&$D_{e_4}$&$D_{e_1+e_2}$&$\textrm{Sing}(X)$&$\rho(\widetilde{X}_{min})$&$h^2(X)_{tra}$\\
			\hline
            \hline
        $C-(2,4)$ &	1)&line&line&conic &conic&line&{\rm $\emptyset$}&18&(3,18,3)\\
        &	2)&line&line&conic &2 lines&line&{\rm $1/2^8$}&21&(3,15,3)\\
        &	3)&line&line&2 lines &2 lines&line&{\rm $1/2^{16}$}&24&(3,12,3)\\
        \hline
         \end{tabular}

    \vspace{0.5em} 
    \begin{tabular}{|c||c|c|c|c|c|c|c||c|c|c|}
    \hline
        case& No.& $D_{e_1}$&$D_{e_2}$&$D_{e_{3}}$&$D_{e_4}$&$D_{e_1+e_2}$&$D_{e_1+e_2+e_3+e_4}$&$\textrm{Sing}(X)$&$\rho(\widetilde{X}_{min})$&$h^2(X)_{tra}$\\
			\hline
            \hline
        $D-(2,4)$ &	1)&line&line&line &line&conic&line&{\rm $\emptyset$}&16&(2,10,2)\\
        &	2)&line&line&line &line&2 lines&line&{\rm $1/2^8$}&18&(2,8,2)\\
       \hline
         \end{tabular}

    \vspace{0.5em} 
    \begin{tabular}{|c||c|c|c|c|c|c|c||c|c|c|}
    \hline
        case& No.& $D_{e_1}$&$D_{e_2}$&$D_{e_{3}}$&$D_{e_4}$&$D_{e_1+e_3}$&$D_{e_2+e_3+e_4}$&$\textrm{Sing}(X)$&$\rho(\widetilde{X}_{min})$&$h^2(X)_{tra}$\\
			\hline
            \hline
        $E-(2,4)$ &	1)&conic&line&line &line&conic&line&{\rm $\emptyset$}&16&(5,28,5)\\
        &	2)&conic&line&line &line&2 lines&line&{\rm $1/2^8$}&20&(5,24,5)\\
        &	3)&2 lines&line&line &line&2 lines&line&{\rm $1/2^{16}$}&24&(5,20,5)\\
        \hline
         \end{tabular}

    \vspace{0.5em} 
    \begin{tabular}{|c||c|c|c|c|c|c|c||c|c|c|}
    \hline
        case& No.& $D_{e_1}$&$D_{e_2}$&$D_{e_{3}}$&$D_{e_4}$&$D_{e_2+e_3}$&$D_{e_2+e_4}$&$\textrm{Sing}(X)$&$\rho(\widetilde{X}_{min})$&$h^2(X)_{tra}$\\
			\hline
            \hline
        $F-(2,4)$ &	1)&conic&conic&line &line&line&line&{\rm $\emptyset$}&15&(5,29,5)\\
        &2)&conic&2 lines&line &line&line&line&{\rm $1/2^8$}&19&(5,25,5)\\
       & 3)&2 lines &conic&line &line&line&line&{\rm $1/2^8$}&20&(5,24,5)\\
       & 4)&2 lines&2 lines&line &line&line&line&{\rm $1/2^{16}$}&24&(5,20,5)\\
        \hline
         \end{tabular}
         
         \vspace{0.5em}
          \begin{tabular}{|c||c|c|c|c|c|c|c||c|c|c|}
    \hline
        case& No.& $D_{e_1}$&$D_{e_2}$&$D_{e_{3}}$&$D_{e_4}$&$D_{e_1+e_3}$&$D_{e_2+e_4}$&$\textrm{Sing}(X)$&$\rho(\widetilde{X}_{min})$&$h^2(X)_{tra}$\\
			\hline
            \hline
        $G-(2,4)$ &	1)&line&line&line &cubic&line&line&{\rm $\emptyset$}&12&(6,42,6)\\
        &	2)&line&line&line &conic+line&line&line&{\rm $1/2^{16}$}&24&(6,30,6)\\
        &	3)&line&line&line &3 lines&line&line&{\rm $1/2^{24}$}&30&(6,24,6)\\
        \hline
         \end{tabular}

    \vspace{0.5em} 
        \hypertarget{tab: (2,4)}{\caption{Building data 
        		of a SNC $(2,4)$-plane $\pi\colon X\to \mathbb P^2$ satisfying Property~$\bigstar$.}} \label{tab: (2,4)}
        \end{flushleft}
	\end{table}
    
    \begin{table}[H]
    \begin{flushleft}
    \footnotesize
    \setlength{\tabcolsep}{5pt}
    \renewcommand{\arraystretch}{1.2}
    
    \hspace*{-3cm}
    \begin{minipage}{\textwidth}
    \begin{tabular}{|c||c|c|c|c|c|c|c||c|c|c|}
    \hline
        case& No.& $D_{e_1}$&$D_{e_2}$&$D_{e_{3}}$&$D_{e_4}$&$D_{e_1+e_2}$&$D_{e_3+e_4}$&$\textrm{Sing}(X)$&$\rho(\widetilde{X}_{min})$&$h^2(X)_{tra}$\\
			\hline
            \hline
        $H-(2,4)$ &	1)&conic&conic&line &line&conic&line&{\rm $\emptyset$}&10&(9,54,9)\\
        &	2)&conic&conic&line &line&2 lines&line&{\rm $1/2^8$}&16&(9,48,9)\\
        &	3)&conic&2 lines&line &line&2 lines&line&{\rm $1/2^{16}$}&22&(9,42,9)\\
        &	4)&2 lines&2 lines&line &line&2 lines&line&{\rm $1/2^{24}$}&28&(9,36,9)\\
        \hline
         \end{tabular}
         
      \vspace{0.5em} 
    \begin{tabular}{|c||c|c|c|c|c|c|c|c||c|c|c|}
    \hline
        case& No.& $D_{e_1}$&$D_{e_2}$&$D_{e_{3}}$&$D_{e_4}$&$D_{e_3+e_4}$&$D_{e_1+e_3+e_4}$&$D_{e_2+e_3+e_4}$&$\textrm{Sing}(X)$&$\rho(\widetilde{X}_{min})$&$h^2(X)_{tra}$\\
			\hline
            \hline
        $I-(2,4)$ &	1)&line&line&line &line&line&line&line&{\rm $\emptyset$}&12&(1,4,1)\\
        \hline
         \end{tabular}

    \vspace{0.5em} 
    \begin{tabular}{|c||c|c|c|c|c|c|c|c||c|c|c|}
    \hline
        case& No.& $D_{e_1}$&$D_{e_2}$&$D_{e_{3}}$&$D_{e_4}$&$D_{e_1+e_2}$&$D_{e_1+e_4}$&$D_{e_2+e_3}$&$\textrm{Sing}(X)$&$\rho(\widetilde{X}_{min})$&$h^2(X)_{tra}$\\
			\hline
            \hline
        $J-(2,4)$ &	1)&line&line&line &line&conic&line&line&{\rm $\emptyset$}&14&(4,20,4)\\
        &	2)&line&line&line &line&2 lines&line&line&{\rm $1/2^8$}&18&(4,16,4)\\
        \hline
         \end{tabular}

    \vspace{0.5em} 
    \begin{tabular}{|c||c|c|c|c|c|c|c|c||c|c|c|}
    \hline
        case& No.& $D_{e_1}$&$D_{e_2}$&$D_{e_{3}}$&$D_{e_4}$&$D_{e_1+e_3}$&$D_{e_2+e_3}$&$D_{e_1+e_3+e_4}$&$\textrm{Sing}(X)$&$\rho(\widetilde{X}_{min})$&$h^2(X)_{tra}$\\
			\hline
            \hline
        $K-(2,4)$ &	1)&line&line&conic &line&conic&line&line&{\rm $\emptyset$}&10&(8,44,8)\\
        &	2)&line&line&conic &line&2 lines&line&line&{\rm $1/2^8$}&16&(8,38,8)\\
        &	3)&line&line&2 lines &line&2 lines&line&line&{\rm $1/2^{16}$}&22&(8,32,8)\\
        \hline
		\end{tabular}
        \vspace{0.5em} 
		\begin{tabular}{|c||c|c|c|c|c|c|c|c||c|c|c|}\hline
			case& No.& $D_{e_1}$&$D_{e_2}$&$D_{e_{3}}$&$D_{e_4}$&$D_{e_1+e_2}$&$D_{e_1+e_3}$&$D_{e_1+e_4}$&$\textrm{Sing}(X)$&$\rho(\widetilde{X}_{min})$&$h^2(X)_{tra}$\\
			\hline
            \hline
        $L-(2,4)$ &	1)&cubic&line&line &line&line&line&line&{\rm $\emptyset$}&4&(9,60,9)\\
         &	2)&conic+line&line&line &line&line&line&line&{\rm $1/2^{16}$}&20&(9,44,9)\\
         &	3)&3 lines&line&line &line&line&line&line&{\rm $1/2^{24}$}&28&(9,36,9)\\
         \hline
         \end{tabular}

    \vspace{0.5em} 
    \begin{tabular}{|c||c|c|c|c|c|c|c|c|c||c|c|c|}
    \hline
         case& No.& $D_{e_1}$&$D_{e_2}$&$D_{e_{3}}$&$D_{e_4}$&$D_{e_1+e_4}$&$D_{e_1+e_2+e_3}$&$D_{e_2+e_3+e_4}$&$D_{e_1+e_2+e_3+e_4}$&$\textrm{Sing}(X)$&$\rho(\widetilde{X}_{min})$&$h^2(X)_{tra}$\\
			\hline
            \hline
        $M-(2,4)$ &	1)&line&line&line &line&line&line&line&line&{\rm $\emptyset$}&12&(3,12,3)\\
         \hline
         \end{tabular}

    \vspace{0.5em} 
    \begin{tabular}{|c||c|c|c|c|c|c|c|c|c||c|c|c|}
    \hline
         case& No.& $D_{e_1}$&$D_{e_2}$&$D_{e_{3}}$&$D_{e_4}$&$D_{e_1+e_2}$&$D_{e_2+e_3}$&$D_{e_3+e_4}$&$D_{e_1+e_2+e_3}$&$\textrm{Sing}(X)$&$\rho(\widetilde{X}_{min})$&$h^2(X)_{tra}$\\
			\hline
            \hline
        $N-(2,4)$ &	1)&conic&line&line &line&line&line&line&line&{\rm $\emptyset$}&10&(7,34,7)\\
        &	2)&2 lines&line&line &line&line&line&line&line&{\rm $1/2^8$}&16&(7,28,7)\\
        \hline
         \end{tabular}

    \vspace{0.5em} 
    \begin{tabular}{|c||c|c|c|c|c|c|c|c|c||c|c|c|}
    \hline
         case& No.& $D_{e_1}$&$D_{e_2}$&$D_{e_{3}}$&$D_{e_4}$&$D_{e_1+e_2}$&$D_{e_1+e_4}$&$D_{e_2+e_4}$&$D_{e_1+e_2+e_4}$&$\textrm{Sing}(X)$&$\rho(\widetilde{X}_{min})$&$h^2(X)_{tra}$\\
			\hline
            \hline
        $O-(2,4)$ &	1)&line&line&conic &line&line&line&line&line&{\rm $\emptyset$}&9&(7,35,7)\\
        &	2)&line&line&2 lines &line&line&line&line&line&{\rm $1/2^8$}&16&(7,28,7)\\
        \hline
         \end{tabular}

    \vspace{0.5em} 
    \begin{tabular}{|c||c|c|c|c|c|c|c|c|c|c||c|c|c|}
    \hline
         case& No.& $D_{e_1}$&$D_{e_2}$&$D_{e_{3}}$&$D_{e_4}$&$D_{e_1+e_2}$&$D_{e_1+e_3}$&$D_{e_3+e_4}$&$D_{e_1+e_2+e_4}$&$D_{e_2+e_3+e_4}$&$\textrm{Sing}(X)$&$\rho(\widetilde{X}_{min})$&$h^2(X)_{tra}$\\
			\hline
            \hline
        $P-(2,4)$ &	1)&line&line&line&line&line&line&line&line&line&{\rm $\emptyset$}&10&(6,24,6)\\
       \hline
         \end{tabular}

    \vspace{0.5em} 
    \begin{tabular}{|c||c|c|c|c|c|c|c|c|c|c||c|c|c|}
    \hline
         case& No.& $D_{e_1}$&$D_{e_2}$&$D_{e_{3}}$&$D_{e_4}$&$D_{e_1+e_2}$&$D_{e_1+e_4}$&$D_{e_2+e_4}$&$D_{e_3+e_4}$&$D_{e_1+e_2+e_4}$&$\textrm{Sing}(X)$&$\rho(\widetilde{X}_{min})$&$h^2(X)_{tra}$\\
			\hline
            \hline
        $Q-(2,4)$ &	1)&line&line&line&conic&line&line&line&line&line&{\rm $\emptyset$}&4&(11,52,11)\\
        &	2)&line&line&line&2 lines&line&line&line&line&line&{\rm $1/2^8$}&12&(11,44,11)\\
        \hline
         \end{tabular}

    \vspace{0.5em} 
    \begin{tabular}{|c||c|c|c|c|c|c|c|c|c|c|c||c|c|c|}
    \hline
         case& No.& $D_{e_1}$&$D_{e_2}$&$D_{e_{3}}$&$D_{e_4}$&$D_{e_1+e_2}$&$D_{e_2+e_3}$&$D_{e_3+e_4}$&$D_{\sum_{i=1}^3e_i}$&$D_{\sum_{i=2}^4e_i}$&$D_{\sum_{i=1}^4e_i}$&$\textrm{Sing}(X)$&$\rho(\widetilde{X}_{min})$&$h^2(X)_{tra}$\\
			\hline
            \hline
        $R-(2,4)$ &	1)&line&line&line&line&line&line&line&line&line&line&{\rm $\emptyset$}&6&(10,40,10)\\
        \hline
         \end{tabular}
     \end{minipage}
        \caption{Building data 
        	of a SNC $(2,4)$-plane $\pi\colon X\to \mathbb P^2$ satisfying Property~$\bigstar$.} \label{tab: (2,4)_part2}
    \end{flushleft}
	\end{table}

    \begin{table}[H]
    \begin{flushleft} 
    \footnotesize
    \setlength{\tabcolsep}{5pt}
    \renewcommand{\arraystretch}{1.2}
    
     \hspace*{-1.3cm}
    \begin{minipage}{\textwidth}
    \begin{tabular}{|c||c|c|c|c|c|c|c||c|c|c|}\hline
			case& No.& $D_{e_1}$&$D_{e_2}$&$D_{e_{3}}$&$D_{e_4}$&$D_{e_5}$&$D_{e_1+e_2+e_3+e_4+e_5}$&$\textrm{Sing}(X)$&$\rho(\widetilde{X}_{min})$&$h^2(X)_{tra}$\\
			\hline
            \hline
        $A-(2,5)$ &	1)&line&line&line &line&line&line&{\rm $\emptyset$}&16&(1,4,1)\\
            \hline
         \end{tabular}

    \vspace{0.5em} 
		\begin{tabular}{|c||c|c|c|c|c|c|c||c|c|c|}\hline
			case& No.& $D_{e_1}$&$D_{e_2}$&$D_{e_{3}}$&$D_{e_4}$&$D_{e_5}$&$D_{e_1+e_2+e_4+e_5}$&$\textrm{Sing}(X)$&$\rho(\widetilde{X}_{min})$&$h^2(X)_{tra}$\\
			\hline
            \hline
        $B-(2,5)$ &	1)&line&line&conic &line&line&line&{\rm $\emptyset$}&27&(5,25,5)\\
         &	2)&line&line&2 lines &line&line&line&{\rm $1/2^{16}$}&32&(5,20,5)\\
            \hline
         \end{tabular}

    \vspace{0.5em} 
    \begin{tabular}{|c||c|c|c|c|c|c|c|c||c|c|c|}\hline
			case& No.& $D_{e_1}$&$D_{e_2}$&$D_{e_{3}}$&$D_{e_4}$&$D_{e_5}$&$D_{e_4+e_5}$&$D_{e_1+e_2+e_3}$&$\textrm{Sing}(X)$&$\rho(\widetilde{X}_{min})$&$h^2(X)_{tra}$\\
			\hline
            \hline
        $C-(2,5)$ &	1)&line&line&line &line&line&line&line&{\rm $\emptyset$}&20&(3,12,3)\\
            \hline
         \end{tabular}

          \vspace{0.5em} 
    \begin{tabular}{|c||c|c|c|c|c|c|c|c||c|c|c|}\hline
			case& No.& $D_{e_1}$&$D_{e_2}$&$D_{e_{3}}$&$D_{e_4}$&$D_{e_5}$&$D_{e_1+e_2}$&$D_{e_1+e_3+e_4+e_5}$&$\textrm{Sing}(X)$&$\rho(\widetilde{X}_{min})$&$h^2(X)_{tra}$\\
			\hline
            \hline
        $D-(2,5)$ &	1)&conic&line&line&line&line&line&line&{\rm $\emptyset$}&24&(9,44,9)\\
        &	2)&2 lines&line&line&line&line&line&line&{\rm $1/2^{16}$}&32&(9,36,9)\\
            \hline
         \end{tabular}

         \vspace{0.5em} 
    \begin{tabular}{|c||c|c|c|c|c|c|c|c||c|c|c|}\hline
			case& No.& $D_{e_1}$&$D_{e_2}$&$D_{e_{3}}$&$D_{e_4}$&$D_{e_5}$&$D_{e_1+e_3}$&$D_{e_2+e_5}$&$\textrm{Sing}(X)$&$\rho(\widetilde{X}_{min})$&$h^2(X)_{tra}$\\
			\hline
            \hline
        $E-(2,5)$ &	1)&line&line&line&conic&line&line&line&{\rm $\emptyset$}&23&(9,45,9)\\
        &	2)&line&line&line&2 lines&line&line&line&{\rm $1/2^{16}$}&32&(9,36,9)\\
            \hline
         \end{tabular}

         \vspace{0.5em} 
    \begin{tabular}{|c||c|c|c|c|c|c|c|c|c||c|c|c|}\hline
			case& No.& $D_{e_1}$&$D_{e_2}$&$D_{e_{3}}$&$D_{e_4}$&$D_{e_5}$&$D_{e_2+e_5}$&$D_{e_1+e_2+e_5}$&$D_{e_2+e_3+e_4+e_5}$&$\textrm{Sing}(X)$&$\rho(\widetilde{X}_{min})$&$h^2(X)_{tra}$\\
			\hline
            \hline
        $F-(2,5)$ &	1)&line&line&line &line&line&line&line&line&{\rm $\emptyset$}&20&(7,28,7)\\
            \hline
         \end{tabular}
         
          \vspace{0.5em} 
    \begin{tabular}{|c||c|c|c|c|c|c|c|c|c||c|c|c|}\hline
			case& No.& $D_{e_1}$&$D_{e_2}$&$D_{e_{3}}$&$D_{e_4}$&$D_{e_5}$&$D_{e_1+e_3}$&$D_{e_2+e_4}$&$D_{e_2+e_4+e_5}$&$\textrm{Sing}(X)$&$\rho(\widetilde{X}_{min})$&$h^2(X)_{tra}$\\
			\hline
            \hline
        $G-(2,5)$ &	1)&line&line&line &line&line&line&conic&line&{\rm $\emptyset$}&16&(15,72,15)\\
        & 2)&line&line&line &line&line&line&2 lines&line&{\rm $1/2^{16}$}&28&(15,60,15)\\
            \hline
         \end{tabular}

          \vspace{0.5em} 
    \begin{tabular}{|c||c|c|c|c|c|c|c|c|c|c||c|c|c|}\hline
			case& No.& $D_{e_1}$&$D_{e_2}$&$D_{e_{3}}$&$D_{e_4}$&$D_{e_5}$&$D_{e_1+e_3}$&$D_{e_2+e_5}$&$D_{e_3+e_5}$&$D_{e_3+e_4+e_5}$&$\textrm{Sing}(X)$&$\rho(\widetilde{X}_{min})$&$h^2(X)_{tra}$\\
			\hline
            \hline
        $H-(2,5)$ &	1)&line&line&line &line&line&line&line&line&line&{\rm $\emptyset$}&16&(13,52,13)\\
            \hline
         \end{tabular}

         \vspace{0.5em} 
    \begin{tabular}{|c||c|c|c|c|c|c|c|c|c|c||c|c|c|}\hline
			case& No.& $D_{e_1}$&$D_{e_2}$&$D_{e_{3}}$&$D_{e_4}$&$D_{e_5}$&$D_{e_1+e_3}$&$D_{e_1+e_5}$&$D_{e_1+e_3+e_5}$&$D_{e_2+e_3+e_4+e_5}$&$\textrm{Sing}(X)$&$\rho(\widetilde{X}_{min})$&$h^2(X)_{tra}$\\
			\hline
            \hline
        $I-(2,5)$ &	1)&line&line&line &line&line&line&line&line&line&{\rm $\emptyset$}&16&(13,52,13)\\
            \hline
         \end{tabular}

         \vspace{0.5em} 
    \begin{tabular}{|c||c|c|c|c|c|c|c|c|c|c|c||c|c|c|}\hline
			case& No.& $D_{e_1}$&$D_{e_2}$&$D_{e_{3}}$&$D_{e_4}$&$D_{e_5}$&$D_{e_1+e_2}$&$D_{e_1+e_4}$&$D_{e_2+e_4}$&$D_{e_3+e_5}$&$D_{e_1+e_2+e_4}$&$\textrm{Sing}(X)$&$\rho(\widetilde{X}_{min})$&$h^2(X)_{tra}$\\
			\hline
            \hline
        $J-(2,5)$ &	1)&line&line&line &line&line&line&line&line&line&line&{\rm $\emptyset$}&8&(21,84,21)\\
            \hline
         \end{tabular}
     \end{minipage}
    \end{flushleft}
    \hypertarget{tab: (2,5)}{\caption{Building data 
    		of a SNC $(2,5)$-plane $\pi\colon X\to \mathbb P^2$ satisfying Property~$\bigstar$.}} \label{tab: (2,5)}
    \end{table}

      \begin{table}[H]
    \begin{flushleft} 
    \footnotesize
    \setlength{\tabcolsep}{5pt}
    \renewcommand{\arraystretch}{1.2}
    \begin{tabular}{|c||c|c|c|c|c|c|c|c||c|c|c|}\hline
			case& No.& $D_{e_1}$&$D_{e_2}$&$D_{e_{3}}$&$D_{e_4}$&$D_{e_5}$&$D_{e_6}$&$D_{e_1+e_2+e_3+e_4+e_5+e_6}$&$\textrm{Sing}(X)$&$\rho(\widetilde{X}_{min})$&$h^2(X)_{tra}$\\
			\hline
            \hline
        $A-(2,6)$ &	1)&line&line&line &line&line&line&line&{\rm $\emptyset$}&36&(7,28,7)\\
            \hline
         \end{tabular}

    \vspace{0.5em} 
    \begin{tabular}{|c||c|c|c|c|c|c|c|c|c||c|c|c|}\hline
			case& No.& $D_{e_1}$&$D_{e_2}$&$D_{e_{3}}$&$D_{e_4}$&$D_{e_5}$&$D_{e_6}$&$D_{e_3+e_6}$&$D_{e_1+e_2+e_4+e_5}$&$\textrm{Sing}(X)$&$\rho(\widetilde{X}_{min})$&$h^2(X)_{tra}$\\
			\hline
            \hline
        $B-(2,6)$ &	1)&line&line&line &line&line&line&line&line&{\rm $\emptyset$}&36&(15,60,15)\\
            \hline
         \end{tabular}

    \vspace{0.5em} 
    \begin{tabular}{|c||c|c|c|c|c|c|c|c|c|c||c|c|c|}\hline
			case& No.& $D_{e_1}$&$D_{e_2}$&$D_{e_{3}}$&$D_{e_4}$&$D_{e_5}$&$D_{e_6}$&$D_{e_1+e_6}$&$D_{e_2+e_4}$&$D_{e_3+e_5}$&$\textrm{Sing}(X)$&$\rho(\widetilde{X}_{min})$&$h^2(X)_{tra}$\\
			\hline
            \hline
        $C-(2,6)$ &	1)&line&line&line &line&line&line&line&line&line&{\rm $\emptyset$}&28&(27,108,27)\\
            \hline
         \end{tabular}
    \end{flushleft}
    \hypertarget{tab: (2,6)}{\caption{Building data 
    		of a SNC $(2,6)$-plane $\pi\colon X\to \mathbb P^2$ satisfying Property~$\bigstar$.}} \label{tab: (2,6)}
    \end{table}

    \begin{table}[H]
    \begin{flushleft} 
    \footnotesize
    \setlength{\tabcolsep}{5pt}
    \renewcommand{\arraystretch}{1.2}
    \begin{tabular}{|c||c|c|c|c||c|c|c|}\hline
			case& No.& $D_{e_1}$&$D_{e_2}$&$D_{e_{1}+e_2}$&$\textrm{Sing}(X)$&$\rho(\widetilde{X}_{min})$&$h^2(X)_{tra}$\\
			\hline
            \hline
        $A-(3,2)$ &	1)&conic&conic&line&{\rm $\emptyset$}&11&(1,8,1)\\
        &	2)&conic&2 lines&line&{\rm $2/3^3$}&13&(1,6,1)\\
        &	3)&2 lines&2 lines&line&{\rm $2/3^6$}&15&(1,4,1)\\
            \hline
         \end{tabular}

    \vspace{0.5em} 
    \begin{tabular}{|c||c|c|c|c||c|c|c|}\hline
			case& No.& $D_{e_1}$&$D_{e_2}$&$D_{2e_2}$&$\textrm{Sing}(X)$&$\rho(\widetilde{X}_{min})$&$h^2(X)_{tra}$\\
			\hline
            \hline
        $B-(3,2)$ &	1)&cubic&line&line&{\rm $1/3^3$}&10&(2,20,2)\\
        &	2)&conic+line&line&line&{\rm $1/3^3$, $2/3^6$}&18&(2,12,2)\\
        &	3)&3 lines&line&line&{\rm $1/3^3$, $2/3^9$}&22&(2,8,2)\\
            \hline
         \end{tabular}

         \vspace{0.5em} 
    \begin{tabular}{|c||c|c|c|c|c||c|c|c|}\hline
			case& No.& $D_{e_1}$&$D_{e_2}$&$D_{e_1+e_2}$&$D_{2(e_1+e_2)}$&$\textrm{Sing}(X)$&$\rho(\widetilde{X}_{min})$&$h^2(X)_{tra}$\\
			\hline
            \hline
        $C-(3,2)$ &	1)&line&line&line&conic&{\rm $1/3^6$}&14&(1,6,1)\\
        &	2)&line&line&line&2 lines&{\rm $1/3^6$, $2/3^3$}&16&(1,4,1)\\
            \hline
         \end{tabular}

    \vspace{0.5em} 
    \begin{tabular}{|c||c|c|c|c|c||c|c|c|}\hline
			case& No.& $D_{e_1}$&$D_{e_2}$&$D_{e_1+2e_2}$&$D_{2(e_1+e_2)}$&$\textrm{Sing}(X)$&$\rho(\widetilde{X}_{min})$&$h^2(X)_{tra}$\\
			\hline
            \hline
        $D-(3,2)$ &	1)&line&line&cubic&line&{\rm $\emptyset$}&1&(3,30,3)\\
        &	2)&line&line&conic+line&line&{\rm $2/3^6$}&13&(3,18,3)\\
         &	3)&line&line&3 lines&line&{\rm $2/3^9$}&19&(3,12,3)\\
            \hline
         \end{tabular}

         \vspace{0.5em} 
    \begin{tabular}{|c||c|c|c|c|c|c||c|c|c|}\hline
			case& No.& $D_{e_1}$&$D_{e_2}$&$D_{2e_2}$&$D_{2e_1+e_2}$&$D_{2(e_1+e_2)}$&$\textrm{Sing}(X)$&$\rho(\widetilde{X}_{min})$&$h^2(X)_{tra}$\\
			\hline
            \hline
        $E-(3,2)$ &	1)&conic&line&line&line&line&{\rm $1/3^3$}&10&(2,12,2)\\
         &	2)&2 lines&line&line&line&line&{\rm $1/3^3$, $2/3^3$}&14&(2,8,2)\\
            \hline
         \end{tabular}
    \end{flushleft}
    \hypertarget{tab: (3,2)}{\caption{Building data 
    		of a SNC $(3,2)$-plane $\pi\colon X\to \mathbb P^2$ satisfying Property~$\bigstar$.}} \label{tab: (3,2)}
    \end{table}

    \begin{table}[H]
    \begin{flushleft} 
    \footnotesize
    \setlength{\tabcolsep}{5pt}
    \renewcommand{\arraystretch}{1.2}
    \begin{tabular}{|c||c|c|c|c|c||c|c|c|}\hline
			case& No.& $D_{e_1}$&$D_{e_2}$&$D_{e_3}$&$D_{2e_{1}+e_2+2e_3}$&$\textrm{Sing}(X)$&$\rho(\widetilde{X}_{min})$&$h^2(X)_{tra}$\\
			\hline
            \hline
        $A-(3,3)$ &	1)&line&conic&line&line&{\rm $\emptyset$}&19&(3,18,3)\\
        &	2)&line&2 lines&line&line&{\rm $2/3^9$}&25&(3,12,3)\\
            \hline
         \end{tabular}

    \vspace{0.5em} 
    \begin{tabular}{|c||c|c|c|c|c|c||c|c|c|}\hline
			case& No.& $D_{e_1}$&$D_{e_2}$&$D_{e_3}$&$D_{e_1+2e_2}$&$D_{e_{1}+2e_3}$&$\textrm{Sing}(X)$&$\rho(\widetilde{X}_{min})$&$h^2(X)_{tra}$\\
			\hline
            \hline
        $B-(3,3)$ &	1)&line&line&line&line&line&{\rm $\emptyset$}&13&(1,4,1)\\
            \hline
         \end{tabular}

    \vspace{0.5em} 
    \begin{tabular}{|c||c|c|c|c|c|c||c|c|c|}\hline
			case& No.& $D_{e_1}$&$D_{e_2}$&$D_{e_3}$&$D_{2e_2}$&$D_{2e_{1}+2e_3}$&$\textrm{Sing}(X)$&$\rho(\widetilde{X}_{min})$&$h^2(X)_{tra}$\\
			\hline
            \hline
        $C-(3,3)$ &	1)&line&line&line&line&line&{\rm $1/3^9$}&22&(2,8,2)\\
            \hline
         \end{tabular}

         \vspace{0.5em} 
    \begin{tabular}{|c||c|c|c|c|c|c|c||c|c|c|}\hline
			case& No.& $D_{e_1}$&$D_{e_2}$&$D_{e_3}$&$D_{2e_1+e_2}$&$D_{2e_{1}+2e_2}$&$D_{e_{1}+2e_2+2e_3}$&$\textrm{Sing}(X)$&$\rho(\widetilde{X}_{min})$&$h^2(X)_{tra}$\\
			\hline
            \hline
        $D-(3,3)$ &	1)&line&line&line&line&line&line&{\rm $\emptyset$}&13&(5,20,5)\\
            \hline
         \end{tabular}
    \end{flushleft}
    \hypertarget{tab: (3,3)}{\caption{Building data 
    		of a SNC $(3,3)$-plane $\pi\colon X\to \mathbb P^2$ satisfying Property~$\bigstar$.}} \label{tab: (3,3)}
    \end{table}
\vspace{-0.5cm}

    \begin{table}[H]
    \begin{flushleft} 
    \footnotesize
    \setlength{\tabcolsep}{5pt}
    \renewcommand{\arraystretch}{1.2}
    \begin{tabular}{|c||c|c|c|c|c|c||c|c|c|}\hline
			case& No.& $D_{e_1}$&$D_{e_2}$&$D_{e_3}$&$D_{e_4}$&$D_{2(e_{1}+e_2+e_3+e_4)}$&$\textrm{Sing}(X)$&$\rho(\widetilde{X}_{min})$&$h^2(X)_{tra}$\\
			\hline
            \hline
        $A-(3,4)$ &	1)&line&line&line&line&line&{\rm $\emptyset$}&31&(5,20,5)\\
            \hline
         \end{tabular}
         \end{flushleft}
    \hypertarget{tab: (3,4)}{\caption{Building data 
    		of a SNC $(3,4)$-plane $\pi\colon X\to \mathbb P^2$ satisfying Property~$\bigstar$.}} \label{tab: (3,4)}
    \end{table}
   \vspace{-1cm}
   
    \begin{table}[H]
    \begin{flushleft} 
    \footnotesize
    \setlength{\tabcolsep}{5pt}
    \renewcommand{\arraystretch}{1.2}
    \begin{tabular}{|c||c|c|c|c|c||c|c|c|}\hline
			case& No.& $D_{e_1}$&$D_{e_2}$&$D_{e_1+e_2}$&$D_{3e_1+3e_2}$&$\textrm{Sing}(X)$&$\rho(\widetilde{X}_{min})$&$h^2(X)_{tra}$\\
			\hline
            \hline
        $A-(5,2)$ &	1)&line&line&line&line&{\rm $2/5^5$}&18&(1,2,1)\\
            \hline
         \end{tabular}

          \vspace{0.5em} 
          \begin{tabular}{|c||c|c|c|c||c|c|c|}\hline
			case& No.& $D_{e_1}$&$D_{e_2}$&$D_{3e_1+4e_2}$&$\textrm{Sing}(X)$&$\rho(\widetilde{X}_{min})$&$h^2(X)_{tra}$\\
			\hline
            \hline
        $B-(5,2)$ &	1)&conic&line&line&{\rm $\emptyset$}&17&(2,12,2)\\
        &	2)&2 lines&line&line&{\rm $4/5^5$}&25&(2,4,2)\\
            \hline
         \end{tabular}
         \end{flushleft}
    \hypertarget{tab: (5,2)}{\caption{Building data 
    		of a SNC $(5,2)$-plane $\pi\colon X\to \mathbb P^2$ satisfying Property~$\bigstar$.}} \label{tab: (5,2)}
    \end{table}
\vspace{-1cm}

    \begin{table}[H]
    \begin{flushleft} 
    \footnotesize
    \setlength{\tabcolsep}{5pt}
    \renewcommand{\arraystretch}{1.2}
    \begin{tabular}{|c||c|c|c|c|c||c|c|c|}\hline
			case& No.& $D_{e_1}$&$D_{e_2}$&$D_{e_3}$&$D_{4(e_1+e_2+e_3)}$&$\textrm{Sing}(X)$&$\rho(\widetilde{X}_{min})$&$h^2(X)_{tra}$\\
			\hline
            \hline
        $A-(5,3)$ &	1)&line&line&line&line&{\rm $\emptyset$}&37&(4,8,4)\\
            \hline
         \end{tabular}
         \end{flushleft}
    \hypertarget{tab: (5,3)}{\caption{Building data 
    		of a SNC $(5,3)$-plane $\pi\colon X\to \mathbb P^2$ satisfying Property~$\bigstar$.}} \label{tab: (5,3)}
    \end{table}
\vspace{-1cm}

    \begin{table}[H]
    \begin{flushleft} 
    \footnotesize
    \setlength{\tabcolsep}{5pt}
    \renewcommand{\arraystretch}{1.2}
    \begin{tabular}{|c||c|c|c|c|c||c|c|c|}\hline
			case& No.& $D_{e_1}$&$D_{e_2}$&$D_{e_1+2e_2}$&$D_{5e_1+4e_2}$&$\textrm{Sing}(X)$&$\rho(\widetilde{X}_{min})$&$h^2(X)_{tra}$\\
			\hline
            \hline
        $A-(7,2)$ &	1)&line&line&line&line&{\rm $\emptyset$}&13&(2,8,2)\\
            \hline
         \end{tabular}

          \vspace{0.5em} 
           \begin{tabular}{|c||c|c|c|c|c||c|c|c|}\hline
			case& No.& $D_{e_1}$&$D_{e_2}$&$D_{e_1+e_2}$&$D_{5e_1+5e_2}$&$\textrm{Sing}(X)$&$\rho(\widetilde{X}_{min})$&$h^2(X)_{tra}$\\
			\hline
            \hline
        $B-(7,2)$ &	1)&line&line&line&line&{\rm $2/7^7$}&26&(3,12,3)\\
            \hline
         \end{tabular}
         \end{flushleft}
    \hypertarget{tab: (7,2)}{\caption{Building data 
    		of a SNC $(7,2)$-plane $\pi\colon X\to \mathbb P^2$ satisfying Property~$\bigstar$.}}\label{tab: (7,2)}
    \end{table}

\bigskip

Federico Fallucca, INdAM, Universit\`a degli Studi di Trento,  Dipartimento di Matematica, I-{38123} Trento, Italy \\
\emph{e-mail} \verb|federico.fallucca@unitn.it|

\smallskip

Matteo Penegini, Universit\`a degli Studi di Pavia,  Dipartimento di Matematica \emph{''Felice Casorati"}, I-27100 Pavia, Italy \\
\emph{e-mail} \verb|matteo.penegini@unipv.it|

\smallskip

Anna Ulivi, Universit\`a degli Studi di Genova, DIMA Dipartimento di Matematica, I-16146 Genova, Italy \\
\emph{e-mail} \verb|anna.ulivi@edu.unige.it|
\end{document}